\documentclass[11pt,reqno]{preprint}
\usepackage[utf8]{inputenc}
\usepackage[full]{textcomp}
\usepackage[osf]{newtxtext}
\usepackage{comment}

\usepackage{amssymb}
\usepackage{mathtools}
\usepackage{hyperref}
\usepackage{breakurl}
\usepackage{mhenvs}
\usepackage{mhequ}
\usepackage{mhsymb}
\usepackage{booktabs}
\usepackage{tikz}
\usepackage{mathrsfs}
\usepackage{longtable}
\usepackage{halloweenmath}
\usepackage{scalerel}
\usepackage{array}
\usepackage{cprotect}

\usepackage{microtype}
\usepackage{comment}
\usepackage{wasysym}
\usepackage{centernot}
\usepackage{enumitem}

\usepackage{soul}

\usetikzlibrary{shapes.misc}
\usetikzlibrary{shapes.symbols}
\usetikzlibrary{shapes.geometric}
\usetikzlibrary{decorations}
\usetikzlibrary{decorations.markings}
\usetikzlibrary{snakes} 

\usetikzlibrary{calc, math}
\usetikzlibrary{external}

\let\oldskull\skull
\def\skull{\mathord{\oldskull}}
	
\DeclareMathAlphabet{\mathbbm}{U}{bbm}{m}{n}

\DeclareFontFamily{U}{BOONDOX-calo}{\skewchar\font=45 }
\DeclareFontShape{U}{BOONDOX-calo}{m}{n}{
  <-> s*[1.05] BOONDOX-r-calo}{}
\DeclareFontShape{U}{BOONDOX-calo}{b}{n}{
  <-> s*[1.05] BOONDOX-b-calo}{}
\DeclareMathAlphabet{\mcb}{U}{BOONDOX-calo}{m}{n}
\SetMathAlphabet{\mcb}{bold}{U}{BOONDOX-calo}{b}{n}

\setlist{noitemsep,topsep=4pt}

\makeatletter
\def\DeclareSymbol#1#2#3{%
	\expandafter\gdef\csname MH@symb@#1\endcsname{\tikzsetnextfilename{symbol#1}%
	\tikz[baseline=#2,scale=0.15,draw=symbols,line join=round,line cap=round]{#3}}%
	\expandafter\gdef\csname MH@symb@#1s\endcsname{\scalebox{0.75}{\tikzsetnextfilename{symbol#1}%
	\tikz[baseline=#2,scale=0.15,draw=symbols,line join=round,line cap=round]{#3}}}%
	\expandafter\gdef\csname MH@symb@#1ss\endcsname{\scalebox{0.65}{\tikzsetnextfilename{symbol#1}%
	\tikz[baseline=#2,scale=0.15,draw=symbols,line join=round,line cap=round]{#3}}}%
	}
\def\<#1>{\ifthenelse{\boolean{mmode}}{\mathchoice{\csname MH@symb@#1\endcsname}{\csname MH@symb@#1\endcsname}{\csname MH@symb@#1s\endcsname}{\csname MH@symb@#1ss\endcsname}}{\csname MH@symb@#1\endcsname}}
\makeatother

\DeclareSymbol{Xi22}{0.5}{\draw (0,0) node[xi] {} -- (-1,1) node[not] {} -- (0,2) node[xi] {};}

\DeclareSymbol{Xi2}{-2}{\draw (0,-0.25) node[xi] {} -- (-1,1) node[xi] {};}
\DeclareSymbol{Xi3}{0}{\draw (0,0) node[xi] {} -- (-1,1) node[xi] {} -- (0,2) node[xi] {};}
\DeclareSymbol{Xi4}{2}{\draw (0,0) node[not] {} -- (-1,1) node[xi] {} -- (0,2) node[xi] {} -- (-1,3) node[xi] {};}
\DeclareSymbol{Xi2X}{-2}{\draw (0,-0.25) node[xi] {} -- (-1,1) node[xix] {};}
\DeclareSymbol{XXi2}{-2}{\draw (0,-0.25) node[xix] {} -- (-1,1) node[xi] {};}

\DeclareSymbol{IXi^2asym}{-1}{\draw (-1,1) node[xired] {} -- (0,-.4);
\draw (1,1) node[xigreen] {} -- (0,-.4);}

\DeclareSymbol{IXi^2asym2}{-1}{\draw (-1,1) node[xigreen] {} -- (0,-.4);
\draw (1,1) node[xired] {} -- (0,-.4);}

\DeclareSymbol{IXi^2asym3}{-1}{\draw (-1,1) node[xired] {} -- (0,-.4);
\draw (1,1) node[xigreen] {} -- (0,-.4);}

\DeclareSymbol{IXi^2}{-1}{\draw (-1,1) node[xi] {} -- (0,-.4);
\draw (1,1) node[xi] {} -- (0,-.4);}

\DeclareSymbol{rIXi^2}{-1}{\draw (-1,1) node[xired] {} -- (0,-.4);
\draw (1,1) node[xired] {} -- (0,-.4);}

\DeclareSymbol{IXi^2green}{-1}{\draw (-1,1) node[xigreen] {} -- (0,-.4);
\draw (1,1) node[xigreen] {} -- (0,-.4);}

\DeclareSymbol{IXi^2_typed}{-1}{\draw (-1,1) node[xiorange] {} -- (0,-.4);
\draw (1,1) node[xigreen] {} -- (0,-.4);}

\DeclareSymbol{IXi^2green*}{-1}{\draw (-1,1) node[xigreen1] {} -- (0,-.4);
\draw (1,1) node[xigreen1] {} -- (0,-.4);}

\DeclareSymbol{IXiI'Xi_notriangle}{-1}{\draw (-1,1) node[xi] {} -- (0,0);
\draw (1,1)[kernels2] node[xi] {} -- (0,0) {};}

\DeclareSymbol{IXiI'Xi}{-1}{\draw (-1,1) node[xi] {} -- (0,0);
\draw (1,1)[kernels2] node[xi] {} -- (0,0) node[not] {};}

\DeclareSymbol{IXiI'Xi_typed1}{-1}{\draw (-1,1) node[xigreen] {} -- (0,0);
\draw (1,1)[kernels2,greennode] node[xired] {} -- (0,0) node[not] {};}

\DeclareSymbol{IXiI'Xi_typed2}{-1}{\draw (-1,1) node[xigreen] {} -- (0,0);
\draw (1,1)[kernels2,rednode] node[xigreen] {} -- (0,0) node[not] {};}

\DeclareSymbol{IXiI'Xi_typed2_notriangle}{-1}{\draw (-1,1) node[xigreen] {} -- (0,0);
\draw (1,1)[kernels2,rednode] node[xigreen] {} -- (0,0) {};}

\DeclareSymbol{IXi^3}{-1}{\draw (-1.3,1) node[xi] {} -- (0,0);
\draw (1.3,1) node[xi] {} -- (0,0);
\draw (0,1.5) node[xi] {} -- (0,0) node[not] {};}

\DeclareSymbol{IXi^3_notriangle}{0}{\draw (-1.3,1) node[xi] {} -- (0,0);
\draw (1.3,1) node[xi] {} -- (0,0);
\draw (0,1.5) node[xi] {} -- (0,0) {};}

\DeclareSymbol{IXi^3_typed}{-1}{\draw (-1.3,1) node[xigreen] {} -- (0,0);
\draw (1.3,1) node[xired] {} -- (0,0);
\draw (0,1) node[xigreen] {} -- (0,0) node[not] {};}

\DeclareSymbol{IXi^3-typed112}{-1}{\draw (-1.3,1) node[xi] {} -- (0,0);
\draw (1.3,1) node[xigreen] {} -- (0,0);
\draw (0,1) node[xi] {} -- (0,0) node[not] {};}

\DeclareSymbol{I'Xi}{-1}{
\draw[kernels2] (1,1) node[xi] {} -- (0,0) node[not] {};}

\DeclareSymbol{I'Xi_notriangle}{-1}{
\draw[kernels2] (1,1) node[xi] {} -- (0,0) {};}

\DeclareSymbol{I'Xigreen-red}{-1}{
\draw[kernels2,rednode] (0,1) node[xi,fill=greennode] {} -- (0,-.4);}

\DeclareSymbol{IXi}{-1}{
\draw (0,1) node[xi] {} -- (0,-.4);}

\DeclareSymbol{IXigreen}{-1}{
\draw (0,1) node[xigreen] {} -- (0,-.4);}

\DeclareSymbol{IXired}{-1}{
\draw (0,1) node[xired] {} -- (0,-.4);}

\DeclareSymbol{green}{-1}{
\draw[greennode,very thick,line cap=round,scale=1.3] (0.05,0.4) -- (0.2,0.55) -- (0,0) -- (0.6,0.6) -- (0.4,0.05) -- (0.55,0.2);}
\DeclareSymbol{red}{-1}{
\draw[rednode,very thick,line cap=round,scale=1.3] (0.05,0.4) -- (0.2,0.55) -- (0,0) -- (0.6,0.6) -- (0.4,0.05) -- (0.55,0.2);}
\DeclareSymbol{lblue}{-1}{
\draw[lbluenode,very thick,line cap=round,scale=1.3] (0.05,0.4) -- (0.2,0.55) -- (0,0) -- (0.6,0.6) -- (0.4,0.05) -- (0.55,0.2);}
\DeclareSymbol{dblue}{-1}{
\draw[dbluenode,very thick,line cap=round,scale=1.3] (0.05,0.4) -- (0.2,0.55) -- (0,0) -- (0.6,0.6) -- (0.4,0.05) -- (0.55,0.2);}
\DeclareSymbol{orange}{-1}{
\draw[orangenode,very thick,line cap=round,scale=1.3] (0.05,0.4) -- (0.2,0.55) -- (0,0) -- (0.6,0.6) -- (0.4,0.05) -- (0.55,0.2);}

\DeclareSymbol{I'XXi}{-1}{
\draw[kernels2] (1,1) node[xi] {} -- (0,0) node[not] {};
\draw (1,1) node[xix] {};}

\DeclareSymbol{I'XXi_notriangle}{-1}{
\draw[kernels2] (1,1) node[xi] {} -- (0,0) {};
\draw (1,1) node[xix] {};}

\DeclareSymbol{IXiI'[IXiI'Xi]}{-1}{\draw[kernels2] (2,2) node[xi] {} -- (1,1);
\draw (0,2) node[xi] {} -- (1,1) node[not] {};
\draw[kernels2] (1,1) -- (0,0) node[not] {};
\draw (-1,1) node[xi] {} -- (0,0);}

\DeclareSymbol{IXiI'[IXiI'Xi]_notriangle}{1}{\draw[kernels2] (2,2) node[xi] {} -- (1,1);
\draw (0,2) node[xi] {} -- (1,1) {};
\draw[kernels2] (1,1) -- (0,0) {};
\draw (-1,1) node[xi] {} -- (0,0);}

\DeclareSymbol{I'[IXiI'Xi]}{-1}{\draw[kernels2] (2,2) node[xi] {} -- (1,1);
\draw (0,2) node[xi] {} -- (1,1) node[not] {};
\draw[kernels2] (1,1) -- (0,0) node[not] {};}

\DeclareSymbol{I'[IXiI'Xi]_notriangle}{1}{\draw[kernels2] (2,2) node[xi] {} -- (1,1);
\draw (0,2) node[xi] {} -- (1,1) {};
\draw[kernels2] (1,1) -- (0,0) {};}

\DeclareSymbol{I[IXi^2]IXi-typed112}{-1}{\draw (-2,2) node[xi] {} -- (-1,1);
\draw (0,2) node[xi] {} -- (-1,1) node[not] {};
\draw (-1,1) -- (0,0) node[not] {};
\draw (1,1) node[xigreen] {} -- (0,0);}

\DeclareSymbol{I[IXi^2]IXi-typed211}{-1}{\draw (-2,2) node[xigreen] {} -- (-1,1);
\draw (0,2) node[xi] {} -- (-1,1) node[not] {};
\draw (-1,1) -- (0,0) node[not] {};
\draw (1,1) node[xi] {} -- (0,0);}

\DeclareSymbol{I[IXiI'Xi]I'Xi}{-1}{\draw (-2,2) node[xi] {} -- (-1,1);
\draw[kernels2] (0,2) node[xi] {} -- (-1,1) node[not] {};
\draw (-1,1) -- (0,0) node[not] {};
\draw[kernels2] (1,1) node[xi] {} -- (0,0);}

\DeclareSymbol{I[IXiI'Xi]I'Xi_notriangle}{1}{\draw (-2,2) node[xi] {} -- (-1,1);
\draw[kernels2] (0,2) node[xi] {} -- (-1,1) {};
\draw (-1,1) -- (0,0) {};
\draw[kernels2] (1,1) node[xi] {} -- (0,0);}

\DeclareSymbol{I[I[Xi]]Xi_notriangle}{1}{
\draw (0,2) node[xi] {} -- (-1,1) {};
\draw[snake=snake , segment length=1pt, segment amplitude=1pt] (-1,1) -- (0,0) node[xi] {};}

\DeclareSymbol{I[I[Xi]]Xi_typed}{1}{
\draw (0,2) node[xigreen] {} -- (-1,1) {};
\draw[snake=snake , segment length=1pt, segment amplitude=1pt] (-1,1) -- (0,0) node[xigreen] {};}

\DeclareSymbol{I[I[Xi]]Xi_typed*}{1}{
\draw (0,2) node[xigreen1] {} -- (-1,1) {};
\draw[snake=snake , segment length=1pt, segment amplitude=1pt] (-1,1) -- (0,0) node[xigreen1] {};}

\DeclareSymbol{I[I'Xi]I'Xi}{1}{\draw[kernels2] (0,2) node[xi] {} -- (-1,1);
\draw (-1,1) node[fill=rednode,not] {} -- (0,0);
\draw[kernels2] (1,1) node[xi] {} -- (0,0);}

\DeclareSymbol{I[I'Xi]I'Xi_notriangle}{1}{\draw[kernels2] (0,2) node[xi] {} -- (-1,1);
\draw (-1,1) {} -- (0,0);
\draw[kernels2] (1,1) node[xi] {} -- (0,0);}

\DeclareSymbol{I[I'Xi]I'Xi_typed}{1}{\draw[kernels2,rednode] (0,2) node[xigreen] {} -- (-1,1);
\draw (-1,1) node[notorange] {} -- (0,0);
\draw[kernels2,rednode] (1,1) node[xigreen] {} -- (0,0);}

\DeclareSymbol{I[I'Xi]I'Xi_typed*}{1}{\draw[kernels2,rednode] (0,2) node[xigreen1] {} -- (-1,1);
\draw (-1,1) node[notorange] {} -- (0,0);
\draw[kernels2,rednode] (1,1) node[xigreen1] {} -- (0,0);}

\DeclareSymbol{I[I'Xi]I'Xi_typed-h}{1}{\draw[kernels2,greennode] (0,2) node[xigreen] {} -- (-1,1);
\draw[snake=snake , segment length=3pt, segment amplitude=1pt] (-1,1) node[notgreen] {} -- (0,0);
\draw[kernels2,greennode] (1,1) node[xigreen] {} -- (0,0);}

\DeclareSymbol{I[I'Xi]I'Xi_typed-h*}{1}{\draw[kernels2,greennode] (0,2) node[xigreen1] {} -- (-1,1);
\draw[snake=snake , segment length=3pt, segment amplitude=1pt] (-1,1) node[notgreen] {} -- (0,0);
\draw[kernels2,greennode] (1,1) node[xigreen1] {} -- (0,0);}

\DeclareSymbol{I[I'Xi]I'Xib}{1}{\draw[kernels2] (0,2) node[xi] {} -- (1,1);
\draw (1,1) node[not] {} -- (0,0);
\draw[kernels2] (-1,1) node[xi] {} -- (0,0);}

\DeclareSymbol{I[I'Xi]I'Xib_notriangle}{1}{\draw[kernels2] (0,2) node[xi] {} -- (1,1);
\draw (1,1) {} -- (0,0);
\draw[kernels2] (-1,1) node[xi] {} -- (0,0);}

\DeclareSymbol{IXiI'[I'Xi]}{1}{\draw[kernels2] (0,2) node[xi] {} -- (1,1);
\draw[kernels2] (1,1) node[not] {} -- (0,0);
\draw (-1,1) node[xi] {} -- (0,0);}

\DeclareSymbol{IXiI'[I'Xi]_notriangle}{1}{\draw[kernels2] (0,2) node[xi] {} -- (1,1);
\draw[kernels2] (1,1)  {} -- (0,0);
\draw (-1,1) node[xi] {} -- (0,0);}

\DeclareSymbol{IXiI'[I'Xi]_typed}{1}{\draw[kernels2,rednode] (0,2) node[xigreen] {} -- (1,1);
\draw[kernels2,rednode] (1,1) node[notorange] {} -- (0,0);
\draw (-1,1) node[xigreen] {} -- (0,0);}

\DeclareSymbol{IXiI'[I'Xi]_typed-h}{1}{\draw[kernels2,greennode] (0,2) node[xigreen] {} -- (1,1);
\draw[kernels2,greennode,snake=snake , segment length=3pt, segment amplitude=1pt] (1,1) node[notgreen] {} -- (0,0);
\draw (-1,1) node[xigreen] {} -- (0,0);}

\DeclareSymbol{IXiI'[I'Xi]_typed*}{1}{\draw[kernels2,rednode] (0,2) node[xigreen1] {} -- (1,1);
\draw[kernels2,rednode] (1,1) node[notorange] {} -- (0,0);
\draw (-1,1) node[xigreen1] {} -- (0,0);}

\DeclareSymbol{IXiI'[I'Xi]_typed-h*}{1}{\draw[kernels2,greennode] (0,2) node[xigreen1] {} -- (1,1);
\draw[kernels2,greennode,snake=snake , segment length=3pt, segment amplitude=1pt] (1,1) node[notgreen] {} -- (0,0);
\draw (-1,1) node[xigreen1] {} -- (0,0);}

\DeclareSymbol{I[I'Xi]}{-1}{\draw[kernels2] (0,2) node[xi] {} -- (1,1);
\draw (1,1) node[not] {} -- (0,0) node[not] {};}

\DeclareSymbol{I[I'Xi]_typed}{1}{\draw[kernels2,rednode] (0,2) node[xigreen] {} -- (-1,1);
\draw (-1,1) node[notorange] {} -- (0,0);}

\DeclareSymbol{I'[I'Xi]}{-1}{\draw[kernels2] (0,2) node[xi] {} -- (1,1);
\draw[kernels2] (1,1) node[not] {} -- (0,0) node[not] {};}

\DeclareSymbol{I'[I'Xi]_notriangle}{-1}{\draw[kernels2] (0,2) node[xi] {} -- (1,1);
\draw[kernels2] (1,1) {} -- (0,0) {};}

\DeclareSymbol{I'[I'Xi]_typed}{1}{\draw[kernels2,rednode] (0,2) node[xigreen] {} -- (1,1);
\draw[kernels2,rednode] (1,1) node[notorange] {} -- (0,0)  {};}

\DeclareSymbol{IXiI'XXi}{-1}{\draw (-1,1) node[xi] {} -- (0,0);
\draw[kernels2] (1,1) node[xi] {} -- (0,0) node[not] {};
\draw (1,1) node[xix] {};}

\DeclareSymbol{IXiI'XXi_notriangle}{-1}{\draw (-1,1) node[xi] {} -- (0,0);
\draw[kernels2] (1,1) node[xi] {} -- (0,0) {};
\draw (1,1) node[xix] {};}

\DeclareSymbol{IXiI'XXi_typed}{-1}{\draw (-1,1) node[xigreen] {} -- (0,-0.3);
\draw[kernels2,rednode] (1,1)  -- (0,-0.3) {};
\draw (1,1) node[xix-green-red] {};
}

\DeclareSymbol{IXiI'XXi_typed*}{-1}{\draw (-1,1) node[xigreen1] {} -- (0,-0.3);
\draw[kernels2,rednode] (1,1)  -- (0,-0.3) {};
\draw (1,1) node[xix-green-red1] {};
}

\DeclareSymbol{IXXiI'Xi}{-1}{\draw (-1,1) node[xix] {} -- (0,0);
\draw[kernels2] (1,1) node[xi] {} -- (0,0) node[not] {};}

\DeclareSymbol{IXXiI'Xi_notriangle}{-1}{\draw (-1,1) node[xix] {} -- (0,0);
\draw[kernels2] (1,1) node[xi] {} -- (0,0) {};}

\DeclareSymbol{IXXiI'Xi_typed}{-1}{\draw (-1,1) node[xix-green-red] {} -- (0,-0.3);
\draw[kernels2,rednode] (1,1) node[xigreen] {} -- (0,-0.3) {};}

\DeclareSymbol{IXXiI'Xi_typed*}{-1}{\draw (-1,1) node[xix-green-red1] {} -- (0,-0.3);
\draw[kernels2,rednode] (1,1) node[xigreen1] {} -- (0,-0.3) {};}

\DeclareSymbol{XiX}{-2.8}{\node[xibx] {};}
\DeclareSymbol{XXi}{-2.8}{\node[xibx] {};}
\DeclareSymbol{tauX}{-2.8}{ \draw[kernels2] (0,0) node[xibx] {};}
\DeclareSymbol{Xi}{-2.8}{\node[xib] {};}
\DeclareSymbol{Xigreen}{-2.8}{\node[xigreen] {};}
\DeclareSymbol{Xired}{-2.8}{\node[xired] {};}
\DeclareSymbol{not}{-2.8}{\node[not,minimum size=1.6mm] {};}
\DeclareSymbol{XiR}{-2.8}{\node[xie,fill=rednode] {};}
\DeclareSymbol{XiY}{-2.8}{\node[xie] {};}
\DeclareSymbol{XiZ}{-2.8}{\node[xid] {};}
\DeclareSymbol{IXiX}{0}{\draw (0,-0.25) node[not] {} -- (0,1.5) node[xix] {};}

\makeatletter 
\newcommand*{\bigcdot}{}
\DeclareRobustCommand*{\bigcdot}{%
  \mathbin{\mathpalette\bigcdot@{}}%
}
\newcommand*{\bigcdot@scalefactor}{.5}
\newcommand*{\bigcdot@widthfactor}{1.15}
\newcommand*{\bigcdot@}[2]{%
  \sbox0{$#1\vcenter{}$}
  \sbox2{$#1\cdot\m@th$}%
  \hbox to \bigcdot@widthfactor\wd2{%
    \hfil
    \raise\ht0\hbox{%
      \scalebox{\bigcdot@scalefactor}{%
        \lower\ht0\hbox{$#1\bullet\m@th$}%
      }%
    }%
    \hfil
  }%
}
\makeatother

\def\act{\bigcdot}

\newcommand{\cut}{\mathfrak{C}}

\def\symset{\mcb{s}}
\def\symcol{\mcb{S}}

\newcommand{\mrd}{\mathop{}\!\mathrm{d}}
\newcommand{\ssep}{\mid}

\newcommand{\subforest}{\hookrightarrow}
\newcommand{\subroot}{%
  \mathrel{\ThisStyle{\vbox{\offinterlineskip\ialign{%
    \hfil##\hfil\cr
    $\scriptscriptstyle r$\cr
    $\SavedStyle\hookrightarrow$\cr
}}}}}

\def\wnorm#1{\lfloor \hspace{-0.29em} \rceil #1 \lfloor \hspace{-0.29em} \rceil}

\newcommand{\mcO}{\mathcal{O}}

\newcommand{\mcH}{\mathcal{H}}

\newcommand{\mcA}{\mathcal{A}}

\newcommand{\mcR}{\mathcal{R}}
\newcommand{\mcC}{\mathcal{C}}
\newcommand{\mcB}{\mathcal{B}}

\newcommand{\mcU}{\mathcal{U}}
\newcommand{\mcL}{\mathcal{L}}

\newcommand{\mcF}{\mathcal{F}}

\newcommand{\mcD}{\mathcal{D}}

\newcommand{\mcP}{\mathcal{P}}
\newcommand{\mcG}{\mathcal{G}}
\newcommand{\mcX}{\mathcal{X}}

\newcommand{\mbA}{\mathbf{A}}

\newcommand{\mbF}{\mathbf{F}}

\newcommand{\mbX}{\mathbf{X}}

\newcommand{\T}{\mathbf{T}}

\def\Lab{\mathfrak{L}}

\def\Deltam{\Delta^{\!-}}

\def\Deltap{\Delta^{\!+}}

\def\Deltaex{\Delta^{\!\mathrm{ex}}}
\def\Deltacut{\Delta^{\!\mathrm{cut}}}

\def\bUpsilon{\boldsymbol{\Upsilon}}
\def\bbUpsilon{\boldsymbol{\bar{\Upsilon}}}

\def\${|\!|\!|}

\def\id{\mathrm{id}}

\def\var#1{#1\textnormal{-var}}
\def\Hol#1{#1\textnormal{-H{\"o}l}}
\def\gr#1{#1\textnormal{-gr}}
\let\vaccent\v
\def\v#1{#1\textnormal{-vee}}

\def\tri#1{#1\textnormal{-tr}}
\def\symm#1{#1\textnormal{-sym}}
\def\Area{\textnormal{Area}}
\def\diam{\textnormal{diam}}

\def\hol{\textnormal{hol}}

\def\expan{\mcb{H}}  

\def\scal#1{{\langle#1\rangle}}

\def\CS{\mathcal{S}}
\def\CR{\mathcal{R}}
\def\CF{\mcb{F}}
\def\CT{\mcb{T}}
\def\CU{\mcb{U}}
\def\bone{\mathbf{1}}
\def\ST{\boldsymbol{\mcb T}}
\def\SU{\boldsymbol{\mcb U}}
\def\SF{\boldsymbol{\mcb F}}

\def\STp{\boldsymbol{\mcb T}_{\!\!+}}
\def\STm{\boldsymbol{\mcb T}_{\!\!-}}

\def\SFm{\boldsymbol{\mcb F}_{\!\!-}}

\let\tto\rightsquigarrow

\def\sol{{\mathop{\mathrm{sol}}}}

\newcommand{\mfu}{\mathfrak{u}}
\newcommand{\mfU}{\mathfrak{U}}

\newcommand{\mfF}{\mathfrak{F}}

\newcommand{\mfT}{\mathfrak{T}}
\newcommand{\mfn}{\mathfrak{n}}

\newcommand{\mfe}{\mathfrak{e}}

\newcommand{\mfL}{\mathfrak{L}}

\newcommand{\mfa}{\mathfrak{a}}
\newcommand{\mfc}{\mathfrak{c}}

\newcommand{\mfR}{\mathfrak{R}}

\newcommand{\mft}{\mathfrak{t}}
\newcommand{\mfm}{\mathfrak{m}}

\newcommand{\mfp}{\mathfrak{p}}

\newcommand{\mfl}{\mathfrak{l}}
\newcommand{\mfh}{\mathfrak{h}}
\newcommand{\mfq}{\mathfrak{q}}

\newcommand{\mfG}{\mathfrak{G}}
\newcommand{\mfO}{\mathfrak{O}}

\newcommand{\mfg}{\mathfrak{g}}

\def\cC{\mathscr{C}}

\def\cD{\mathscr{D}}

\def\can{\textnormal{\scriptsize can}}
\def\sym{\textnormal{\small \textsc{sym}}}
\def\gsym{\textnormal{\small \textsc{gsym}}}
\def\BPHZ{\textnormal{\small \textsc{bphz}}}

\def\YM{\textnormal{\small \textsc{ym}}}

\def\init{\textnormal{init}}

\def\bXi{\boldsymbol{\Xi}}

\newcommand{\ex}{\mathrm{ex}}

\newcommand{\ad}{\mathrm{ad}}
\newcommand{\Ad}{\mathrm{Ad}}

\def\combplus[#1,#2,#3,#4]{\binom{#1\ {\scriptstyle #4} }{#2\ #3}}

\def\singlescalegenvert[#1,#2]{\hat{H}^{#2}_{#1}}
\def\multiscalegenvert[#1,#2]{H^{#2}_{#1}}

\def\moll{\chi}

\usepackage{stmaryrd}

\def\PP{\mathbb{P}}

\def\nr[#1]{\tilde{N}[#1]} 
\def\inn[#1]{\mathring{N}[#1]}
\def\nrinn[#1]{\hat{N}_{#1}} 
\def\nrmod[#1,#2]{\tilde{N}_{#1}(#2)}
\def\nrinnmod[#1,#2]{\hat{N}_{#1}(#2)}

\def\ident[#1]{\underline{#1}}

\def\mylink#1#2{\mathrel{\vbox{\offinterlineskip\ialign{%
    \hfil##\hfil\cr
    $\scriptscriptstyle#1$\cr
    \noalign{\kern0.1ex}
    $#2$\cr
}}}}
\def\mysublink[#1]#2#3{\mathrel{\vbox{\offinterlineskip\ialign{%
    \hfil##\hfil\cr
    $\scriptscriptstyle#2$\cr
    \noalign{\kern0.1ex}
    $#3$\cr
    \noalign{\kern-0.2ex}
    \smash{\raisebox{-\height}{\hbox{$\scriptscriptstyle #1$}}}\cr
    \noalign{\kern0.2ex}
}}}}

\def\smooth{\mcb{C}^\infty}
\def\fon[#1]{\cC_{#1}}

\def\mincompproj[#1]{\mfp_{#1}}

\def\Proj_#1{\mathop{\mathrm{Proj}_{#1}}}

\def\negrenorm[#1]{\mfR_{#1}}
\def\topnegrenorm[#1]{\overline{\mfR}_{#1}}

\def\quotedge[#1]{E^{q}_{#1}}

\def\posrenorm[#1]{\mcC_{#1}}
\def\topposrenorm[#1]{\overline{\mcC_{#1}}}
\def\cutsmod[#1]{\mathbb{C}_{+,#1}}

\def\fullcutsmod[#1]{\cut_{#1}}

\def\emptyset{{\centernot\Circle}}

\colorlet{symbols}{blue!30!black!50}
\colorlet{testcolor}{green!60!black}
\colorlet{darkblue}{blue!60!black}
\colorlet{darkgreen}{green!60!black}
\definecolor{darkergreen}{rgb}{0.0, 0.5, 0.0}

\definecolor{purple}{rgb}{0.55,0.05,0.8}

\colorlet{redkernel}{red!80}

\def\symbol#1{\textcolor{symbols}{#1}}
\def\pr{\textcolor{purple}}

\def\1{\mathbf{\symbol{1}}}

\def\hotimes{\mathbin{\hat\otimes}}

\newcommand*{\mathcolor}{}
\def\mathcolor#1#{\mathcoloraux{#1}}
\newcommand*{\mathcoloraux}[3]{%
  \protect\leavevmode
  \begingroup
    \color#1{#2}#3%
  \endgroup
}

\makeatletter
\pgfdeclareshape{crosscircle}
{
  \inheritsavedanchors[from=circle] 
  \inheritanchorborder[from=circle]
  \inheritanchor[from=circle]{north}
  \inheritanchor[from=circle]{north west}
  \inheritanchor[from=circle]{north east}
  \inheritanchor[from=circle]{center}
  \inheritanchor[from=circle]{west}
  \inheritanchor[from=circle]{east}
  \inheritanchor[from=circle]{mid}
  \inheritanchor[from=circle]{mid west}
  \inheritanchor[from=circle]{mid east}
  \inheritanchor[from=circle]{base}
  \inheritanchor[from=circle]{base west}
  \inheritanchor[from=circle]{base east}
  \inheritanchor[from=circle]{south}
  \inheritanchor[from=circle]{south west}
  \inheritanchor[from=circle]{south east}
  \inheritbackgroundpath[from=circle]
  \foregroundpath{
    \centerpoint%
    \pgf@xc=\pgf@x%
    \pgf@yc=\pgf@y%
    \pgfutil@tempdima=\radius%
    \pgfmathsetlength{\pgf@xb}{\pgfkeysvalueof{/pgf/outer xsep}}%
    \pgfmathsetlength{\pgf@yb}{\pgfkeysvalueof{/pgf/outer ysep}}%
    \ifdim\pgf@xb<\pgf@yb%
      \advance\pgfutil@tempdima by-\pgf@yb%
    \else%
      \advance\pgfutil@tempdima by-\pgf@xb%
    \fi%
    \pgfpathmoveto{\pgfpointadd{\pgfqpoint{\pgf@xc}{\pgf@yc}}{\pgfqpoint{-0.707107\pgfutil@tempdima}{-0.707107\pgfutil@tempdima}}}
    \pgfpathlineto{\pgfpointadd{\pgfqpoint{\pgf@xc}{\pgf@yc}}{\pgfqpoint{0.707107\pgfutil@tempdima}{0.707107\pgfutil@tempdima}}}
    \pgfpathmoveto{\pgfpointadd{\pgfqpoint{\pgf@xc}{\pgf@yc}}{\pgfqpoint{-0.707107\pgfutil@tempdima}{0.707107\pgfutil@tempdima}}}
    \pgfpathlineto{\pgfpointadd{\pgfqpoint{\pgf@xc}{\pgf@yc}}{\pgfqpoint{0.707107\pgfutil@tempdima}{-0.707107\pgfutil@tempdima}}}
  }
}
\makeatother

\colorlet{greennode}{green!50!black}
\colorlet{rednode}{red!50!black}
\colorlet{lbluenode}{blue!25}
\colorlet{dbluenode}{blue}
\colorlet{orangenode}{orange}

\definecolor{connection}{rgb}{0.7,0.1,0.1}

\tikzset{
root/.style={circle,fill=black!50,inner sep=0pt, minimum size=3mm},
        var/.style={circle,fill=black!10,draw=black,inner sep=0pt, minimum size=3mm},
        kernel/.style={semithick,shorten >=2pt,shorten <=2pt},
        kernel1/.style={thick},
        kernels/.style={snake=zigzag,shorten >=2pt,shorten <=2pt,segment amplitude=1pt,segment length=4pt,line before snake=2pt,line after snake=5pt,},
        rho/.style={densely dashed,semithick,shorten >=2pt,shorten <=2pt},
           testfcn/.style={dotted,semithick,shorten >=2pt,shorten <=2pt},
           tau/.style={circle,inner sep=1pt,draw=black,fill=white,text=black,thin},
        renorm/.style={shape=circle,fill=white,inner sep=1pt},
        labl/.style={shape=rectangle,fill=white,inner sep=1pt},
        xi/.style={very thin,circle,fill=lbluenode,draw=symbols,inner sep=0pt,minimum size=1.2mm},
        xigreen/.style={very thin,circle,fill=greennode,draw=symbols,inner sep=0pt,minimum size=1.2mm},
        xigreen1/.style={very thin,rectangle,fill=greennode,draw=symbols,inner sep=0pt,minimum size=1.2mm},
        xired/.style={very thin,circle,fill=rednode,draw=symbols,inner sep=0pt,minimum size=1.2mm},
        xilblue/.style={very thin,circle,fill=lbluenode,draw=symbols,inner sep=0pt,minimum size=1.2mm},
        xiorange/.style={very thin,circle,fill=orangenode,draw=symbols,inner sep=0pt,minimum size=1.2mm},
        xix/.style={crosscircle,fill=lbluenode,draw=symbols,inner sep=0pt,minimum size=1.2mm},
xix-green-red/.style={circle, fill=greennode!70!white,draw=rednode,inner sep=0pt,minimum size=1.6mm,append after command={node [inner sep=0pt,minimum size=0.8mm,thick, draw = rednode, cross out]{}}},
xix-green-red1/.style={rectangle, fill=greennode!70!white,draw=rednode,inner sep=0pt,minimum size=1.5mm,append after command={node [inner sep=0pt,minimum size=1mm,thick, draw = rednode, cross out]{}}},
	xib/.style={very thin,circle,fill=lbluenode,draw=symbols,inner sep=0pt,minimum size=1.6mm},
	xie/.style={very thin,circle,fill=greennode,draw=symbols,inner sep=0pt,minimum size=1.6mm},
	xid/.style={very thin,circle,fill=lbluenode,draw=symbols,inner sep=0pt,minimum size=1.6mm},
	xibx/.style={crosscircle,fill=lbluenode,draw=symbols,inner sep=0pt,minimum size=1.6mm},
	kernels2/.style={ultra thick,draw=symbols,segment length=12pt},
	not/.style={thin,regular polygon, regular polygon sides=3,draw=connection,fill=connection,inner sep=0pt,minimum size=1.2mm},
	dot/.style={thin,circle,fill=black,draw=black,inner sep=0pt,minimum size=0.5mm},
	notlblue/.style={thin,regular polygon, regular polygon sides=3,draw=lbluenode,fill=lbluenode,inner sep=0pt,minimum size=1.2mm},
	notorange/.style={thin,regular polygon, regular polygon sides=3,draw=orangenode,fill=orangenode,inner sep=0pt,minimum size=1.2mm},
	notgreen/.style={thin,regular polygon, regular polygon sides=3,draw=greennode,fill=greennode,inner sep=0pt,minimum size=1.2mm},
	>=stealth,
  }

\newtheorem{assumption}[lemma]{Assumption}

\colorlet{darkblue}{blue!90!black}
\colorlet{darkred}{red!90!black}
\colorlet{darkgreen}{green!70!black}
\let\comm\comment

\def\s{\mathfrak{s}}

\newcommand{\e}{\varepsilon}

\def\K{\mathfrak{K}}

\def\bbar#1{\bar{\bar #1}}
\def\${|\!|\!|}

\def\?{{\color{red}?}}

\def\id{\mathrm{id}}

\def\restr{\mathbin{\upharpoonright}}

\def\Id{\mathrm{id}}
\def\proj{\mathbf{p}}
\def\PP{\mathbf{P}}

\def\Set{\mathop{\mathrm{Set}}\nolimits}
\def\Hom{\mathord{\mathrm{Hom}}}
\def\Homb{\mathord{\overline{\mathrm{Hom}}}}
\def\VHomb{\mathord{\overline{\mathrm{SHom}}}}
\def\Iso{\mathord{\mathrm{Iso}}}
\def\TStruc{\mathord{\mathrm{TStruc}}}
\def\SSet{\mathord{\mathrm{SSet}}}
\def\Cas{\mathord{\mathrm{Cas}}}

\def\Func{\mathbf{F}}

\def\PPi{\boldsymbol{\Pi}}
\def\ttau{\boldsymbol{\tau}}
\def\ff{\boldsymbol{f}}
\def\ob{\mathord{\mathrm{Ob}}}

\def\Vec{{\mathrm{Vec}}}

\def\id{\mathrm{id}}

\def\bbar#1{\bar{\bar #1}}
\def\bxi{\boldsymbol{\xi}}

\def\Lab{\mathfrak{L}}

\def\dash{\leavevmode\unskip\kern0.18em--\penalty\exhyphenpenalty\kern0.18em}
\def\slash{\leavevmode\unskip\kern0.15em/\penalty\exhyphenpenalty\kern0.15em}

\def\hotimes{\mathbin{\hat\otimes}}
\def\one{\mathbbm{1}}

\newcommand{\roof}[1]{\lceil #1 \rceil}

\newtheorem{example}[lemma]{Example}

\newcounter{tables}

\let\basepoint\logof
\def\logof{\mathord{{\basepoint}}} 

\title{Langevin dynamic for the 2D Yang--Mills measure}
\author{Ajay Chandra$^1$, Ilya Chevyrev$^2$, Martin Hairer$^1$, and Hao Shen$^3$}

\institute{Imperial College London \\ \email{a.chandra@imperial.ac.uk, m.hairer@imperial.ac.uk} \and University of Edinburgh, \email{ichevyrev@gmail.com} \and University of Wisconsin--Madison, \email{pkushenhao@gmail.com}}

\begin{document}
\maketitle
\begin{abstract}
We define a natural state space and Markov process associated to the stochastic Yang--Mills heat 
flow in two dimensions.

To accomplish this we first introduce a space of distributional connections for which holonomies along sufficiently regular curves (Wilson loop observables) and the action of an associated group of gauge transformations are both well-defined and satisfy good continuity properties. 
The desired state space is obtained as the corresponding space of orbits under this group action and is shown to be a Polish space when equipped with a natural Hausdorff metric.

To construct the Markov process we show that the stochastic Yang--Mills heat flow takes values in our space of connections and use the ``DeTurck trick'' of introducing a time dependent gauge transformation to show invariance, in law, of the solution under gauge transformations.

Our main tool for solving for the Yang--Mills heat flow is the theory of regularity structures and along the way we also develop a ``basis-free'' framework for applying the theory of regularity structures in the context of vector-valued noise \dash this provides a conceptual framework for interpreting several previous constructions and 
we expect this framework to be of independent interest. 
\end{abstract}
\setcounter{tocdepth}{2}

\tableofcontents

\section{Introduction}

The purpose of this paper and the companion article~\cite{CCHSPrep} is to study the Langevin dynamic associated to the Euclidean Yang--Mills (YM) measure.
Formally, the YM measure is written
\begin{equ}\label{eq:YM_measure}
\mrd\mu_\YM(A) = \mathcal{Z}^{-1}\exp\big[-
S_{\YM}(A)\big] \mrd A\;,
\end{equ}
where $\mrd A$ is a formal Lebesgue measure on the space of connections of a 
principal $G$-bundle $P\to M$, $G$ is a compact Lie group, and $\mathcal{Z}$ is a normalisation constant.
The YM action is given by 
\begin{equ}\label{eq:YM_energy}
S_{\YM}(A) \eqdef \int_M
|F_A(x)|^2 \mrd x\;,
\end{equ}
where $F_A$ is the curvature $2$-form of $A$, the norm $|F_A|$ is given by an $\Ad$-invariant inner product on the Lie algebra $\mfg$ of $G$, and $\mrd x$ is a Riemannian volume measure on the space-time manifold $M$.
The YM measure plays a fundamental role in high energy physics, constituting one of the 
components of the Standard Model, and its rigorous construction largely remains open, see~\cite{JW06, Chatterjee18} and the references therein.
The action $S_\YM$ in addition plays a significant role in geometry, see e.g.~\cite{AB83, DK90}.

For the rest of our discussion we will take $M=\T^d$, the $d$-dimensional torus equipped with the Euclidean distance and normalised Lebesgue measure, and the principal bundle $P$ to be trivial.
In particular, we will identify the space of connections on $P$ with $\mfg$-valued $1$-forms on $\T^d$ (implicitly fixing a global section).
The results of this paper almost exclusively focus on the case $d=2$, and the case $d=3$ is studied in~\cite{CCHSPrep}.

A postulate of gauge theory is that all physically relevant quantities should be invariant under the action
of the gauge group, which consists of the automorphisms of the principal bundle $P$ fixing the base space.
In our setting, the gauge group can be identified with maps $g\in\mfG^\infty = \mcC^\infty(\T^d,G)$, and the corresponding action on connections is given by
\begin{equ}\label{eq:gauge_tranform}
A \mapsto A^g \eqdef gAg^{-1} - (\mrd g)g^{-1}\;.
\end{equ}
Equivalently, the $1$-form $A^g$ represents the same connection as $A$ but in a new frame 
(i.e.\ global section) determined by $g$.
The physically relevant object is therefore not the connection $A$ itself,
but its orbit $[A]$ under the action \eqref{eq:gauge_tranform}.


In addition to the challenge of rigorously interpreting~\eqref{eq:YM_measure}
due to the infinite-dimensionality of the space of connections,
gauge invariance poses an additional difficulty that is not encountered in theories such as the $\Phi^p$ models.
Indeed, since $S_\YM$ is invariant under the action of the infinite-dimensional gauge group $\mfG^\infty$
(as it should be to represent a physically relevant theory), the interpretation of~\eqref{eq:YM_measure}
as a probability measure on the space of connections runs into the problem of the impossibility of
constructing a measure that is ``uniform'' on each gauge orbit.
Instead, one would like to quotient out the action of the gauge group and build the measure 
on the space of gauge orbits, but this 
introduces a new difficulty in that it is even less clear what the reference ``Lebesgue measure'' means in this case. 

A natural approach to study the YM measure is to consider the Langevin dynamic associated with the action $S_\YM$.
Indeed, this dynamic is expected to be naturally gauge covariant and one can aim to 
use techniques from PDE theory to understand its behaviour.
Denoting by $\mrd_A$ the covariant derivative associated with $A$ and by $\mrd_A^*$ its adjoint, the equation governing the Langevin dynamic is formally given by
\begin{equ}\label{eq:SYM_no_deturck}
\partial_t A = -\mrd_A^* F_A + \xi\;.
\end{equ}
In coordinates this reads, for $i=1,\ldots, d$ and with summation over $j$ implicit,\footnote{Implicit summation over repeated indices will be in place throughout the paper with departures from this convention explicitly specified.}
\begin{equ}[e:SYMcoord]
\partial_t A_i
= \xi_i + \Delta A_i - \partial_{ji}^2 A_j 
+ \big[A_j,2\partial_j A_i - \partial_i A_j + [A_j,A_i]\big]
+ [\partial_j A_j, A_i]\;,
\end{equ}
where $\xi_{1},\dots,\xi_{d}$ are independent $\mfg$-valued space-time white noises on $\R\times\T^d$ with
covariance induced by an $\Ad$-invariant scalar product on $\mfg$. (We fix such a scalar product for the
remainder of the discussion.)
Equation~\eqref{eq:SYM_no_deturck} was the original motivation of Parisi--Wu~\cite{ParisiWu} in their introduction of
stochastic quantisation.
This field has recently received renewed interest due to a
development of tools able to study singular SPDEs~\cite{Hairer14, GIP15}, and has 
proven fruitful in the study and an alternative construction of the scalar $\Phi^4$ quantum 
field theories~\cite{MW17Phi42,MW17Phi43,AK20,MoinatWeber18,GH21}
(see also~\cite{BG18} for a related construction).

A very basic issue with~\eqref{eq:SYM_no_deturck} is the lack of ellipticity of the term $\mrd_A^* F_A$, 
which is a reflection of the invariance of the action $S_\YM$ under the gauge group\footnote{Since $A^g$ is a solution for any solution $A$ and fixed element $g \in \CG$, there are non-smooth solutions. Note also that the highest order term 
in the right-hand side is $d^* d A$ which annihilates all exact forms and therefore cannot be elliptic.}.
A well-known solution to this problem is to realise that if we take any sufficiently regular functional
$A \mapsto H(A) \in \CC^\infty(\T^2,\mfg)$ and consider instead of \eqref{eq:SYM_no_deturck}
the equation
\begin{equ}\label{eq:SYM_deturck}
\partial_t A = -\mrd_A^* F_A + \mrd_A H(A) + \xi\;,
\end{equ}
then, at least formally, solutions to \eqref{eq:SYM_deturck} are gauge equivalent to those of \eqref{eq:SYM_no_deturck}
in the sense that there exists a time-dependent gauge transformation mapping one into the other one,
at least in law.
This is due to the fact that the tangent space of 
the gauge orbit at $A$ (ignoring issues of regularity \slash topology for the moment)
is given by terms of the form $\mrd_A \omega$, where $\omega$ is an arbitrary $\mfg$-valued $0$-form.

A convenient choice of $H$ is given by $H(A) = -\mrd^* A$ which yields the 
so-called DeTurck--Zwanziger term~\cite{zwanziger81,deturck83} 
\begin{equ}
-\mrd_A\mrd^* A =  (\partial_i + [A_i,\cdot])\partial_j A_j\,\mrd x_i\;.
\end{equ}
This allows to cancel out the term $\partial_{ji}^2 A_j$ appearing in \eqref{e:SYMcoord} 
and thus renders the equation parabolic,
while still keeping the solution to the modified equation gauge-equivalent 
to the original one.
We note that the idea to use this modified equation to study properties of the heat flow has proven a useful tool in geometric analysis~\cite{deturck83, Donaldson, CG13} and has appeared in works on stochastic quantisation in the physics literature~\cite{zwanziger81,BHST87II,DH87}.

With this discussion in mind, the equation we focus on, also referred to in the sequel as the stochastic Yang--Mills (SYM) equation, is given in coordinates by
\begin{equ}\label{eq:SYM}
\partial_t A_i = \Delta A_i + \xi_i + \big[A_j,2\partial_j A_i - \partial_i A_j + [A_j,A_i]\big]\;.
\end{equ}
Our goal is to show the existence of a natural space of gauge orbits such that (appropriately renormalised) solutions to~\eqref{eq:SYM}  define a canonical Markov process on this space.
In addition, one desires a class of gauge invariant observables to be defined on this orbit space
which is sufficiently rich to separate points; a popular class is that of Wilson loop observables (another being the lasso variables of Gross~\cite{Gross85}),
which are defined in terms of holonomies of the connection and a variant of which is known to separate the gauge orbits in the smooth setting~\cite{Sengupta92}.

One of the difficulties in carrying out this task is that any reasonable definition for the state space should be supported on gauge 
orbits of distributional connections, and it is a priori not clear
how to define holonomies (or other gauge-invariant observables) for such connections.
In fact, it is not even clear how to carry out the construction to ensure that the orbits form a reasonable (e.g.\ Polish) space, given that the quotient of a Polish space by the action of a Polish group will typically yield a highly 
pathological object from a measure-theoretical perspective. (Think of even simple cases like the quotient of
$L^2([0,1])$ by the action of $H^1_0([0,1])$ given by $(x,g) \mapsto x + \iota g$ with $\iota \colon H_0^1 \to L^2$
the canonical inclusion map or the quotient of the torus $\T^2$ by the action of $(\R,+)$ given by an irrational rotation.)

We now describe our main results on an informal level, postponing a precise
formulation to Section~\ref{sec:main_results}, and mention connections with the existing literature and several open problems.

\subsection{Outline of results}

Our first contribution is to identify a natural space of distributional connections $\Omega^1_\alpha$, which can be seen as a refined analogue of the classical H{\"o}lder--Besov spaces, along with an associated gauge group.
An important feature of this space is that holonomies along all sufficiently regular curves (and thus Wilson loops and their variants) are canonically defined for each connection in $\Omega^1_\alpha$ and are continuous functions of the connection and curve.
In addition, the associated space of gauge orbits is a Polish space and thus well behaved from the viewpoint of probability theory. A byproduct of the construction of $\Omega^1_\alpha$ is a parametrisation-independent
way of measuring the regularity of a curve which relates to $\beta$-H{\"o}lder regular curves with $\beta \in (1,2)$
(in the sense that they are differentiable with $\beta-1$-H{\"o}lder derivative) 
in a way that is strongly reminiscent of how $p$-variation relates to H{\"o}lder regularity for $\beta \le 1$.  

In turn, we show that the SPDE~\eqref{eq:SYM} can naturally be solved in the space $\Omega^1_\alpha$ through mollifier approximations.
More precisely, we show that for any mollifier $\moll^\eps$ at scale $\eps\in(0,1]$ and $C\in L_{G}(\mfg,\mfg)$ (where $L_{G}(\mfg,\mfg)$ consists of all linear operators from $\mfg$ to itself which commute with $\Ad_g$ for any $g \in G$), the solutions to the renormalised SPDE  
\begin{equ}\label{eq:renormalised_SYM}
\partial_t A_i = 
\Delta A_i +\moll^\eps*\xi_i + C A_i + \big[A_j,2\partial_j A_i - \partial_i A_j + [A_j,A_i]\big]
\end{equ}
converge as $\Omega^1_\alpha$-valued processes as $\eps\to 0$ (with a possibility of finite-time blow-up).
Observe that the addition of the mass term in~\eqref{eq:renormalised_SYM}
(as well as the choice of mollification with respect to a fixed coordinate system)
breaks gauge-covariance for any $\eps>0$.
Our final result is that gauge-covariance can be restored in the $\eps\to0$ limit.
Namely, we show that for each non-anticipative mollifier $\moll$,
there exists an essentially \textit{unique}
choice for $C$ (depending on $\moll$) such that 
in the limit $\eps\to0$, the law of the gauge orbit $[A(t)]$ is 
independent of $\moll$ and depends only on the gauge orbit $[A(0)]$ of the initial condition.
This provides the construction of the aforementioned canonical Markov process
associated to~\eqref{eq:SYM} on the space of gauge orbits.

We mention that a large part of the solution theory for~\eqref{eq:SYM} is now automatic and 
follows from the theory of regularity structures~\cite{Hairer14,BHZ19, CH16,BCCH21}.
In particular, these works guarantee that a suitable renormalisation procedure
yields convergence of the solutions inside some H{\"o}lder--Besov space.
The points which are not automatic are that the limiting solution 
indeed takes values in the space $\Omega^1_\alpha$, that it is gauge invariant, and that no 
diverging counterterms are required for the convergence of \eqref{eq:renormalised_SYM}.
One contribution of this article is to adapt the
algebraic framework of regularity structures developed in~\cite{BHZ19, BCCH21}
to address the latter point.
Precisely, we give a natural renormalisation procedure for SPDEs  of the form
\begin{equ}\label{eq:general_SPDE}
(\partial_t - \mathscr{L}_\mft) A_\mft = F_{\mft}(\mbA, \bxi)\;, \quad \mft\in\mfL_+\;,
\end{equ}
with vector-valued noises and solutions.
Here, for some finite index sets $\mfL_\pm$,  $(\mathscr{L}_\mft)_{\mft\in\mfL_+}$ are differential operators, $\mbA$ and $\bxi$ represent the jet of $(A_\mft)_{\mft\in\mfL_+}$
and $(\xi_\mfl)_{\mfl\in\mfL_-}$ which take values in vector spaces $(W_\mft)_{\mft\in \mfL_+}$
and $(W_\mft)_{\mft\in \mfL_-}$ respectively,
and the nonlinearities $(F_{\mft})_{\mft\in\mfL_+}$ are smooth and local.
We give a systematic way to build a regularity structure associated 
to~\eqref{eq:general_SPDE} and to derive the renormalised equation without ever choosing a basis of the spaces $W_\mft$.

\begin{example}
In addition to~\eqref{eq:SYM}, an equation of interest which fits into this framework comes from the Langevin dynamic of the Yang--Mills--Higgs Lagrangian 
\begin{equ}\label{eq:Higgs_Lagrangian}
\frac12\int_{\T^d} \Bigl(|F_A|^2 + |\mrd_A \Phi|^2 -  m^2|\Phi|^2 + \frac12 |\Phi|^4\Bigr)\mrd x\;,
\end{equ}
where $A$ is a $1$-form taking values in a Lie sub-algebra $\mfg$ of the anti-Hermitian operators on $\C^N$, and $\Phi$ is a $\C^N$-valued function.
The associated SPDE (again with DeTurck term) reads
\begin{equation}\label{eq:YMH_SPDE}
\begin{split}
\partial_t A &= -\mrd_A^* F_A -\mrd_A\mrd^* A - \mathbf{B}(\Phi\otimes \mrd_A\Phi) + \xi^A\;,
\\
\partial_t \Phi &= -\mrd_A^*\mrd_A \Phi + (\mrd^* A)\Phi - \Phi |\Phi|^2 - m^2\Phi + \xi^\Phi\;,
\end{split}
\end{equation}
where $\mathbf{B}\colon\C^N\otimes\C^N \to \mfg$ is the $\R$-linear map that satisfies $\scal{h,\mathbf{B}(x\otimes y)}_\mfg = \Re\scal{ h x,y}_{\C^N}$ for all $h \in \mfg$.

One of the consequences of our framework is that the renormalisation counter\-terms of~\eqref{eq:YMH_SPDE} can all be constructed from iterated applications of $\mathbf{B}$, the Lie bracket $[\cdot,\cdot]_\mfg$, and 
the product $(A,\Phi) \mapsto A\Phi$.
\end{example}

\subsection{Relation to previous work}

There have been several earlier works on the construction of an orbit space.
Mitter--Viallet~\cite{Mitter} showed that the space of gauge orbits modelled on $H^k$ for
$k > {\frac d2} +1$ is a smooth Hilbert manifold.
More recently, Gross~\cite{Gross2,GrossSingular}
has made progress on the analogue in $H^{1/2}$ in dimension $d=3$.

An alternative (but related) route to give meaning to the YM measure is to directly define a stochastic process indexed by a class of gauge invariant observables (e.g.\ Wilson loops).
This approach was undertaken in earlier works on the 2D YM measure~\cite{Driver89,Sengupta97,Levy03,Levy06} which have successfully given explicit representations of the measure for general compact manifolds and principal bundles.
It is not clear, however, how to 
extract from these works a space of gauge orbits with a well-defined probability measure,
which is somewhat closer to the physical interpretation of the measure. (This is a kind of non-linear analogue to 
Kolmogorov's standard question of finding a probability measure on a space of ``sufficiently regular''
functions that matches a given consistent collection of $n$-point distributions.)
In addition, this setting is ill-suited for the study of the Langevin dynamic since it is far from clear how to interpret a realisation of such a stochastic process as the initial condition for a PDE.

A partial answer was obtained in~\cite{Chevyrev18YM} where it was 
shown that a gauge-fixed version of the YM measure (for a simply-connected structure group $G$)
can be constructed in a Banach space of distributional connections 
which could serve as the space of initial conditions of the PDE~\eqref{eq:SYM}.
Section~\ref{sec:state_space} of this paper extends part of this 
earlier work by providing a strong generalisation of the spaces used 
therein (e.g.\ supporting holonomies along all sufficiently regular 
paths, while only axis-parallel paths are handled in~\cite{Chevyrev18YM})
and constructing an associated canonical space of gauge orbits.

Another closely related work was recently carried out in~\cite{Shen18}.
It was shown there that the lattice gauge covariant Langevin dynamic of the scalar Higgs model (the Lagrangian of which is given by~\eqref{eq:Higgs_Lagrangian} without the $|\Phi|^4$ term and with an abelian Lie algebra) in $d=2$ can be appropriately modified by a DeTurck--Zwanziger term and renormalised to yield local-in-time solutions in the continuum limit.
The  mass renormalisation term $C A_i$ as in~\eqref{eq:renormalised_SYM}
is absent in~\cite{Shen18} due to the fact that the lattice gauge theory preserves the exact gauge symmetry,
while a divergent mass renormalisation for the Higgs field $\Phi$  is still needed but preserves gauge invariance. 
In addition, convergence of a natural class of gauge-invariant observables was shown over short time intervals; but there was no description for the orbit space.

\subsection{Open problems}

It is natural to conjecture that the Markov process constructed in 
this paper possesses a unique invariant measure, for which the 
associated stochastic process indexed by Wilson loops agrees with the YM measure constructed in~\cite{Sengupta97, Levy03,Levy06}.
Such a result would be one of the few known rigorous connections between the 
YM measure and the YM energy functional~\eqref{eq:YM_energy} (another connection is made in~\cite{LevyNorris06} through a large deviations principle).
A possible approach would be to show that the gauge-covariant lattice 
dynamic for the discrete YM measure converges to the solution to the SYM equation~\eqref{eq:SYM} 
identified in this paper.
Combined with a gauge fixing procedure as in~\cite{Chevyrev18YM} and
an argument of Bourgain~\cite{Bourgain94} along the lines of~\cite{HM18}, 
this convergence would prove the result
(as well as strong regularity properties of the YM measure obtained from the description of the orbit space in this paper).
The main difficulty to overcome is the lack of general stochastic estimates for the lattice which are available in the continuum thanks to~\cite{CH16}.

Our results do not exclude finite-time blow-up of solutions to SYM~\eqref{eq:SYM},
not even in the quotient space. (Since gauge orbits are unbounded, non-explosion of solutions 
to~\eqref{eq:SYM} is a stronger property than non-explosion of the Markov process on gauge orbits constructed
in this article.)
It would be of interest to determine whether the solution to SYM survives almost surely  for all 
time for any initial condition. The weaker case of the Markov 
process would be handled by the above conjecture combined 
with the strong Feller property~\cite{HM16} and irreducibility~\cite{HS19} which both hold in this case.
The analogous result is known for the $\Phi^4_d$ SPDE in $d=2,3$~\cite{MoinatWeber18}.
Long-time existence of the deterministic YM heat flow in $d=2,3$ is also known~\cite{Rade92,CG13}, but it is not clear how to adapt these methods to the stochastic setting.

It is also unclear how to extend all the results of this paper to the 3D setting.
In~\cite{CCHSPrep} we analyse the SPDE~\eqref{eq:SYM} for $d=3$ and show that the solutions take values in a suitable state space to which gauge equivalence extends in a canonical way.
We furthermore show the same form of gauge-covariance in law as in this paper, which allows us to construct a Markov process on the corresponding orbit space.
However, an important result missing in~\cite{CCHSPrep} in comparison to this article is the existence of a gauge group which acts transitively on the gauge orbits.
We give further details therein.

\subsection{Outline of the paper}

The paper is organised as follows.
In Section~\ref{sec:main_results}, we give a precise formulation of our main results concerning the SPDE~\eqref{eq:SYM} and the associated Markov process on gauge orbits.
In Section~\ref{sec:state_space} we provide a detailed study of the space of distributional $1$-forms $\Omega^1_\alpha$ used in the construction of the state space of the Markov process.
In Section~\ref{sec:SHE} we study the stochastic heat equation as an $\Omega^1_\alpha$-valued process.

In Section~\ref{sec:vector-valued_noises} we give a canonical, basis-free framework for constructing regularity structures associated to SPDEs with vector-valued noise. 
Moreover, we generalise the main results of \cite{BCCH21} on formulae for renormalisation counterterms in the scalar setting and obtain analogous vectorial formulae. 
We expect this framework to be useful in for a variety of systems of SPDE whose natural formulation involve vector-valued noise 
\dash in the context of \eqref{eq:SYM} this framework allows us to directly obtain expressions for renormalisation counterterms in terms of Lie brackets and to use symmetry arguments coming from the $\Ad$-invariance of the noises.  

In Section~\ref{sec:solution_theory} we prove local well-posedness of the SPDE~\eqref{eq:SYM},
and in Section~\ref{sec:gauge_equivar} we show that gauge covariance holds in law
for a specific choice of renormalisation procedure 
which allows us to construct the canonical Markov process on gauge orbits.

\subsection{Notation and conventions}
\label{subsec:notation}

We collect some notation and definitions used throughout the paper.
We denote by $\R_+$ the interval $[0,\infty)$ and we identify the torus $\T^2$ with the set $\big[-\frac12,\frac12\big)^2$.
We equip $\T^2$ with the geodesic distance, which, by an abuse of notation, we denote $|x-y|$, and $\R\times \T^2 $ 
with the parabolic distance $|(t,x)-(s,y)|=\sqrt{|t-s|}+|x-y|$.

A \emph{mollifier} $\moll$ is a smooth function on space-time $\R\times\R^2$ (or just space $\R^2$) with support 
in the ball $\{z \ssep |z| < \frac14 \}$ such that $\int\moll = 1$.
We will assume that any space-time mollifier $\moll$ we use satisfies $\moll(t,x_{1},x_{2}) = \moll(t,-x_{1},x_{2}) = \moll(t,x_{1},-x_{2})$.\footnote{This is for convenience, so that some renormalisation constants vanish, but is not strictly necessary.}
A space-time mollifier is called \emph{non-anticipative}
if it has support in the set $\{(t,x)\ssep t \geq 0\}$.

Consider a separable Banach
space $(E, |\cdot|)$.
For $\alpha\in[0,1]$ and a metric space $(F,d)$, we denote by $\mcC^{\Hol\alpha}(F,E)$ the set of all functions $f\colon F\to E$ such that
\begin{equ}
|f|_{\Hol\alpha} \eqdef \sup_{x,y} \frac{|f(x)-f(y)|}{d(x,y)^\alpha} < \infty\;,
\end{equ}
where the supremum is over all distinct $x,y\in F$.
We further denote by $\mcC^\alpha(F,E)$ the space of all functions $f\colon F\to E$ such that
\begin{equ}
|f|_{\mcC^{\alpha}} \eqdef |f|_{\infty} + |f|_{\Hol\alpha} < \infty\;,
\end{equ}
where $|f|_\infty \eqdef \sup_{x\in F} |f(x)|$.
For $\alpha>1$, we define 
$\mcC^\alpha(\T^2,E)$
(resp. $\mcC^\alpha(\R\times\T^2,E)$)  to be the space of $k$-times differentiable functions (resp. functions that are $k_0$-times differentiable in $t$ and $k_1$-times differentiable in $x$  with $2k_0+k_1\le k$), where $k\eqdef\roof\alpha-1$, with $(\alpha-k)$-H{\"o}lder continuous $k$-th derivatives.

For $\alpha<0$, let $r \eqdef -\roof{\alpha-1}$ and $\mcB^r$ denote the set of all smooth functions $\psi \in \mcC^\infty(\T^2)$ with $|\psi|_{\mcC^r} \leq 1$ and support in the ball $|z|<\frac14$.
Let $(\mcC^\alpha(\T^2,E),|\cdot|_{\mcC^\alpha})$ denote the space of distributions $\xi \in \mcD'(\T^2,E)$ for which
\[
|\xi|_{\mcC^\alpha} \eqdef \sup_{\lambda\in(0,1]}\sup_{\psi \in \mcB^r} \sup_{x \in \T^2} \frac{|\scal{\xi,\psi^\lambda_x}|}{\lambda^\alpha} < \infty\;,
\]
where $\psi^\lambda_x(y) \eqdef \lambda^{-2}\psi(\lambda^{-1}(y-x))$.
For $\alpha=0$, we define $\mcC^0$ to simply be $L^\infty(\T^2,E)$,
and use $\mcC(\T^2,E)$ to denote the space of continuous functions,
both spaces being equipped with the $L^\infty$ norm.
For any $\alpha\in\R$, we denote by $\mcC^{0,\alpha}$ the closure  of smooth functions in $\mcC^\alpha$.
We drop $E$ from the notation and write simply $\mcC(\T^2)$, $\mcC^{\alpha}(\T^2)$, etc. whenever $E=\R$.

For a space $\CB$ of $E$-valued functions (or distributions) on $\T^2$, we denote by $\Omega\CB$\label{Omega_CB page ref} the space of $E$-valued $1$-forms $A = \sum_{i=1}^2 A_i \mrd x_i$ where $A_1, A_2 \in\CB$.
If $\CB$ is equipped with a (semi)norm $|\cdot|_\CB$, we define
\begin{equ}
|A|_{\Omega\CB} \eqdef \sum_{i=1}^2 |A_i|_{\CB}\;.
\end{equ}
When $\CB$ is of the form $\mcC(\T^2,E)$, $\mcC^{\alpha}(\T^2,E)$, etc., we write simply $\Omega\mcC$, $\Omega\mcC^{\alpha}$, etc. for $\Omega\CB$.

Given two real vector spaces $V$ and $W$ we write $L(V,W)$ for the set of all linear operators from $V$ to $W$.
If $V$ is equipped with a topology, we write $V^*$ for its topological dual,
and otherwise we write $V^*$ for its algebraic dual.
As mentioned in the introduction, we also write $L_{G}(\mfg,\mfg) = \{C \in L(\mfg,\mfg):\ C \Ad_{g} = \Ad_{g} C \textnormal{ for all } g \in G\}$.

\subsubsection{State space with blow-up}
\label{subsubsec:blowup}

Given a metric
space $(F,d)$, we extend it with a cemetery state by setting
$\hat F = F\sqcup\{\skull\}$ and
equipping it 
with the topology whose basis sets are the balls of $F$
and the complements of closed balls.
We use the convention $d(\skull,f) \eqdef \infty$ for all $f\in F$
and observe that $\hat F$ is metrisable with metric
\[
\hat d(f,g) = d(f,g)\wedge
\big(h[f]+h[g]\big)\;,
\]
where $h\colon\hat F\to [0,1]$ is defined by $h[f]=(1+d(f,o))^{-1}$, where $o\in F$ is 
any fixed element (see the proof of~\cite[Thm.~2]{Mandelkern89} for a similar statement).

Denote $\R_+=[0,\infty)$.
Given $f \in \CC(\R_{+},\hat F)$, we define $T[f] = \inf\{ t \ge 0: f(t) = \skull\}$.
We then define\label{page_ref:F_sol}
\begin{equ}
F^{\sol} \eqdef \Big\{ f \in \mcC(\R_+,\hat F) \,:\,
f\restr_{[T[f],\infty)} \equiv \skull \Big\}\;.
\end{equ}
We should think of $F^\sol$ as the state space of dynamical systems with values in $F$ which can 
blow up in finite time and cannot be `reborn'.
We equip $F^\sol$ with the following metric, which is an extension of that defined in the
arXiv version of~\cite[Sec.~2.7.2]{BCCH21} with the benefit that it can be defined even when $F$ is a non-linear space.

Consider the cone $C\hat F = \big( [0,1]\times \hat F \big) /\sim$, where $x\sim y$ $\Leftrightarrow$ [$x=y$] or [$x=(0,f),y=(0,g)$ for some $f,g\in \hat F$].
Treating $\hat F \simeq \{1\}\times \hat F$ as a subset of $C\hat F$, we extend the metric $\hat d$ to $C\hat F$ by
\[
\hat d((a,f),(b,g)) = |a-b| + (a\wedge b)\hat d(f,g)\;.
\]
(Using that $\hat d\leq 2$ on $\hat F$, the verification that $\hat d$ is a metric on $C\hat F$ is routine.)

For any $f \in F^\sol$, we define its running supremum by
\begin{equ}
S_f(t) \eqdef \sup_{s \le t}  d(f(s),o) \in [0,\infty] \;,
\end{equ}
where $o \in F$ is a fixed element.
We also fix a smooth non-increasing function $\psi \colon \R \to [0,1]$ with
derivative supported in $[1,2]$, $\psi(1) = 1$, and $\psi(2) = 0$. Given $L\ge 1$, 
we define $\Theta_L(f) \in \CC(\R_+,C \hat F)$ by 
\begin{equ}
\Theta_L(f)(t) = \big( \psi(S_f(t)/L), f(t) \big)\;.
\end{equ}
We then equip $F^\sol$ with a metric $D(\cdot,\cdot) \eqdef \sum_{L=1}^{\infty} 2^{-L} D_{L}(\cdot,\cdot)$, where for $f,g \in F^\sol$
\begin{equ}
D_L(f,g) \eqdef \sup_{t \in [0,L]} \hat d \big( \Theta_L(f)(t), \Theta_L(g)(t) \big)
\;.
\end{equ}
The point of these definitions is that they equip $F^\sol$ with a metric such that 
a function $f$ exploding at time $T$ is close to a function $g$ that `tracks' $f$ closely up to time $T-\eps$,
but then remains finite, or possibly explodes at some later time.
In particular, we have the following lemma, the proof of which is identical to the proof of the arXiv version of~\cite[Lem.~2.19]{BCCH21}.

\begin{lemma}\label{lem:alt_conv}
Consider $f,f_1,f_2,\ldots \in F^\sol$. The following statements are equivalent.
\begin{enumerate}[label=(\roman*)]
\item\label{pt:D_conv} $D(f_n,f)\to 0$;
\item\label{pt:D_conv_alt} for every $L>0$,
\begin{equ}
\lim_{n\to\infty} \sup_{t\in[0,T_L]} d(f_n(t),f(t)) = 0\;,
\end{equ}
where $T_L \eqdef L\wedge\inf\{t>0\,:\, d(f_n(t),o) \vee d(f(t),o) > L\}$.
\end{enumerate}
\end{lemma}

We remark here that if $F$ is separable, then so are $(C\hat F,\hat d)$ and $(F^\sol,D)$
(to show this, one can for example adapt the argument of~\cite[Thm~4.19]{KechrisSet}).
Furthermore, if $(F,d)$ is a complete metric space,
it is not difficult to see that $(C\hat F,\hat d)$ is also complete.
On the other hand,
the following example shows that $(F^\sol,D)$ may not be complete even if $(F,d)$ is complete.

\begin{example}\label{ex:not_complete}
Consider $F=\R$ and a sequence $f_n$
such that $f_n(1-2^{-k})=2^k$ for $k=0,\ldots, n$, and then constant on $[1-2^{-n},\infty)$.
On the intervals $[1-2^{-k+1},1-2^{-k}]$ for $k=1,\ldots, n$, suppose $f_n$ goes from $2^{k-1}$ down to $0$ and then back up to $2^k$ linearly.
Then clearly $f_n$ is Cauchy for $D$ and its limit $f$ exists pointwise as a function on $\CC([0,1),\R)$,
but its naive extension as $f(t)=\skull$ for $t\geq 1$ is \textit{not} an element of $\CC(\R_+,\hat\R)$ because it fails to converge to $\skull$ as $t\nearrow1$.
\end{example}

\subsection*{Acknowledgements}

{\small
We would like to thank Andris Gerasimovi\vaccent cs for many discussions regarding the derivation of the
renormalised equation in Sections~\ref{sec:solution_theory} and~\ref{sec:gauge_equivar}.
MH gratefully acknowledges support by the Royal Society through a research professorship.
IC was partly funded by a Junior Research Fellowship of St John's College, Oxford, during the writing of this article.
HS gratefully acknowledges support by NSF via DMS-1712684 / DMS-1909525, DMS-1954091 and CAREER DMS-2044415.
AC gratefully acknowledges financial support from the Leverhulme Trust via an Early Career
Fellowship, ECF-2017-226.
}

\section{Main results}
\label{sec:main_results}

In this section, we give a precise formulation of the main results
described in the introduction.
\subsection{State space and solution theory for SYM equation}
Our first result concerns the state space of the Markov process.
We collect the main features of this space in the following theorem along with precise references, and refer the reader to Section~\ref{sec:state_space} for a detailed study.

\begin{theorem}\label{thm:state_space}
For each $\alpha\in(\frac23,1)$, there exists a Banach space $\Omega^1_\alpha$ of distributional $\mfg$-valued $1$-forms on $\T^2$ with the following properties.
\begin{enumerate}[label=(\roman*)]
\item\label{pt:hol_well-defined} For each $A\in\Omega^1_\alpha$ and $\gamma\in\mcC^{1,\beta}([0,1],\T^2)$
with $\beta\in(\frac{2}{\alpha}-2,1]$, the holonomy $\hol(A,\gamma)\in G$ is well-defined
and, on bounded balls of $\Omega^1_\alpha\times \mcC^{1,\beta}([0,1],\T^2)$, 
is a H{\"o}lder continuous function of $(A,\gamma)$
with distances between $\gamma$'s measured in the supremum metric.
In particular, Wilson loop observables are well-defined on $\Omega^1_\alpha$.
(See Theorem~\ref{thm:extension_to_curves} and Proposition~\ref{prop:p-var_cont}
combined with Young ODE theory~\cite{Lyons94,FrizHairer}.)
\item There are canonical embeddings with the classical H{\"o}lder--Besov spaces
\begin{equ}
\Omega\mcC^{0,\alpha/2}\hookrightarrow \Omega^1_\alpha\hookrightarrow \Omega\mcC^{0,\alpha-1}\;.
\end{equ}
(See Section~\ref{subsec:smooth_one_forms}.)

\item Let $\mfG^{0,\alpha}$ denote the closure of smooth functions in $\mcC^{\Hol\alpha}(\T^2,G)$.
Then there is a continuous group action of $\mfG^{0,\alpha}$ on $\Omega^1_\alpha$.
Furthermore there exists a metric $D_\alpha$ on the quotient space of gauge orbits $\mfO_\alpha\eqdef \Omega^1_\alpha/\mfG^{0,\alpha}$ which induces
the quotient topology and such that $(\mfO_\alpha,D_\alpha)$ is separable and complete.
(See Corollary~\ref{cor:group_action_smooth_closure} and Theorem~\ref{thm:Polish}.)

\item Gauge orbits in $\mfO_\alpha$ are uniquely determined by conjugacy classes of holonomies along loops. (See Proposition~\ref{prop:gauge_equiv}.)
\end{enumerate}
\end{theorem}

\begin{remark}
Analogous spaces could be defined on any manifold, but it is not clear whether higher dimensional versions are useful for the
study of the stochastic YM equation. 
\end{remark}

\begin{remark}\label{rem:holonomy_param_indep}
Since $\hol(A,\gamma)$ is independent of the parametrisation of $\gamma$, the ``right'' way of measuring its
regularity should also be parametrisation-independent, which is not the case of $\CC^{1,\beta}$. This
is done in Definition~\ref{def:pvar} which might be of independent interest.
\end{remark}

We now turn to the results on the SPDE.
Let us fix $\alpha\in(\frac23,1)$ and $\eta\in(\frac\alpha4-\frac12,\alpha-1]$.
Note that $\alpha>\frac23$ implies that such $\eta$ always exists.
%
Let us fix for the remainder of this section a space-time mollifier $\moll$
as defined in Section~\ref{subsec:notation} and
denote $\moll^\eps(t,x) \eqdef \eps^{-4}\moll(t\eps^{-2},\eps^{-1}x)$.
We also fix i.i.d.\ $\mfg$-valued white noises $(\xi_i)_{i=1}^2$ on $\R\times\T^2$ and write
$\xi^\eps_i \eqdef \xi_i * \moll^\eps$.
Fix some $C\in L_{G}(\mfg,\mfg)$ independent of $\eps$, and for each $\eps\in(0,1]$ consider the system of PDEs on $\R_+ \times \T^2$, with $i \in \{1,2\}$,
\begin{equs}\label{eq:SPDE_for_A}
\partial_t A^\eps_i &= 
\Delta A^\eps_i +\xi^\eps_i + C A^\eps_i + [A^\eps_j,2\partial_j A^\eps_i - \partial_i A^\eps_j + [A^\eps_j,A^\eps_i]] \;,
\\
A^\eps(0) &= a \in \Omega^1_\alpha\;.
\end{equs}

\begin{theorem}[Local existence]\label{thm:local_exist}
The solution $A^\eps$ converges in $(\Omega^{1}_{\alpha})^{\sol}$ in probability as $\eps \to 0$ to an $(\Omega^{1}_{\alpha})^\sol$-valued 
random variable $A$.
\end{theorem}
\begin{remark}\label{rem:2d_no_renorm}
Note that the RHS of \eqref{eq:SPDE_for_A} does not contain any divergent mass counterterm \dash that is a term of the form $C^{\eps}A_{i}$ with $\lim_{\eps \downarrow 0} C^{\eps} = \infty$. 
The fact that one can obtain a non-trivial limit without such a counterterm is specific to working in two space dimensions and is not true in three space dimensions. 
\end{remark}
\begin{remark}  
Note that we could take the initial condition $a\in\Omega\mcC^\eta$, and the analogous statement would hold either with $(\Omega^{1}_{\alpha})^{\sol}$ replaced by $(\Omega\mcC^\eta)^{\sol}$ or, in the definition of $(\Omega^{1}_{\alpha})^{\sol}$ dropping continuity at $t=0$.
\end{remark}
\begin{remark}\label{rem:A=Psi+B}
As one would expect, the roughest part of the solution $A$ is already captured by the solutions $\Psi$
to the stochastic heat equation. (In fact, one has $A=\Psi+B$ where $B$ belongs to 
in $\mcC^{1-\kappa}$ for any $\kappa>0$.)
Hence, fine regularity properties of $A$ can be inferred from those of $\Psi$.
In particular, one could sharpen the above result to encode time regularity of the solution $A$ at the expense of taking smaller values of $\alpha$, cf. Theorem~\ref{thm:SHE_regularity}.
\end{remark}
\begin{remark}\label{rem:simple_vs_reductive}
From our assumption that $G$ is compact, it follows that $\mfg$ is reductive, namely it can be written as the direct sum of simple Lie algebras and an abelian Lie algebra.
Note that if $\mfh$ is one of the simple components, then every
$C\in L_G(\mfg,\mfg)$ preserves $\mfh$ and its restriction to $\mfh$
is equal to $\lambda\id_\mfh$ for some $\lambda\in\R$;
indeed, $\mfh$ is the Lie algebra of a compact simple Lie group (see the proof of~\cite[Thm~4.29]{Knapp02} for a similar statement) and thus its complexification is also simple, and the claim follows readily from Schur's lemma.
Furthermore, since these components are orthogonal under the Ad-invariant inner product on $\mfg$ introduced in \eqref{eq:YM_energy}, 
 each white noise $\xi_{i}$ also splits into independent noises, each valued in the abelian or a simple component. 
 Eq.~\eqref{eq:SPDE_for_A} then decouples into a system of equations, each for a simple or abelian component, which means that it suffices to prove Theorem~\ref{thm:local_exist} in the  case of a simple Lie algebra
 for which we can take $C \in \R$ (this is the approach we take in our analysis of this SPDE).
 In the abelian case, \eqref{eq:SPDE_for_A} is just a linear stochastic heat equation taking values in an abelian Lie algebra with $C$ a linear map (commuting with $\Ad$) from the abelian Lie algebra to itself, for which the solution theory is standard.
\end{remark}
We give the proof of Theorem~\ref{thm:local_exist} in Section~\ref{sec:solution_theory}.
In principle a large part of the proof is by now automatic and follows from the series of results~\cite{Hairer14,CH16,BHZ19,BCCH21}. 
Key facts which don't follow from general principles are that the solution takes values in the space $\Omega^1_\alpha$ (but this only requires one to show that the SHE takes values in it) and 
more importantly that no additional renormalisation is required.
%
%
However, if one were to directly apply  the framework of~\cite{Hairer14,CH16,BHZ19,BCCH21},
one would have to expand the system 
 with respect to a basis of $\mfg$ 
 into a system of equations driven by $d\times \dim(\mfg)$ independent $\R$-valued  {\it scalar} space-time white noises.
The renormalised equation computed using \cite{BCCH21} would then have to be rewritten to be taken back to the setting of vector valued noises. 
In particular, verifying that the renormalisation counterterm takes the form prescribed above would be both laborious and not very illuminating. 
We instead choose to work with \eqref{eq:SPDE_for_A} intrinsically and, in Section~\ref{sec:vector-valued_noises}, develop a framework for applying the theory of regularity structures and the formulae of \cite{BCCH21} directly to equations with vector valued noise. 

When working with scalar noises, a labelled decorated combinatorial tree $\tau$, which represents some space-time process, corresponds to a one dimensional subspace of our regularity structure. 
On the other hand, if our noises take values in some vector space $W$, then it is natural\footnote{See Section~\ref{sec:motivation} for more detail on why this is indeed natural.} for $\tau$ to index a subspace of our regularity structure 
isomorphic to a partially symmetrised tensor product of copies of $W^*$, where the particular symmetrisation is determined by the symmetries of $\tau$. 

One of our key constructions in Section~\ref{sec:vector-valued_noises} is a functor $\Func_{\cdot}(\act)$ which maps labelled decorated combinatorial trees, which we view as objects in a category of ``symmetric sets'', to these partially symmetrised tensor product spaces in the category of vector of spaces. 
In other words, operations\slash morphisms between these trees analogous to the products and the coproducts of \cite{BHZ19} are mapped, under this functor, to corresponding linear maps between the vector spaces they index.
This allows us to construct a regularity structure, with associated structure group and renormalisation group, without performing any basis expansions.

We also show that this functor behaves well under direct sum decompositions of the vector spaces $W$, which 
allows us to verify that our constructions in the vector noise setting are consistent with the regularity 
structure that would be obtained in the scalar setting if one performed a basis expansion. 
This last point allows us to transfer results from the setting of scalar noise to that of vector noise. One of 
our main results in that section is Proposition~\ref{prop:renormalisation_preserves_coherence} which reformulates 
the renormalisation formulae of \cite{BCCH21} in the vector noise setting.
\subsection{Gauge covariance and the Markov process on orbits}
\label{subsec:gauge_covar_results}
The reader may wonder why we don't simply enforce $C = 0$ in \eqref{eq:SPDE_for_A}
since this is allowed in our statement. 
One reason is that although the limit of $A^\eps$ exists for 
such a choice, it would depend in general on the choice of mollifier $\moll$.
More importantly, our next result shows that it is possible to counteract this by choosing $C$ as
a function of $\moll$ in such a way that not only the limit is independent of
the choice of $\moll$, but the canonical projection of $A$ onto $\mfO_\alpha$ is independent (in law!) 
of the choice of representative of the initial condition.
This then allows us to use this SPDE to construct a ``nice'' Markov
process on the gauge orbit space $\mfO_\alpha$, which would \textit{not} be the case for 
any other choice of $C$.

We first discuss the (lack of) gauge invariance of the mollified equation~\eqref{eq:SPDE_for_A} from a geometric perspective.
Recall that the natural state space for $A$ is the space $\mcA$ of (for now smooth) connections  on a principal $G$-bundle $P$ (which we assume is trivial for the purpose of this article).
The space of connections is an affine space modelled on the vector space $\Omega^1(\T^2, \Ad(P))$, the space of $1$-forms on $\T^2$ with values in the adjoint bundle.
In what follows, we drop the references to $\T^2$ and $\Ad(P)$.

Recall furthermore that the covariant derivative is a map $\mrd_A \colon \Omega^k\to \Omega^{k+1}$
with adjoint $\mrd^*_A \colon \Omega^{k+1}\to \Omega^{k}$.
Hence, the correct geometric form of the DeTurck--Zwanziger term $\mrd_A \mrd^* A$
is really $\mrd_A \mrd_Z^* (A-Z) = \mrd_A\mrd_A^*(A-Z)$, where $Z$ is 
the canonical flat connection associated with the global section of 
$P$ which we implicitly chose at the very start (this choice for $Z$ 
is only for convenience -- any fixed ``reference'' connection $Z$ 
will lead to a parabolic equation for $A$, e.g., the initial 
condition of $A$ is used as $Z$ in~\cite[Sec.~6.3]{DK90}).
The mollification operator $\moll^\eps \colon \Omega^k \to \Omega^k$ also depends on our 
global section (or equivalently, on $Z$).

If we endow $\Omega^k$ with the distance coming from its natural $L^2$ Hilbert space structure then, 
for any $g \in \mfG^\infty = \mcC^\infty(\T^2,G)$, the adjoint action $\Ad_g \colon \Omega^k \to \Omega^k$ 
is an isometry with the covariance properties
$\Ad_g(A-Z) = A^g - Z^g$ and $\Ad_g \mrd_{A}\omega = \mrd_{A^g} \Ad_g \omega$.
Finally, recall that $F_A$ is a $2$-form in $\Omega^2$, and satisfies $\Ad_g F_A = F_{A^g}$.

With these preliminaries in mind, for any $\xi^\eps \in \mcC^\infty([0,T],\Omega^1)$, we rewrite the PDE~\eqref{eq:SPDE_for_A} as
\begin{equ}
\partial_t A = -\mrd_A^* F_A - \mrd_A \mrd^*_A(A-Z) + \xi^\eps + C(A-Z)\;, \quad A(0) = a \in \mcA\;,
\end{equ}
where $C\in \R$ is a constant.
Note that the right and left-hand sides take values in $\Omega^1$.
For a time-dependent gauge transformation $g \in \mcC^\infty([0,T],\mfG^\infty)$, we have that $B \eqdef A^g$ satisfies
\begin{equ}
\partial_t B = \Ad_g \partial_t A - \mrd_{B}[(\partial_t g)g^{-1}]\;.
\end{equ}
In particular, if $g$ satisfies
\begin{equ}\label{eq:PDE_for_g}
(\partial_t g) g^{-1} = \mrd^*_B(Z^g-Z)\;,
\end{equ}
then $B$ solves
\begin{equ}\label{eq:PDE_for_bar_A}
\partial_t B = -\mrd_{B}^* F_{B}
-\mrd_{B} \mrd^*_{B}(B - Z) + \Ad_g\xi^\eps + C(B-Z^g)\;,
\quad B(0) = a^{g(0)} \in \mcA\;.
\end{equ}
The claimed gauge covariance of~\eqref{eq:SPDE_for_A} is then a consequence of the non-trivial fact
that one can choose the constant $C$ in such a way that, as $\eps \to 0$, $B$ converges to the same 
limit in law as the SPDE~\eqref{eq:SPDE_for_A} started from $a^{g(0)}$, i.e.
\begin{equ}[e:tildeA]
\partial_t \tilde A = -\mrd_{\tilde A}^* F_{\tilde A}
-\mrd_{\tilde A} \mrd^*_{\tilde A}(\tilde A - Z)
+ \xi^\eps + C(\tilde A-Z)\;,
\quad \tilde A(0) = a^{g(0)} \in \mcA\;.
\end{equ}
We now  make this statement precise. 
Written in coordinates, the equations for the gauge transformed system are given by
\begin{equs}[2]\label{eq:SPDE_for_B}
\partial_t B_i &= 
\Delta B_i +g\xi^\eps_i g^{-1} + C B_i + C (\partial_i g) g^{-1}
\\
&\qquad
+[B_j,2\partial_j B_i - \partial_i B_j + [B_j,B_i]]\;,&\qquad B(0) &= a^{g(0)} \in \Omega_\alpha^1\;,
\\
(\d_t g)g^{-1} &= \partial_j((\partial_j g)g^{-1})+ [B_j, (\partial_j g)g^{-1}]\;,&\qquad g(0) &\in \mfG^{0,\alpha}\;.
\end{equs}

In our gauge covariance argument it will be useful to drop extraneous information and only keep track of how $g$ and $\bar{g}$ act on our gauge fields.
To this end, for $\alpha > \frac{1}{2}$, we define the  \emph{reduced gauge group} $\mathring{\mathfrak{G}}^{0,\alpha}$\label{mrmfG^0,alpha page ref} to be the quotient of $\mathfrak{G}^{0,\alpha}$ by the kernel\footnote{Note that this kernel consists of all constant gauge transformations $g(\cdot) = \tilde{g}$ with $\tilde{g}$ in the kernel of the adjoint representation of $G$.} of its action on $\Omega \CC^{\alpha-1}$ . 
We denote the corresponding projection map by $\mathfrak{G}^{0,\alpha} \ni g \mapsto [g] \in \mathring{\mathfrak{G}}^{0,\alpha}$.

The desired gauge covariance is then stated as follows.
%
%
\begin{theorem}\label{thm:gauge_covar}
For every non-anticipative space-time mollifier $\moll$, one has the following results.
\begin{enumerate}[label=(\roman*)]
\item \label{pt:gauge_covar}  There exists a unique $\eps$-independent $\bar C\in L_{G}(\mfg,\mfg)$ with the following property.
For every $C \in L_{G}(\mfg,\mfg)$, $a\in\Omega_\alpha^1$, and $g(0)\in\mfG^{0,\alpha}$,
let $(B,g)$ be the solution to~\eqref{eq:SPDE_for_B}
and $(\bar{A},\bar{g})$ be the solution to
\begin{equs}[2]\label{eq:SPDE_for_bar_A}
\partial_t \bar{A}_i &= 
\Delta \bar{A}_i + \moll^\eps * (\bar{g}\xi_i \bar{g}^{-1}) + C \bar{A}_i
+ (C-\bar C)(\partial_i \bar{g} )\bar{g}^{-1}
\\
&\qquad +[\bar{A}_j,2\partial_j \bar{A}_i - \partial_i \bar{A}_j + [\bar{A}_j,\bar{A}_i]]\;,&
\bar{A}(0) &= a^{g(0)}\;,
\\
(\d_t \bar{g})\bar{g}^{-1} &= \partial_j((\partial_j \bar{g})\bar{g}^{-1})+ [\bar{A}_j, (\partial_j \bar{g})\bar{g}^{-1}]\;,& \bar{g}(0) &= g(0)\;,
\end{equs}
where we set $\bar{g}\equiv 1$ on $(-\infty,0)$.
Then, for every $\eps>0$,~\eqref{eq:SPDE_for_bar_A} is well-posed in the sense that, replacing $\xi_i$ by $\bone_{t<0}\xi_i + \bone_{t\geq0}\xi^\delta_i$, $(\bar A, \bar g)$ converges to a smooth maximal solution in $(\Omega^1_\alpha\times \mfG^{0,\alpha})^\sol$ as $\delta\downarrow0$.

Furthermore, $(\bar{A},[\bar{g}])$ and $(B,[g])$ converge in probability to the same limit in $(\Omega_{1}^{\alpha} \times \mathring{\mathfrak{G}}^{0,\alpha})^\sol$ as $\eps\to 0$.

\item \label{pt:indep_moll} If $\bar C$ is as stated in item~\ref{pt:gauge_covar}, then the limiting 
solution $A$ to~\eqref{eq:SPDE_for_A} with $C=\bar C$ is independent of $\moll$ (as long as it is 
non-anticipative). 
\end{enumerate} 
\end{theorem}
As discussed above, $B = A^g$ (i.e.\ $B$ is pathwise gauge equivalent to $A$) for any choice of $C$.
On the other hand, if $\moll$ is non-anticipative, then $\moll^\eps * (\bar{g}\xi_i \bar{g}^{-1})$ is equal in law to $\xi_i^\eps$ by It{\^o} isometry since $\bar{g}$ is adapted, so that when $C=\bar C$, 
the law of $\bar{A}$ does not depend on $\bar g$ anymore and 
$\bar{A}$ is equal in law to the 
process $\tilde A$ defined in~\eqref{e:tildeA}, obtained by starting the dynamics for $A$ from $a^{g(0)}$. 
The theorem therefore proves the desired form of gauge covariance for the choice $C = \bar C$.
\begin{remark}\label{rem:simple_vs_reductive_2}
Again, as in Remark~\ref{rem:simple_vs_reductive} it suffices to prove Theorem~\ref{thm:gauge_covar} in the case of a simple Lie algebra for which one has $\bar{C} \in \R$. 
In this case, for non-anticipative $\moll$, we will show that $\bar C =
\lambda \lim_{\eps \downarrow 0} 
\int \mrd z\ \moll^\eps(z) (K*K^\eps)(z)$,  
where $K$ is the heat kernel, $K^\eps=\moll^\eps*K$,
and $\lambda<0$ is such that $\lambda\id_\mfg$ is the quadratic Casimir
in the adjoint representation.
\end{remark}

We give the proof of Theorem~\ref{thm:gauge_covar} at the end of Section~\ref{sec:renorm_for_system}.
As with Theorem~\ref{thm:local_exist}, the strategy is to lift both systems~\eqref{eq:SPDE_for_B} and~\eqref{eq:SPDE_for_bar_A} to fixed point problems in an appropriate space of modelled distributions.
We compare the two systems using $\eps$-dependent norms on a suitable regularity structure.
The products $g\xi^\eps g^{-1}$ and $\bar g \xi\bar g^{-1}$, however, cause singularities at the $t=0$ time slice
which are not encountered, for example, in the analogous strategy for the equation~\eqref{eq:SPDE_for_A}.
To handle these singularities, we employ a similar strategy to~\cite{MateBoundary}
and decompose $g$ into the heat flow of its initial condition, which we handle with integration operators requiring auxiliary distributions $\omega$ as input (the $\omega$ are obtained through stochastic estimates),
and the remainder with improved behaviour at $t=0$, which we handle with special spaces of modelled distributions described in Appendix~\ref{app:Singular modelled distributions}.

We finally turn to the associated Markov process on gauge orbits.
To state the way in which our Markov process is canonical
we introduce a particular class of $\Omega^1_\alpha$-valued processes
which essentially captures the ``nice'' ways to run the SPDE~\eqref{eq:SPDE_for_A}
and restart it from different representatives of gauge orbits.

For a metric space $X$, denote by $D_0(\R_+,X)$ 
the space of functions $A\colon \R_+\to X$ which are c{\`a}dl{\`a}g on $(0,\infty)$ and for which $A(0+)\eqdef \lim_{t\downarrow0}A(t)$ exists.
We can naturally identify $D_0(\R_+,X)$ with $X\times D(\R_+,X)$, where $D$ is the usual Skorokhod space of c{\`a}dl{\`a}g functions, and we equip $D_0$ with the metric induced by this identification.
For the remainder of this section, by a ``white noise'' we mean a pair of i.i.d.\ $\mfg$-valued
white noises $\xi=(\xi_1,\xi_2)$ on $\R\times \T^2$.

\begin{definition}\label{def:generative}
Setting $\hat\Omega^1_\alpha \eqdef \Omega^{1}_\alpha\cup\{\skull\}$,
a probability measure $\mu$ on $D_0(\R_+,\hat\Omega^{1}_\alpha)$ is called generative if 
there exists a filtered probability space $(\mcO,\mcF,(\mcF_t)_{t \geq 0}, \P)$
supporting a white noise $\xi$
for which the filtration $(\mcF_t)_{t \geq 0}$ is admissible
(i.e.\ $\xi$ is adapted to $(\mcF_t)_{t \geq 0}$ and $\xi\restr[t,\infty)$ is independent of $\mcF_t$ for all $t \geq 0$),
and a random variable $A\colon\mcO\to D_0(\R_+,\hat\Omega^{1}_\alpha)$
with the following properties.
\begin{enumerate}
\item The law of $A$ is $\mu$ and $A(0)$ is $\mcF_0$-measurable.

\item\label{pt:fix_at_zero}
There exists an $\mcF_0$-measurable random variable $g_{0}\colon\mcO\to\mfG^{0,\alpha}$ such that $A(0+) = A(0)^{g_{0}}$.
(We use the convention $\skull^g=\skull$ for all $g\in\mfG^{0,\alpha}$.)

\item\label{pt:dynamics} For any $0\leq s \leq t$, let
$\Phi_{s,t} \colon \hat\Omega^1_\alpha \to \hat\Omega^1_\alpha$ denote the (random)
solution map in the $\eps\to0$ limit of~\eqref{eq:SPDE_for_A}
with a non-anticipative mollifier $\moll$ and constant $C=\bar C$ from part~\ref{pt:gauge_covar} of
Theorem~\ref{thm:gauge_covar}.\footnote{Note that $\Phi$ exists and is independent of the choice of $\moll$
by part~\ref{pt:indep_moll} of Theorem~\ref{thm:gauge_covar}.}
There exists a non-decreasing sequence of stopping times $(\sigma_j)_{j=0}^\infty$,
such that $\sigma_0=0$ almost surely and, for all $j \geq 0$,
\begin{enumerate}
\item\label{pt:solves_SYM} 
$A(t) = \Phi_{\sigma_j,t}(A(\sigma_j+))$ for all $t\in[\sigma_j,\sigma_{j+1})$, and

\item\label{pt:gauge_at_jumps} there exists an $\mcF_{\sigma_{j+1}}$-measurable random variable $g_{j+1}\colon\mcO\to\mfG^{0,\alpha}$  such that
$A(\sigma_{j+1}) = \Phi_{\sigma_j,\sigma_{j+1}}(A(\sigma_j+))^{g_{j+1}}$.
\end{enumerate}

\item\label{pt:honest_blow_up} Let $T^* \eqdef \inf\{t \geq 0 \ssep A(t) = \skull\}$.
Then a.s. $\lim_{j\to\infty}\sigma_j=T^*$.
Furthermore, on the event $T^*<\infty$, a.s. $A\equiv \skull$ on $[T^*,\infty)$ and
\begin{equ}
\lim_{t\nearrow T^*}\inf_{g\in \mfG^{0,\alpha}} |A(t)^g|_\alpha = \infty\;.
\end{equ}
\end{enumerate}
If there exists $a\in\hat\Omega^1_\alpha$ such that 
$A(0)=a$ almost surely, then we call $a$ the initial condition of $\mu$.
\end{definition}

\begin{remark}
In the setting of Definition~\ref{def:generative},
if $B\colon\mcO\to\hat\Omega^1_\alpha$ is $\mcF_s$-measurable, 
then $t\mapsto \Phi_{s,t}(B)$ is adapted to $(\mcF_t)_{t \geq 0}$.
In particular, the conditions on the process $A$ imply that $A$ is adapted to $(\mcF_t)_{t \geq 0}$.
\end{remark}

Denote as before
$\hat\mfO_\alpha \eqdef \mfO_\alpha\sqcup\{\skull\}$
and let $\pi\colon\hat\Omega^1_\alpha\to\hat\mfO_\alpha$ be the canonical projection.
Note that if $\mu$ is generative and $\hat\mfO_\alpha$ is equipped with the quotient topology, then the pushforward $\pi_*\mu$ is a probability measure
on $\mcC(\R_+,\hat\mfO_\alpha)$ (rather than just on $D_0(\R_+,\hat\mfO_\alpha)$)
thanks to items~\ref{pt:fix_at_zero},~\ref{pt:solves_SYM}, and~\ref{pt:gauge_at_jumps} of Definition~\ref{def:generative}.

We equip $\mfO_\alpha$ with the metric $D_\alpha$ from Theorem~\ref{thm:state_space}.
Note that the metric on $\hat\mfO_\alpha$ from Section~\ref{subsubsec:blowup}
induces a topology finer than the quotient topology, so it is not automatic that $\pi_*\mu$ is a probability measure on $\mfO_\alpha^\sol$.
The fact that this is case is part of the following theorem which shows the existence
of the Markov process on the space of gauge orbits announced in the introduction.

%

%
\begin{theorem}\label{thm:Markov_process}
\begin{enumerate}[label=(\roman*)]
\item \label{pt:generative_exists} For every $a\in\hat\Omega_{\alpha}^1$, there exists a generative 
probability measure $\mu$ with initial condition $a$.
Moreover, one can take in Definition~\ref{def:generative}
$(\mcF_t)_{t \geq 0}$ to be the filtration generated 
by any white noise 
and the process $A$ itself to be Markov.

\item \label{pt:unique_Markov}
There exists a unique family of probability measures $\{\P^x\}_{x\in\hat\mfO_\alpha}$
such that $\pi_*\mu=\P^x$
for every $x\in \hat\mfO_\alpha$, $a\in x$,
and generative probability measure $\mu$ with initial condition $a$.
Furthermore, $\{\P^x\}_{x\in\hat\mfO_\alpha}$ define the transition probabilities of a time homogeneous Markov process on $\hat\mfO_\alpha$.
\end{enumerate}
\end{theorem}
The proof of Theorem~\ref{thm:Markov_process} is given at the end of Section~\ref{subsec:Markov_process}.
\section{Construction of the state space}
\label{sec:state_space}
The aim of this section is to find a space of distributional $1$-forms and a corresponding group of gauge transformations which can be used to construct the state space for our Markov process.
We would like our space to be sufficiently large to contain $1$-forms with components that ``look like'' a free field, but sufficiently
small that there is a meaningful notion of integration along smooth enough curves.
Our space of $1$-forms is a strengthened version of that constructed in~\cite{Chevyrev18YM}, the main difference being that we 
do not restrict our notion of integration to axis-parallel paths.
\subsection{Additive functions on line segments}
Let $\mcX$\label{mcX page ref} denote the set of oriented line segments in $\T^2$ of length at most $\frac14$.
Specifically, denoting $B_r \eqdef \{v \in \R^2 \ssep |v| \leq r\}$, we define $\mcX \eqdef \T^2 \times B_{1/4}$ (first coordinate is the initial point, second coordinate is the direction).
We fix for the remainder of this section a Banach space $E$.\label{E page ref}

\begin{definition}
We say that $\ell=(x,v),\bar\ell=(\bar x,\bar v)\in \mcX$ are joinable if $\bar x = x+v$ and there exist $w\in\R^2$ and $c,\bar c\in [-\frac14,\frac14]$ such that $|w|=1$, $v=cw$, $\bar v = \bar c w$, and $|c+\bar c| \leq \frac14$.
In this case, we denote $\ell\sqcup\bar\ell \eqdef (x,(c+\bar c)w) \in \mcX$.
We say that a function $A\colon\mcX\to E$ is additive if $A(\ell\sqcup\bar\ell) = A(\ell)+A(\bar\ell)$ for all joinable $\ell,\bar\ell \in \mcX$.
Let $\Omega=\Omega(\T^2,E)$\label{Omega page ref} denote the space of all measurable $E$-valued additive functions on $\mcX$.
\end{definition}
Note that additivity implies that $A(x,0) = 0$ for all $x\in\T^2$ and $A\in\Omega$.
\begin{remark}\label{rem:line_integral}
For $A\in\Omega$, one should think of $A(\ell)$ as the line integral along $\ell$ of a homogeneous function on the tangent bundle of $\T^2$.
To wit, any measurable function $B\colon\T^2\times \R^2 \to E$ which is bounded on $\T^2\times B_1$ and homogeneous in the sense that $B(x,cv)=cB(x,v)$ for all $(x,v)\in\T^2\times \R^2$ and $c\in\R$, defines an element $A\in\Omega$ by
\begin{equ}
A(x,v) \eqdef \int_0^1 B(x+tv, v) \mrd t\;.
\end{equ}
We will primarily be interested in the case that $B$ is a $1$-form, i.e., $B(x,v)$ is linear in $v$, and we discuss this situation in Section~\ref{subsec:smooth_one_forms}.
However, many definitions and estimates turn out to be more natural in the general setting of $\Omega$.
\end{remark}

For $\ell = (x,v) \in \mcX$, let us denote by $\ell_i \eqdef x$ and $\ell_f \eqdef x+v$ the initial and final point of $\ell$ respectively.
We define a metric $d$ on $\mcX$ by
\begin{equ}
d(\ell,\bar\ell) \eqdef |\ell_i-\bar\ell_i|\vee|\ell_f-\bar\ell_f|\;.
\end{equ}
For $\ell = (x,v) \in \mcX$, let $|\ell| \eqdef |v|$ denote its length.

\begin{definition}\label{def:rho}
We say that $\ell,\bar\ell \in \mcX$ are \emph{far} if $d(\ell,\bar\ell) > \frac14 (|\ell| \wedge |\bar\ell|)$.
Define the function $\rho \colon \mcX^2 \to [0,\infty)$ by\label{rho page ref}
\begin{equ}
\rho(\ell,\bar\ell) \eqdef
\begin{cases}
|\ell| + |\bar\ell| &\text{if $\ell,\bar\ell$ are far,}
\\
|\ell_i - \bar\ell_i| + |\ell_f-\bar\ell_f|  + \Area(\ell,\bar\ell)^{1/2} &\text{otherwise,}
\end{cases}
\end{equ}
where $\Area(\ell,\bar\ell)$ is the area of the convex hull of of the points $(\ell_i,\ell_f,\bar\ell_f,\bar\ell_i)$ (which is well-defined\footnote{Since we are on the torus, the definition of this ``convex hull'' can be ambiguous if $\ell,\bar\ell \in \mcX$ are far.} whenever $\ell,\bar\ell$ are not far), see the last example of Figure~\ref{fig:arrows}.
\end{definition}

\begin{figure}[h]
\centering
\begin{tikzpicture}[scale=1.1]
\draw[white] (2.2,0.8) -- (2.2,0.9);
\fill[blue!5] (0,0) circle (.5);
\fill[blue!5] (2,0) circle (.5);
\draw[thick,->] (0,0) to (2,0);
\fill[black] (0,0) circle (.03);
\draw[thick,gray,->] (-.3,.2) to (2.2,0.8);
\fill[gray] (-.3,.2) circle (.03);
\end{tikzpicture}
\hspace{1cm}
\begin{tikzpicture}[scale=1.1]
\fill[blue!5] (0,0) circle (.5);
\fill[blue!5] (2,0) circle (.5);
\draw[thick,->] (0,0) to (2,0);
\fill[black] (0,0) circle (.03);
\draw[thick,gray,<-] (-.3,.2) to (2.1,-0.2);
\fill[gray] (2.1,-.2) circle (.03);
\end{tikzpicture}
\hspace{1cm}
\begin{tikzpicture}[scale=1.1]
\fill[blue!5] (0,0) circle (.5);
\fill[blue!5] (2,0) circle (.5);
\fill[blue!15] (0,0) -- (-.25,.3) -- (2,0) -- (2.1,-0.35) -- (0,0);
\draw[thick,->] (0,0) to (2,0);
\fill[black] (0,0) circle (.03);
\draw[thick,gray,->] (-.25,.3) to (2.1,-0.35);
\fill[gray] (-.25,.3) circle (.03);
\end{tikzpicture}
\caption{Given a line segment $\ell$ shown as a black arrow, another (gray) line segment $\bar\ell$
with $|\bar\ell| \ge |\ell|$ is far from $\ell$ if and only if its endpoints don't both lie in the light blue circles (of radius $|\ell|/4$) around the corresponding endpoints of $\ell$. For instance, $\ell$ and
$\bar\ell$ shown in the first two pictures above are far, but they aren't in the last one.
In the last picture, the shaded quadrilateral is the convex hull of the endpoints.}\label{fig:arrows}
\end{figure}
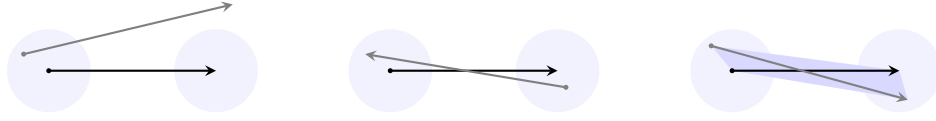

\begin{remark}\label{rem:inframetric}
If $\ell,\bar\ell\in\mcX$ are not far, then their lengths are of the same order, and $\Area(\ell,\bar\ell)$ is of the same order as $|\ell|\big[d(\bar\ell_i,\ell) + d(\bar\ell_f,\ell)\big]$, where, denoting $\ell=(x,v)$, we have set $d(y,\ell) \eqdef \inf_{t \in [-1,2]} |x+tv - y|$ (note the set $[-1,2]$ in the infimum instead of $[0,1]$).
In particular, it readily follows that although $\rho$ isn't a metric in general, it is a semimetric admitting a constant $C \geq 1$ such that for all $a,b,c\in\mcX$
\begin{equ}\label{eq:inframetric}
\rho(a,b) \leq C (\rho(a,c) + \rho(b,c))\;.
\end{equ}
\end{remark}

For $\alpha\in[0,1]$, we define the (extended) norm on $\Omega$\label{norm_alpha page ref}
\begin{equ}[e:normA]
|A|_\alpha \eqdef \sup_{\rho(\ell,\bar\ell)>0}\frac{|A(\ell)-A(\bar\ell)|}{\rho(\ell,\bar\ell)^\alpha}\;.
\end{equ}
We also write $\Omega_\alpha$\label{Omega_alpha page ref} for 
the Banach space $\{A \in \Omega \ssep |A|_\alpha < \infty\}$ equipped with the norm $|\cdot|_\alpha$.

\begin{remark}
By additivity, any element of $\Omega$ extends uniquely to an additive function on
all line segments, not just those of length less than $1/4$ and we will use this extension in the sequel 
without further mention. However, the supremum in \eqref{e:normA} is restricted to these ``short''
line segments.
\end{remark}

\begin{remark}
Since we know that $A \in \Omega$ vanishes on line segments of zero length, 
it follows that $|A(\ell)| \le |A|_\alpha |\ell|^\alpha$, so that 
despite superficial appearances \eqref{e:normA} is a norm on $\Omega_\alpha$ and not just a seminorm.
\end{remark}

We now introduce several other (semi)norms which will be used in the sequel.

\begin{definition}
Define the (extended) norm on $\Omega$
\begin{equ}
|A|_{\gr\alpha} \eqdef \sup_{|\ell| > 0}\frac{|A(\ell)|}{|\ell|^\alpha}\;.
\end{equ}
Let $\Omega_{\gr\alpha}$ denote the Banach space $\{A \in \Omega \ssep |A|_{\gr\alpha} < \infty\}$ equipped with the norm $|\cdot|_{\gr\alpha}$.
\end{definition}

\begin{definition}\label{def:alpha_vee}
We say that a pair $\ell,\bar\ell \in \mcX$ \emph{form a vee} if they are not far, have the same length $|\ell|=|\bar\ell|$, and have the same initial point $\ell_i=\bar\ell_i$.
Define the (extended) seminorm on $\Omega$
\begin{equ}
|A|_{\v\alpha} \eqdef \sup_{\ell\neq\bar\ell}\frac{|A(\ell)-A(\bar\ell)|}{\Area(\ell,\bar\ell)^{\alpha/2}}\;,
\end{equ}
where the supremum is taken over all distinct $\ell,\bar\ell \in \mcX$ forming a vee.
\end{definition}

\begin{definition}
For a line segment $\ell = (x,v) \in \mcX$, let us denote the associated subset of $\T^2$ by
\begin{equ}
\iota(\ell) \eqdef \iota(x,v) \eqdef \{x+cv\,:\, c \in [0,1)\}\;.
\end{equ}
For an integer $n \geq 3$, an \emph{$n$-gon} is a tuple $P = (\ell^1,\ldots, \ell^n) \in \mcX^n$ such that 
\begin{itemize}
\item $\ell^1_i=\ell^n_f$, and $\ell^j_i = \ell^{j-1}_f$ for all $j=2,\ldots, n$,
\item $\iota(\ell^j)\cap\iota(\ell^k) = \emptyset$ for all distinct $j,k \in \{ 1,\ldots, n\}$, and
\item $\iota(\ell^1)\cup\ldots\cup\iota(\ell^n)$ has diameter at most $\frac14$.
\end{itemize}
(See Figure~\ref{fig:shapes} for examples.) A $3$-gon is called a \emph{triangle}.
\end{definition}

Note that an $n$-gon $P$ splits $\T^2$ into two connected components, one of which is simply connected and we denote by $\mathring P$.
We further note that this split allows us to define when two $n$-gons have the same orientation.
For measurable subsets $X,Y \in \T^2$, let $X\triangle Y$ denote their symmetric difference, and $|X|$ denote the Lebesgue measure of $X$.
For $n$-gons $P_1, P_2$, let us denote $|P_1| \eqdef |\mathring P_1|$ and
\begin{equ}
|P_1; P_2| \eqdef
\begin{cases}
|\mathring P_1 \triangle \mathring P_2| & \text{ if $P_1,P_2$ have the same orientation}
\\
| P_1| + | P_2| & \text{ otherwise}\;,
\end{cases}
\end{equ}
(see Figure~\ref{fig:shapes}) which we observe defines a metric on the set of $n$-gons.

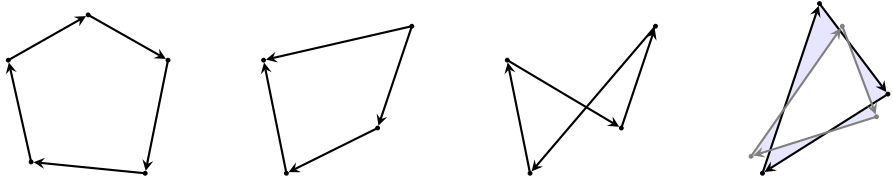
\begin{figure}[h]
\centering
\begin{tikzpicture}[scale=1.5]
\coordinate[dot] (1) at (0,0.1);
\coordinate[dot] (2) at (-0.2,1);
\coordinate[dot] (3) at (0.5,1.4);
\coordinate[dot] (4) at (1.2,1);
\coordinate[dot] (5) at (1,0);
\draw[thick,->] (1.center) -- (2);
\draw[thick,->] (2.center) -- (3);
\draw[thick,->] (3.center) -- (4);
\draw[thick,->] (4.center) -- (5);
\draw[thick,->] (5.center) -- (1);
\end{tikzpicture}
\hspace{1cm}
\begin{tikzpicture}[scale=1.5]
\coordinate[dot] (1) at (0,0);
\coordinate[dot] (2) at (-0.2,1);
\coordinate[dot] (3) at (1.1,1.3);
\coordinate[dot] (4) at (0.8,0.4);
\draw[thick,->] (1.center) -- (2);
\draw[thick,->] (3.center) -- (2);
\draw[thick,->] (3.center) -- (4);
\draw[thick,->] (4.center) -- (1);
\end{tikzpicture}
\hspace{1cm}
\begin{tikzpicture}[scale=1.5]
\coordinate[dot] (1) at (0,0);
\coordinate[dot] (2) at (-0.2,1);
\coordinate[dot] (3) at (1.1,1.3);
\coordinate[dot] (4) at (0.8,0.4);
\draw[thick,->] (1.center) -- (2);
\draw[thick,->] (2.center) -- (4);
\draw[thick,->] (4.center) -- (3);
\draw[thick,->] (3.center) -- (1);
\end{tikzpicture}
\hspace{1cm}
\begin{tikzpicture}[scale=1.5]
\coordinate (1c) at (0,0);
\coordinate (2c) at (0.5,1.5);
\coordinate (3c) at (1.1,0.7);

\coordinate (1bc) at (-0.1,0.15);
\coordinate (2bc) at (0.7,1.3);
\coordinate (3bc) at (1,0.5);

\fill[blue!10] (1c) -- (2c) -- (3c) -- (1c);
\fill[blue!10] (1bc) -- (2bc) -- (3bc) -- (1bc);

\begin{scope}
\clip (1c) -- (2c) -- (3c) -- (1c);
\fill[white] (1bc) -- (2bc) -- (3bc) -- (1bc);
\end{scope}

\coordinate[dot] (1) at (1c);
\coordinate[dot] (2) at (2c);
\coordinate[dot] (3) at (3c);

\coordinate[dot,gray] (1b) at (1bc);
\coordinate[dot,gray] (2b) at (2bc);
\coordinate[dot,gray] (3b) at (3bc);

\draw[thick,->] (1c) -- (2);
\draw[thick,->] (2c) -- (3);
\draw[thick,->] (3c) -- (1);

\draw[thick,gray,->] (1bc) -- (2b);
\draw[thick,gray,->] (2bc) -- (3b);
\draw[thick,gray,->] (3bc) -- (1b);

\end{tikzpicture}
\caption{The first collection of segments forms a $5$-gon, while the second and third do
\textit{not} form $4$-gons. On the right, the distance $|P_1;P_2|$ between the black triangle $P_1$ and grey triangle $P_2$ 
triangle is given by the area shaded in blue.}\label{fig:shapes}
\end{figure}

\begin{definition}
Let $P=(\ell^1,\ldots,\ell^n)$ be an $n$-gon.
For $A\in\Omega$, we denote
\begin{equ}
A(\partial P) \eqdef \sum_{j=1}^n A(\ell^j)\;.
\end{equ}
For $\alpha \in [0,1]$ we define the quantities
\begin{equ}[e:normP]
|A|_{\tri\alpha} \eqdef \sup_{|P|>0} \frac{|A(\partial P)|}{|P|^{\alpha/2}}\;,
\end{equ}
where the supremum is taken over all triangles $P$ with $|P|>0$, and
\begin{equ}
|A|_{\symm\alpha} \eqdef \sup_{|P;\bar P|>0} \frac{|A(\partial P) - A(\partial \bar P)|}{|P;\bar P|^{\alpha/2}}\;,
\end{equ}
where the supremum is taken over all triangles $P,\bar P$ with $|P;\bar P|>0$.
\end{definition}

The motivation behind each norm is the following.
\begin{itemize}
\item The norm $|\cdot|_{\v\alpha}$ facilitates the analysis of gauge transformations (Section~\ref{subsec:gauge_transforms}).
\item The norm $|\cdot|_{\symm\alpha}$ is helpful in extending the domain of definition of $A \in \Omega_\alpha$ to a wider class of curves (Section~\ref{subsec:regular_curves}).
\item The norm $|\cdot|_{\tri\alpha}$ is simpler but equivalent to $|\cdot|_{\symm\alpha}$.
Furthermore, the values $A(\partial P)$ can be evaluated using Stokes' theorem (e.g., as in Lemma~\ref{lem:Kolmog_tri_bound}).
\end{itemize}
We show now that each of these norms, when combined with $|\cdot|_{\gr\alpha}$, is equivalent to $|\cdot|_\alpha$.

\begin{theorem}\label{thm:norm_equiv}
There exists $C \geq 1$ such that for all $\alpha \in [0,1]$ and $A \in \Omega$
\begin{equ}
C^{-1} |A|_\alpha \leq |A|_{\gr\alpha} + |A|_{\bullet} \leq C|A|_\alpha\;,
\end{equ}
where $\bullet$ is any one of $\v\alpha$, $\tri\alpha$, or $\symm\alpha$.
\end{theorem}

For the proof, we require the following lemmas.

\begin{lemma}\label{lem:n-gon_tri_bound}
For $\alpha \in [0,1]$ and $n \geq 3$, it holds that
\begin{equ}
\sup_{|P|>0} \frac{|A(\partial P)|}{|P|^{\alpha/2}} \leq C_n |A|_{\tri\alpha}\;,
\end{equ}
where the supremum is taken over all $n$-gons $P$ with $|P|>0$,
and where $C_3 \eqdef 1$ and for $n \geq 4$
\begin{equ}
C_n \eqdef \frac{C_{n-1}+(C_{n-1}^{-2/(2-\alpha)})^{\alpha/2}}{(1+C_{n-1}^{-2/(2-\alpha)})^{\alpha/2}}\;.
\end{equ}
\end{lemma}

\begin{proof}
This readily follows by induction, using the two ears theorem \cite{Ears} (every polygon admits a triangulation)
and the fact that $C_n$ is the optimal constant 
such that 
$x^{\alpha/2}+C_{n-1}y^{\alpha/2} \leq C_n (x+y)^{\alpha/2}$ for all $x,y \geq 0$.
\end{proof}

\begin{lemma}\label{lem:symm_tri_bound}
There exists $C \geq 1$ such that for all $\alpha \in [0,1]$ and $A \in \Omega$
\begin{equ}
|A|_{\tri\alpha} \leq |A|_{\symm\alpha} \leq C |A|_{\tri\alpha}\;.
\end{equ}
\end{lemma}

\begin{proof}
The first inequality is obvious by taking $\bar P$ in the definition of $|A|_{\symm\alpha}$ as any degenerate triangle.
For the second, let $P_1,P_2$ be two triangles.
We need only consider the case that $P_1,P_2$ are oriented in the same direction.
Observe that there exist $k \leq 6$ and $Q_1,\ldots, Q_k$, where each $Q_i$ is an $n$-gon with $n \leq 7$,
such that $|\mathring P_1 \triangle \mathring P_2| = \sum_{i=1}^k |\mathring Q_i|$, and such that 
$A(\partial P_1) - A(\partial P_2) = \sum_{i=1}^k A(\partial Q_i)$ (see Figure~\ref{fig:n-gons}).
It then follows from Lemma~\ref{lem:n-gon_tri_bound} that
\begin{equ}
|A(\partial P_1) - A(\partial P_2)| \lesssim |A|_{\tri\alpha} \sum_{i=1}^k|\mathring Q_i|^{\alpha/2} \lesssim  |A|_{\tri\alpha}|\mathring P_1 \triangle \mathring P_2|^{\alpha/2}\;,
\end{equ}
as required.
\end{proof}

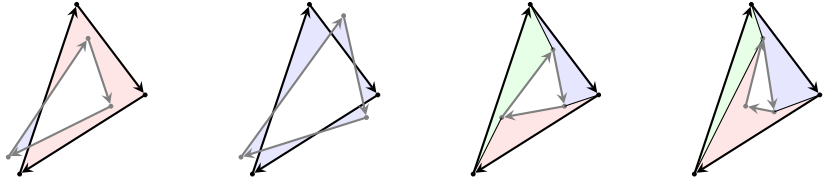
\begin{figure}[h]
\centering
\begin{tikzpicture}[scale=1.5]
\coordinate (1c) at (0,0);
\coordinate (2c) at (0.5,1.5);
\coordinate (3c) at (1.1,0.7);

\coordinate (1bc) at (-0.1,0.15);
\coordinate (2bc) at (0.6,1.2);
\coordinate (3bc) at (0.8,0.6);

\fill[red!10] (1c) -- (2c) -- (3c) -- (1c);
\fill[blue!10] (1bc) -- (2bc) -- (3bc) -- (1bc);

\begin{scope}
\clip (1c) -- (2c) -- (3c) -- (1c);
\fill[white] (1bc) -- (2bc) -- (3bc) -- (1bc);
\end{scope}

\coordinate[dot] (1) at (1c);
\coordinate[dot] (2) at (2c);
\coordinate[dot] (3) at (3c);

\coordinate[dot,gray] (1b) at (1bc);
\coordinate[dot,gray] (2b) at (2bc);
\coordinate[dot,gray] (3b) at (3bc);

\draw[thick,->] (1c) -- (2);
\draw[thick,->] (2c) -- (3);
\draw[thick,->] (3c) -- (1);

\draw[thick,gray,->] (1bc) -- (2b);
\draw[thick,gray,->] (2bc) -- (3b);
\draw[thick,gray,->] (3bc) -- (1b);

\end{tikzpicture}
\hspace{1cm}
\begin{tikzpicture}[scale=1.5]
\coordinate (1c) at (0,0);
\coordinate (2c) at (0.5,1.5);
\coordinate (3c) at (1.1,0.7);

\coordinate (1bc) at (-0.1,0.15);
\coordinate (2bc) at (0.8,1.4);
\coordinate (3bc) at (1,0.5);

\fill[blue!10] (1c) -- (2c) -- (3c) -- (1c);
\fill[blue!10] (1bc) -- (2bc) -- (3bc) -- (1bc);

\begin{scope}
\clip (1c) -- (2c) -- (3c) -- (1c);
\fill[white] (1bc) -- (2bc) -- (3bc) -- (1bc);
\end{scope}

\coordinate[dot] (1) at (1c);
\coordinate[dot] (2) at (2c);
\coordinate[dot] (3) at (3c);

\coordinate[dot,gray] (1b) at (1bc);
\coordinate[dot,gray] (2b) at (2bc);
\coordinate[dot,gray] (3b) at (3bc);

\draw[thick,->] (1c) -- (2);
\draw[thick,->] (2c) -- (3);
\draw[thick,->] (3c) -- (1);

\draw[thick,gray,->] (1bc) -- (2b);
\draw[thick,gray,->] (2bc) -- (3b);
\draw[thick,gray,->] (3bc) -- (1b);

\end{tikzpicture}
\hspace{1cm}
\begin{tikzpicture}[scale=1.5]
\coordinate (1c) at (0,0);
\coordinate (2c) at (0.5,1.5);
\coordinate (3c) at (1.1,0.7);

\coordinate (1bc) at (0.25,0.5);
\coordinate (2bc) at (0.7,1.1);
\coordinate (3bc) at (0.8,0.6);

\fill[green!10] (1c) -- (1bc) -- (2bc) -- (2c);
\fill[blue!10] (2c) -- (2bc) -- (3bc) -- (3c);
\fill[red!10] (3c) -- (3bc) -- (1bc) -- (1c);

\begin{scope}
\clip (1c) -- (2c) -- (3c) -- (1c);
\fill[white] (1bc) -- (2bc) -- (3bc) -- (1bc);
\end{scope}

\coordinate[dot] (1) at (1c);
\coordinate[dot] (2) at (2c);
\coordinate[dot] (3) at (3c);

\coordinate[dot,gray] (1b) at (1bc);
\coordinate[dot,gray] (2b) at (2bc);
\coordinate[dot,gray] (3b) at (3bc);

\draw[thick,->] (1c) -- (2);
\draw[thick,->] (2c) -- (3);
\draw[thick,->] (3c) -- (1);

\draw[] (1c) -- (1bc);
\draw[] (2c) -- (2bc);
\draw[] (3c) -- (3bc);

\draw[thick,gray,->] (1bc) -- (2b);
\draw[thick,gray,->] (2bc) -- (3b);
\draw[thick,gray,->] (3bc) -- (1b);
\end{tikzpicture}
\hspace{1cm}
\begin{tikzpicture}[scale=1.5]
\coordinate (1c) at (0,0);
\coordinate (2c) at (0.5,1.5);
\coordinate (3c) at (1.1,0.7);

\coordinate (1bc) at (0.45,0.6);
\coordinate (2bc) at (0.6,1.2);
\coordinate (3bc) at (0.7,0.55);

\fill[green!10] (1c) -- (2bc) -- (2c);
\fill[blue!10] (2c) -- (2bc) -- (3bc) -- (3c);
\fill[red!10] (3c) -- (3bc) -- (1bc) -- (2bc) -- (1c);

\begin{scope}
\clip (1c) -- (2c) -- (3c) -- (1c);
\fill[white] (1bc) -- (2bc) -- (3bc) -- (1bc);
\end{scope}

\coordinate[dot] (1) at (1c);
\coordinate[dot] (2) at (2c);
\coordinate[dot] (3) at (3c);

\coordinate[dot,gray] (1b) at (1bc);
\coordinate[dot,gray] (2b) at (2bc);
\coordinate[dot,gray] (3b) at (3bc);

\draw[thick,->] (1c) -- (2);
\draw[thick,->] (2c) -- (3);
\draw[thick,->] (3c) -- (1);

\draw[] (1c) -- (2bc);
\draw[] (2c) -- (2bc);
\draw[] (3c) -- (3bc);

\draw[thick,gray,->] (1bc) -- (2b);
\draw[thick,gray,->] (2bc) -- (3b);
\draw[thick,gray,->] (3bc) -- (1b);
\end{tikzpicture}
\caption{Examples of $Q_1,\ldots, Q_k$ appearing in the proof of Lemma~\ref{lem:symm_tri_bound}.
In the first figure, $k=2$ where $Q_1$ is bounded by blue shaded region and $Q_2$ is bounded by red shaded region; $Q_1$ is a triangle and $Q_2$ is a $7$-gon.
In the second figure, $k=6$ and all $Q_i$ are triangles bounded by blue shaded regions.
In the last two figures, $P_1$ is contained inside $P_2$, in which case we can choose
$k=3$ and $Q_1,Q_2,Q_3$ either as triangles, $4$-gons, or $5$-gons, shaded in blue, red, and green, and
given in this example by rotating clockwise each outer arrow until it hits the inner triangle.}\label{fig:n-gons}
\end{figure}

\begin{proof}[of Theorem~\ref{thm:norm_equiv}]
We show first
\begin{equ}\label{eq:init_alpha}
C^{-1} |A|_\alpha \leq |A|_{\gr\alpha} + |A|_{\v\alpha} \leq C|A|_\alpha\;.
\end{equ}
The second inequality in~\eqref{eq:init_alpha} is clear (without even assuming that $A$ is additive) since $|\ell| = \rho(\ell,\bar\ell)$ whenever $|\bar\ell|=0$, and for any $\ell,\bar\ell \in \mcX$ forming a vee, we have $\rho(\ell,\bar\ell) \lesssim \Area(\ell,\bar\ell)^{1/2}$.

It remains to show the first inequality in~\eqref{eq:init_alpha}.
If $\ell,\bar\ell$ are far, then clearly $|A(\ell)-A(\bar\ell)| \lesssim \rho(\ell,\bar\ell)^{\alpha}|A|_{\gr\alpha}$.
Supposing now that $\ell,\bar\ell$ are not far,
we want to show that
\begin{equ}\label{eq:A_gr_vee_bound}
|A(\ell)-A(\bar\ell)| \lesssim \rho(\ell,\bar\ell)^{\alpha}(|A|_{\gr\alpha}+|A|_{\v\alpha})\;.
\end{equ}
Consider the line segment $a$ with initial point $\ell_i$ and endpoint $\bar\ell_f$, and the line segment $\bar a \in \mcX$ such that $\bar a = (\ell_i,c(\bar\ell_f-\ell_i))$ for some $c>0$ and $|\bar a| = |\ell|$ (see Figure~\ref{fig:as}).
Note that it is possible that $|a|>\frac14$, and thus $a\notin\mcX$, however $A(a)$ still makes sense by additivity of $A$.
Observe that $|a_f - \bar a_f| \lesssim |\ell_f-\bar\ell_f|$ and $\Area(\ell,\bar a) \lesssim \Area(\ell,\bar\ell)$ (for the latter, note that $a$ is contained inside the convex hull of $\ell,\bar\ell$, and that $\bar a$ is at most twice the length of $a$).

\begin{figure}[h]
\centering
\begin{tikzpicture}[scale=1.6]
\coordinate (li) at (0,0);
\coordinate (lf) at (2,0);

\coordinate (bli) at (-.25,.3);
\coordinate (blf) at (2.2,-0.4);

\fill[blue!5] (li) circle (.5);
\fill[blue!5] (lf) circle (.5);

\draw[thick,->] (li) to (lf);

\draw[thick,gray,->] (bli) to (blf);

\fill[black] (li) circle (.03);
\fill[gray] (li) circle (.03);

\draw[red,thick,->] (li.-10.3) -- (-10.3:2);

\draw [black,blue,domain=-20:20] plot ({2*cos(\x)}, {2*sin(\x)});
\end{tikzpicture}
\caption{The black arrow represents $\ell$, the grey arrow represents $\bar \ell$, and the red arrow represents $\bar a$.
The blue arc represents the circle of radius $|\ell|$ centred at $\ell_i$.}\label{fig:as}
\end{figure}
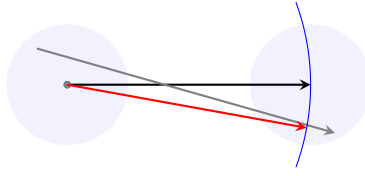

Observe that $\bar a$ and $\ell$ form a vee because $\bar\ell_f$ is in the ball $B\eqdef \{z\in\T^2\,:\,|z-\ell_f|\leq \frac14|\ell|\}$
and therefore $\bar a_f$, which is the intersection point of the ray from $\ell_i$ to $\bar \ell_f$ with the circle of radius $|\ell|$ centred at $\ell_i$, is also inside $B$.
Therefore, breaking up $A(a)$ into $A(\bar a)$ and a remainder, we see by additivity of $A$ that
\begin{equ}
|A(\ell) - A(a)| \lesssim |A|_{\v\alpha}\Area(\ell,\bar\ell)^{\alpha/2} + |A|_{\gr\alpha}|\ell_f-\bar\ell_f|^\alpha\;.
\end{equ}
By symmetry, one obtains
\begin{equ}
|A(\bar\ell)-A(a)| \lesssim |A|_{\v\alpha} \Area(\ell,\bar\ell)^{\alpha/2} + |A|_{\gr\alpha}|\ell_i-\bar\ell_i|^{\alpha}\;,
\end{equ}
which proves~\eqref{eq:A_gr_vee_bound}.

For the remaining inequalities, one can readily see that
\begin{equ}
|A|_{\v\alpha} \lesssim |A|_{\gr\alpha} + |A|_{\tri\alpha}\;,
\end{equ}
so that the claim follows if we can show that
\begin{equ}\label{eq:tri_alpha_bound}
|A|_{\tri\alpha} \lesssim |A|_{\alpha}\;.
\end{equ}
For this, consider a triangle $P = (\ell^1,\ell^2,\ell^3)$ and assume without loss of generality that $|\ell^1| \geq |\ell^2| \geq |\ell^3|$.
Suppose first that $P$ is right-angled.
If $\ell^1,\ell^2$ are far, then $\sum_{j=1}^3|\ell^j| \lesssim |P|^{1/2}$, while if $\ell^1,\ell^2$ are not far, then $\rho(\ell^1,\ell^2) \lesssim |P|^{1/2}$.
In either case, $|A(\partial P)| \leq |A|_\alpha |P|^{1/2}$.
For general $P$, we can split $P$ into two right-angled triangles $P_1,P_2$ with $|P_1|+|P_2| = |P|$ and $A(\partial P)=A(\partial P_1)+A(\partial P_2)$ and apply the previous case, which proves~\eqref{eq:tri_alpha_bound}.
The conclusion follows from Lemma~\ref{lem:symm_tri_bound}.
\end{proof}

For $A\in\Omega$ and $\ell=(x,v) \in\mcX$, define the function $\ell_A \colon [0,1] \to E$ by
\begin{equ}
\ell_A(t) \eqdef A(x,tv)\;.
\end{equ}

\begin{lemma}\label{lem:Hol_paths}
There exists a constant $C > 0$ such that for all $\alpha \in [0,1]$, $A \in \Omega$, and $\ell,\bar\ell \in \mcX$ forming a vee, one has
\begin{equ}\label{eq:ell_A_Hol}
|\ell_A|_{\Hol\alpha} \leq |\ell|^\alpha|A|_{\gr\alpha}\;,\qquad
|\ell_A - \bar\ell_A|_{\Hol{\frac \alpha 2}} \leq C\Area(\ell,\bar\ell)^{\alpha/2} |A|_{\alpha}\;.
\end{equ}
\end{lemma}

\begin{proof}
The first inequality is obvious by additivity of $A$.
For the second, let $0 \leq s < t \leq 1$ and denote by $\ell^{s,t}$ the sub-segment of $\ell = (x,v)$ with initial point $x+sv$ and final point $x+tv$.
We claim that
\begin{equ}\label{eq:rho_ell_st}
\rho(\ell^{s,t},\bar\ell^{s,t}) \lesssim |t-s|^{1/2}\Area(\ell,\bar\ell)^{1/2}\;.
\end{equ}
Indeed, observe that $\Area(\ell,\bar\ell) \asymp |\ell||\ell_f-\bar\ell_f|$ as a consequence of the fact that $|\ell| = |\bar \ell|$ and $\ell_i = \bar \ell_i$ by
the definition of ``forming a vee''. One furthermore has the identities $|\ell^{s,t}|=|\bar\ell^{s,t}|=|t-s||\ell|$, and $|\ell^{s,t}_f-\bar\ell^{s,t}_f| = t|\ell_f-\bar\ell_f|$.
Hence, if $\ell^{s,t},\bar\ell^{s,t}$ are far, then we must have
\begin{equ}
|t-s||\ell| \lesssim t|\ell_f-\bar\ell_f| \leq |\ell_f-\bar\ell_f|\;,
\end{equ}
and thus
\begin{equ}
\rho(\ell^{s,t},\bar\ell^{s,t}) = 2|t-s||\ell| \lesssim |t-s|^{1/2}\Area(\ell,\bar\ell)^{1/2}\;.
\end{equ}
On the other hand, if $\ell^{s,t},\bar\ell^{s,t}$ are not far, then we must have
\begin{equ}
t|\ell_f-\bar\ell_f| \lesssim |t-s||\ell| \;,
\end{equ}
and thus
\begin{equ}
t^{1/2}|\ell_f-\bar\ell_f| \lesssim |t-s|^{1/2}|\ell|^{1/2}|\ell_f-\bar\ell_f|^{1/2} \asymp |t-s|^{1/2}\Area(\ell,\bar\ell)^{1/2}\;.
\end{equ}
It follows that
\begin{equ}
\rho(\ell^{s,t},\bar\ell^{s,t}) \leq 2t|\ell_f-\bar\ell_f| + |t-s|^{1/2}\Area(\ell,\bar\ell)^{1/2} \lesssim |t-s|^{1/2}\Area(\ell,\bar\ell)^{1/2}\;,
\end{equ}
which proves~\eqref{eq:rho_ell_st}.
It follows that
\begin{equ}
|\ell_A(t)-\ell_A(s) - \bar\ell_A(t)+\bar\ell_A(s)| = |A(\ell^{s,t}) - A(\bar\ell^{s,t})| \lesssim |t-s|^{\alpha/2}\Area(\ell,\bar\ell)^{\alpha/2} |A|_{\alpha}\;,
\end{equ}
concluding the proof.
\end{proof}

\subsection{Extension to regular curves}
\label{subsec:regular_curves}

In this subsection we show that any element $A \in \Omega_\alpha$ extends to a well-defined functional on sufficiently regular curves $\gamma \colon [0,1] \to \T^2$.
Given that $A(\gamma)$ should be invariant under reparametrisation of $\gamma$, we first provide a way to
measure the regularity of $\gamma$ in a parametrisation invariant 
way, and later provide relations to more familiar spaces of paths (namely paths in $\mcC^{1,\beta}$).

For a function $\gamma \colon [s,t] \to \T^2$, we denote by
$\diam(\gamma) \eqdef \sup_{u,v\in[s,t]}|\gamma(u)-\gamma(v)|$ the diameter of $\gamma$.
We assume throughout this subsection that all 
functions $\gamma\colon[s,t] \to \T^2$ under consideration 
have diameter at most $\frac14$.

We call a \emph{partition} of an interval $[s,t]$ a finite collection of subintervals 
$D = \{[t_i,t_{i+1}] \ssep i \in \{0,\ldots, n-1\}\}$, with $t_0=s < t_1 < \ldots < t_{n-1} < t_n = t$, 
and we write $\CD([s,t])$ for the set of all partitions.
For a function $\gamma \colon [s,t] \to \T^2$ and $D \in \CD([s,t])$, let $\gamma^D$ be the piecewise affine interpolation of $\gamma$ along $D$.
Note that if $\gamma$ is piecewise affine, then there exists $D \in \CD([s,t])$ and elements $\ell_i = (x_i, v_i) \in \CX$ such
that, for $u \in [t_i,t_{i+1}]$, one has $\gamma(u) = x_i + v_i (u-t_i)/(t_{i+1}-t_i)$.
$A(\gamma)$ is then canonically defined by $A(\gamma) = \sum_i A(\ell_i)$. (This is independent of the choice of $t_i$ and $\ell_i$ parametrising $\gamma$.)

\begin{definition}
Let $A \in\Omega$ and $\gamma \colon [0,1] \to \T^2$.
We say that $A$ \emph{extends} to $\gamma$ if the limit
\begin{equ}\label{eq:A_gamma_limit}
\gamma_A(t) \eqdef \lim_{|D| \to 0} A(\gamma^D\restr_{[0,t]})
\end{equ}
exists for all $t \in [0,1]$,
where $D \in \CD([0,1])$ and
$|D|\eqdef \max_{[a,b] \in D}|b-a|$.
\end{definition}

The following definition provides a convenient, parametrisation invariant way to determine if a given $A \in \Omega$ extends to $\gamma$.

\begin{definition}\label{def:pvar}
Let $\gamma \colon [0,1] \to \T^2$ be a function.
The \emph{triangle process} associated to $\gamma$ is defined to be the function $P$ defined on $[0,1]^3$, taking values in the set of triangles, such that $P_{sut}$ is the triangle formed by $(\gamma(s),\gamma(u),\gamma(t))$.

For two functions $\gamma,\bar\gamma \colon [0,1] \to \T^2$, and a subinterval $[s,t] \subset [0,1]$, define
\begin{equ}
|\gamma;\bar\gamma|_{[s,t]} \eqdef \sup_{u \in [s,t]} |P_{sut} ; \bar P_{sut}|^{1/2}\;,
\end{equ}
where $P,\bar P$ are the triangle processes associated to $\gamma,\bar\gamma$ respectively.
For $\alpha \in [0,1]$, define further
\begin{equ}
|\gamma;\bar\gamma|_{\alpha;[s,t]} \eqdef \sup_{D \in \CD([s,t])} \sum_{[a,b] \in D} |\gamma;\bar\gamma|_{[a,b]}^{\alpha}\;.
\end{equ}
We denote $|\gamma|_{\alpha;[s,t]} \eqdef |\gamma;\bar\gamma|_{\alpha;[s,t]}$ where $\bar\gamma$ is any constant path.
We drop the reference to the interval $[s,t]$ whenever $[s,t]=[0,1]$.
\end{definition}

We note the following basic properties of $|\cdot;\cdot|_{\alpha}$:
\begin{itemize}
\item $|\cdot;\cdot|_{\alpha}$ is symmetric and satisfies the triangle inequality but defines only a pseudometric rather than a metric since any two affine paths are at distance $0$ from each other.

\item The map $\alpha\mapsto |\gamma;\bar\gamma|_{\alpha}$ is decreasing in $\alpha$ for any $\gamma,\bar\gamma\colon[0,1]\to \T^2$.

\item For a typical smooth curve, $|P_{sut}|$ is of order $|t-s|^3$ (cf.~\eqref{eq:P_Hol_bound} below).
It follows that $|\gamma|_\alpha<\infty$ for all smooth $\gamma\colon[0,1]\to\T^2$ if and only if $\alpha\geq \frac23$.
\end{itemize}

Recall (see, e.g., \cite[Def.1.6]{FV10}) that a \emph{control} is a continuous, super-additive function $\omega \colon \{(s,t) \ssep 0\leq s \leq t \leq 1\} \to \R_+$ such that $\omega(t,t) = 0$.
Here super-additivity means that $\omega(s,t)+\omega(t,u)\le \omega(s,u)$ for any $s\le t\le u$.

\begin{lemma}\label{lem:control}
Let $\gamma,\bar\gamma\in\mcC([0,1],\T^2)$ such that $|\gamma;\bar\gamma|_\alpha < \infty$.
Then $\omega \colon (s,t) \mapsto |\gamma;\bar\gamma|_{\alpha;[s,t]}$ is a control.
\end{lemma}

The proof of Lemma~\ref{lem:control} follows in the same way as the more classical statement that $(s,t) \mapsto |\gamma|^p_{\var p;[s,t]}$ is a control, see e.g.\ the proof of~\cite[Prop.~5.8]{FV10} (note that continuity is the only subtle part).

\begin{theorem}\label{thm:extension_to_curves}
Let $0 \leq \alpha < \bar\alpha \leq 1$ and denote $\theta \eqdef \bar\alpha/\alpha$.
Let $A \in\Omega$ with $|A|_{\symm{\bar\alpha}} < \infty$ and $\gamma \in \mcC([0,1],\T^2)$ such that $|\gamma|_\alpha < \infty$.
Then $A$ extends to $\gamma$ and for any partition $D$ of $[0,1]$
\begin{equ}\label{eq:A_gamma_D_bound}
|A(\gamma^D) - A(\gamma)| \leq 2^{\theta}\zeta(\theta)|A|_{\symm{\bar\alpha}}\sum_{[s,t] \in D} |\gamma|_{\alpha;[s,t]}^{\theta}\;,
\end{equ}
where $\zeta$ is the classical Riemann zeta function.
Let $\bar\gamma \in \mcC([0,1], \T^2)$ be another path such that $|\bar\gamma|_\alpha < \infty$.
Then
\begin{equ}\label{eq:A_gamma_bar_gamma_bound}
|A(\gamma) - A(\bar\gamma)| \leq |A(\ell)-A(\bar\ell)| + 2^{\theta}\zeta(\theta)|A|_{\symm{\bar\alpha}} |\gamma;\bar\gamma|_\alpha^{\theta}\;,
\end{equ}
where $\ell,\bar\ell \in\mcX$ are the line segments connecting $\gamma(0),\gamma(1)$ and $\bar\gamma(0),\bar\gamma(1)$ respectively.
\end{theorem}

\begin{proof}
Define $\omega(s,t) \eqdef |\gamma;\bar\gamma|_{\alpha;[s,t]}$, which we note is a control by Lemma~\ref{lem:control}.
Let $D$ be a partition of $[0,1]$.
We will apply Young's partition coarsening argument to show that
\begin{equ}\label{eq:gamma_D_unif_bound}
|A(\gamma^D) - A(\bar \gamma^D)| \leq |A(\ell)-A(\bar\ell)| + |\gamma;\bar\gamma|_{\alpha}^{\theta}2^{\theta}\zeta(\theta) |A|_{\symm{\bar\alpha}}\;.
\end{equ}
Let $n$ denote the number of points in $D$.
If $n=2$, then the claim is obvious.
If $n \geq 3$, then by superadditivity of $\omega$ there exist two adjacent subintervals $[s,u],[u,t] \in D$ such that $\omega(s,t) \leq 2\omega(0,1)/(n-1)$.
Let $P,\bar P$ denote the triangle process associated with $\gamma,\bar\gamma$ respectively.
Observe that
\begin{equs}
|A(\partial P_{sut}) - A(\partial \bar P_{sut})|
&\leq |\gamma;\bar\gamma|_{[s,t]}^{\bar\alpha}|A|_{\symm{\bar\alpha}}
\leq \omega(s,t)^{\theta} |A|_{\symm{\bar\alpha}}
\\
&\leq (2\omega(0,1)/(n-1))^{\theta}|A|_{\symm{\bar\alpha}}\;.
\end{equs}
Merging the intervals $[s,u],[u,t] \in D$ into $[s,t]$ yields a coarser partition $D'$ and we see that
\begin{equs}
\big|A(\gamma^D)-A(\bar\gamma^D) - \big(A(\gamma^{D'}) - A(\bar \gamma^{D'})\big)\big|
&= |A(\partial P_{sut}) - A(\partial \bar P_{sut})|
\\
&\leq (2\omega(0,1)/(n-1))^{\theta} |A|_{\symm{\bar\alpha}}\;.
\end{equs}
Proceeding inductively, we obtain~\eqref{eq:gamma_D_unif_bound}.
It remains only to show that~\eqref{eq:A_gamma_limit} exists for $t=1$ and satisfies~\eqref{eq:A_gamma_D_bound}.
By Lemma~\ref{lem:control}, we have
\begin{equ}
\lim_{\eps \to 0} \sup_{|D| < \eps} \sum_{[s,t] \in D} |\gamma|_{\alpha;[s,t]}^{\theta} 
 = 0\;.
\end{equ}
Observe that if $D'$ is a refinement of $D$, then we can apply the uniform bound~\eqref{eq:gamma_D_unif_bound} to every $[s,t]\in D$ to obtain
\begin{equ}
|A(\gamma^D)-A(\gamma^{D'})| \leq 2^{\theta}\zeta(\theta)|A|_{\symm{\bar\alpha}}\sum_{[s,t] \in D}|\gamma|_{\alpha;[s,t]}^{\theta}\;,
\end{equ}
from which the existence of~\eqref{eq:A_gamma_limit} and the bound~\eqref{eq:A_gamma_D_bound} follow.
\end{proof}

For a metric space $(X,d)$ and $p\in[1,\infty)$, recall that the $p$-variation $|x|_{\var p}$ of a path $x\colon[0,1]\to X$ is given by
\begin{equ}
|x|_{\var p}^p \eqdef \sup_{D \in \CD([0,1])} \sum_{[s,t] \in D} d(x(s),x(t))^p\;.
\end{equ}
Our interest in $p$-variation stems from the Young integral~\cite{Young, Lyons94, FrizHairer} which ensures that ODEs driven by finite $p$-variation paths are well-defined.

\begin{corollary}\label{cor:p-var_bound}
Let $0 \leq \alpha < \bar \alpha \leq 1$, $\eta \in (0,1]$, and $p \geq \frac1\eta$. Consider $\gamma \in \mcC([0,1],\T^2)$ with $|\gamma|_\alpha < \infty$ and $A \in \Omega$ with $|A|_{\symm{\bar\alpha}} + |A|_{\gr\eta} < \infty$.
Then
\begin{equ}
|\gamma_A|_{\var p} \leq |A|_{\gr\eta} |\gamma|^{\eta}_{\var{p\eta}} + 2^{\bar\alpha/\alpha}\zeta(\bar\alpha/\alpha) |A|_{\symm{\bar\alpha}}|\gamma|_{\alpha}^{\bar\alpha/\alpha}\;.
\end{equ}
\end{corollary}


\begin{proof}
For any $[s,t] \subset [0,1]$,~\eqref{eq:A_gamma_bar_gamma_bound} implies that
\begin{equ}
|\gamma_A(t)-\gamma_A(s)| \leq |A|_{\gr\eta} |\gamma(t)-\gamma(s)|^{\eta} +  2^{\bar\alpha/\alpha}\zeta(\bar\alpha/\alpha) |A|_{\symm{\bar\alpha}}|\gamma|_{\alpha;[s,t]}^{\bar\alpha/\alpha}\;,
\end{equ}
from which the conclusion follows by Minkowski's inequality.
\end{proof}

The following result provides a convenient (now parametrisation dependent) way to control the quantity $|\gamma;\bar\gamma|_\alpha$.
For $\beta \in [0,1]$, let $\mcC^{1,\beta}([0,1],\T^2)$ denote the space of differentiable functions $\gamma \colon [0,1] \to \T^2$ with $\dot \gamma \in \mcC^\beta$.
Recall that $|\cdot|_\infty$ denotes the supremum norm.

\begin{proposition}\label{prop:gamma_alpha_Hol_bound}
There exists $C > 0$ such that for all $\alpha \in [\frac23,1]$, $\beta \in [\frac{2}{\alpha}-2,1]$, and $\kappa \in [\frac{2-\alpha}{\alpha(1+\beta)},1]$, and $\gamma,\bar\gamma \in\mcC^{1,\beta}([0,1],\T^2)$, it holds that
\begin{equ}
|\gamma;\bar\gamma|_{\alpha} \leq C\Big[(|\dot\gamma|_\infty + |\dot{\bar\gamma}|_\infty) (|\dot\gamma|_{\Hol\beta} + |\dot{\bar\gamma}|_{\Hol\beta})^{\kappa}|\gamma - \bar\gamma|_\infty^{1-\kappa}\Big]^{\alpha/2}\;.
\end{equ}
\end{proposition}

\begin{proof}
Let $P,\bar P$ denote the triangle process associated to $\gamma,\bar\gamma$ respectively.
For $0 \leq s < u < t \leq 1$, observe that
\begin{equs}
|P_{sut}|
&\leq |\gamma(t)-\gamma(s)|\Big|\gamma(u)-\gamma(s) - \frac{|u-s|}{|t-s|}(\gamma(t)-\gamma(s))\Big|
\\
&\leq |t-s| |\dot\gamma|_\infty \Big| \frac{1}{|t-s|}\int_s^u \int_s^t |\dot\gamma(r) - \dot\gamma(q)| \mrd q \mrd r \Big|
\\
& \leq |t-s|^{2+\beta} |\dot\gamma|_\infty |\dot\gamma|_{\Hol\beta}\;.\label{eq:P_Hol_bound}
\end{equs}
Furthermore,
\begin{equ}
|P_{sut};\bar P_{sut}| \lesssim \sum_{q \neq r} (|\gamma(q) - \gamma(r)| + |\bar\gamma(q)-\bar\gamma(r)|) (|\gamma(q)-\bar\gamma(q)| + |\gamma(r)-\bar\gamma(r)|)\;,
\end{equ}
where the sum is over all $2$-subsets $\{q,r\}$ of $\{s,u,t\}$, whence
\begin{equ}
|P_{sut};\bar P_{sut}| \lesssim|t-s|(|\dot\gamma|_\infty + |\dot{\bar\gamma}|_\infty) |\gamma-\bar\gamma|_\infty\;.\label{eq:P_unif_bound}
\end{equ}
Interpolating between~\eqref{eq:P_Hol_bound} and~\eqref{eq:P_unif_bound}, we have for any $\kappa \in [0,1]$
\begin{equ}
|P_{sut};\bar P_{sut}| \lesssim (|\dot\gamma|_\infty + |\dot{\bar\gamma}|_\infty) (|\dot\gamma|_{\Hol\beta} + |\dot{\bar\gamma}|_{\Hol\beta})^{\kappa}|t-s|^{1+\kappa+\beta\kappa} |\gamma - \bar\gamma|_\infty^{1-\kappa}\;.
\end{equ}
The conclusion follows by taking $\kappa \geq \frac{2-\alpha}{\alpha(1+\beta)}$ so that $\alpha(1+\kappa+\beta\kappa)/2 \geq 1$.
\end{proof}

%
%

We end this subsection with a result on the continuity in $p$-variation of $\gamma_A$ jointly in $(A,\gamma)\in\Omega_\alpha\times\mcC^{1,\beta}([0,1],\T^2)$.
For $\beta\in[0,1]$, a \emph{ball} in $\mcC^{1,\beta}$ is any set of the form
\begin{equ}
\{\gamma\in \mcC^{1,\beta}([0,1],\T^2) \ssep
|\dot\gamma|_\infty + |\dot\gamma|_{\Hol\beta} \leq R\}
\end{equ}
for some $R\geq 0$.

\begin{proposition}\label{prop:p-var_cont}
Let $\alpha\in(\frac23,1]$, $p>\frac1\alpha$, and $\beta\in(\frac2{\alpha}-2,1]$.
There exists $\delta>0$ such that for all $A,\bar A\in \Omega_\alpha$
\begin{equ}
|\gamma_A-\bar\gamma_{\bar A}|_{\var{p}} \lesssim |A-\bar A|_{\alpha} + |A|_{\alpha} |\gamma-\bar\gamma|^\delta_\infty
\end{equ}
uniformly over $\gamma,\bar\gamma$ in balls of $\mcC^{1,\beta}$.
\end{proposition}

\begin{proof}
Note that $|\gamma|_{\var1}$ is trivially bounded by $|\dot\gamma|_{\infty}$.
Furthermore, for $\bar\alpha\in[\frac23,1]$, it follows from Proposition~\ref{prop:gamma_alpha_Hol_bound} that $|\gamma|_{\bar\alpha}$ is uniformly bounded on balls in $\mcC^{1,\bar\beta}$ with $\bar\beta = \frac{2}{\bar\alpha}-2$.
As a consequence (using that $\beta>\frac2{\alpha}-2$), it follows from Corollary~\ref{cor:p-var_bound} that $|\gamma_A|_{\var{\frac1\alpha}} \lesssim |A|_{\alpha}$ uniformly on balls in $\mcC^{1,\beta}$.

On the other hand, by~\eqref{eq:A_gamma_bar_gamma_bound} and Proposition~\ref{prop:gamma_alpha_Hol_bound} (using again that $\beta>\frac2{\alpha}-2$), there exists $\eps>0$ such that $|\gamma_A-\bar\gamma_A|_\infty \lesssim |A|_{\alpha}|\gamma-\bar\gamma|^{\eps}_\infty$ uniformly over balls in $\mcC^{1,\beta}$.
Applying the interpolation estimate for $p>\frac1\alpha$
and $x\colon[0,1]\to E$
\begin{equ}
|x|_{\var p} \leq (|x|_{\var{\frac1\alpha}})^{\frac{1}{\alpha p}} (2|x|_\infty)^{1-\frac{1}{\alpha p}}\;,
\end{equ}
it follows that for some $\delta>0$
\begin{equ}
|\gamma_A-\bar\gamma_A|_{\var{p}} \lesssim |A|_{\alpha}|\gamma-\bar\gamma|^{\delta}_\infty
\end{equ}
uniformly on balls in $\mcC^{1,\beta}$.

Finally, it follows again from Corollary~\ref{cor:p-var_bound} and Proposition~\ref{prop:gamma_alpha_Hol_bound} that
\begin{equ}
|\gamma_{A}-\gamma_{\bar A}|_{\var{\frac1\alpha}} = |\gamma_{A-\bar A}|_{\var{\frac1\alpha}} \lesssim |A-\bar A|_{\alpha}
\end{equ}
uniformly over balls in $\mcC^{1,\beta}$, from which the conclusion follows.
\end{proof}

\subsection{Closure of smooth \texorpdfstring{$1$}{1}-forms}
\label{subsec:smooth_one_forms}

Recall that $\Omega\mcC\eqdef\Omega\mcC(\T^2,E)$ denotes the Banach space of continuous $E$-valued $1$-forms.
Following Remark~\ref{rem:line_integral}, there exists a canonical map 
\begin{equ}[e:inclusion]
\imath \colon \Omega\mcC \to \Omega_{\gr 1}
\end{equ}
defined by
\begin{equ}
\imath A(x,v) \eqdef \int_0^1 \sum_{i=1}^2 A_i(x+tv) v_i \mrd t\;,
\end{equ}
which is injective and satisfies $|\imath A|_{\gr1} \leq |A|_\infty$.

%

\begin{definition}\label{def:closure_smooth_1_forms}
For $\alpha\in [0,1]$, let $\Omega^1_\alpha$ and $\Omega^1_{\gr\alpha}$ denote the closure of $\imath(\Omega\mcC^\infty)$ in $\Omega_\alpha$ and $\Omega_{\gr\alpha}$ respectively.
\end{definition}

\begin{remark}\label{rem:Holder_Omega_embedding}
Recalling notation from Section~\ref{subsec:notation}, it is easy to see that $\imath$ embeds $\Omega\mcC^{\alpha/2}$ and $\Omega\mcC^{0,\alpha/2}$ continuously into $\Omega_{\alpha}$ and $\Omega^1_\alpha$ respectively (and that the exponent $\alpha/2$ is sharp in the sense that $\Omega\mcC^\beta$ and $\Omega\mcC^{0,\beta}$ do not embed into $\Omega_{\alpha}$ and $\Omega^1_\alpha$ for any $\beta<\alpha/2$).
\end{remark}

\begin{remark}\label{rem:separable}
Note that since any element of $\Omega\mcC^\infty$ can be approximated by a trigonometric polynomial with 
rational coefficients,
$\Omega^1_\alpha$ is a separable Banach space whenever $E$ is separable.
\end{remark}

We now construct a continuous, linear map $\pi\colon\Omega_{\gr\alpha} \to \Omega\mcC^{\alpha-1}$ which is a left inverse to $\imath$.
Consider $\alpha \in (0,1]$ and $A\in \Omega_{\gr\alpha}$.
For $z\in \T^2$, $i \in \{1,2\}$ and $t \in \R_+$, we then set
$X^{z}_{i}(t) \eqdef A(z,t e_i)$ (which is well-defined by the additivity of $A$ even
when $t \ge 1/4$). Note also that $|X^{z}_{i}|_{\Hol\alpha} \leq  |A|_{\gr\alpha}$ by Lemma~\ref{lem:Hol_paths}.
In a similar way, for $\psi \in \mcC^1(\T^2)$, consider $Y^z_i \in \mcC^1([0,1],\R)$ given by 
$Y^{z}_i (t)\eqdef \psi(z + t e_i)$.
We then define the $E$-valued distribution $\pi_i A\in\mcD'(\T^2,E)$ by
\begin{equ}
\scal{\pi_i A,\psi} \eqdef \int_{P_i}  \int_0^{1} Y^{z}_i(t) \mrd X^{z}_i(t)\mrd z\;,
\end{equ}
where the inner integral is in the Young sense and $P_i = \{z \in [0,1)^2\,:\, z_i = 0\}$.
Combining the components yields the linear map $\pi$.

 One can show that $|\pi A|_{\Omega\mcC^{\alpha-1}} \lesssim |A|_{\gr\alpha}$ (see~\cite[Prop.~3.21]{Chevyrev18YM}) and that $\pi$ is a left inverse of $\imath$ in the sense that, for all $A\in\Omega\mcC$, $\pi(\imath A) = A$ as distributions.
Furthermore, we see that $\pi$ is injective on $\Omega^1_{\gr\alpha}$ and one has a direct way to recover $A\in \Omega^1_{\gr\alpha}$ from $\pi A$ through mollifier approximations (see also Proposition~\ref{prop:strongly_continuous}).
In particular, we can identify $\Omega^1_{\gr\alpha}$ (and a fortiori $\Omega^1_\alpha$) with a subspace of $\Omega\mcC^{\alpha-1}$.

\begin{proposition}\label{prop:Omega^1_alpha_char}
Let $\alpha\in(0,1)$ and $A \in \Omega_{\gr\alpha}^1$.
Then  $A$ is a continuous function on $\mcX$, and
\begin{equ}\label{eq:A_ell_vanish}
\lim_{\eps\to0} \sup_{|\ell|<\eps} |\ell|^{-\alpha}|A(\ell)| = 0\;.
\end{equ}
\end{proposition}

\begin{proof}
Let $A \in \Omega^1_{\gr\alpha}$ and $\delta>0$.
Consider a sequence $(A^n)_{n \geq 1}$ in $\Omega\mcC^\infty$ such that $\lim_{n\to\infty}|\imath A^n - A|_{\gr\alpha} = 0$ and define functions $B^n \colon \mcX \to E$ by
\begin{equ}
B^n(\ell) \eqdef \begin{cases}
|\ell|^{-\alpha}\imath A^n(\ell) &\text{ if $|\ell|>0$,}
\\
0 &\text{ otherwise.}
\end{cases}
\end{equ}
We define $B\colon\mcX\to E$ in the same way with $\imath A^n$ replaced by $A$.
Observe that, since $A^n \in\Omega\mcC^\infty$, $B^n$ is a continuous function on $\mcX$.
Furthermore, $\lim_{n\to\infty}|\imath A^n - A|_{\gr\alpha} = 0$ is equivalent to
\begin{equ}
\lim_{n \to \infty} \sup_{\ell\in\mcX} |B^n(\ell) - B(\ell)|=0\;.
\end{equ}
Hence $B$ is continuous on $\mcX$, from which continuity of $A$ and~\eqref{eq:A_ell_vanish} follow.
\end{proof}

\subsection{Gauge transformations}
\label{subsec:gauge_transforms}

For the remainder of the section, we fix a compact Lie group $G$\label{G page ref}
with Lie algebra $\mfg$\label{mfg page ref}.
We equip $\mfg$ with an arbitrary norm and henceforth take $E=\mfg$ as our Banach space.
Since $G$ is compact, we can assume without loss of generality that $G$ 
(resp.\ $\mfg$) is a Lie subgroup of unitary matrices 
(resp.\ Lie subalgebra of anti-Hermitian matrices), so that both $G$ 
and $\mfg$ are embedded in some normed linear space $F$ of matrices.
When we write expressions of the form $|g-h|$ with $g,h \in G$, we
implicitly identify them in $F$ and interpret $|\cdot|$ as the norm of $F$. Different
choices of unitary representation yield equivalent distances, so the precise choice is unimportant.
For $g\in G$, we denote by $\Ad_g \colon \mfg \to \mfg$ the adjoint action $\Ad_g(X) = gXg^{-1}$.

For $\alpha \in [0,1]$ and a function $g\colon\T^2\to F$, recall the definition of the seminorm 
$|g|_{\Hol\alpha}$ and norm $|g|_\infty$.
We denote by $\mfG^\alpha$\label{mfG^alpha page ref} the subset $\mcC^\alpha(\T^2,G)$, which we 
note is a topological group.

\begin{definition}\label{def:gauge}
Let $\alpha \in (0,1]$, $A\in\Omega_{\gr\alpha}$, $\beta \in (\frac 12,1]$ with $\alpha + \beta > 1$, and $g \in \mfG^\beta$.
Define $A^g \in \Omega$ by
\begin{equ}
A^g(\ell) \eqdef \int_0^1 \bigl(\Ad_{g(x+tv)} \mrd \ell_A(t) - \big[\mrd g(x+tv) \big] g^{-1}(x+tv)\bigr)\;,
\end{equ}
where $\ell = (x,v) \in \mcX$, and where both terms make sense as $\mfg$-valued Young integrals since $\alpha+\beta > 1$ and $\beta > \frac12$.
In the case $\alpha > \frac12$, for $A,\bar A \in \Omega_{\gr\alpha}$ we write $A\sim\bar A$ if there exists $g \in \mfG^\alpha$ such that $A^g = \bar A$.
\end{definition}

Note that, in the case that $A$ is a continuous $1$-form and $g$ is $\mcC^1$, we have $\mrd \ell_A(t) = A(x+tv)(v) \mrd t$, hence
\begin{equ}
A^g(x) = \Ad_{g(x)} A(x) - \big[\mrd g(x)\big] g^{-1}(x) \;,
\end{equ}
as one expects from interpreting $A$ as a connection.
However, in the interpretation of $A$ as a $1$-form, the more natural map is $A\mapsto A^g-0^g$, which is linear and makes sense for any $\beta\in(0,1]$ such that $\alpha+\beta>1$
(here $0$ is an element of $\Omega_{\gr\alpha}$ and, despite the notation, $0^g$ is in general non-zero).

The main result of this subsection is the following.

\begin{theorem}\label{thm:grp_action}
Let $\beta \in (\frac23,1]$ and $\alpha \in (0,1]$ such that $\alpha + \frac\beta2 > 1$ and $\frac\alpha2 + \beta > 1$.
Then the map $(A,g) \mapsto A^g$ is a continuous map from $\Omega_\alpha\times\mfG^\beta$ (resp. $\Omega_{\gr\alpha}\times\mfG^\beta$) into $\Omega_{\alpha\wedge\beta}$ (resp. $\Omega_{\gr{\alpha\wedge\beta}}$)
and is furthermore uniformly continuous on every ball.
If $\alpha\leq \beta$, then this map defines a left-group action, i.e., $(A^h)^g = A^{gh}$.
\end{theorem}

We give the proof of Theorem~\ref{thm:grp_action} at the end of this subsection.
We begin by analysing the case $A=0$.

\begin{proposition}\label{prop:mfG_bound}
Let $\alpha \in (\frac23,1]$ and $g \in \mfG^\alpha$.
Then $|0^g|_\alpha \lesssim |g|_{\Hol\alpha}\vee|g|_{\Hol\alpha}^2$, where the proportionality constant depends only on $\alpha$.
\end{proposition}

For the proof of Proposition~\ref{prop:mfG_bound}, we require several lemmas.

\begin{lemma}\label{lem:ell_g_bound}
Let $\alpha \in [0,1]$, $g \in \mfG^\alpha$, and $\ell = (x,v)$, $\bar\ell=(\bar x, \bar v) \in \mcX$ forming a vee.
Consider the path $\ell_g \colon [0,1] \to G$ given by $\ell_g(t) = g(x+tv)$, and similarly for $\bar\ell_g$.
Then
\begin{equ}\label{eq:ell_g_Hol}
|\ell_g|_{\Hol\alpha} \leq |\ell|^\alpha|g|_{\Hol\alpha}
\end{equ}
and
\begin{equ}\label{eq:ell_g_Hol_diff}
|\ell_g-\bar\ell_g|_{\Hol{\alpha/2}} \lesssim |g|_{\Hol\alpha} \Area(\ell,\bar\ell)^{\alpha/2}
\end{equ}
for a universal proportionality constant.
\end{lemma}

\begin{proof}
We have $|\ell_g(t) - \ell_g(s)| \leq |g|_{\Hol\alpha}|t-s|^\alpha|\ell|^\alpha$, which proves~\eqref{eq:ell_g_Hol}.
For~\eqref{eq:ell_g_Hol_diff}, we have
\begin{equs}
|\ell_g(t) - \ell_g(s) - \bar\ell_g(t)+\bar\ell_g(s)| &\leq |g|_{\Hol\alpha} \big[ (2t^\alpha|\ell_f-\bar\ell_f|^\alpha)\wedge (|t-s|^\alpha|\ell|^\alpha)\big]
\\
&\lesssim 2|g|_{\Hol\alpha} |t-s|^{\alpha/2}\Area(\ell,\bar\ell)^{\alpha/2}\;,
\end{equs}
where in the second inequality we used interpolation and the fact that $\Area(\ell,\bar\ell) \asymp |\ell_f-\bar\ell_f||\ell|$.
\end{proof}

\begin{lemma}\label{lem:g_gr_bound}
Let $\alpha \in (\frac12,1]$ and $g \in \mfG^\alpha$.
Then $|0^g|_{\gr\alpha} \lesssim |g|_{\Hol\alpha}\vee |g|_{\Hol\alpha}^2$, where the proportionality constant depends only on $\alpha$.
\end{lemma}

\begin{proof}
Let $\ell = (x,v) \in \mcX$.
Then by~\eqref{eq:ell_g_Hol} and Young's estimate
\begin{equ}
|0^g(\ell)| = \Big|\int_0^1 \mrd g(x+tv)\,g^{-1}(x+tv) \Big| \lesssim (1+|g|_{\Hol\alpha}|\ell|^{\alpha})|\ell|^{\alpha}|g|_{\Hol\alpha}\;,
\end{equ}
which implies the claim.
\end{proof}

\begin{lemma}\label{lem:g_vee_bound}
Let $\alpha \in (\frac23,1]$ and $g \in \mfG^\alpha$.
Then $|0^g|_{\v\alpha} \lesssim |g|_{\Hol\alpha}\vee|g|_{\Hol\alpha}^2$, where the proportionality constant depends only on $\alpha$.
\end{lemma}

\begin{proof}
Let $\ell = (x,v), \bar\ell = (x, \bar v) \in \mcX$ form a vee.
Then, denoting $Y_t \eqdef g^{-1}(x+tv), \bar Y_t \eqdef g^{-1}(x+t\bar v)$, and $X_t \eqdef g(x+tv), \bar X_t \eqdef g(x+t\bar v)$, we have
\begin{equs}
|0^g(\ell) - 0^g(\bar\ell)|
&= \Big| \int_0^1 Y_t \mrd X_t - \int_0^1 \bar Y_t \mrd \bar X_t \Big|
\\
&\leq \Big| \int_0^1 (Y_t-\bar Y_t) \mrd X_t \Big| + \Big| \int_0^1 \bar Y_t \mrd (X_t - \bar X_t) \Big|\;.
\end{equs}
Using~\eqref{eq:ell_g_Hol_diff},~\eqref{eq:ell_g_Hol}, Young's estimate, and the fact that $Y_0 = \bar Y_0$, we have
\begin{equs}
\Big|\int_0^1 (Y_t-\bar Y_t) \mrd X_t\Big|
&\lesssim |Y-\bar Y|_{\Hol{\alpha/2}}|X|_{\Hol\alpha} 
\\
&\lesssim |g|_{\Hol\alpha}^2\Area(\ell,\bar\ell)^{\alpha/2}|\ell|^\alpha
\end{equs}
and
\begin{equs}
\Big| \int_0^1 \bar Y_t \mrd (X_t - \bar X_t) \Big|
&\lesssim (1+|\bar Y|_{\Hol{\alpha}})|X-\bar X|_{\Hol{\alpha/2}}
\\
&\lesssim (1+|\ell|^\alpha|g|_{\Hol\alpha})|g|_{\Hol\alpha}\Area(\ell,\bar\ell)^{\alpha/2}\;,
\end{equs}
thus concluding the proof.
\end{proof}

\begin{proof}[of Proposition~\ref{prop:mfG_bound}]
Combining the equivalence of norms $|\cdot|_\alpha \asymp |\cdot|_{\gr\alpha} + |\cdot|_{\v\alpha}$ from Theorem~\ref{thm:norm_equiv} with Lemmas~\ref{lem:g_gr_bound} and~\ref{lem:g_vee_bound} yields the proof.
\end{proof}

For the lemmas which follow, recall that the quantity $A^g-0^g$ makes sense for all $A\in\Omega_{\gr\alpha}$ and $g\in\mfG^\beta$ provided that $\alpha,\beta\in(0,1]$ with $\alpha+\beta>1$.

\begin{lemma}\label{lem:group_action_gr}
Let $\alpha,\beta \in (0,1]$ such that $\alpha+\beta>1$, $A\in\Omega_{\gr\alpha}$, and $g\in\mfG^\beta$. Then
\begin{equ}\label{eq:mfG_gr_bound}
|A^g - 0^g - A|_{\gr{\alpha}} \lesssim (|g-1|_\infty + |g|_{\Hol\beta})|A|_{\gr\alpha}\;,
\end{equ}
where the proportionality constant depends only on $\alpha$ and $\beta$.
\end{lemma}

\begin{proof}
Let $\ell=(x,v)\in\mcX$, $A \in\Omega$, and $g \colon \T^2 \to G$.
Using notation from Lemma~\ref{lem:ell_g_bound}, note that
\begin{equ}\label{eq:Ag_A_diff}
(A^g-0^g-A)(\ell) = \int_0^1 (\Ad_{\ell_g(t)}-1)\mrd \ell_A (t)\;.
\end{equ}
Using~\eqref{eq:ell_A_Hol},~\eqref{eq:ell_g_Hol}, and Young's estimate, we obtain
\begin{equ}
|A^g(\ell)-0^g(\ell)-A(\ell)|
\lesssim (|g-1|_\infty + |\ell|^\beta|g|_{\Hol\beta})|\ell|^{\alpha}|A|_{\gr\alpha}\;,
\end{equ}
which proves~\eqref{eq:mfG_gr_bound}.
\end{proof}

\begin{lemma}\label{lem:group_action_vee}
Let $\alpha,\beta \in (0,1]$ such that $\frac\beta2+\alpha > 1$ and $\frac\alpha2 + \beta > 1$, $A \in \Omega_\alpha$, and $g \in \mfG^\beta$.
Then
\begin{equ}\label{eq:mfG_vee_bound}
|A^g -0^g - A|_{\v{\alpha\wedge \beta}} \lesssim (|g-1|_\infty + |g|_{\Hol\beta})|A|_{\alpha}\;,
\end{equ}
where the proportionality constant depends only on $\alpha$ and $\beta$.
\end{lemma}

\begin{proof}
Let $\ell,\bar\ell \in \mcX$ form a vee.
Recall the identity~\eqref{eq:Ag_A_diff}.
By~\eqref{eq:ell_A_Hol},~\eqref{eq:ell_g_Hol_diff}, and Young's estimate (since $\frac\beta2 + \alpha > 1$), we have
\begin{equ}
\Big|\int_0^1 (\Ad_{\ell_g(t)}-\Ad_{\bar\ell_g(t)}) \mrd \ell_A (t) \Big| \lesssim |g|_{\Hol\beta}\Area(\ell,\bar\ell)^{\beta/2}|\ell|^\alpha |A|_{\gr\alpha}\;.
\end{equ}
Similarly, using~\eqref{eq:ell_A_Hol},~\eqref{eq:ell_g_Hol}, and Young's estimate (since $\beta+\frac\alpha2 > 1$), we have
\begin{equ}
\Big|\int_0^1 (\Ad_{\bar\ell_g(t)} -1)\mrd (\ell_A (t) - \bar\ell_A(t)) \Big| \lesssim (|g-1|_\infty +|g|_{\Hol\beta}|\ell|^\beta) |A|_\alpha \Area(\ell,\bar\ell)^{\alpha/2}\;.
\end{equ}
Note that the integrals on the left-hand sides of the previous two bounds add to $(A^g-0^g-A)(\ell) - (A^g-0^g-A)(\bar\ell)$,
from which~\eqref{eq:mfG_vee_bound} follows.
\end{proof}

\begin{proof}[of Theorem~\ref{thm:grp_action}]
The fact that the action of $\mfG^\beta$ maps $\Omega_{\gr\alpha}$ into $\Omega_{\gr{\alpha\wedge \beta}}$ follows from Proposition~\ref{prop:mfG_bound} and Lemma~\ref{lem:group_action_gr}.
The fact that the action of $\mfG^\beta$ maps $\Omega_\alpha$ into $\Omega_{\alpha\wedge \beta}$ follows by combining the equivalence of norms $|\cdot|_\alpha \asymp |\cdot|_{\gr\alpha} + |\cdot|_{\v\alpha}$ from Theorem~\ref{thm:norm_equiv} with Proposition~\ref{prop:mfG_bound} and Lemmas~\ref{lem:group_action_gr} and~\ref{lem:group_action_vee}.
The fact that $(A,g)\mapsto A^g$ is uniformly continuous on any given ball in $\Omega_{\gr\alpha}\times\mfG^\beta$ (resp. $\Omega_{\alpha}\times\mfG^\beta$)
follows from writing
\begin{equ}
A^g-B^h =   \big((A-B)^h-0^h - (A-B)\big) -\big((A^g)^{hg^{-1}} - 0^{hg^{-1}} - A^g\big) - 0^{hg^{-1}} + (A-B)
\end{equ}
and noting that, again by Proposition~\ref{prop:mfG_bound} and Lemma~\ref{lem:group_action_gr},
the $|\cdot|_{\gr\alpha}$ norm of all four terms is
of order
\begin{equ}
|A-B|_{\gr{\alpha\wedge\beta}} + |hg^{-1}|_{\Hol\alpha} + |hg^{-1}-1|_\infty
\end{equ}
uniformly over all $(A,g),(B,h)$ in the given ball (similarly for the $|\cdot|_\alpha$ norm using Lemma~\ref{lem:group_action_vee}).
Finally, if $\alpha \leq \beta$, the fact that $A^{g h} = (A^{h})^g$ follows from the identity
\begin{equ}
\mrd (g h)\,(gh)^{-1} = (\mrd g) \,g^{-1} + \Ad_g \big[ (\mrd h) \,h^{-1}\big]\;.
\end{equ}
Combining all of these claims completes the proof.
\end{proof}

\subsection{Holonomies and recovering gauge transformations}
\label{subsec:recover_gauge_transform}

The main result of this subsection, Proposition~\ref{prop:gauge_equiv}, provides a way to recover the gauge 
transformation that transforms between gauge equivalent elements of $\Omega_{\gr\alpha}$.
This result can be seen as a version of~\cite[Prop.~2.1.2]{Sengupta92} for the non-smooth case (see also~\cite[Lem.~3]{LevyNorris06}).

Let us fix $\alpha \in (\frac12,1]$ throughout this subsection.
For $\ell\in\mcX$ and $A \in \Omega_{\gr\alpha}$, the ODE
\begin{equ}
\mrd y(t) = y(t) \mrd \ell_A(t)\;,\qquad y(0) = 1\;,
\end{equ}
admits a unique solution $y\colon[0,1]\to G$ as a Young integral (thanks to Lemma~\ref{lem:Hol_paths}).
Furthermore, the map $\ell_A \mapsto y$ is locally Lipschitz when both sides are equipped with $|\cdot|_{\Hol\alpha}$.
We define the holonomy of $A$ along $\ell$ as $\hol(A,\ell) \eqdef y(1)$.
As usual, we extend the definition $\hol(A,\gamma)$ to any piecewise affine path $\gamma\colon[0,1]\to\T^2$ by taking the ordered product of the holonomies along individual line segments.
\begin{remark}\label{rem:extension_of_hol}
Recall from item~\ref{pt:hol_well-defined} of Theorem~\ref{thm:state_space} that,
provided $\alpha\in(\frac23,1]$ and $\beta\in(\frac2\alpha-2,1]$,
the holonomy $\hol(A,\gamma)$ is well-defined for all paths $\gamma$ piecewise in $\mcC^{1,\beta}$
(rather than only piecewise affine)
and all $A \in \Omega^1_\alpha$.
\end{remark}
For any $g\in\mfG^\alpha$ and any piecewise affine path $\gamma$, note the familiar identity
\begin{equ}\label{eq:hol_g_action}
\hol(A^g,\gamma) = g(\gamma(0)) \,\hol(A,\gamma)\, g(\gamma(1))^{-1}\;.
\end{equ}
For $x,y\in\T^2$, let $\mcL_{xy}$ denote the set of piecewise affine paths $\gamma\colon[0,1]\to\T^2$ with $\gamma(0)=x$ and $\gamma(1)=y$.

\begin{proposition}\label{prop:gauge_equiv}
Let $\alpha\in(\frac12,1]$ and $A,\bar A \in \Omega_{\gr\alpha}$.
Then the following are equivalent:
\begin{enumerate}[label=(\roman*)]
\item \label{point:A_sim} $A\sim \bar A$.
\item \label{point:hol_sim} there exists $x\in\T^2$ and $g_0 \in G$ such that $\hol(\bar A,\gamma)=g_0\hol(A,\gamma)g_0^{-1}$ for all $\gamma\in\mcL_{xx}$.
\item \label{point:hol_simStrong} for every $x\in\T^2$ there exists $g_x \in G$ such that $\hol(\bar A,\gamma)=g_x\hol(A,\gamma)g_x^{-1}$ for all $\gamma\in\mcL_{xx}$.
\end{enumerate}
Furthermore, if~\ref{point:hol_sim} holds, then there exists a unique $g \in \mfG^\alpha$ such that $g(x) = g_0$ and $A^g = \bar A$.
The element $g$ is determined by
\begin{equ}\label{eq:g_def}
g(y) = \hol(\bar A,\gamma_{xy})^{-1} g_0 \hol(A,\gamma_{xy})\;,
\end{equ}
where $\gamma_{xy}$ is any element of $\mcL_{xy}$,
and satisfies
\begin{equ}\label{eq:g_Hol_bound}
|g|_{\Hol\alpha} \lesssim |A|_{\gr\alpha}+|\bar A|_{\gr\alpha}\;.
\end{equ}
\end{proposition}

\begin{proof}
The implication~\ref{point:A_sim}~$\Rightarrow$~\ref{point:hol_simStrong} is clear from~\eqref{eq:hol_g_action} and
the implication~\ref{point:hol_simStrong}~$\Rightarrow$~\ref{point:hol_sim} is trivial.
Hence suppose~\ref{point:hol_sim} holds.
Let us define $g(y)$ using~\eqref{eq:g_def}, which we note does not depend on the choice of path $\gamma_{xy}\in\mcL_{xy}$.
Then one can readily verify the bound~\eqref{eq:g_Hol_bound} and that $A^g = \bar A$, which proves~\ref{point:A_sim}.
The fact that $g$ is the unique element in $\mfG^\alpha$ such that $g(x)=g_0$ and $A^g=\bar A$ follows again from~\eqref{eq:hol_g_action}.
\end{proof}

\subsection{The orbit space}
\label{subsec:orbit_space}

We define and study in this subsection the space of gauge orbits of the Banach space $\Omega^1_\alpha$.
Let $\mfG^{0,\alpha}$\label{mfG^0,alpha page ref} denote the closure of $\mcC^\infty(\T^2,G)$ in $\mfG^{\alpha}$.
The following is a simple corollary of Theorem~\ref{thm:grp_action}.

\begin{corollary}\label{cor:group_action_smooth_closure}
Let $\alpha \in (\frac23,1]$.
Then $(A,g)\mapsto A^g$ is a continuous left group action of $\mfG^{0,\alpha}$ on $\Omega^1_\alpha$ and on $\Omega^1_{\gr\alpha}$. 
\end{corollary}

\begin{proof}
It holds that $A^g\in\imath\Omega\mcC^\infty$ whenever $A\in\imath\Omega\mcC^\infty$ and $g\in\mcC^\infty(\T^2,G)$.
The conclusion follows from Theorem~\ref{thm:grp_action} by continuity of $(A,g) \mapsto A^g$.
\end{proof}

We are now ready to define our desired space of orbits.

\begin{definition}\label{def:orbit_space}
For $\alpha \in (\frac23,1]$, let $\mfO_{\alpha}$\label{mfO_alpha page ref} denote the space of orbits $\Omega^1_{\alpha}/\mfG^{0,\alpha}$ equipped with the quotient topology.
For every $A \in \Omega_\alpha^1$, let
$\mfO_\alpha \ni [A] \eqdef \{A^g \,:\, g \in \mfG^{0,\alpha}\} \subset \Omega_\alpha^1$ denote the corresponding
gauge orbit.
We likewise define $\mfO_{\gr\alpha}\eqdef \Omega^1_{\gr\alpha}/\mfG^{0,\alpha}$.
\end{definition}

We next show that the restriction to the subgroup $\mfG^{0,\alpha}$ is natural in the sense that $\mfG^{0,\alpha}$ is precisely the stabiliser of $\Omega^1_\alpha$ and $\Omega^1_{\gr\alpha}$. For this, we use the following
version of the standard fact that the closure of smooth functions yield the ``little H{\"o}lder'' spaces.
\begin{lemma}\label{lem:mfG_0_alpha_char}
For $\alpha\in (0,1)$, one has $g\in\mfG^{0,\alpha}$ if and only if 
$\lim_{\eps\to0} \sup_{|x-y|<\eps}|x-y|^{-\alpha}|g(x)-g(y)| = 0$.\qed
\end{lemma}
We then have the following general statement.

\begin{proposition}\label{prop:mfG_0_alpha_action}
Let $\alpha\in(\frac12,1]$ and $A\in\Omega^1_{\gr\alpha}$.
Suppose that $A^g\in\Omega^1_{\gr\alpha}$ for some $g\in\mfG^\alpha$.
Then $g\in\mfG^{0,\alpha}$.
\end{proposition}

%

\begin{proof}
Suppose first $\alpha\in (\frac12,1)$.
Then by Proposition~\ref{prop:Omega^1_alpha_char},
\begin{equ}\label{eq:A_A_g_lim}
\lim_{\eps\to0}\sup_{|\ell|<\eps} \frac{|A(\ell)|+|A^g(\ell)|}{|\ell|^{\alpha}} = 0\;.
\end{equ}
Combining~\eqref{eq:A_A_g_lim} with the expression for $g$ in~\eqref{eq:g_def}, we conclude that
 $g\in\mfG^{0,\alpha}$ by Lemma~\ref{lem:mfG_0_alpha_char}.
Suppose now $\alpha=1$.
Then $\Omega^1_{\gr 1}=\Omega\CC$
and $\mfG^{0,1}=\{g\in\mfG^1\,:\,\mrd g\in \CC\}$.
Furthermore $(\mrd g) g^{-1}=\Ad_g A-A^g$,
from which continuity of $\mrd g$ follows.
\end{proof}

In general, the quotient of a Polish space by the continuous action of a Polish group has no nice properties.
In the remainder of this subsection, we show that the space $\mfO_\alpha$
for $\alpha\in(\frac23,1)$ is itself a Polish space 
and we exhibit a metric $D_\alpha$ for its topology. 
We first show that these orbits are very well behaved in the following sense.

\begin{lemma}\label{lem:closed}
Let $\alpha\in(\frac23,1]$.
For every $A \in \Omega_\alpha^1$ (resp. $A \in \Omega_{\gr\alpha}^1$),  the gauge orbit $[A]$ is closed in $\Omega^1_\alpha$ (resp. $\Omega^1_{\gr\alpha}$).
\end{lemma}

\begin{proof}
Since $\Omega_\alpha^1$ is a separable Banach space, it suffices to show that,
for every $B \in \Omega_\alpha^1$ and any sequence $A_n \in [B]$ such that $A_n \to A$ in $\Omega_\alpha^1$,
one has $A \in [B]$. 
Since the $A_n$ are uniformly bounded, the corresponding gauge transformations $g_n$ such that $A_n = B^{g_n}$ are uniformly bounded
in $\mfG^\alpha$ by~\eqref{eq:g_Hol_bound}.
Since $\mfG^\alpha \subset \mfG^\beta$ compactly for $\beta < \alpha$, 
we can assume modulo passing to a subsequence
that $g_n \to g$ in $\mfG^\beta$, which implies that $A = B^g$ by Theorem~\ref{thm:grp_action}. 
Since however we know that $A \in \Omega_\alpha^1$, we conclude that $g \in \mfG^{0,\alpha}$ by Proposition~\ref{prop:mfG_0_alpha_action},
so that $A \in [B]$ as required. 
The proof for $\Omega^1_{\gr\alpha}$ is the same.
\end{proof}

In the next step, we introduce a complete metric $k_\alpha$ on $\Omega_\alpha^1$ which generates the same topology as $|\cdot|_\alpha$, but 
shrinks distances at infinity so that, for large $r$, points on the sphere with radius $r$ are close to each other but such that the spheres with radii $r$ and $2r$ are still far apart.
We then define the metric $D_\alpha$ on $\mfO_\alpha$ as the Hausdorff distance associated with $k_\alpha$.

Until the end of the section, let $(\CO,\|\cdot\|)$ be a Banach space.

\begin{definition}
For $A,B\in\CO$, set
\begin{equ}
K(A,B) \eqdef \frac{|\|A\|-\|B\|| + 1}{(\|A\|\wedge\|B\|) + 1} (\|A-B\| \wedge 1)\;,
\end{equ}
and define the metric
\begin{equ}[e:defkalpha]
k(A,B) \eqdef \inf_{Z_0,\ldots, Z_n} \sum_{i=1}^n K(Z_{i-1},Z_i)\;,
\end{equ}
where the inf is over all finite sequences $Z_0,\ldots, Z_n\in\CO$ with $Z_0=A$ and $Z_n=B$.
\end{definition}

Note that
\begin{equ}\label{eq:K_upper_bound}
k(A,B)\leq K(A,B)\leq \frac{1}{r+1}
\end{equ}
for all $A,B$ in the sphere
\begin{equ}
S_r \eqdef \{C\in \CO \,:\, \|C\|=r\}\;.
\end{equ}
On the other hand, for $r_1,r_2 > 0$, if $A\in S_{r_1}$ and $B\in S_{r_2}$, then
\begin{equ}\label{eq:K_lower_bound_sphere}
K(A,B) \geq \frac{|r_1-r_2|}{r_1\wedge r_2+1}\;,
\end{equ}
and if $r>0$ and $A,B$ are in the ball
\begin{equ}
B_r \eqdef \{C\in\CO \,:\, \|C\| \leq r\}\;,
\end{equ}
then
\begin{equ}\label{eq:K_lower_bound_ball}
K(A,B) \geq \frac{\|A-B\|\wedge 1}{r+1}\;.
\end{equ}

\begin{lemma}\label{lem:k_lower_bound}
If $A\in S_r$ and $B\in S_{r+h}$ for some $r,h>0$, then
\begin{equ}
k(A,B) \geq \log\Big(1+\frac{h}{r+1}\Big)\;.
\end{equ}
\end{lemma}
%

\begin{proof}
Consider a sequence $Z_0=A,Z_1,\ldots, Z_n=B$
and let $r_i \eqdef \|Z_i\|$.
Then
\begin{equs}
\sum_{i=1}^n K(Z_{i-1},Z_i)
&\geq \sum_{i=1}^n \frac{|r_{i}-r_{i-1}|}{
r_i\wedge r_{i-1}+1}
\geq
\sum_{i=1}^n
\int_{r_i\wedge r_{i-1}}^{r_i \vee r_{i-1}}
\frac{1}{x+1} \mrd x
\\
&\geq \int_r^{r+h} \frac{1}{x+1}\mrd x
=\log\Big(\frac{r+h+1}{r+1}\Big)\;.
\end{equs}
where in the first bound we used~\eqref{eq:K_lower_bound_sphere}.
\end{proof}

\begin{proposition}
The metric space $(\CO,k)$ is complete and $k$ metrises the same topology as $\|\cdot\|$.
\end{proposition}

\begin{proof}
It is obvious that $k$ is weaker than the metric induced by $\|\cdot\|$.
On the other hand, if $A_n$ is a $k$-Cauchy sequence, then $\sup \|A_n\|<\infty$ by Lemma~\ref{lem:k_lower_bound}.
It readily follows from~\eqref{eq:K_lower_bound_ball} that $k(A_n,A)\to 0$ and $\|A-A_n\|\to 0$ for some $A\in\CO$.
\end{proof}

We now apply the above construction with $\CO=\Omega^1_\alpha$ and $\CO=\Omega^1_{\gr\alpha}$
and denote the corresponding metric by $k_\alpha$ and $k_{\gr\alpha}$ respectively.

\begin{definition}
Let $\alpha\in(\frac23,1]$.
We denote by $D_\alpha$ (resp. $D_{\gr\alpha}$) the Hausdorff distance on $\mfO_\alpha$ (resp. $\mfO_{\gr\alpha}$) associated with $k_\alpha$ (resp. $k_{\gr\alpha}$).
\end{definition}

\begin{theorem}\label{thm:Polish}
Let $\alpha\in(\frac23,1]$.
The metric space $(\mfO_\alpha,D_\alpha)$ is complete and $D_\alpha$ metrises the quotient topology on $\mfO_\alpha$.
In particular, $\mfO_\alpha$ is a Polish space.
The same holds for $(\mfO_{\gr\alpha},D_{\gr\alpha})$.
\end{theorem}

We prove Theorem~\ref{thm:Polish} only for $(\mfO_\alpha,D_\alpha)$; the proof for $(\mfO_{\gr\alpha},D_{\gr\alpha})$ is exactly the same.

\begin{lemma}\label{lem:points_on_spheres}
Let $\alpha\in(\frac23,1]$ and $x\in\mfO_\alpha$.
Then for all $r> \inf_{A\in x} |A|_\alpha$ there exists $A\in x$ with $|A|_\alpha=r$.
\end{lemma}

\begin{proof}
For any $A\in \Omega^1_\alpha$,
from the identity~\eqref{eq:hol_g_action},
we can readily construct a continuous function $g \colon [0,\infty) \to\mfG^{0,\alpha}$
such that $g(0)\equiv 1$ and $\lim_{t\to\infty}|A^{g(t)}|_\alpha=\infty$.
The
conclusion follows by continuity of $g\mapsto |A^g|_\alpha$ (Corollary~\ref{cor:group_action_smooth_closure}).
\end{proof}

\begin{lemma}\label{lem:k_to_D_conv}
Suppose $\alpha\in(\frac23,1]$ and $k_\alpha(A_n,A) \to 0$. Then $D_\alpha([A],[A_n]) \to 0$.
\end{lemma}

\begin{proof}
Consider $\eps>0$.
Observe that~\eqref{eq:K_upper_bound}, the fact 
that $\sup_n |A_n|_\alpha < \infty$ by Lemma~\ref{lem:k_lower_bound}, and Lemma~\ref{lem:points_on_spheres} together imply that 
there exists $r>0$ sufficiently large such that 
the Hausdorff distance for $k_\alpha$ between
$[A_n]\cap (\Omega^1_\alpha\setminus B^r_\alpha)$
and $[A]\cap (\Omega^1_\alpha\setminus B^r_\alpha)$ is at most $\eps$ for all $n$ sufficiently large.
On the other hand, for any $r>0$, $g\in\mfG^{0,\alpha}$, and $X,Y\in\Omega^1_\alpha$ such that $X,X^g \in B^r_\alpha$, it follows from Lemmas~\ref{lem:group_action_gr} and~\ref{lem:group_action_vee} and the identity
\begin{equ}
X^g-Y^g = \big((X-Y)^g-0^g-(X-Y)\big) + (X-Y)
\end{equ}
that
\begin{equ}
|X^g - Y^g|_\alpha \lesssim (1+|g|_{\Hol\alpha})|X-Y|_\alpha
\end{equ}
where $|g|_{\Hol\alpha} \lesssim r$ due to~\eqref{eq:g_Hol_bound}.
It follows that
\begin{equ}
\sup_{X\in[A_n]\cap B^r_\alpha} \inf_{Y\in [A]} k_\alpha(X,Y) + \sup_{X\in[A]\cap B^r_\alpha} \inf_{Y\in [A_n]} k_\alpha(X,Y) \to 0\;,
\end{equ}
which concludes the proof.
\end{proof}

\begin{lemma}\label{lem:D_to_k_conv}
Let $\alpha\in(\frac23,1]$ and suppose that $[A_n]$ is a $D_\alpha$-Cauchy sequence.
Then there exist $B\in\Omega^1_\alpha$ and representatives $B_n\in[A_n]$ 
such that $k_\alpha(B_n,B)\to 0$.
\end{lemma}

\begin{proof}
We can assume that $A_n$ are ``almost minimal'' representatives of $[A_n]$ in the sense that $|A_n|_\alpha \leq 1+\inf_{g\in\mfG^{0,\alpha}} |A_n^g|_\alpha$.
By Lemma~\ref{lem:k_lower_bound} and the definition of the Hausdorff distance, we see that $\sup_{n \geq 1} |A_n|_\alpha < \infty$, from which it is easy to extract a $k_\alpha$-Cauchy sequence $B_n\in [A_n]$.
\end{proof}

\begin{proof}[of Theorem~\ref{thm:Polish}]
Since every $x\in\mfO_\alpha$ is closed by Lemma~\ref{lem:closed}, $D_\alpha(x,y)=0$ if and only if $x=y$.
The facts that $D_\alpha$ is a complete metric and that it metrises the quotient topology both follow from Lemmas~\ref{lem:k_to_D_conv} and~\ref{lem:D_to_k_conv}.
\end{proof}

\section{Stochastic heat equation}
\label{sec:SHE}

We investigate in this section the regularity of the stochastic heat equation (which is the ``rough part'' of the SYM) with respect to the spaces introduced in Section~\ref{sec:state_space}.
For the remainder of the article, we will focus on the space of ``$1$-forms'' $\Omega^1_\alpha$.

\subsection{Regularising operators}

The main result of this subsection, Proposition~\ref{prop:convolution_estimate}, provides a convenient way to extend regularising properties of an operator $K$ to the spaces $\Omega^1_\alpha$. 
This will be particularly helpful in deriving Schauder estimates and controlling the effect of mollifiers (Corollaries~\ref{cor:Schauder} and~\ref{cor:mollif_estimate}).

Let $E$ be a Banach space throughout this subsection, and consider a linear map $K\colon\mcC^{\infty}(\T^2,E) \to \mcC(\T^2,E)$.
We denote also by $K$ the linear map $K\colon\Omega\mcC^\infty \to \Omega\mcC$ obtained by componentwise extension.
We denote by $\hat K \colon \imath\Omega\mcC^\infty \to \imath \Omega\mcC$ the natural ``lift'' of $K$ given by $\hat K(\imath A) \eqdef \imath(KA)$.
We say that $K$ is translation invariant if $K$ commutes with all translation operators $T_v\colon f\mapsto f(\cdot+v)$.
For $\theta\geq 0$, we denote
\begin{equ}
|K|_{\mcC^{\theta}\to L^\infty} \eqdef \sup\{|K(f)|_\infty \ssep f\in\mcC^\infty(\T^2,E)\;,\;\; |f|_{\mcC^{\theta}}=1\}\;.
\end{equ}
In general, for normed spaces $X,Y$ and a linear map $K\colon D(K) \to Y$, where $D(K) \subset X$, we denote
\begin{equ}
|K|_{X\to Y} \eqdef \sup \{|K (x)|_Y \ssep x\in D(K)\;,\; |x|_X = 1 \} \;.
\end{equ}
If $D(K)$ is dense in $X$, then $K$ does of course extend uniquely to all of $X$ if 
$|K|_{X\to Y} < \infty$. The reason for this setting is that it will be convenient to consider $D(K)$ as fixed
and to allow $X$ to vary.



\begin{proposition}\label{prop:convolution_estimate}
Let $0 < \bar\alpha \leq \alpha \leq 1$.
Let $K\colon\mcC^{\infty}(\T^2,E) \to \mcC(\T^2,E)$ be a translation invariant linear map.
Then
\begin{equ}\label{eq:rho_conv_estimate}
|\hat K|_{\Omega_{\alpha}^1\to\Omega_{\bar\alpha}^1} \lesssim |K|_{\mcC^{(\alpha-\bar\alpha)/2}\to L^\infty}\;.
\end{equ}
Furthermore, if $\bar\alpha \in [\frac\alpha2,\alpha]$, then for all $A\in\imath \Omega\mcC^\infty$
\begin{equ}\label{eq:gr_conv_estimate}
|\hat K A|_{\gr{\bar\alpha}} \lesssim |K|_{\mcC^{\alpha-\bar\alpha}\to L^\infty} |A|_\alpha^{2(\alpha-\bar\alpha)/\alpha} |A|_{\gr\alpha}^{(2\bar\alpha-\alpha)/\alpha}\;.
\end{equ} 
The proportionality constants in both inequalities are universal.
\end{proposition}

\begin{proof}
We suppose that $|K|_{\mcC^{\alpha-\bar\alpha}\to L^\infty}<\infty$, as otherwise there is nothing to prove.
Let $A\in\Omega\mcC^\infty$ and observe that, for $(x,v)\in\mcX$,
\begin{equs}
\int_0^1 (K A_i)(x+tv) \mrd t
&= \int_0^1 T_{tv}(KA_i)(x) \mrd t = \int_0^1 K [T_{tv} A_i] (x) \mrd t
\\
&= K \Big[\int_0^1 T_{tv} A_i \mrd t\Big](x) = K\Big[\int_0^1 A_i(\cdot + tv) \mrd t\Big](x)\;,
\end{equs}
where we used translation invariance of $K$ in the second equality, and the boundedness of $K$ in the third equality.
In particular, it follows from the definition of $\imath\colon\Omega\mcC\to\Omega$ that
\begin{equ}\label{eq:hat_K_identity}
\hat K(\imath A)(x,v) = K[\imath A(\cdot,v)](x)\;.
\end{equ}
We will first prove~\eqref{eq:gr_conv_estimate}.
We claim that for any $\theta\in[0,1]$
\begin{equ}\label{eq:imath_A_mcC_bound}
|\imath A(\cdot,v)|_{\mcC^{(\theta\alpha/2)}} \lesssim |\imath A|_\alpha^\theta |\imath A|_{\gr\alpha}^{1-\theta} |v|^{\alpha(1-\theta/2)}
\end{equ}
for a universal proportionality constant.
Indeed, note that $|\imath A(x,v)|_\infty \leq |\imath A|_{\gr\alpha}|v|^\alpha$ which is bounded above by the right-hand side of~\eqref{eq:imath_A_mcC_bound} for any $\theta\in[0,1]$.
Furthermore, we have for all $x,y\in\T^2$
\begin{equ}\label{eq:A_x_y_bound}
|\imath A(x,v)-\imath A(y,v)| \lesssim \big[|\imath A|_{\alpha} |v|^{\alpha/2} |x-y|^{\alpha/2}\big] \wedge \big[|\imath A|_{\gr\alpha} |v|^{\alpha}\big]\;,
\end{equ}
for a universal proportionality constant,
from which~\eqref{eq:imath_A_mcC_bound} follows by interpolation.
If $\bar\alpha\in[\frac\alpha2,\alpha]$, then we can take $\theta = 2(\alpha-\bar\alpha)/\alpha$
in~\eqref{eq:imath_A_mcC_bound} and combine with~\eqref{eq:hat_K_identity} to obtain~\eqref{eq:gr_conv_estimate}.

We now prove~\eqref{eq:rho_conv_estimate}.
Consider $\ell = (x,v), \bar\ell = (\bar x,\bar v) \in \mcX$.
If $\ell,\bar\ell$ are far, then the necessary estimate follows from~\eqref{eq:gr_conv_estimate}.
Hence, suppose $\ell,\bar\ell$ are not far.
Consider the function $\Psi\in\mcC^\infty(\T^2,E)$ given by $\Psi(y) \eqdef \imath A(y,v)-\imath A(y+\bar x-x,\bar v)$.
Note that~\eqref{eq:hat_K_identity} implies
\begin{equ}\label{eq:K_Psi_identity}
(K\Psi)(x) = \hat K(\imath A)(\ell) - \hat K(\imath A)(\bar \ell)\;.
\end{equ}
We claim that for any $\theta\in[0,1]$
\begin{equ}\label{eq:Psi_mcC_bound}
|\Psi|_{\mcC^{\alpha\theta/2}} \lesssim |A|_\alpha |v|^{\alpha\theta/2} \rho(\ell,\bar\ell)^{\alpha(1-\theta)}\;.
\end{equ}
Indeed, note that $|\Psi|_\infty \leq |A|_{\alpha}\rho(\ell,\bar\ell)^{\alpha}$, which is bounded above (up to a universal constant) by the right-hand side of~\eqref{eq:Psi_mcC_bound} for any $\theta\in[0,1]$.
Furthermore, since $(y,v),(y+\bar x-x,\bar v)$ are also not far for every $y \in \T^2$, and since $|v| \asymp |\bar v|$, we have
\begin{equs}
|\Psi(y)-\Psi(z)|
&=
|A(y,v)-A(y+\bar x-x,\bar v) - A(z, v) + A(z+\bar x-x,\bar v)| \\
&\lesssim |A|_\alpha \big[ \rho(\ell,\bar\ell)^{\alpha} \wedge \big( |v|^{\alpha/2} |y- z|^{\alpha/2} \big) \big]\label{eq:A_x_y_rho_bound}
\end{equs}
for a universal proportionality constant,
from which~\eqref{eq:Psi_mcC_bound} follows by interpolation.
Taking $\theta = \frac{\alpha-\bar\alpha}{\alpha} \Leftrightarrow \bar\alpha = \alpha(1-\theta)$
in~\eqref{eq:Psi_mcC_bound} and combining with~\eqref{eq:K_Psi_identity} proves~\eqref{eq:rho_conv_estimate}.
\end{proof}

As a consequence of Proposition~\ref{prop:convolution_estimate}, any linear map $K\colon\mcC^{\infty}(\T^2,E) \to \mcC(\T^2,E)$ with $|K|_{\mcC^{(\alpha-\bar\alpha)/2}\to L^\infty} < \infty$ uniquely determines a bounded linear map $\hat K\colon \Omega^1_\alpha\to \Omega^{1}_{\bar\alpha}$ which intertwines with $K$ through the embedding $\imath \colon \Omega\mcC \to \Omega$.
The same applies to $\Omega^1_{\gr\alpha}\to \Omega^{1}_{\gr{\alpha}}$ if $|K|_{L^\infty\to L^\infty} < \infty$.
In the sequel, we will denote $\hat K$ by the same symbol $K$ without further notice.

We give two useful corollaries of Proposition~\ref{prop:convolution_estimate}.
For $t\geq 0$, let $e^{t\Delta}$ denote the heat semigroup acting on $\mcC^\infty(\T^2,E)$.

\begin{corollary}\label{cor:Schauder}
Let $0 < \bar\alpha \leq \alpha \leq 1$.
Then for all $A \in \Omega^1_\alpha$, it holds that
\begin{equ}
\label{eq:rho_Schauder}
|(e^{t\Delta}-1)A|_{\bar\alpha} \lesssim t^{(\alpha-\bar\alpha)/4}|A|_{\alpha}\;,
\end{equ}
where the proportionality constant depends only on $\alpha-\bar\alpha$.
\end{corollary}

\begin{proof}
Recall the classical estimate for $\kappa\in[0,1]$
\begin{equ}
|(e^{t\Delta}-1)|_{\mcC^{\kappa}\to L^\infty} \leq |(e^{t\Delta}-1)|_{\mcC^{\Hol\kappa}\to L^\infty} \lesssim t^{\kappa/2}\;.
\end{equ}
The claim then follows from~\eqref{eq:rho_conv_estimate} by taking $\kappa=(\alpha-\bar\alpha)/2$.
\end{proof}

\begin{remark}
The appearance of $t^{(\alpha-\bar\alpha)/4}$ in~\eqref{eq:rho_Schauder} may seem unusual since one instead has $t^{(\alpha-\bar\alpha)/2}$ in the classical Schauder estimates for the H{\"o}lder norm $|\cdot|_{\mcC^{\alpha}}$.
The exponent $(\alpha-\bar\alpha)/4$ is however sharp (which can be seen by looking at the Fourier basis), and is consistent with the embedding of $\mcC^{\alpha/2}$ into $\Omega_\alpha$ (Remark~\ref{rem:Holder_Omega_embedding}).
\end{remark}

\begin{corollary}\label{cor:mollif_estimate}
Let $0<\bar\alpha\leq\alpha\leq1$ and $\kappa\in[0,1]$.
Let $\moll$ be a mollifier on $\R\times \R^2$ and consider a function $A\colon\R\to \Omega_\alpha^1$. 
Then for any interval $I\subset\R$
\begin{equ}
\sup_{t\in I} |(\moll^\eps*A) (t) - A(t)|_{\bar\alpha} \lesssim |\moll|_{L^1}\Big(\eps^{(\alpha-\bar\alpha)/2}\sup_{t\in I_{\eps}}|A(t)|_\alpha + \eps^{2\kappa}|A|_{\mcC^{\Hol\kappa}(I_{\eps},\Omega_{\bar\alpha})}\Big)\;,
\end{equ}
where $I_\eps$ is the $\eps^2$ fattening of $I$, and the proportionality constant is universal.
\end{corollary}

\begin{proof}
For $t\in\R$ define $m(t) \eqdef \int_{\T^2} \moll^\eps(t,x)\mrd x$ and denote by $\moll^\eps(t)$ the convolution operator $[\moll^\eps(t) f](x) \eqdef \scal{\moll^\eps(t,x-\cdot),f(\cdot)}$ for $f\in\mcD'(\T^2)$.
Observe that for any $\theta\in[0,1]$
\begin{equ}
|m(t) f - \moll^\eps(t) f|_{L^\infty} \leq \eps^{\theta}|\moll^\eps(t,\cdot)|_{L^1(\T^2)} |f|_{\mcC^{\Hol\theta}(\T^2)}\;.
\end{equ}
In particular, $|m(t) - \moll^\eps(t)|_{\mcC^{\Hol{(\alpha-\bar\alpha)/2}} \to L^\infty} \leq \eps^{(\alpha-\bar\alpha)/2}|\moll^\eps(t,\cdot)|_{L^1(\T^2)}$.
Hence, for any $t\in I$,
\begin{equs}
|(\moll^\eps*A)(t) - A(t)|_{\bar\alpha}
&\leq \int_\R |(\moll^\eps(s) - m(s))A(t-s)|_{\bar\alpha} \mrd s
\\
&\quad + \int_\R |m(s)(A(t-s) - A(t))|_{\bar\alpha} \mrd s
\\
&\lesssim \int_{\R} |\moll^\eps(s,\cdot)|_{L^1} \eps^{(\alpha-\bar\alpha)/2} |A(t-s)|_\alpha \mrd s
\\
&\quad + \int_\R |m(s)| |s|^{\kappa}|A|_{\mcC^{\Hol\kappa}(I_\eps,\Omega_{\bar\alpha})} \mrd s
\\
&\leq |\moll|_{L^1} \eps^{(\alpha-\bar\alpha)/2} \sup_{t\in I_\eps}|A(t)|_\alpha + \eps^{2\kappa} |\moll|_{L^1}|A|_{\mcC^{\Hol\kappa}(I_\eps,\Omega_{\bar\alpha})}\;,
\end{equs}
where we used~\eqref{eq:rho_conv_estimate} in the second inequality.
\end{proof}

Another useful property is that the heat semigroup is 
strongly continuous on $\Omega^1_\alpha$.
To show this, we need the following lemma.
For a function $\omega\colon\R_+ \to \R_+$, let $\mcC^\omega(\T^2,E)$
denote the space of continuous functions $f\colon \T^2\to E$ with
\begin{equ}
|f|_{\mcC^\omega} \eqdef \sup_{x\neq y} \frac{|f(x)-f(y)|}{\omega(|x-y|)} < \infty\;.
\end{equ}

\begin{lemma}\label{lem:K_modulus_bound}
Let $\alpha\in(0,1]$ and $K\colon\mcC^\infty(\T^2,E) \to \mcC(\T^2,E)$ be a translation invariant linear map with $|K|_{L^\infty\to L^\infty}<\infty$.
Let $A\in\Omega^1_\alpha$ and $\omega\colon\R_+ \to \R_+$, and
suppose that for all
$x,y\in\T^2$, $v,\bar v\in B_{1/4}$, and $h\in\R^2$,
\begin{equ}\label{eq:A_x_y_bound_2}
|A(x,v)-A(y,v)| \leq |A|_{\gr\alpha}|v|^\alpha\omega(|x-y|)
\end{equ}
and
\begin{equ}\label{eq:A_x_y_rho_bound_2}
|A(x,v)-A(x+h,\bar v) - A(y,v)+A(y+h,\bar v)|
\leq |A|_\alpha\rho(\ell,\bar \ell)^\alpha \omega(|x-y|)\;,
\end{equ}
where $\ell=(x,v)$ and $\bar\ell=(x+h,\bar v)$.
Then
\begin{equ}
|K A|_{\gr\alpha} \leq |K|_{\mcC^\omega\to L^\infty} |A|_{\gr\alpha}
\qquad \text{ and }
\qquad |K A|_{\alpha} \leq |K|_{\mcC^\omega\to L^\infty} |A|_{\alpha}\;.
\end{equ}
\end{lemma}

\begin{proof}
The proof is essentially the same as that of Proposition~\ref{prop:convolution_estimate};
one simply replaces~\eqref{eq:A_x_y_bound} by~\eqref{eq:A_x_y_bound_2}
and~\eqref{eq:A_x_y_rho_bound} by~\eqref{eq:A_x_y_rho_bound_2}.
\end{proof}

\begin{proposition}\label{prop:strongly_continuous}
Let $\alpha\in(0,1]$.
The heat semigroup $e^{t\Delta}$ is strongly 
continuous on $\Omega^{1}_{\gr\alpha}$ and $\Omega^1_\alpha$.
\end{proposition}

\begin{proof}
Observe that for every $A\in\imath\Omega\mcC^\infty$ 
there exists a bounded modulus of continuity $\omega\colon\R_+\to\R_+$ such that~\eqref{eq:A_x_y_bound_2} 
and~\eqref{eq:A_x_y_rho_bound_2} hold.
On the other hand, recall that for every bounded modulus of continuity $\omega\colon\R_+\to\R_+$
\begin{equ}
\lim_{t\to0}|e^{t\Delta}-1|_{\mcC^\omega\to L^\infty} = 0\;.
\end{equ}
It follows from Lemma~\ref{lem:K_modulus_bound} that
$\lim_{t\to0} |e^{t\Delta}A-A|_{\gr\alpha} = 0$
for every $A\in\Omega\mcC^\infty$,
and the same for the norm $|\cdot|_\alpha$, from which 
the conclusion follows by density of $\imath\Omega\mcC^\infty$ in $\Omega^{1}_{\gr\alpha}$ and $\Omega^1_\alpha$.
\end{proof}

\subsection{Kolmogorov bound}

In this subsection, let $\xi$ be a $\mfg$-valued Gaussian random distribution on $\R\times \T^2$.
We assume that there exists $C_\xi > 0$ such that
\begin{equ}\label{eq:covariance_bound}
\E[|\scal{\xi,\phi}|^2] \leq C_\xi|\phi|_{L^2(\R\times\T^2)}^2
\end{equ}
for all smooth compactly supported $\phi \colon \R\times\T^2\to\mfg$.
Let $\xi_1,\xi_2$ be two i.i.d.\ copies of $\xi$, and let $\Psi = \sum_{i=1}^2 \Psi_i \mrd x_i$ solve the stochastic heat equation $(\partial_t - \Delta) \Psi = \sum_{i=1}^2\xi_i \mrd x_i$ on $\R_+ \times \T^2$ with initial condition $\iota\Psi(0) \in \Omega^1_\alpha$.

\begin{lemma}\label{lem:tri_Sobolev}
Let $P$ be a triangle with inradius $h$.
Let $\kappa \in (0,1)$ and $p \in [1,\kappa^{-1})$, and let $W^{\kappa,p}$ denote the Sobolev--Slobodeckij space on $\T^2$.
Then $|\bone_{\mathring P}|_{W^{\kappa,p}}^p \lesssim |P| h^{-\kappa p}$, where the proportionality constant depends only on $\kappa p$.
\end{lemma}

\begin{proof}
Using the definition of Sobolev--Slobodeckij spaces, we have
\begin{equs}
|\bone_{\mathring P}|^p_{W^{\kappa,p}}
&= \int_{\T^2}\int_{\T^2} \frac{|\bone_{\mathring P}(x)-\bone_{\mathring P}(y)|^p}{|x-y|^{\kappa p + 2}} \mrd x \mrd y
\\
&= 2\int_{\mathring P} \int_{\T^2\setminus \mathring P} |x-y|^{-\kappa p -2} \mrd x\mrd y
\\
&\leq 2\int_{\mathring P} \mrd x \int_{|y-x|>d(x,\partial P)} \mrd y |x-y|^{-\kappa p -2}
\\
&\lesssim \int_{\mathring P} d(x,\partial P)^{-\kappa p} \mrd x
\\
&= \int_0^1 |\{x \in \mathring P \ssep d(x,\partial P) < \delta\}| \delta^{-1-\kappa p}\mrd \delta\;.
\end{equs}
Note that the integrand is non-zero only if $\delta < h$, in which case $|\{x \in \mathring P \ssep d(x,\partial P) < \delta\}| \lesssim |\partial P|\delta$, where $|\d P|$ denotes the length of the perimeter.
Hence, since $\kappa p < 1$,
\begin{equ}
|\bone_{\mathring P}|^p_{W^{\kappa,p}} \lesssim \int_0^h |\partial P| \delta^{-\kappa p} \mrd \delta \lesssim |\partial P| h^{-\kappa p + 1} \asymp |P| h^{-\kappa p}\;,
\end{equ}
as claimed.
\end{proof}

\begin{lemma}\label{lem:Kolmog_tri_bound}
Let $\kappa \in (0,1)$ and suppose $\Psi(0) = 0$.
Then for any triangle $P$ with inradius $h$
\begin{equ}
\E[|\Psi(t) (\partial P)|^2 ] \lesssim C_\xi t^\kappa|P|h^{-2\kappa} \leq C_\xi t^\kappa |P|^{1-\kappa}\;,
\end{equ}
where the proportionality constant depends only on $\kappa$.
\end{lemma}

\begin{proof}
By Stokes' theorem, we have
\begin{equ}
|\Psi(t)(\partial P)| = |\scal{\partial_1 \Psi_2(t) - \partial_2 \Psi_1(t) , \bone_{\mathring P}}|\;.
\end{equ}
Observe that
\begin{equ}
\scal{\partial_1 \Psi_2(t),\bone_{\mathring P}} = -\int_{\R\times \T^2}\xi_2(s,y) \bone_{s \in [0,t]} [e^{(t-s)\Delta}\partial_1\bone_{\mathring P}](y) \mrd s\mrd y\;.
\end{equ}
Hence, by~\eqref{eq:covariance_bound},
\begin{equ}
\E[|\scal{\partial_1 \Psi_2(t),\bone_{\mathring P}}|^2] \leq C_\xi\int_{0}^t |e^{s\Delta}\partial_1\bone_{\mathring P}|_{L^2}^2 \mrd s\;.
\end{equ}
By the estimate $|e^{s\Delta} f|_{L^2} \lesssim s^{(-1+\kappa)/2}|f|_{H^{-1+\kappa}}$, we have
\begin{equ}
|e^{s\Delta}\partial_1\bone_{\mathring P}|_{L^2}^2 \lesssim s^{-1+\kappa} |\partial_1 \bone_{\mathring P}|_{H^{-1+\kappa}}^2\;.
\end{equ}
Since $|\partial_1 f|_{H^{-1+\kappa}} \lesssim |f|_{H^{\kappa}}$, we have by Lemma~\ref{lem:tri_Sobolev}
\begin{equ}
\E[|\scal{\partial_1 \Psi_2(t),\bone_{\mathring P}}|^2] \lesssim C_\xi \int_0^t s^{-1+\kappa} |P| h^{-2\kappa} \mrd s \lesssim C_\xi t^{\kappa} |P|h^{-2\kappa}\;.
\end{equ}
Likewise for the term $\scal{\partial_2 \Psi_1(t),\bone_{\mathring P}}$, and the conclusion follows from the inequality $\pi h^2 \leq |P|$.
\end{proof}

\begin{lemma}\label{lem:delta_ell_Sobolev}
Let $\ell = (x,v) \in \mcX$ and consider the distribution $\scal{\delta_\ell,\psi} \eqdef \int_0^1 |v|\psi(x+tv) \mrd t$.
Then, for any $\kappa \in (\frac{1}{2},1)$,
\begin{equ}
|\delta_\ell|_{H^{-\kappa}} \lesssim |\ell|^\kappa\;,
\end{equ}
where the proportionality constant depends only on $\kappa$.
\end{lemma}

\begin{proof}
By rotation and translation invariance, we may assume $\ell = (0,|\ell| e_1)$.
For $k = (k_1,k_2) \in \Z^2$, we have $\scal{\delta_\ell, e^{2\pi i \scal{k,\cdot}}} = (e^{2\pi i k_1 |\ell|} - 1)/(2\pi i k_1)$ if $k_1\neq 0$ and $\scal{\delta_\ell, e^{2\pi i \scal{k,\cdot}}}=|\ell|$ if $k_1=0$.
Hence
\begin{equs}
|\delta_\ell|_{H^{-\kappa}}^2
&= \sum_{k \in \Z^2} |\scal{\delta_\ell, e^{2\pi i \scal{k,\cdot}}}|^2 (1+k_1^2+k_2^2)^{-\kappa}
\\
&\lesssim \sum_{k \in \Z^2} (|\ell|^2 \wedge k_1^{-2})(1+k_1^2 + k_2^2)^{-\kappa}
\\
&\lesssim \sum_{k\in\Z} (|\ell|^2 \wedge k^{-2})(1+k)^{1-2\kappa}\;.
\end{equs}
Splitting the final sum into $|k| \leq |\ell|^{-1}$ and $|k| > |\ell|^{-1}$ yields the desired result.
\end{proof}

\begin{lemma}\label{lem:Kolmog_gr_bound}
Let $\kappa \in (0,\frac12)$ and suppose $\Psi(0) = 0$.
Then for any $\ell \in \mcX$
\begin{equ}
\E[|\Psi(t)(\ell)|^2] \lesssim C_\xi t^{\kappa} |\ell|^{2-2\kappa}\;,
\end{equ}
where the proportionality constant depends only on $\kappa$.
\end{lemma}

\begin{proof}
Observe that $\Psi(t)(\ell) = \sum_{i=1}^2 |v|^{-1}v_i\scal{\Psi_i(t),\delta_\ell}$, where we used the notation of Lemma~\ref{lem:delta_ell_Sobolev}.
Furthermore,
\begin{equ}
\scal{\Psi_i(t),\delta_\ell} = \int_{\R\times \T^2}\xi_i(s,y) \bone_{s \in [0,t]} [e^{(t-s)\Delta}\delta_\ell](y) \mrd s\mrd y\;.
\end{equ}
Hence, by~\eqref{eq:covariance_bound},
\begin{equ}
\E\big[\scal{\Psi_i(t),\delta_\ell}^2\big] \leq C_\xi \int_{0}^t |e^{s\Delta}\delta_\ell|_{L^2}^2 \mrd s\;.
\end{equ}
The estimate $|e^{s\Delta} f|_{L^2} \lesssim s^{(\kappa-1)/2}|f|_{H^{\kappa-1}}$ implies
$|e^{s\Delta}\delta_\ell|_{L^2}^2 \lesssim s^{\kappa-1} |\delta_\ell|_{H^{\kappa-1}}^2$.
Hence, by Lemma~\ref{lem:delta_ell_Sobolev},
\begin{equ}
\E\big[\scal{\Psi_i(t),\delta_\ell}^2\big] \lesssim C_\xi \int_0^t s^{-1+\kappa} |\ell|^{2-2\kappa} \mrd s \lesssim C_\xi t^{\kappa} |\ell|^{2-2\kappa}\;,
\end{equ}
and the claim follows from the bound $||v|^{-1}v_i| \lesssim 1$.
\end{proof}

Since our ``index space'' $\mcX$ and ``distance'' function $\rho$ are not entirely standard, we spell out the following Kolmogorov-type criterion.

\begin{lemma}\label{lem:Kolmogorov_fixed_time}
Let $A$ be a $\mfg$-valued stochastic process indexed by $\mcX$ such that,
 for all joinable $\ell,\bar\ell \in \mcX$,
$A(\ell\sqcup\bar\ell) = A(\ell)+A(\bar\ell)$ almost surely.
Suppose that there exist $p \geq 1$, $M > 0$, and $\alpha \in (0,1]$ such that for all $\ell\in\mcX$
\begin{equ}
\E[|A(\ell)|^p] \leq M |\ell|^{p\alpha}\;,
\end{equ}
and for all triangles $P$
\begin{equ}
\E[|A(\partial P)|^p ] \leq M |P|^{p\alpha/2}\;.
\end{equ}
Then there exists a modification of $A$ (which we denote by the same letter) which is a.s.\ a continuous function on $\mcX$.
Furthermore, for every $\bar\alpha \in (0,\alpha-\frac{16}{p})$, there exists $\lambda > 0$, depending only on $p,\alpha,\bar\alpha$, such that
\begin{equ}
\E[|A|_{\bar\alpha}^p] \leq \lambda M\;.
\end{equ}
\end{lemma}
\begin{proof}
Observe that for any $\ell,\bar\ell\in\mcX$, we can write $A(\ell)-A(\bar\ell)=A(\partial P_1)+A(\partial P_2) + A(a) - A(b)$, where $|P_1|+|P_2| \leq \rho(\ell,\bar\ell)^2$ and $|a|+|b|\leq \rho(\ell,\bar\ell)$
(if $\ell,\bar\ell$ are far, then $a=\ell$, $b=\bar\ell$, and $P_1,P_2$ are empty).
It follows that for all $\ell,\bar\ell \in \mcX$
\begin{equ}\label{eq:rho_p_bound}
\E[|A(\ell)-A(\bar\ell)|^p ] \lesssim M \rho(\ell,\bar\ell)^{p\alpha}\;,
\end{equ}
where the proportionality constant depends only on $p,\alpha$.
For $N \geq 1$ let $D_N$ denote the set of line segments in $\mcX$ whose start and end points have dyadic coordinates of scale $2^{-N}$, and let $D=\cup_{N\geq 1} D_N$.
For $r>0,\ell\in\mcX$, define $B_{\rho}(r,\ell) \eqdef \{\bar\ell \in \mcX \ssep \rho(\bar\ell,\ell) \leq r\}$.
From Definition~\ref{def:rho} and Remark~\ref{rem:inframetric}, we see that for some $K >0$, the family $\{B_{\rho}(K2^{-N},\ell)\}_{\ell\in D_{2N}}$ covers $\mcX$ (quite wastefully) for every $N \geq 1$.
It readily follows, using~\eqref{eq:inframetric}, that for any $\bar\alpha \in (0,1]$
\begin{equ}\label{eq:A_dyadic_bound}
\sup_{\ell,\bar\ell \in D} \frac{|A(\ell)-A(\bar\ell)|^p}{\rho(\ell,\bar\ell)^{\bar\alpha p}} \lesssim \sum_{N \geq 1} \sum_{\substack{a,b\in D_{2N}\\ \rho(a,b) \leq K2^{-N}}} 2^{N\bar\alpha p} |A(a)-A(b)|^p\;.
\end{equ}
Observe that $|D_{2N}| \leq 2^{8N}$, and thus the second sum has at most $2^{16N}$ terms.
Hence, for $\bar\alpha \in(0,\alpha-\frac{16}{p})\Leftrightarrow 16+p(\bar\alpha-\alpha)<0$, we see from~\eqref{eq:rho_p_bound} that the expectation of the right-hand side of~\eqref{eq:A_dyadic_bound} is bounded by $\lambda(p,\alpha,\bar\alpha)M$. The conclusion readily follows as in the classical Kolmogorov continuity theorem.
\end{proof}
By equivalence of moments for Gaussian random variables, Lemmas~\ref{lem:Kolmog_tri_bound},~\ref{lem:Kolmog_gr_bound}, and~\ref{lem:Kolmogorov_fixed_time} yield the following lemma.
\begin{lemma}\label{lem:Kolmog_small_time}
Suppose $\Psi(0) = 0$.
Then for any $p > 16$, $\alpha \in (0,1)$, and $\bar\alpha \in (0,\alpha-\frac{16}{p})$, there exists $C > 0$, depending only on $p,\bar\alpha,\alpha$, such that for all $t \geq 0$
\begin{equ}
\E[|\Psi(t)|_{\bar\alpha}^p] \leq C C_\xi^{p/2} t^{p(1-\alpha)/2}\;.
\end{equ}
\end{lemma}
We are now ready to prove the following continuity theorem.
\begin{theorem}\label{thm:SHE_regularity}
Let $0 < \bar\alpha < \alpha < 1$, $\kappa \in (0,\frac{\alpha-\bar\alpha}{4})$, and suppose $\Psi(0) \in \Omega_\alpha$.
Then for all $p\geq 1$ and any $T > 0$
\begin{equ}
\E\Big[ \sup_{0 \leq s < t \leq T} \frac{|\Psi(t)-\Psi(s)|_{\bar\alpha}^p}{|t-s|^{p\kappa}}\Big]^{1/p} \lesssim |\Psi(0)|_\alpha + C_\xi^{1/2}
\end{equ}
where the proportionality constant depends only on $p,\alpha,\bar\alpha,\kappa,T$.
\end{theorem}
\begin{proof}
Let $0 \leq s \leq t \leq T$ and observe that
\begin{equ}
\Psi(t)-\Psi(s) = (e^{(t-s)\Delta}-1)e^{s\Delta}\Psi(0) + (e^{(t-s)\Delta}-1)\tilde\Psi(s) + \hat\Psi(t)\;,
\end{equ}
where $\tilde\Psi \colon [0,s] \to \Omega_{\bar\alpha}$ and $\hat\Psi \colon [s,t] \to \Omega_{\bar\alpha}$ driven by $\xi$ with zero initial conditions. 
By Corollary~\ref{cor:Schauder},
\begin{equ}
|(e^{(t-s)\Delta}-1)e^{s\Delta}\Psi(0)|_{\bar\alpha} \lesssim |t-s|^{(\alpha-\bar\alpha)/4}|\Psi(0)|_\alpha\;.
\end{equ}
Likewise, by Corollary~\ref{cor:Schauder} and Lemma~\ref{lem:Kolmog_small_time}, for any $\beta > \alpha+\frac{16}{p}$
\begin{equ}
\E[|(e^{(t-s)\Delta}-1)\tilde\Psi(s)|^p_{\bar\alpha}] \lesssim C_\xi^{p/2} |t-s|^{p(\alpha-\bar\alpha)/4}s^{p(1-\beta)/2}\;.
\end{equ}
Finally, by Lemma~\ref{lem:Kolmog_small_time}, for any $\beta > \bar\alpha+\frac{16}{p}$
\begin{equ}
\E[|\hat\Psi(t)|^p_{\bar\alpha}] \lesssim C_\xi^{p/2} |t-s|^{p(1-\beta)/2}\;.
\end{equ}
In conclusion, for all $p$ sufficiently large,
\begin{equ}
\E[|\Psi(t)-\Psi(s)|^p_{\bar\alpha}] \lesssim |t-s|^{p(\alpha-\bar\alpha)/4} (|\Psi(0)|_\alpha^p+ C_\xi^{p/2})\;.
\end{equ}
The conclusion follows by the classical Kolmogorov continuity criterion.
\end{proof}
\begin{corollary}\label{cor:SHE_mollif_converge}
Let $\moll$ be a mollifier on $\R\times\T^2$.
Suppose that $\xi$ is a $\mfg$-valued white noise and denote $\xi^\eps\eqdef \moll^\eps*\xi$.
Suppose that $\Psi(0)=0$ and let $\Psi^\eps$ solve $(\partial_t-\Delta)\Psi^\eps=\xi^\eps$ on $\R_+\times \T^2$ with zero initial condition $\Psi^\eps(0)=0$.
Let $\alpha\in(0,1)$, $T>0$, $\kappa\in(0,\frac{1-\alpha}{4})$, and $p \geq 1$.
Then
\begin{equ}
\E\Big[\sup_{t\in[0,T]}|\Psi^\eps(t)-\Psi(t)|_{\alpha}^p\Big]^{1/p} \lesssim \eps^{2\kappa}|\moll|_{L^1}\;,
\end{equ}
where the proportionality constant depends only on $\alpha,\kappa,T,p$.
\end{corollary}
\begin{proof}
Observe that, by Theorem~\ref{thm:SHE_regularity},
\begin{equ}\label{eq:Psi_E_bound}
\E\Big[\sup_{t\in[0,\eps^2]}|\Psi(t)|_{\alpha}^p\Big]^{1/p} \lesssim \eps^{2\kappa}\;.
\end{equ}
Furthermore, for any $\phi\in L^2(\R\times\T^2)$, by Young's inequality,
\begin{equ}
\E[\scal{\xi^\eps,\phi}^2]=|\moll^\eps*\phi|_{L^2}^2 \leq |\phi|_{L^2}^2|\moll^\eps|^2_{L^1}\;.
\end{equ}
Hence $\xi^\eps$ satisfies~\eqref{eq:covariance_bound} with $C_{\xi^\eps} \eqdef |\moll|_{L^1}^2$.
It follows again by Theorem~\ref{thm:SHE_regularity} that
\begin{equ}\label{eq:Psi_eps_E_bound}
\E\Big[\sup_{t\in[0,\eps^2]}|\Psi^\eps(t)|_{\alpha}^p\Big]^{1/p} \lesssim \eps^{2\kappa}|\moll|_{L^1}\;.
\end{equ}
It remains to estimate $\E[\sup_{t\in[\eps^2,T]}|\Psi(t)-\Psi^\eps(t)|_{\alpha}^p]$.

Denoting $I\eqdef[\eps^2,T]$, observe that by Corollary~\ref{cor:mollif_estimate}, for any $\bar\alpha\in[\alpha,1]$
\begin{multline*}
\E\Big[\sup_{t\in I}|\Psi(t) - \moll^\eps*\Psi(t)|^p_{\alpha}\Big]^{1/p}
\\
\lesssim |\moll|_{L^1} \Big\{\eps^{(\bar\alpha-\alpha)/2} \E\Big[\sup_{t\in I_\eps}|\Psi(t)|_{\bar\alpha}^p\Big]^{1/p} + \eps^{2\kappa}\E\Big[|\Psi|_{\mcC^{\Hol\kappa}(I_\eps,\Omega_{\alpha})}^p\Big]^{1/p}\Big\}\;.
\end{multline*}
Both expectations are finite provided $\bar\alpha<1$, and thus the right-hand side 
is bounded above by a multiple of $\eps^{2\kappa}|\moll|_{L^1}$.

We now estimate $\E[\sup_{t\in I}|\moll^\eps*\Psi-\Psi^\eps|^p_\alpha]$.
Let us denote by $\bone_+$ the indicator on the set $\{(t,x)\in\R\times\T^2\ssep t \geq 0\}$.
Observe that $\bone_+(\moll^\eps * \xi)(t,x)$ and $\moll^\eps*(\bone_+\xi)(t,x)$ both vanish if $t<-\eps^2$ and agree if $t>\eps^2$.
In particular, $\moll^\eps*\Psi$ and $\Psi^\eps$ both solve the (inhomogeneous) heat equation on $[\eps^2,\infty)\times\T^2$ with the same source term but with possibly different initial conditions.
To estimate these initial conditions, for $s\in[-\eps^2,\eps^2]$, let us denote by $\moll^\eps(s)$ the convolution operator $[\moll^\eps(s) f](x) \eqdef \scal{\moll^\eps(s,x-\cdot),f(\cdot)}$ for $f\in\mcD'(\T^2)$.
Observe that
\begin{equ}
|\moll^\eps(s)|_{L^\infty\to L^\infty} \leq \mu(s)\eqdef \int_{\T^2} |\moll^\eps(s,x)|\mrd x\;,
\end{equ}
and thus $|\moll^\eps(s)A|_{\alpha} \lesssim \mu(s) |A|_{\alpha}$ for any $A\in\Omega_\alpha$ by~\eqref{eq:rho_conv_estimate}.
Hence
\begin{equs}
\E[|\moll^\eps*\Psi(\eps^2)|^p_{\alpha}]^{1/p}
&= \E\Big[\Big|\int_{\R}\moll^\eps(s)\Psi(\eps^2-s)\mrd s \Big|_{\alpha}^p\Big]^{1/p}
\\
&\lesssim |\moll|_{L^1} \E\Big[\sup_{t\in[0,2\eps^2]} |\Psi|^p_{\alpha}\Big]^{1/p}
\lesssim |\moll|_{L^1}\eps^{2\kappa}\;.
\end{equs}
As a result, by Theorem~\ref{thm:SHE_regularity} and recalling that $\xi^\eps$ satisfies~\eqref{eq:covariance_bound} with $C_{\xi^\eps}=|\moll|^2_{L^1}$, we obtain
\begin{equ}
\E[|\moll^\eps*\Psi(\eps^2)|^p_{\alpha}]^{1/p} + \E[|\Psi^\eps(\eps^2)|^p_{\alpha}]^{1/p} \lesssim \eps^{2\kappa}|\moll|_{L^1}\;.
\end{equ}
Finally,
\begin{equ}
\E[\sup_{t\geq \eps^2}|\moll^\eps*\Psi(t) - \Psi^\eps(t)|_{\alpha}^p ]^{1/p} \lesssim \E[|\moll^\eps*\Psi(\eps^2) - \Psi^\eps(\eps^2)|_{\alpha}^p ]^{1/p} \lesssim \eps^{2\kappa}|\moll|_{L^1}
\end{equ}
where we used Corollary~\ref{cor:Schauder} in the first inequality.
\end{proof}

\section{Regularity structures for vector-valued noises}
\label{sec:vector-valued_noises}

\subsection{Motivation}
\label{sec:motivation}

As already mentioned in the introduction, the aim of this section
is to provide a solution \slash renormalisation theory for SPDEs of the form 
\begin{equ}\label{e:SPDE}
(\partial_t - \mathscr{L}_\mft) A_\mft = F_{\mft}(\mbA, \bxi)\;, \quad \mft\in\mfL_+\;,
\end{equ}
where the nonlinearities $(F_\mft)_{\mft \in \Lab_+}$, linear operators $(\mathscr{L}_\mft)_{\mft \in \Lab_+}$,
and noises $(\xi_\mft)_{\mft \in \Lab_-}$, satisfy the assumptions required for the  general theory
of \cite{Hairer14,CH16,BCCH21,BHZ19} to apply. The problem is that this theory assumes that 
the different components of the solutions $A_\mft$ and of its driving noises $\xi_\mft$ 
are scalar-valued. While this is not a restriction
in principle (simply expand solutions and noises according to some arbitrary basis 
of the corresponding spaces), it makes it rather unwieldy to obtain an expression
for the precise form of the counterterms generated by the renormalisation procedure
described in \cite{BCCH21}. 

Instead, one would much prefer a formalism in which the vector-valued natures of
both the solutions and the driving noises are preserved. To motivate our construction,
consider the example
of a $\mfg$-valued noise $\xi$, where $\mfg$ is some finite-dimensional vector space. One way of describing
it in the context of \cite{Hairer14} would be to choose a basis $\{e_1,\ldots,e_n\}$ of $\mfg$ and to consider
a regularity structure $\CT$ with basis vectors $\Xi_i$ endowed with a model $\PPi$ such that 
$\PPi \Xi_i = \xi_i$ with $\xi_i$ such that $\xi = \sum_{i=1}^n \xi_i e_i$. 
We could then also consider the element $\boldsymbol{\Xi} \in \CT \otimes \mfg$
obtained by setting $\boldsymbol{\Xi} = \sum_{i=1}^n\Xi_i \otimes e_i$. 
When applying the model to $\boldsymbol{\Xi}$ (or rather its first factor), we
then obtain $\PPi \boldsymbol{\Xi} = \sum_{i=1}^n\xi_i\, e_i = \xi$ as expected.
A cleaner coordinate-independent way of achieving the same result is to view the subspace
 $\CT[\Xi]\subset \CT$ spanned by the $\Xi_i$ as a copy
of $\mfg^*$, with $\PPi$ given by $\PPi \Xi_g = g(\xi)$ for any element $\Xi_g$ in this copy of $ \mfg^*$.  
In this way, 
$\boldsymbol{\Xi} \in \CT[\Xi] \otimes \mfg \simeq \mfg^* \otimes \mfg$ is simply given 
by $\boldsymbol{\Xi} = \id_{\mfg}$, where $\id_{\mfg}$ denotes the identity map $\mfg \to \mfg$, modulo
the canonical correspondence $L(X,Y) \simeq X^* \otimes Y$.
(Here and below we will use the notation $X\simeq Y$ to denote the existence of a
\textit{canonical} isomorphism between objects $X$ and $Y$.)
	
\begin{remark}
This viewpoint is consistent with the natural correspondence between a $\mfg$-valued rough path $\mbX$ and a model $\PPi$.
Indeed, while $\mbX$ takes values in $\mcH^*$, 
the tensor series \slash Grossman--Larson algebra over $\mfg$,
one evaluates the model $\PPi$ against elements of its predual $\mcH$,
the tensor \slash Connes--Kreimer algebra over $\mfg^*$, see~\cite[Sec.~4.4]{Hairer14} and~\cite[Sec.~6.2]{BCFP19}.
\end{remark}

Imagine now a situation in which we are given $\mfg_1$ and $\mfg_2$-valued noises $\xi_1 = \<XiR>$
and $\xi_2 = \<XiY>$, as well as an integration kernel $K$ which we draw as a plain line,
and consider the symbol $\<IXi^2asym>$. It seems natural in view of the above discussion to 
associate it with a subspace of $\CT$ isomorphic to $\mfg_1^*\otimes \mfg_2^*$ 
(let's borrow the notation from \cite{Mate2} and denote this subspace as 
$\<IXi^2asym> \otimes \mfg_1^*\otimes \mfg_2^*$) and 
to have the canonical model act on it as
\begin{equ}
\PPi (\<IXi^2asym> \otimes g_1 \otimes g_2) = \bigl(K * g_1(\xi_1)\bigr)\,\bigl(K * g_2(\xi_2)\bigr)\;,
\end{equ}
for $g_i \in \mfg_i^*$.
It would appear that such a construction necessarily breaks the commutativity of the product
since in the same vein one would like to associate to $\<IXi^2asym2>$ a copy of 
$\mfg_2^*\otimes \mfg_1^*$, but this can naturally be restored by simply postulating that in $\CT$
one has the identity
\begin{equ}[e:postulate-comm]
(\<IXired> \otimes g_1)\cdot(\<IXigreen> \otimes g_2) = \<IXi^2asym> \otimes g_1 \otimes g_2 = \<IXi^2asym2> \otimes g_2 \otimes g_1\
= (\<IXigreen> \otimes g_2) \cdot (\<IXired> \otimes g_1)\;.
\end{equ}
This then forces us to associate to $\<IXi^2>$ a copy of the symmetric tensor product $\mfg_1^*\otimes_s \mfg_1^*$.
The goal of this section is to provide a functorial description of such considerations which allows us to 
transfer algebraic identities for regularity structures of trees of the type considered in \cite{BHZ19} to the present
setting where each noise (or edge) type $\mft$ is associated to a vector space $\mfg_{\mft}^*$.
This systematises previous constructions like \cite[Sec.~3.1]{Mate2} or \cite[Sec.~3.1]{PhilippMall}
where similar considerations were made in a rather ad hoc manner.
Our construction bears a resemblance to that of \cite{TomTrees} who introduced a similar formalism in the context
of rough paths, but our formalism is more functorial and better suited for our purposes.

\begin{remark}
We will see that our framework will also allow us to easily accommodate vector valued kernels, see the discussion around \eqref{e:KSPDE} and Remark~\ref{rem:vect_kernel_example}.
\end{remark}

\subsection{Symmetric sets and symmetric tensor products}
\label{subsec:sym_tensor_prod}

Fix a collection $\Lab$ of types and recall that a ``typed set'' $T$ consists of a finite set (which we denote again by $T$)
together with a map $\mft \colon T \to \Lab$.
For any two typed sets $T$ and $\bar T$, write $\Iso(T, \bar T)$ for the set of all type-preserving
bijections from $T \to \bar T$. 
\begin{definition}\label{def:sym-typed-set} 
A \textit{symmetric set} $\symset$ consists of a non-empty index set $A_\symset$, as well as a triple 
$\symset = (\{T_\symset^a\}_{a\in A_\symset},\{\mft^a_\symset\}_{a\in A_\symset}, \{\Gamma_\symset^{a,b}\}_{a,b \in A_\symset})$ where 
$(T_\symset^a ,\mft_\symset^a)$ is a typed set
and $\Gamma_\symset^{a,b} \subset \Iso(T_\symset^b, T_\symset^a)$ are non-empty sets such that, for any $a,b,c\in A_\symset$,
\begin{equs}
\gamma \in \Gamma_\symset^{a,b} \quad&\Rightarrow \quad \gamma^{-1} \in \Gamma_\symset^{b,a}\;,\\
\gamma \in \Gamma_\symset^{a,b}\;,\quad
\bar \gamma \in \Gamma_\symset^{b,c}\quad&\Rightarrow \quad
\gamma \circ \bar \gamma \in \Gamma_\symset^{a,c}\;.
\end{equs}
In other words, a symmetric set is a connected groupoid inside $\Set_\Lab$, 
the category of typed sets endowed with type-preserving maps.
\end{definition}

\begin{example}\label{ex:trees}
An example of a symmetric set is obtained via ``a tree with $\Lab$-typed leaves''.
A concrete tree $\tau$ is defined by fixing a vertex set $V$, which we can take without loss of generality
as a finite subset of $\N$ (which we choose to play the role of the set of all possible vertices), together with an (oriented) edge set $E \subset V \times V$ so that the
resulting graph is a rooted tree with all edges oriented towards the root, as well
as a  labelling $\mft \colon L \to \Lab$, with $L \subset V$ the set of leaves.
However, when we draw\footnote{By convention, we always draw the root at the bottom of the tree.} a``tree with $\Lab$-typed leaves'' such as $\<IXi^3-typed112>$, we are actually specifying an isomorphism class of trees since we are not specifying $V$, $E$, and $\mft$ as concrete sets. Thus $\<IXi^3-typed112>$ corresponds to an infinite isomorphism class of trees $\ttau$
with each $\tau \in \ttau$ being a concrete representative of $\ttau$. 

For any two concrete representatives $\tau_1 = (V_1, E_1, \mft_1) $ and $\tau_2 = (V_2, E_2, \mft_2)$ in the isomorphism class $\<IXi^3-typed112>$, we have two distinct tree isomorphisms $\gamma\colon V_1 \sqcup E_1 \to V_2\sqcup E_2$ 
which preserve
the typed tree structure, since it doesn't matter how the two vertices of type \<Xi> get mapped
onto each other. In this way, we have an unambiguous way of viewing
$\ttau = \<IXi^3-typed112>$ as a symmetric set with local symmetry 
group isomorphic to $\Z_2$.
\end{example}

\begin{remark}
Each of the sets $\Gamma_\symset^{a,a}$ forms a group and, by connectedness, these are all (not necessarily canonically) isomorphic.
We will call this isomorphism class the ``local symmetry group'' of $\symset$.
\end{remark}

Given a typed set $T$ and a symmetric set $\symset$, we define 
\begin{equ}
\Homb(T,{\symset}) = \Big(\bigcup_{a\in A_\symset} \Iso(T, T_{\symset}^a)\Big) \Big/ \Gamma_{{\symset}}\;,
\end{equ}
i.e.\ we postulate that $\Iso(T,T_\symset^a) \ni \phi \sim \tilde \phi \in \Iso(T,T_\symset^b)$
if and only if there exists $\gamma \in \Gamma_{\symset}^{a,b}$ such that $\phi = \gamma \circ \tilde \phi$. 
Note that, by connectedness of $\Gamma_\symset$, any equivalence class in $\Homb(T,{\symset})$ has, for every $a \in A_{\symset}$, at least one representative $\phi^{a} \in \Iso(T, T_{\symset}^a) $. 
Given two symmetric sets $\symset$ and $\bar\symset$, we also define the set $\VHomb(\symset,\bar \symset)$
of ``sections'' by
\begin{equ}
\VHomb(\symset,\bar \symset) = \big\{\Phi = (\Phi_a)_{a \in A_\symset}\,:\, \Phi_a \in \Vec \big(\Homb(T_\symset^a,\bar \symset)\big)\big\}\;,
\end{equ}
where $\Vec(X)$ denotes the real vector space spanned by a set $X$.
We then have the following definition.
\begin{definition}\label{def:morphism}
A \textit{morphism} between two symmetric sets $\symset$ and $\bar{\symset}$ is a $\Gamma_\symset$-invariant
section; namely, an element of 
\begin{equ}
\Hom(\symset,\bar{\symset}) = \{\Phi \in \VHomb(\symset,\bar \symset) \,:\, \Phi_b = \Phi_a \circ \gamma_{a,b} \quad \forall a,b\in A_\symset\;,\; \forall \gamma_{a,b} \in \Gamma_{\symset}^{a,b}\}\;.
\end{equ}
Here, we note that right composition with $\gamma_{a,b}$  gives a well-defined map from $\Homb(T_\symset^b,\bar \symset)$ to $\Homb(T_\symset^a,\bar \symset)$ and we extend this to $\Vec(\Homb(T_\symset^b,\bar \symset))$ by linearity.
Composition of morphisms is defined in the natural way by
\begin{equ}
(\bar \Phi\circ \Phi)_a = \bar \Phi_{\bar a} \circ \Phi_a^{(\bar a)}\;,
\end{equ}
where $\Phi_a^{(\bar a)}$ denotes an arbitrary representative of $\Phi_a$ in 
$\Vec\big(\Iso(T_\symset^a, T_{\bar \symset}^{\bar a})\big)$ and composition is
extended bilinearly.
It is straightforward to verify that this is independent of the choice of $\bar a$
and of representative $\Phi_a^{(\bar a)}$ thanks to the invariance property 
$\bar \Phi_{\bar a} \circ \gamma_{\bar a,\bar b} = \bar \Phi_{\bar b}$, as well as the postulation of the equivalence relation in the definition of $\Homb(T_{\symset}^a,\bar\symset)$.
\end{definition}
\begin{remark}\label{rem:generalisation}
A natural generalisation of this construction is obtained by replacing $\Lab$ by an arbitrary finite category.
In this case, typed sets are defined as before, with each element having as type an object of $\Lab$.
Morphisms between typed sets $A$ and $\bar A$ are then given by maps
$\phi\colon A \to \bar A \times \Hom_\Lab$ such that, writing $\phi = (\phi_0, \vec \phi)$,
one has $\vec\phi(a) \in \Hom_\Lab(\mft(a),\mft(\phi_0(a)))$ for every $a \in A$. Composition is defined
in the obvious way by ``following the arrows'', namely
\begin{equ}
(\psi \circ \phi)_0 = \psi_0 \circ  \phi_0\;,\qquad (\psi \circ \phi)(a) = \vec \psi(\phi_0(a))  \circ \vec\phi(a)\;,
\end{equ}
where the composition on the right takes place in $\Hom_\Lab$. 
The set $\Iso(A,\bar A)$ is then defined as those morphisms $\phi$ such that $\phi_0$ is a bijection, but we
do \textit{not} impose that $\vec \phi(a)$ is an isomorphism in $\Lab$ for $a \in A$.
\end{remark}

\begin{remark}
Note that, for any symmetric set $\symset$, there is a natural identity element $\id_{\symset} \in \Hom(\symset,\symset)$ given by
$a \mapsto [\id_{T_{\symset}^a}]$, with $[\id_{T_{\symset}^a}]$ denoting the equivalence class of $\id_{T_{\symset}^a}$
in $\Homb(T_\symset^a, \symset)$. 
In particular, symmetric sets form a category, 
which we denote by $\SSet$ (or $\SSet_\Lab$).\label{def:SSet}
\end{remark}
\begin{remark}
We choose to consider formal linear combinations
in our definition of $\VHomb$ since otherwise the resulting definition of $\Hom(\symset,\bar \symset)$
would be too small for our purpose.
\end{remark}
\begin{remark}\label{rem:singleton}
An important special case is given by the case when $A_\symset$ and $A_{\bar \symset}$ are singletons.
In this case, $\Hom(\symset,\bar{\symset})$ can be viewed as a subspace of $\Vec\big(\Homb(T_\symset, \bar \symset)\big)$,
$\Homb(T_\symset, \bar \symset) = \Iso(T_\symset, T_{\bar \symset})/\Gamma_{\bar \symset}$, and $\Gamma_{\bar \symset}$
is a subgroup of $\Iso(T_{\bar \symset}, T_{\bar \symset})$.
\end{remark}
\begin{remark}
An alternative, more symmetric, way of viewing morphisms of $\SSet$ is as two-parameter maps
\begin{equ}
A_\symset \times A_{\bar \symset} \ni (a,\bar a) \mapsto \Phi_{\bar a, a} \in \Vec \bigl(\Iso(T_\symset^a, T_{\bar \symset}^{\bar a})\bigr)\;,
\end{equ}
which are invariant in the sense that, for any 
$\gamma_{a,b} \in \Gamma_\symset^{a,b}$ and
$\bar \gamma_{\bar a,\bar b} \in \Gamma_{\bar \symset}^{\bar a,\bar b}$, one 
has the identity
\begin{equ}[e:Phi-aa-gamma-ab]
\Phi_{\bar a,a}\circ \gamma_{a,b} = \bar \gamma_{\bar a,\bar b} \circ \Phi_{\bar b,b}\;.
\end{equ}
Composition is then given by 
\begin{equ}
(\bar \Phi\circ \Phi)_{\bbar a, a} = \bar \Phi_{\bbar a, \bar a} \circ \Phi_{\bar a, a}\;,
\end{equ}
for any \textit{fixed} choice of $\bar a$ (no summation).
Indeed, it is easy to see that for any choice of $\bar{a}$, $\bar \Phi\circ \Phi$ satisfies \eqref{e:Phi-aa-gamma-ab}.
To see that our definition does not depend on the choice of  $\bar{a}$, note that, for any $\bar{b} \in A_{\bar{\symset}}$, we can take an element $\gamma_{\bar a,\bar b}\in \Gamma_\symset^{\bar a,\bar b}$ 
(which is non-empty set by definition) and use \eqref{e:Phi-aa-gamma-ab} to write
\[
\id_{T^{\bbar a}_{\symset},T^{\bbar a}_{\symset}} \circ \bar \Phi_{\bbar a, \bar b} \circ \Phi_{\bar b, a}
= \bar \Phi_{\bbar a, \bar a}\circ \bar\gamma_{\bar a,\bar b} \circ \Phi_{\bar b, a}
 = \bar \Phi_{\bbar a, \bar a} \circ \Phi_{\bar a, a}\circ \id_{T^a_{\symset},T^{a}_{\symset}}\;.
 \]
We write $\Hom_{2}(\symset,\bar{\symset})$ of the set of morphisms, as described above, between $
\symset$ and $\bar{\symset}$.
To see that this notion of morphism gives an equivalent category note that the map(s) $\iota_{\symset,\bar{\symset}}\colon \Hom(\symset,\bar{\symset}) \rightarrow \Hom_{2}(\symset,\bar{\symset})$, 
given by mapping $\Gamma_{\symset}$ equivalence classes to their symmetrised sums, 
is a bijection and maps compositions in $\Hom$ to the corresponding compositions in $\Hom_{2}$. 
\end{remark}
\begin{remark}\label{rem:tensor-symset}
The category $\SSet$ of symmetric sets just described is an $\R$-linear symmetric monoidal category, 
with tensor product $\symset \otimes \bar{\symset}$ given by 
\begin{equs}
{}&A  = A_\symset \times A_{\bar \symset}\;,\quad
T^{(a,\bar a)} = T_\symset^a \sqcup T_{\bar \symset}^{\bar a}\;,\quad
\mft^{(a,\bar a)} = \mft_\symset^a \sqcup \mft_{\bar \symset}^{\bar a}\;,\\
\quad
{}&\Gamma^{(a,\bar a),(b,\bar b)} =
\{ \gamma \sqcup \bar{\gamma}: \gamma \in  \Gamma_\symset^{a,b},\; \bar{\gamma} \in  \Gamma_{\bar \symset}^{\bar a,\bar b}
\}\;.
\end{equs} 
and unit object $\one$ given by $A_{\one} = \{\bullet\}$ a singleton and $T_\one^\bullet =\emptyset$.
\end{remark}
\begin{remark}\label{rem:canonicalIsomorphism}
We will sometimes encounter the situation where a pair $(\symset, \bar \symset)$ of symmetric sets
naturally comes with elements $\Phi_a \in \Homb(T_\symset^a, \bar \symset)$ such that 
$\Phi \in \Hom(\symset, \bar \symset)$. In this case, $\Phi$ is necessarily an isomorphism 
which we call the ``canonical isomorphism'' between $\symset$ and $\bar \symset$.
Note that this notion of ``canonical'' is not intrinsic to $\SSet$ but relies on additional 
structure in general.

More precisely, consider a category $\CC$ that is concrete over typed sets (i.e.\ such that objects
of $\CC$ can be viewed as typed sets and morphisms as type-preserving maps between them). Then, 
any collection $(T^a)_{a \in A}$ of isomorphic objects of $\CC$ yields a symmetric set $\symset$ by taking
for $\Gamma$ the groupoid of all $\CC$-isomorphisms between them. Two symmetric sets obtained in this way
such that the corresponding collections $(T^a)_{a \in A}$ and $(\bar T^b)_{b \in \bar A}$ 
consist of objects that are $\CC$-isomorphic are then canonically isomorphic 
(in $\SSet$) by taking for $\Phi_a$ the set of all $\CC$-isomorphisms from $T^a$ to any of the $\bar T^b$.
Note that this does \textit{not} in general mean that there isn't another isomorphism between these
objects in $\SSet$!
\end{remark}
\begin{example}\label{ex:Hom}
We now give an example where we compute $\Hom(\act,\act)$ and $\Homb(\act,\act)$. 
Consider $\ttau = \<I[IXi^2]IXi-typed112>$,
fix some representative $\tau\in\ttau$, and write $T_{\tau}=\{x,y,z\} \subset V_\tau \subset \N$, with 
$\mft_{\tau}(x)=\mft_{\tau}(y)=\<Xi>$ 
and $\mft_{\tau}(z_1)=\<XiY>$ \dash the local symmetry group is then isomorphic to $\Z_2$,
acting on $T_\tau$ by permuting $\{x,y\}$.
We also introduce a second isomorphism class $ \bar{\ttau}  = \<I[IXi^2]IXi-typed211>$
which has trivial local symmetry group and fix a representative $\bar{\tau}$ of $ \bar{\ttau} $ which coincides, as a typed set, with $\tau$. 

It is easy to see that $\Homb(T_{\bar{\tau}},\scal{\ttau})$ 
consists of only one equivalence class, while  
$\Homb(T_\tau,\scal{\bar{\ttau}})$ consists of two equivalence classes, 
which we call $\phi$ and $\tilde\phi$.
$\Hom(\langle \ttau \rangle,\langle \bar{\ttau} \rangle)$ then consists of the linear span of a ``section'' $\Phi$
such that, restricting to the representative $\tau$,
$\Phi_\tau=\phi+\tilde\phi$, since the action of 
$\Z_2$ on $\tau$ swaps $\phi$ and $\tilde \phi$.
\end{example}

\subsubsection{Symmetric tensor products}

A \emph{space assignment} $V$ for $\Lab$ is a tuple of vector spaces $ V = (V_\mft)_{\mft\in \Lab}$.
We say a space assignment $V$ is finite-dimensional if $\dim(V_{\mft}) < \infty$ for every $\mft \in \Lab$. 
For the rest of this subsection we fix an arbitrary (not necessarily finite-dimensional) space assignment $(V_\mft)_{\mft\in \Lab}$. 

For any (finite) typed set $T$, we 
write  $V^{\otimes T}$ for the tensor product defined as the linear span of elementary
tensors of the form $v = \bigotimes_{x \in T} v_x$ with $v_x \in V_{\mft(x)}$, subject to the
usual identifications suggested by the notation.
Given $\psi \in \Iso(T,\bar T)$ for two typed sets, we can then interpret it as a linear map 
$V^{\otimes T} \to V^{\otimes \bar T}$ by
\begin{equ}[e:actionMaps]
v = \bigotimes_{x \in T} v_x \mapsto
\psi \cdot v = \bigotimes_{y \in \bar T} v_{\psi^{-1}(y)}\;.
\end{equ}
In particular, given a symmetric set $\symset$, elements $a,b \in A_\symset$, and $\gamma \in \Gamma_\symset^{a,b}$,
we view $\gamma$ as a map from $V^{\otimes T_\symset^b}$ to $V^{\otimes T_\symset^a}$.
We then define the vector space $V^{\otimes \symset} \subset \prod_{a \in A_\symset}V^{\otimes T_\symset^a}$ by
\begin{equ}[e:def-V-tensor-symset]
V^{\otimes \symset} = \Big\{(v^{(a)})_{a\in A_\symset} 
\,:\, v^{(a)} = \gamma_{a,b} \cdot v^{(b)}\quad \forall a,b\in A_\symset\;,\; \forall \gamma_{a,b} \in \Gamma_{\symset}^{a,b}\Big\}\;.
\end{equ}
Note that for every $a \in A_\symset$, we have a natural symmetrisation map $\pi_{\symset,a} \colon V^{\otimes T_\symset^a} \to V^{\otimes \symset}$ given by
\begin{equ}\label{eq:translating_on_section}
(\pi_{\symset,a} v)^{(b)} = {\frac{1}{|\Gamma_\symset^{b,a}|}} \sum_{\gamma \in \Gamma_\symset^{b,a}} \gamma \cdot v\;,
\end{equ}
an important property of which is that
\begin{equ}[e:covpi]
\pi_{\symset,a} \circ \gamma_{a,b} = \pi_{\symset,b}\;,\quad \forall a,b \in A_\symset\;,\quad \forall \gamma_{a,b} \in \Gamma_\symset^{a,b}\;.
\end{equ}
Furthermore, these maps are left inverses to the natural inclusions $\iota_{\symset,a} \colon V^{\otimes \symset} \to V^{\otimes T_\symset^a}$
given by $(v^{(b)})_{b\in A_\symset} \mapsto v^{(a)}$.
\begin{remark}\label{rem:universal_property_and_motivation}
Suppose that we are given a symmetric set $\symset$. 
For each $a \in A_{\symset}$, if we view $\symset_{a}$ as a symmetric set in its own right with $A_{\symset_a} = \{a\}$, then $V^{\otimes \symset_{a}}$ is a partially symmetrised tensor product. 
In particular $V^{\otimes \symset_{a}}$ is again characterised by a universal property, namely it allows one to uniquely factorise multilinear maps on $V^{T^{a}_{\symset}}$ that are, for every $\gamma \in \Gamma_{\symset}^{a,a}$, $\gamma$-invariant in the sense that they are invariant under a permutation of their arguments like \eqref{e:actionMaps} with $\psi = \gamma$. 

This construction (where $|A_{\symset}| = 1$) is already enough to build the vector spaces that we would want to associate to combinatorial trees as described in Section~\ref{sec:motivation}.
A concrete combinatorial tree, that is a tree with a fixed vertex set and edge set along with an associated type map, will allow us to construct a symmetric set with  $|A_{\symset}| = 1$. 

We now turn to another feature of our construction, namely that we allow $|A_{\symset}| > 1$. 
The main motivation is that when we work with combinatorial trees, what we really want is to work with are \emph{isomorphism classes of such trees}, and so we want our construction to capture that we can allow for many different ways for the same concrete combinatorial tree to be realised. 
In particular we will use $A_{\symset}$ to index a variety of different ways to realise the same combinatorial trees as a concrete set of vertices and edges with type map. 
The sets $\Gamma^{a,b}_{\symset}$ then encode a particular set of chosen isomorphisms linking different combinatorial trees in the same isomorphism class.  
Once they are fixed, the maps \eqref{eq:translating_on_section} allow us to move between the different vector spaces that correspond to different concrete realisations of our combinatorial trees.
In particular, once $\symset$ has been fixed, for every $a, b \in A_\symset$ one has fixed canonical isomorphisms 
\begin{equ}\label{eq:canonical_iso_one_to_all}
V^{\otimes \symset_{a}} \simeq V^{\otimes \symset_{b}} \simeq V^{\otimes \symset}
\end{equ}
which can be written explicitly using the maps $\pi_{\symset,\bullet}$ of \eqref{eq:translating_on_section}.

In addition to meaningfully resolving\footnote{See for instance Remarks~\ref{rem:canonical} and~\ref{rem:echo-moti5-2}.} the ambiguity between working with isomorphism classes of objects like trees and concrete instances in those isomorphism classes, this flexibility is crucial for the formulation and proof of Proposition~\ref{prop:nat_transform}.  
\end{remark}
\begin{remark}
Given a space assignment $V$ there is a natural notion of a dual space assignment given by $V^{\ast} = (V_{\mft}^{\ast})_{\mft \in \Lab}$.
There is then a canonical inclusion 
\begin{equ}\label{eq:dual_tensors} 
(V^{\ast})^{\otimes \symset} \hookrightarrow (V^{\otimes \symset})^{\ast} \;.
\end{equ}
Thanks to \eqref{eq:canonical_iso_one_to_all} it suffices to prove \eqref{eq:dual_tensors} when $A_{\symset} = \{a\}$. 

Let $\iota$ be the canonical inclusion from $(V^{\ast})^{\otimes T_\symset^a}$ into $(V^{\otimes T_\symset^a})^{\ast}$ and let $r$ be the canonical surjection from $\big(
V^{\otimes T_\symset^a}
\big)^{\ast}$ to $(V^{\otimes \symset})^{\ast}$. 
The desired inclusion in \eqref{eq:dual_tensors} is then given by the restriction of $r \circ \iota$ to $(V^{\ast})^{\otimes \symset}$. 
To see the claimed injectivity of this map, suppose that for some $w \in (V^{\ast})^{\otimes \symset}$ one has $\iota(w)(v) = 0$ for all $v \in V^{\otimes \symset}$. 
Then, we claim that $\iota(w) = 0$ since for arbitrary $v' \in V^{\otimes T_\symset^a}$ we have 
\begin{equ}
\iota(w)(v') = \iota\Big( 
| \Gamma_{\symset}^{a,a}|^{-1} 
\sum_{\gamma \in \Gamma_{\symset}^{a,a}} \gamma \cdot w
\Big)
(v')=
\iota(
 w
)
\Big( 
|\Gamma_{\symset}^{a,a}|^{-1}
\sum_{\gamma \in \Gamma_{\symset}^{a,a}}
\gamma^{-1} \cdot v'
\Big) = 0
\end{equ}
where in the first equality we used that $w \in (V^{\ast})^{\otimes \symset}$ while in the last equality we used that the sum in the expression before is in $V^{\otimes \symset}$. 

If the space assignment $V$ is finite-dimensional then our argument above shows that we have a canonical isomorphism
\begin{equ}\label{eq:dual_tensors_finite_dim} 
(V^{\ast})^{\otimes \symset} \simeq (V^{\otimes \symset})^{\ast} \;.
\end{equ}
\end{remark}
\begin{example}
Continuing with Example~\ref{ex:Hom}
and denoting $\symset = \langle \tau\rangle$,
given  vector spaces $V_{\<Xi>}$, $V_{\<XiY>}$, an element in $V^{\otimes \symset}$
can be identified with a formal sum over all representations of vectors of the form
\[
v^{\<Xi>}_{x_1} \otimes v^{\<Xi>}_{y_1} \otimes v^{\<XiY>}_{z_1} 
+v^{\<Xi>}_{y_1} \otimes v^{\<Xi>}_{x_1} \otimes v^{\<XiY>}_{z_1} \;,
\]
namely, it is partially symmetrised such that
for another representation, the two choices of $\gamma$ both satisfy the requirement in \eqref{e:def-V-tensor-symset}.
The projection $\pi_{\symset,a} $ then plays the role of symmetrisation.
\end{example}

We now fix two symmetric sets $\symset$ and $\bar \symset$.
Given $\Phi \in \Homb(T_\symset^a,\bar{\symset})$, it 
naturally defines a linear map $F^a_\Phi\colon V^{\otimes T_\symset^a} \to V^{\otimes \bar{\symset}}$ by
\begin{equ}[e:defFPhi]
F^a_\Phi v = \pi_{\bar \symset, \bar a} \bigl(\Phi^{(\bar a)} \cdot v\bigr)\;,
\end{equ}
where $\bar a \in A_{\bar \symset}$ and $\Phi^{(\bar a)}$ denotes any representative of $\Phi$ in $\Iso(T_\symset^a, T_{\bar \symset}^{\bar a})$.
Since any other such choice $\bar b$ and $\Phi^{(\bar b)}$ is related to the previous one by composition to the left with an element of 
$\Gamma_{\bar \symset}^{\bar a,\bar b}$, it follows from \eqref{e:covpi} that \eqref{e:defFPhi} is independent of these choices.
We extend \eqref{e:defFPhi} to $\Vec\big(\Homb(T_\symset^a,\bar{\symset})\big)$ by linearity.

Since $(\phi \circ \psi) \cdot v = \phi \cdot (\psi\cdot v)$ by the definition \eqref{e:actionMaps}, we conclude that,
for $v = (v^{(a)})_{a \in A_\symset} \in V^{\otimes \symset}$, $\Phi = (\Phi_a)_{a \in A_\symset} \in \Hom(\symset,\bar \symset)$, 
as well as $\gamma_{a,b} \in \Gamma_{\symset}^{a,b}$, we have the identity
\begin{equ}
F_{\Phi_b}^b v^{(b)} = F_{\Phi_a\circ \gamma_{a,b}}^b v^{(b)}  = F_{\Phi_a}^a (\gamma_{a,b} \cdot v^{(b)}) = F_{\Phi_a}^a v^{(a)}\;,
\end{equ}
so that $F_\Phi$ is well-defined as a linear map from $V^{\otimes \symset}$ to $V^{\otimes \bar \symset}$ by
\begin{equ}[e:def-F_Phi]
F_\Phi v = F_{\Phi_a}^a \iota_{\symset,a} v\;,
\end{equ}
which we have just seen is independent of the choice of $a$.

The following lemma shows that this construction defines a monoidal functor $\Func_V$ mapping $\symset$ to $V^{\otimes \symset}$ and  
$\Phi$ to $F_\Phi$ between the category $\SSet$ of symmetric sets and the category $\Vec$ of 
vector spaces. \label{def:CF-V}
%
\begin{lemma}
Consider symmetric sets $\symset$, $\bar \symset$, $\bbar\symset$,
and morphisms $\Phi\in\Hom(\symset,\bar \symset)$ and $\bar\Phi\in \Hom(\bar \symset,\bbar\symset)$.
Then $F_{\bar\Phi\circ\Phi}= F_{\bar\Phi}\circ F_{\Phi}$.
\end{lemma}
\begin{proof}
Since we have by definition
\begin{equ}
F_{\bar\Phi} F_{\Phi} v = F_{\bar\Phi} \pi_{\bar \symset, \bar a} \bigl(\Phi_a^{(\bar a)} \cdot \iota_{\symset, a}v\bigr)
=  \pi_{\bbar \symset, \bbar a} \bigl(\bar \Phi_{\bar b}^{(\bbar a)} \cdot \iota_{\bar \symset, \bar b}\pi_{\bar \symset, \bar a} \bigl(\Phi_a^{(\bar a)} \cdot \iota_{\symset, a}v\bigr)\bigr)\;,
\end{equ}
for any arbitrary choices of $a \in A_\symset$, $\bar a, \bar b \in A_{\bar\symset}$, $\bbar a \in A_{\bbar\symset}$, 
it suffices to note that, for any $w = (w^{\bar{c}})_{\bar{c} \in A_{\bar{\symset}}} \in V^{\otimes \bar{\symset}}$, 
\begin{equ}
\bar \Phi_{\bar b}^{(\bbar a)} \cdot \iota_{\bar \symset, \bar b}\pi_{\bar \symset, \bar a} w
= {\frac{1}{|\Gamma_{\bar \symset}^{\bar b, \bar a}|}} \sum_{\gamma \in \Gamma_{\bar \symset}^{\bar b, \bar a}} (\bar \Phi_{\bar b}^{(\bbar a)} \circ \gamma) \cdot w^{(\bar{a})} = \bar \Phi_{\bar a}^{(\bbar a)} \cdot w^{(\bar{a})}\;,
\end{equ}
as an immediate consequence of the definition of $\Hom(\bar \symset, \bbar\symset)$.
\end{proof}
\begin{remark}\label{rem:U-on-Vtensor}
A useful property is the following.
Given two space assignments $V$ and $W$ and a collection of linear maps $U_\mft\colon V_\mft \to W_\mft$, this induces a natural
transformation $\Func_V \to \Func_W$.
Indeed, for any typed set $\symset$, it yields a collection of linear maps $U_\symset^a \colon V^{\otimes T_\symset^a} \to W^{\otimes T_\symset^a}$
by
\begin{equ}
U_\symset^a \bigotimes_{x \in T_\symset^a} v_x = \bigotimes_{x \in T_\symset^a} U_{\mft_\symset^a(x)} v_x\;.
\end{equ}
This in turn defines a linear map $U_\symset\colon  V^{\otimes \symset} \to W^{\otimes \symset}$ in the natural way.
It is then immediate that, for any $\Phi \in \Hom(\symset,\bar{\symset})$, one has the identity
\begin{equ}
U_{\bar{\symset}} \circ \Func_V(\Phi) = \Func_W(\Phi) \circ U_{\symset}\;,
\end{equ}
so that this is indeed a natural transformation.
\end{remark}

\begin{remark}
In the more general context of Remark~\ref{rem:generalisation}, this construction proceeds similarly. The only
difference is that now a space assignment $V$ is a functor $\Lab \to \Vec$ mapping objects
$\mft$ to spaces $V_\mft$ and morphisms $\phi \in \Hom_\Lab(\mft, \bar \mft)$ to linear
maps $V_\phi \in L(V_\mft, V_{\bar \mft})$. In this case, an element $\phi \in \Iso(T, \bar T)$
naturally yields a linear map $V^{\otimes T} \to V^{\otimes \bar T}$ by
\begin{equ}
v = \bigotimes_{x \in T} v_x \mapsto
\phi \cdot v = \bigotimes_{y \in \bar T} V_{\vec \phi(\phi_0^{-1}(y))} v_{\phi_0^{-1}(y)}\;.
\end{equ}
The remainder of the construction is then essentially the same.
\end{remark}

\begin{remark}
One may want to restrict oneself to a smaller category than $\Vec$ by enforcing additional
``nice'' properties on the spaces $V_\mft$. For example, it will be convenient below to replace it 
by some category of topological vector spaces. 
Other natural examples of replacements for $\Vec$ could be the category of Banach spaces endowed with a 
choice of cross-product, the category of Hilbert spaces, the category of finite-dimensional vector bundles
over a fixed base manifold, etc.
\end{remark}

\subsubsection{Typed structures}\label{subsec:typed_struct}

It will be convenient to consider the larger category $\TStruc$ of \textit{typed structures}. 
\begin{definition}\label{def:typed_struct}
We define $\TStruc$ to be the category obtained by freely adjoining 
countable products to $\SSet$.
We write $\TStruc_\Lab$ for $\TStruc$ when we want to emphasize the dependence of this category on 
the underlying label set $\Lab$.
\end{definition}
\begin{remark}\label{rem:TStruc}
An object $\symcol$ in the category $\TStruc$ can be viewed as a countable (possibly finite) index set $\CA$ and, 
for every $\alpha \in \CA$, a symmetric set $\symset_\alpha \in \ob(\SSet)$.
This typed structure is then equal to $\prod_{\alpha \in \CA} \symset_{\alpha}$, where $\prod$ denotes the 
categorical product. (When $\CA$ is finite it coincides with the coproduct and we will then also
write $\bigoplus_{\alpha \in \CA} \symset_{\alpha}$ and call it the ``direct sum'' in the sequel.)
Morphisms
between $\symcol$ and $\bar \symcol$ can be viewed as ``infinite matrices'' $M_{\bar \alpha, \alpha}$ with $\alpha \in \CA$, 
$\bar \alpha \in \bar \CA$,
 $M_{\bar \alpha, \alpha} \in \Hom(\symset_\alpha,\symset_{\bar \alpha})$ and the property that, 
 for every $\bar\alpha \in \bar\CA$, one has $M_{\bar \alpha, \alpha} = 0$
 for all but finitely many values of $\alpha$. 
Composition of morphisms is performed in the natural way, analogous to matrix multiplication. 

We remark that the index set $\CA$ here
has nothing to do with the index set $A_\symset$
in Definition~\ref{def:sym-typed-set}.
\end{remark}

Note that $\TStruc$ is still symmetric monoidal with the tensor product 
behaving distributively over the direct sum if we enforce $(\prod_{\alpha \in \CA} \symset_{\alpha} )\otimes ( \prod_{\beta \in \CB} \symset_{\beta}) 
= \prod_{(\alpha,\beta) \in \CA\times \CB} \big(\symset_{\alpha}\otimes \symset_{\beta}\big)$,
with $\symset_{\alpha}\otimes \symset_{\beta}$ as in Remark~\ref{rem:tensor-symset}, and define
the tensor product of morphisms in the natural way.

\begin{remark}
The choice of adjoining countable products (rather than coproducts) is that we will use this construction
in Section~\ref{subsec:nonlinearities} to describe the general solution to the algebraic fixed point problem associated
to \eqref{e:SPDE} as an infinite formal series.
\end{remark}

If the space assignments $V_\mft$ are finite-dimensional, the functor $\Func_V$ then naturally extends to 
an additive monoidal functor from $\TStruc$ to the category of topological vector spaces (see for 
example \cite[Sec.~4.5]{MonCat}). Note that in particular one has $\Func_V(\symcol) = \prod_{\alpha\in\CA}\Func_V(\symset_\alpha)$.
\subsection{Direct sum decompositions of symmetric sets}\label{subsec: direct sum decomp}
In Section~\ref{subsec:sym_tensor_prod}
we showed how, given a set of labels $\Lab$ and space assignment $(V_{\mft})_{\mft \in \Lab}$, we 
can ``extend'' this space assignment so that we get an appropriately symmetrised vector space 
$V^{\otimes \symset}$ for any symmetric set $\symset$ (or, more generally,
for any typed structure) typed by $\Lab$. 
In this subsection we will investigate how this construction behaves under a direct sum decomposition for the space assignment $(V_{\mft})_{\mft \in \Lab}$ that is encoded via a corresponding ``decomposition'' on the set $\Lab$.
\begin{definition}\label{def:label-decompose}
Let $\CP(A)$ denote the powerset of a set $A$.\label{powerset page ref}
Given two distinct finite sets of labels $\Lab$ and $\bar \Lab$ as well as a 
map $\proj \colon \Lab \to \CP(\bar \Lab) \setminus \{\emptyset\}$,
such that $\{ \proj(\mft): \mft \in \Lab\}$ is a partition of $\bar\Lab$, we call $\bar{\Lab}$ a type decomposition of $\Lab$ under $\proj$.
If we are also given space assignments $(V_{\mft})_{\mft \in \mfL}$ for $\mfL$ and $(\bar{V}_{\mfl})_{\mfl \in \bar{\mfL}}$ for $\bar{\mfL}$ with the property that
\begin{equ}\label{eq: type vector space decomp}
V_\mft = \bigoplus_{\mfl \in \proj(\mft)} \bar V_\mfl \qquad\textnormal{ for every }\mft \in \mfL\;,
\end{equ} 
then we say that $(\bar{V}_{\mfl})_{\mfl \in \bar{\mfL}}$ is a decomposition of $(V_{\mft})_{\mft \in \mfL}$. 
For $\mfl \in \proj(\mft)$, we write $\PP_{\mfl} \colon V_{\mft} \to \bar V_{\mfl}$
for the projection induced by \eqref{eq: type vector space decomp}.
\end{definition}
For the remainder of this subsection we fix a set of labels $\Lab$, a space assignment $(V_{\mft})_{\mft \in \mfL}$, along with a type decomposition $\bar{\Lab}$ of $\Lab$ under  $\proj$ and a space assignment $(\bar{V}_{\mfl})_{\mfl \in \bar{\mfL}}$ that is a decomposition of $(V_{\mft})_{\mft \in \mfL}$. 
To shorten notations, for functions $\mft \colon B \to \Lab$ and $\mfl \colon B \to \bar\Lab$
with  any set $B$, we write 
$\mfl \tto \mft$ as a shorthand for the relation $\mfl(p) \in \proj(\mft(p))$ for every $p \in B$.
Given any symmetric set
$\symset$ with label set $\Lab$ and any $a \in A_\symset$,
we write $\hat L_\symset^a = \{\mfl\colon T_\symset^a \to \bar \Lab\,:\, \mfl \tto \mft_\symset^a\}$,
and we consider on $\hat L_\symset = \bigcup_{a\in A_\symset} \hat L_\symset^a$ the equivalence 
relation $\sim$ given by
\begin{equ}[e:defsim]
\hat L_\symset^a\ni \mfl \sim \bar \mfl \in \hat L_\symset^b \quad\Leftrightarrow\quad \exists \gamma_{b,a} \in \Gamma_\symset^{b,a}\,:\,  
\mfl = \bar\mfl \circ \gamma_{b,a}\;.
\end{equ}
We denote by $L_\symset \eqdef \hat L_\symset / {\sim}$ the set of such equivalence classes,
which we note is finite.
\begin{example}
In this example we describe the vector space associated to $\<IXi^2>$.
For our space assignment we start with a labelling set $\mfL=\{\star\}$, where $\star$ represents the noise, and we assume that the actual noise driving the equation we're interested in 
is a distribution with values in a finite-dimensional vector space $V_{\star}$. 
If we are given a direct sum decomposition $V_{\star} = V_{(\star,1)} \oplus V_{(\star,2)}$, 
we can ``split'' the noise into two by introducing a new set of labels 
$\bar\mfL=\mfL\times \{1,2\}$ and setting $\proj(\star)=\{(\star,1),(\star,2)\}$. 

Recall our convention described in 
Example~\ref{ex:trees}: $\<IXi^2>$ represents an isomorphism class of trees and
any concrete tree $\tau$ in that class is realised by a vertex set which is a subset of
3 elements of $\N$ and in which 2 of those elements are the leaves labelled by $\star$.
Let $\symset$ be the symmetric set associated to $\<IXi^2>$ and $\tau \in A_{\symset}$ be a concrete tree with $T_\tau=\{x,y\}$, $\mft_\tau(x)=\mft_\tau(y)=\star$. The local symmetry group in this 
example is isomorphic to $\Z_2$. 

The set $\hat L^{\tau}_\symset$
 consists of 4 elements which we denote by
\begin{equ}\label{eq:expansion_of_tree}
 \<rIXi^2>,\; \<IXi^2green>,\; \<IXi^2asym>, \textnormal{ and } \<IXi^2asym2>\;.
\end{equ} 
In the symbols above, the left leaf corresponds to $x$ and the right one to $y$.
We thus obtain four labellings on  $T_\tau $ by $\bar{\Lab}$ where
$(\star,1)$ is associated to $\<Xired>$
and 
$(\star,2)$ is  associated to $\<Xigreen>$.

However, if we were to interpret the symbols of \eqref{eq:expansion_of_tree} as isomorphism classes of trees labelled by $\bar{\Lab}$ then $\<IXi^2asym>$ and $\<IXi^2asym2>$ are the same isomorphism class and this is reflected by the fact that $|L_\symset|=3$.
The isomorphism class of $\<IXi^2>$ is associated to a vector space isomorphic to $V_{\star}\otimes_{s} V_{\star}$.
Our construction will decompose (see Proposition~\ref{prop:nat_transform}) the vector space for $\<IXi^2>$ into a direct sum of three vector spaces corresponding to the isomorphism classes 
$\<rIXi^2>$, $\<IXi^2green>$, and $\<IXi^2asym>$ which are, respectively, isomorphic to $V_{(\star,1)}\otimes_{s} V_{(\star,1)}$, $V_{(\star,2)}\otimes_{s} V_{(\star,2)}$, and $V_{(\star,1)} \otimes V_{(\star,2)}$.
\end{example}
Given an equivalence class $Y \in L_\symset$ and $a \in A_{\symset}$, we define $Y_a = Y \cap \hat L_\symset^a$ (which we note is non-empty due to the connectedness of $\Gamma_\symset$).
We then define a symmetric set
$\symset_Y$ by
\begin{equs}
A_{\symset_Y} &= \{(a, \mfl)\,:\, a \in A_\symset\;,\quad \mfl \in Y_a\}\;,\qquad
T_{\symset_Y}^{(a,\mfl)} = T_\symset^a\;,\qquad
\mft_{\symset_Y}^{(a,\mfl)} = \mfl\;,\\
\Gamma_{\symset_Y}^{(a,\mfl),(b,\bar \mfl)} &= \{\gamma \in \Gamma_{\symset}^{a,b}\,:\, \bar \mfl = \mfl \circ \gamma\}\;.
\end{equs}
\begin{remark}\label{rem:Y}
The definition \eqref{e:defsim} of our equivalence relation 
guarantees that $\Gamma_{\symset_Y}$ is connected. The definition of $\Gamma_{\symset_Y}$ furthermore 
yields a morphism of groupoids $\Gamma_{\symset_Y} \to \Gamma_{\symset}$
which is easily seen to be surjective.
\end{remark}
With these notations at hand, we can define a functor 
$\proj^*$ from $\SSet_\Lab$ to $\TStruc_{\bar\Lab}$ as follows. 
Given any $\symset\in \ob(\SSet_\Lab)$, we define
\begin{equ}[e:linkpi3]
\proj^* \symset = \bigoplus_{Y \in L_\symset} \symset_Y \in \ob\big(\TStruc_{\bar \Lab} \big)\;.
\end{equ}
To describe how $\proj^*$ acts on morphisms,
let us fix two symmetric sets $\symset, \bar \symset \in \ob(\SSet_\Lab)$, a choice 
of $Y \in L_{\symset}$, $(a,\mfl) \in A_{\symset_{Y}}$, as well as an element 
$\phi \in \Homb(T_\symset^a,\bar \symset)$. We then let $\phi\cdot\mfl \subset \hat L_{\bar \symset}$ 
be given by 
\begin{equ}
\phi \cdot \mfl = \mfl \circ \phi^{-1} \eqdef \{ \bar{\mfl}\ :\ \exists \psi \in \phi  \textnormal{ with } \bar{\mfl} = \mfl \circ \psi^{-1}\}\;, 
\end{equ} 
where we recall that $\phi \subset \bigcup_{\bar{a} \in A_{\bar{\symset}}} \Iso(T^{a}_{\symset},T^{\bar{a}}_{\bar{\symset}})$ is a $\Gamma_{\bar\symset}$-equivalence class of bijections. 
It follows from the definitions of the equivalence relation on $\hat L_{\bar \symset}$ and of $\Homb(T_\symset^a,\bar \symset)$ that one actually has $\phi \cdot \mfl \in L_{\bar \symset}$.
We then define
\begin{equ}
\proj^{\ast}_{(a,\mfl)}
\phi \in \bigoplus_{\bar Y \in L_{\bar \symset}} \Vec \bigl(\Homb(T^{(a,\mfl)}_{\symset_Y}, \bar \symset_{\bar Y})\bigr)\;,
\end{equ}
by simply setting
\begin{equ}[e:defproj*]
\proj^{\ast}_{(a,\mfl)}
\phi
= \phi \in \Homb(T^{(a,\mfl)}_{\symset_Y}, \bar \symset_{\phi \cdot \mfl})
\subset \bigoplus_{\bar Y \in L_{\bar \symset}} \Vec \bigl(\Homb(T^{(a,\mfl)}_{\symset_Y}, \bar \symset_{\bar Y})\bigr)\;,
\end{equ}
which makes sense since $T^{(a,\mfl)}_{\symset_Y} = T_\symset^a$, $T_{\bar \symset_{\bar Y}}^{(\bar a, \bar \mfl)} = T_{\bar \symset}^{\bar a}$, and since $\Gamma_{\bar \symset_{\phi \cdot \mfl}} \to \Gamma_{\bar \symset}$ is surjective.

We take a moment to record an important property of this construction.
Given any $(a,\mfl), (b,\hat{\mfl}) \in A_{\symset_{Y}}$, $\gamma \in \Gamma^{(a,\mfl), (b,\hat{\mfl})}_{\symset_{Y}}$, and $\phi \in \Homb(T_{\symset}^{a},\bar{\symset})$, it follows from 
our definitions that
\begin{equ}\label{eq: moving domain basepoint}
(\proj^{\ast}_{(a,\mfl)} \phi) \circ \gamma
=
\proj^{\ast}_{(b,\hat \mfl)} (\phi \circ \gamma)
\end{equ}
where on the right-hand side of \eqref{eq: moving domain basepoint} we are viewing $\gamma$ as an element of $\Gamma^{a,b}_{\symset}$, which indeed maps $\Homb(T^{a}_{\symset},\bar{\symset})$ into $\Homb(T^{b}_{\symset},\bar{\symset})$ by right composition. 

Extending \eqref{e:defproj*} by linearity, we obtain a map
\[
\proj^{\ast}_{(a,\mfl)}\colon\Vec \big( \Homb(T_\symset^a,\bar \symset)\big) \rightarrow \bigoplus_{\bar Y \in L_{\bar \symset}} \Vec \big( \Homb(T^{(a,\mfl)}_{ \symset_{Y}},\bar{\symset}_{\bar Y}) \big)\;.
\]
We then use this to construct a map $\proj^{\ast}_{Y}\colon\Hom(\symset,\bar \symset) \rightarrow \Hom(\symset_Y, \proj^{\ast}\bar\symset)$ as follows. 
For any $\Phi = (\Phi_a)_{a\in A_\symset} \in  \Hom(\symset,\bar \symset)$, we set
\[
(\proj^{\ast}_{Y}\Phi)_{(a,\mfl)}
= \proj^{\ast}_{(a,\mfl)} \Phi_{a}\;,\qquad \forall (a,\mfl) \in A_{\symset_Y}\;.
\]
To show that this indeed belongs to $\Hom(\symset_Y, \proj^{\ast}\bar\symset)$,
note that, for any $(a,\mfl), (b,\hat{\mfl}) \in A_{\symset_Y}$ and $\gamma \in \Gamma^{(a,\mfl),(b,\hat{\mfl})}_{\symset_{Y}} \subset \Gamma_\symset^{a,b}$, we have
\begin{equs}
(\proj_{Y}^{\ast}\Phi)_{(a,\mfl)} \circ \gamma 
&= (\proj_{(a,\mfl)}^{\ast}\Phi_a) \circ \gamma
= \big(\proj_{(b,\hat{\mfl})}^{\ast}(\Phi_a\circ \gamma)\big) 
= \big(\proj_{(b,\hat{\mfl})}^{\ast}\Phi_b\big) 
= (\proj_{Y}^{\ast}\Phi)_{(b,\hat{\mfl})} \;.
\end{equs}
In the second equality we used the property \eqref{eq: moving domain basepoint} and in the 
third equality we used that $\Phi_{a} \circ \gamma = \Phi_{b}$ which follows from our assumption that $\Phi \in \Hom(\symset,\bar{\symset})$ \dash recall that here we are viewing $\gamma$ as an element of $\Gamma^{a,b}_{\symset}$.

Finally, we then obtain the desired map $\proj^{\ast}\colon\Hom(\symset,\bar{\symset}) \rightarrow \Hom(\proj^{\ast}\symset,\proj^{\ast}\bar{\symset})$ by setting, for $\Phi \in \Hom(\symset,\bar{\symset})$, 
\[
\proj^{\ast}\Phi = 
\bigoplus_{Y \in L_{\symset}}
\proj^{\ast}_{Y}\Phi\;.
\] 
The fact that $\proj^*$ is a functor (i.e.\ preserves composition of morphisms) is an almost
immediate consequence of \eqref{e:defproj*}. Indeed, given 
$\phi \in \Homb(T_\symset^a, \bar \symset)$ and $\bar \Phi \in \Hom(\bar \symset, \bbar \symset)$,
it follows immediately from \eqref{e:defproj*} that
\begin{equ}
\proj^*_{\phi \cdot \mfl} \bar \Phi \circ \proj^*_{(a,\mfl)}\phi
= \proj^*_{(a,\mfl)} \bigl(\bar \Phi \circ \phi\bigr)\;,
\end{equ}
where we view $\bar \Phi \circ \phi$ as an element of $\Vec \bigl(\Homb(T_\symset^a, \bbar \symset)\bigr)$. 
It then suffices to note that $\proj^*_{\bar Y} \bar \Phi \circ \proj^*_{(a,\mfl)}\phi = 0$
for $\bar Y \neq \phi \cdot \mfl$, which then implies that 
\begin{equ}
\proj^* \bar \Phi \circ \proj^*_{(a,\mfl)}\phi
= \proj^*_{(a,\mfl)} \bigl(\bar \Phi \circ \phi\bigr)\;,
\end{equ}
and the claim follows.
Note also that $\proj^*$ is monoidal in the sense that 
$\proj^*(\symset \otimes \bar \symset) = \proj^*(\symset) \otimes \proj^*(\bar \symset)$ and similarly for morphisms,
modulo natural transformations. 
\begin{remark}
One property that can be verified in a rather straightforward way is that
if we define $(\bar\proj \circ  \proj)(\mft) = \bigcup_{\mfl \in \proj(\mft)}\bar\proj(\mfl)$, then
\begin{equ}
(\bar\proj \circ  \proj)^* =  \bar\proj^* \circ  \proj^*\;,
\end{equ}
again modulo natural transformations. This is because triples $(a,\mfl,\bar \mfl)$
with $\bar \mfl \tto \mfl \tto \mft_a$ are in natural bijection with pairs $(a,\bar \mfl)$.
Note that this identity crucially uses that the sets $\proj(\mft)$ are all disjoint.
\end{remark}
Our main interest in the functor $\proj^*$ is that it will perform the corresponding direct sum decompositions
at the level of partially symmetric tensor products of the spaces $V_\mft$. 
This claim is
formulated as the following proposition.
\begin{proposition}\label{prop:nat_transform}
One has $\Func_{\bar V} \circ \proj^* = \Func_{V}$, modulo natural transformation.
\end{proposition}
\begin{proof}
Fix $\symset \in \SSet_\Lab$. 
Given any $a \in A_{\symset}$ and elementary tensor $v^{(a)} \in V^{\otimes T^{a}_\symset}$ of the form
$
v^{(a)} = \bigotimes_{x \in T^{a}_\symset} v^{(a)}_x
$,
we first note that we have the identity
\begin{equ}[e:decompv]
v^{(a)} = \bigotimes_{x \in T^{a}_\symset} \sum_{\mfl \in \proj(\mft_\symset(x))} \PP_{\mfl} v^{(a)}_x
= \sum_{\mfl\tto \mft_\symset} \bigotimes_{x \in T^{a}_\symset} \PP_{\mfl(x)} v^{(a)}_x\;,
\end{equ}
where $\PP_{\mfl}$ is defined below \eqref{eq: type vector space decomp}.
This suggests the following definition for a map 
\begin{equ}
\iota_\symset \colon \prod_{a \in A_\symset} V^{\otimes T_\symset^a}
\to \prod_{(a,\mfl) \in \hat L_\symset}
\bar V^{\otimes T_{\symset_{[a,\mfl]}}^{(a,\mfl)}}\;,
\end{equ}
where $[a,\mfl]\in L_{\symset}$ is the equivalence class that $(a,\mfl)$ belongs to.
Given $v = (v^{(a)})_{a\in A_\symset}$ with $v^{(a)} = \bigotimes_{x \in T_\symset^a} v_x^{(a)}$,  
we set
\begin{equ}
(\iota_\symset v)_{(a,\mfl)} \eqdef \bigotimes_{x \in T_\symset^a} \PP_{\mfl(x)} v_x^{(a)}\;,
\end{equ}
which is clearly invertible with inverse given by 
$
(\iota_\symset^{-1} w)_{a} = \sum_{\mfl \tto \mft_a} \bigotimes_{x \in T_\symset^a} w_x^{(a,\mfl)}
$.
Note now that 
\begin{equs}
 \Func_V(\symset) &\subset \prod_{a \in A_\symset} V^{\otimes T_\symset^a}\;, \\
\Func_{\bar V}(\proj^*\symset)
&= \bigoplus_{Y \in L_\symset} \Func_{\bar V}(\symset_Y)
\subset \bigoplus_{Y \in L_\symset} \prod_{(a,\mfl) \in A_{\symset_Y}}
\bar V^{\otimes T_{\symset_Y}^{(a,\mfl)}}
\simeq \prod_{(a,\mfl) \in \hat L_\symset}
\bar V^{\otimes T_{\symset_{[a,\mfl]}}^{(a,\mfl)}}\;,
\end{equs}
where we used that $L_\symset$ is finite in the final line.
Furthermore, both $\iota_\symset$ and $\iota_\symset^{-1}$ preserve these subspaces,
and we can thus view
$\iota_\symset$ as an isomorphism of vector spaces between
$\Func_V(\symset)$ and $\Func_{\bar V}(\proj^*\symset)$. The fact that, for 
$\Phi \in \Hom(\symset, \bar \symset)$, one has
\begin{equ}
\iota_{\bar \symset} \circ \Func_V(\Phi) = \Func_{\bar V}(\proj^*\Phi) \circ\iota_\symset\;,
\end{equ}
is then straightforward to verify.
\end{proof}
\subsection{Symmetric sets from trees and forests}
\label{sec:symtrees}
Most of the symmetric sets entering our constructions will be generated from {\it finite labelled rooted trees}  (sometimes just called ``trees'' for simplicity) and their associated automorphisms. 
A finite labelled rooted tree $\tau = (T,\rho,\mft,\mfn)$  consists of a tree 
$T = (V,E)$ with finite vertex set $V$, edge set $E \subset V\times V$ 
and root $\rho \in V$, 
endowed with a type $\mft \colon E \to \Lab$ and label $\mfn \colon V \cup E \to \N^{d+1}$.
We also write $\mfe \colon E \to \Lab \times \N^{d+1}$ for the map $\mfe = (\mft, \mfn \restr E)$.
Note that the ``smallest'' possible tree, usually denoted by $\bone$, is given by $V = \{\rho\}$, $E = \emptyset$ 
and $\mfn(\rho) = 0$;
we denote by $\mbX^k$ with $k\in  \N^{d+1}$ the same tree but with  $\mfn(\rho) = k$.
For convenience, we consider edges as directed towards the root in the sense that we always
have $e = (e_-, e_+)$ with $e_+$ nearer to the root. 
Note that one can naturally 
extend the map $\mft$ to $V \setminus \{\rho\}$ by setting 
$\mft(v) = \mft(e)$ for the unique edge $e$ such that $e_- = v$.

An {\it isomorphism} between two labelled rooted trees is a bijection between their edge and
vertex sets that preserves their connective structure, their roots, and their labels $\mft$ and $\mfn$.
We then denote by $\mfT$ \label{mfT page ref} the set of isomorphism classes of rooted labelled trees with vertex sets
that are subsets of $\N$.\footnote{The choice of $\N$ here is of course irrelevant; any 
set of infinite cardinality would do. The only reason for this restriction is to make
sure that elements of $\mfT$ are sets.} 

Given $\ttau \in \mfT$, we assign to it a symmetric set $\symset = \scal{\ttau}$.
In particular, we fix $A_\symset = \ttau$ and, for every
$\tau \in \ttau$, we set
$T_{\symset}^\tau = E_\tau$ (the set of edges of $\tau$), $\mft_\symset^\tau$ the type map of $\tau$,
and, for $\tau_1,\tau_2 \in \ttau$, we let $\Gamma_\symset^{\tau_1,\tau_2}$
be the set of all elements of $\Iso(T_{\symset}^{\tau_2},T_{\symset}^{\tau_1})$
obtained from taking a tree isomorphism from $\tau_{2}$ to $\tau_{1}$ and then
restricting this map to the set of edges $E_{\tau_2}$.
We also define the object $\scal{\mfT}$ in $\TStruc$
given by $\scal{\mfT} = \prod_{\ttau \in \mfT} \scal{\ttau}$.

\label{def:scal-tau}
Given an arbitrary labelled rooted tree $\tau$, we also write $\scal{\tau}$ for the symmetric set
with $A_{\scal{\tau}}$ a singleton, 
$T_{\scal{\tau}}$ and $\mft_{\scal{\tau}}$
as above, and $\Gamma$ the set of all automorphisms of $\tau$.
The following remark is crucial for our subsequent use of notations.
\begin{remark}\label{rem:canonical}
By definition, given any labelled rooted tree $\tau$, there exists exactly 
one $\ttau \in \mfT$ such that its 
elements are tree isomorphic to $\tau$ and exactly \textit{one} element of $\Homb(T_{\scal{\tau}}, \scal{\ttau})$ whose
representatives are tree isomorphisms, so we are in the setting of Remark~\ref{rem:canonicalIsomorphism}.
As a consequence, we can, for all intents and purposes, identify $\scal{\tau}$ with 
$\scal{\ttau}$.
As an example, in Section~\ref{sec:Products and coproducts} this observation allows us to define various morphisms on $\langle \mfT \rangle$ and $\langle \mfT \rangle \otimes \langle \mfT \rangle$ by defining operations at the level of trees $\tau \in \ttau \in \mfT$ with fully specified vertex and edge sets rather than working with the isomorphism class $\ttau$.
\end{remark}
\begin{remark}\label{rem:echo-moti5-2}
In the example in Section~\ref{sec:motivation},
the two labelled rooted trees $\tau=\<IXi^2asym>$ and $ \bar\tau=\<IXi^2asym2> $ (both with trivial labels $\mfn=0$ say)
are isomorphic. The canonical isomorphism $\Phi \in \Hom(\scal{\tau},\scal{\bar\tau})$ is then simply the map matching the same types.
The map $\Func_V(\Phi)\colon  V^{\otimes \scal{\tau}}\to  V^{\otimes \scal{\bar\tau}}$  is then a canonical isomorphism \dash this is where the 
middle identity in the motivation \eqref{e:postulate-comm} is encoded.
\end{remark} 
A labelled rooted forest $ f = (F,\mathfrak{P},\mft,\mfn)$ is defined as consisting of a 
finite forest\footnote{Recall that a forest is a graph without cycles, so that every connected component is a tree.} 
$F = (V,E)$, where again $V$ is the set of vertices, $E$ the set of edges, each connected component $T$ of 
$F$ has a unique distinguished root $\rho $ with $\mathfrak{P} \subset V$ the set of all these roots, and $\mft$ and $\mfn$ are both as 
before. Note that we allow for the empty forest, that is the case where $V = E = \emptyset$, and that any finite 
labelled rooted tree $(T,\rho,\mft,\mfn)$ is also a forest (where $\mathfrak{P} = \{\rho\}$). 

Two labelled rooted forests are considered isomorphic if there is a bijection between their edge and
vertex sets that preserves their connective structure, their roots, and their labels $\mft$ and $\mfn$ \dash note that we allow automorphisms of labelled rooted forests to swap connected components of the forest. 
Given a labelled rooted forest $f$, we then write $\scal{f}$ for the corresponding symmetric set 
constructed similarly to above, now with tree automorphisms replaced by forest automorphisms. 

We denote by $\mfF$\label{mfF page ref} the set of isomorphism classes of rooted labelled forests 
with vertex sets in $\N$,
which can naturally be viewed as the unital commutative monoid 
generated by $\mfT$ with unit given by the empty forest. 
In the same way as above, we assign to an element $\ff\in\mfF$
a symmetric set $\scal{\ff}$
and we write $\scal{\mfF} = \prod_{\ff \in \mfF} \scal{\ff}\in \ob(\TStruc)$.

Before continuing our discussion we take a moment to describe where we are going. 
In many previous works on regularity structures, in particular in \cite{BHZ19}, the vector space 
underlying a regularity structure is given by $\Vec(\mfT(R))$ for a subset $\mfT(R) \subset \mfT$ 
determined by some rule $R$. 
The construction and action of the structure and renormalisation groups was then described by 
using combinatorial operations on elements of $\mfT$ and $\mfF$.

Here our point of view is different.
Our concrete regularity structure will be obtained by applying the functor $\Func_{V}$ to $\scal{\mfT(R)}$, an object in $\TStruc$. 
We will refer to $\scal{\mfT(R)}$ as an ``abstract'' regularity structure. 
In particular, trees $\ttau \in \mfT(R)$ will not be interpreted as basis vectors for our regularity structure 
anymore, but instead serve as an indexing set for subspaces canonically isomorphic to $\Func_{V}(\scal{\ttau})$. 
In the case when $V_\mft \simeq \R$ for all $\mft \in \Lab$, this is of course equivalent,
but in general it is not.
Operations like integration, tree products, forest products, and co-products on the regularity structures defined in \cite{BHZ19} were previously given in terms of operations on $\mfT$ and \slash or $\mfF$. 
In order to push these operations to our concrete regularity structure, we will in the next section describe how to interpret them as morphisms between the corresponding typed structures, which then 
allows us to push them through to ``concrete'' regularity structures using $\Func_{V}$. 
\begin{remark}\label{rem:trees_remain_trees}
Although the definition of $\mfT_{\Lab}$ depends on the choice of $\Lab$, this definition is compatible with $\proj^{\ast}$ in the following sense.
Given $\ttau  \in \mfT_{\Lab}$, if $L_{\scal{\ttau}}$ is defined as in the definition immediately below \eqref{e:defsim}, then $L_{\scal{\ttau}}$ can be identified with a subset of $\mfT_{\bar{\Lab}}$. 
In particular, we overload notation and define a map $\proj\colon \mfT_{\Lab} \rightarrow \CP (\mfT_{\bar{\Lab}})\setminus \{\emptyset\}$ by setting $\proj (\ttau) = L_{\scal{ \ttau}}$ so $\proj^{\ast} \langle \ttau \rangle \simeq \bigoplus_{\bar{\ttau} \in \proj(\ttau)} \scal{ \bar{\ttau}}$. 
It is also easy to see that $\{ \proj(\ttau): \ttau \in \mfT_{\Lab}\}$ is a partition of $\mfT_{\bar{\Lab}}$ so that
\[
\proj^{\ast}
\scal{\mfT_{\Lab}}
\simeq
\scal{\mfT_{\bar{\Lab}}}\;.
\]
Analogous statements hold for the sets of forests $\mfF_{\Lab}$ and $\mfF_{\bar{\Lab}}$. 
\end{remark}
\subsubsection{Integration and products}
\label{sec:Products and coproducts}
We start by recalling the tree product. 
Given two rooted labelled trees $\tau = (T,\rho,\mft,\mfn)$ and $\bar{\tau} = (\bar{T},\bar{\rho},\bar{\mft},\bar{\mfn})$ the tree product of $\tau$ and $\bar{\tau}$, which we denote $\tau  \bar{\tau}$, is a rooted labelled tree defined as follows. 
Writing $\tau \bar{\tau} = (\hat{T},\hat{\rho},\hat{\mft},\hat{\mfn})$, one sets $\hat{T} \eqdef (T \sqcup \bar{T})/\{\rho,\bar{\rho}\}$, namely $\hat{T}$ is the rooted tree obtained by taking the rooted trees $T$ and $\bar{T}$ and identifying the roots $\rho$ and $\bar{\rho}$ into a new root $\hat{\rho}$. 
Writing $T = (V,E)$, $\bar{T} = (\bar{V},\bar{E})$, and $\hat{T} = (\hat{V},\hat{E})$, we have a canonical identification of $\hat{E}$ with $E \sqcup \bar{E}$ and $\hat{V} \setminus \{\hat{\rho}\}$ with $(V \sqcup \bar{V}) \setminus \{\rho,\bar{\rho}\}$. 
With these identifications in mind, $\hat{\mft}$ is obtained from the concatenation of $\mft$ and $\bar{\mft}$. 
We also set
\[
\hat{\mfn}(a) \eqdef
\begin{cases}
\mfn(a) 
& 
\textnormal{ if }a \in E \sqcup (V \setminus \{\rho\}),\\
\bar{\mfn}(a) 
& 
\textnormal{ if }a \in \bar{E} \sqcup (\bar{V} \setminus \{\bar{\rho}\}),\\
\mfn(\rho) + \bar{\mfn}(\bar{\rho})
&
\textnormal{ if }a = \hat{\rho}\;.
\end{cases}
\]
We remark that the tree product is well-defined and commutative at the level of isomorphism classes.  

In order to push this tree product through our functor, we want to encode it as a morphism
$\CM \in \Hom(\scal{\mfT} \otimes \scal{\mfT}, \scal{\mfT})$. It is of course sufficient for this to 
define elements $\CM \in \Hom(\scal{\ttau} \otimes \scal{\bar \ttau}, \scal{\ttau \bar \ttau})$
for any $\ttau, \bar \ttau \in \mfT$, which in turn is given by
\begin{equ}[e:idenIso]
\scal{\ttau} \otimes \scal{\bar \ttau} \simeq \scal{\tau} \otimes \scal{\bar \tau} \to \scal{\tau \bar \tau} \simeq \scal{\ttau \bar \ttau}\;,
\end{equ}
where the two canonical isomorphisms are the ones given by Remark~\ref{rem:canonicalIsomorphism}
and the morphism in $\Hom\big (\langle \tau \rangle \otimes \langle \bar{\tau} \rangle, \langle \tau \bar{\tau} \rangle)$ is obtained
as follows.
Note that the same typed set $(\hat{E},\hat{\mft})$ underlies both the symmetric sets $\langle \tau \rangle \otimes \langle \bar{\tau} \rangle$ and $\langle \tau \bar{\tau} \rangle$ and that $\Gamma_{\langle \tau \rangle \otimes \langle \bar{\tau} \rangle}$ is a subgroup (possibly proper) of $\Gamma_{\langle \tau \bar{\tau} \rangle}$. 
Therefore, the only natural element of $\Hom(\scal{\tau} \otimes \scal{\bar \tau}, \scal{\tau \bar \tau})$ is the 
equivalence class of the identity in 
$\Homb(E \sqcup \bar E,\langle \tau \bar{\tau} \rangle)$ (in the notation of Remark~\ref{rem:singleton}).
It is straightforward to verify that $\CM$ constructed in this way is independent 
of the choices $\tau \in \ttau$ and $\bar \tau \in \bar \ttau$.

The ``neutral element'' $\eta \in \Hom(I,\scal{\mfT})$ for $\CM$,
where $I$ denotes the unit object 
in $\TStruc$ (corresponding to the empty symmetric set), is given by 
the canonical isomorphism $I \to \scal{\bone}$
with $\bone$ denoting the 
tree with a unique vertex and $\mfn = 0$ as before,
composed with the canonical inclusion $\scal{\bone} \to \scal{\mfT}$. One does indeed have
$\CM\circ (\eta \otimes \id) = \CM\circ (\id \otimes \eta) = \id$, with equalities 
holding modulo the identifications $\scal{\mfT} \simeq \scal{\mfT} \otimes I \simeq I\otimes \scal{\mfT}$.
Associativity holds in a similar way, namely $\CM\circ(\id\otimes\CM) = \CM\circ(\CM\otimes \id)$
as elements of $\Hom(\scal{\mfT}\otimes\scal{\mfT}\otimes\scal{\mfT},\scal{\mfT})$.
\begin{remark}
Another important remark is that the construction of the product $\CM$ respects the 
functors $\proj^*$ in the same way as the construction of $\scal{\mfT}$ does.
\end{remark}
As mentioned above, $\mfF$ is viewed as the free unital commutative monoid 
generated by $\mfT$ with unit given by the empty forest (which we denote by $\emptyset$). 
We can interpret this product in the following way. 
Given two rooted labelled forests $f = (F,\mathfrak{P},\mft,\mfn)$ and $\bar{f} = (\bar{F},\bar{\mathfrak{P}}, \bar{\mft}, \bar{\mfn})$ we define the forest product  $f \cdot \bar{f} = (\hat{F}, \hat{\mathfrak{P}},\hat{\mft},\hat{\mfn})$ by $\hat{F} = F \sqcup  \bar{F}$, $\hat{\mathfrak{P}} = \mathfrak{P} \sqcup \bar{\mathfrak{P}}$, $\hat{\mft} = \mft \sqcup \bar{\mft}$, and $\hat{\mfn} = \mfn \sqcup \bar{\mfn}$. 
Again, it is easy to see that this product is well-defined and commutative at the level 
of isomorphism classes. 
As before, writing $\hat{F} = (\hat{V},\hat{E})$ and noting that
the same typed set $(\hat{E},\hat{\mft}) = (E \sqcup \bar E, \mft\sqcup \bar \mft)$ 
underlies both symmetric sets
$\langle f \rangle \otimes \langle \bar{f} \rangle$ and $\langle f \cdot \bar{f} \rangle$ and that the symmetry group of the former is a subgroup of that of the latter,
there is a natural morphism $\Hom(\langle f \rangle \otimes \langle \bar{f} \rangle,\langle f \cdot \bar{f} \rangle)$ given by the 
equivalence class of the identity.
As before, this yields a product morphism in $\Hom(\scal{\mfF} \otimes \scal{\mfF}, \scal{\mfF})$, this time with 
the canonical isomorphism between $I$ and $\scal{\emptyset}$ (with $\emptyset$ the empty forest) playing the role of the neutral element.

We now turn to integration.
Given any $\mfl \in \mfL$ and $\tau = (T,\rho,\mft,\mfn) \in \mfT$ we define a new rooted labelled tree $\mcb{I}_{(\mfl,0)}(\tau) = (\bar{T},\bar{\rho},\bar{\mft},\bar{\mfn}) \in \mfT$ as follows. 
The tree $\bar{T} = (\bar{V},\bar{E})$ is obtained from $T = (V,E)$ by setting $\bar{V} \eqdef V \sqcup \{\bar{\rho}\}$ and $\bar{E} = E \sqcup \{\bar{e}\}$ where $\bar{e} = (\rho,\bar{\rho})$, that is one adds a new root vertex to the tree $T$ and connects it to the old root with an edge. 
We define $\bar{\mft}$ to be the extension of $\mft$ to $\bar{E}$ obtained by setting $\mft(\bar{e}) = \mfl$ and $\bar{\mfn}$ to be the extension of $\mfn$ obtained by setting $\bar{\mfn}(\bar{e}) = \bar{\mfn}(\bar{\rho}) = 0$. 
We encode this into a morphism $\mcb{I}_{(\mfl,0)} \in \Hom( \scal{\mfT} \otimes\scal{\mfl}, \scal{\mfT})$
where $\scal{\mfl}$ denotes the symmetric set with a single element $\bullet$ of type $\mfl$.
For this, it suffices to exhibit natural morphisms    
\begin{equ}[e:planting-morph]
\Hom\big(\scal{ \tau }\otimes\scal{\mfl} , \scal{ \mcb{I}_{(\mfl,0)}(\tau)}\big)\;,
\end{equ}
which are given by the equivalence class of $\iota\colon  E \sqcup\{\bullet\} \rightarrow \bar{E}$ 
in $\Homb( E \sqcup\{\bullet\}, \scal{ \mcb{I}_{(\mfl,0)}(\tau)})$, where $\iota$
is the identity on $E$ and  $\iota(\bullet) = \bar{e}$. It is immediate that this 
respects the automorphisms of $\tau$ and therefore defines indeed an element of $\Hom\big(\scal{ \tau }\otimes\scal{\mfl} , \scal{ \mcb{I}_{(\mfl,0)}(\tau)}\big)$.
The construction above also gives us corresponding morphisms $\mcb{I}_{(\mfl,p)} \in \Hom(\scal{\mfT}\otimes\scal{\mfl}, \scal{\mfT})$, for any $p \in \N^{d+1}$, if we exploit the canonical isomorphism $\scal{ \mcb{I}_{(\mfl,0)}(\tau)} \simeq \scal{\mcb{I}_{(\mfl,p)}(\tau)}$ where $\mcb{I}_{(\mfl,p)}(\tau)$ is constructed just as $\mcb{I}_{(\mfl,0)}(\tau)$, the only difference being that one sets $\bar{\mfn}(\bar{e}) = p$. 
\subsubsection{Coproducts}\label{subsec:coprod}
In order to build a regularity structure, we will also need analogues of the maps $\Deltap$ and $\Deltam$
as defined in \cite{BHZ19}.
The following construction will
be very useful: given $\tau \in \ttau \in \mfT$ and $f \in \ff \in \mfF$, we write $f \subforest \tau$ for 
the specification of an injective map $\iota \colon T_f \to T_{\tau}$ which preserves connectivity,
orientation, and type (but roots of $f$ may be mapped to arbitrary vertices of $\tau$).
We also impose that $\mfn_{f}(e) = \mfn_\tau(\iota e)$ for every edge $e \in E_{f}$ and that polynomial 
vertex labels are increased by $\iota$ in the sense that $\mfn_f(x) \le \mfn_\tau(\iota x)$
for all $x \in V_f$. 
Given $f \hookrightarrow \tau$, we also write $\d E_f \subset E_\tau \setminus \iota(E_f)$ for the set of 
edges $e$ ``incident to $f$'' in the sense that $e_+ \in \iota(V_f)$. 
We consider inclusions $\iota \colon f \hookrightarrow \tau$ and 
$\bar \iota \colon \bar f \hookrightarrow \tau$ to be ``the same'' if there exists a forest isomorphism
$\phi \colon f \to \bar f$ such that $\iota = \bar \iota \circ \phi$.
(We do however consider them as distinct if they differ by a tree isomorphism of the target $\tau$!)

Given a label $\mfe \colon \d E_f \to \N^{d+1}$,
we write $\pi \mfe \colon V_f \to \N^{d+1}$ for the map given by 
$\pi \mfe(x) = \sum_{e_+ = \iota x} \mfe(e)$
and we write $f_\mfe$ for the forest $f$, but with $\mfn_f$ replaced by $\mfn_f + \pi \mfe$.
We then write $\tau / f_\mfe \in \mfT$ for the tree constructed as follows.
Its vertex set is given by $V_{\tau} / {\sim_f}$, where ${\sim_f}$ is the equivalence relation
given by $x \sim_f y$ if and only if $x,y \in \iota(V_f)$ and $\iota^{-1}x$ and $\iota^{-1}y$
belong to the same connected component of $f$.
The edge set of $\tau / f_\mfe$ is given by $E_{\tau} \setminus \iota(E_f)$, and types  and the root 
are inherited from $\tau$. Its edge label is given by $e \mapsto \mfn_\tau(e) + \mfe(e)$.
Noting that vertices of $\tau / f_\mfe$ are subsets of $V_\tau$,
its vertex label is given by $x \mapsto \sum_{y \in x} \big(\mfn_\tau(y) - \mfn_f(\iota^{-1} y)\big)$
with the convention that $\mfn_f$ is extended additively to subsets.
(This is positive by our definition of ``inclusion''.) 

This construction then naturally defines an `extraction \slash contraction' operation
$(f \hookrightarrow \tau)_\mfe \in \Homb(E_\tau, \scal{f_\mfe} \otimes \scal{\tau / f_\mfe})$ similarly to above.
(Using $\iota$, the edge set of $\tau$ is canonically identified with the disjoint union 
of the edge set of $f_\mfe$ with that of $\tau / f_\mfe$.) Note that this is well-defined in the sense
that two identical (in the sense specified above) inclusions yield identical (in the sense of
canonically isomorphic) elements 
of $\Homb(E_\tau, \scal{f_\mfe} \otimes \scal{\tau / f_\mfe})$.
We also define $\bar f / f_\mfe$ and $(f \hookrightarrow \bar f)_\mfe$ for a forest $\bar f$ in the analogous way.

We also define a ``cutting'' operation in a very similar way. Given two trees $\tau$ and $\bar \tau$, 
we write $\bar \tau \subroot \tau$ if $\bar \tau \subforest \tau$ (viewing $\bar \tau$ as a forest with a single tree) 
and the injection $\iota$ furthermore
maps the root of $\bar \tau$ onto that of $\tau$.
With this definition at hand, we define ``extraction'' and ``cutting'' operators
\begin{equs}
\Deltaex[\tau] &\in \Vec \big( \Homb(E_\tau,\scal{\mfF} \otimes \scal{\mfT}) \big) \;,&\qquad
\Deltaex[\tau] &= \sum_{f \hookrightarrow \tau} \sum_{\mfe}{\frac{1}{\mfe!}}\binom{\tau}{f} (f \hookrightarrow \tau)_\mfe\;, \\
\Deltacut[\tau] &\in \Vec \big( \Homb(E_\tau,\scal{\mfT} \otimes \scal{\mfT})
\big) \;,&\qquad
\Deltacut[\tau] &= \sum_{\bar\tau \subroot \tau} \sum_{\mfe}{\frac{1}{\mfe!}}\binom{\tau}{\bar \tau} (\bar \tau \subroot \tau)_\mfe\;.
\end{equs}
Here, the inner sum runs over $\mfe \colon \d E_f \to \N^d$ ($\mfe \colon \d E_{\bar \tau} \to \N^d$ in the second case) and 
the binomial coefficient $\binom{\tau}{f}$ is defined as
\begin{equ}
\binom{\tau}{f} = \prod_{x \in V_f} \binom{\mfn_\tau(\iota x)}{\mfn_f(x)}\;.
\end{equ}
We also view $\Homb(E_\tau, \scal{f_\mfe} \otimes \scal{\tau / f_\mfe})$
as a subset of $\Homb(E_\tau,\scal{\mfF} \otimes \scal{\mfT})$ via the canonical
maps $\scal{\tau} \simeq \scal{\ttau} \hookrightarrow \scal{\mfT}$ and similarly for $\scal{\mfF}$.
\begin{lemma}
One has $\Deltaex[\tau] \in \Hom(\scal{\tau},\scal{\mfF} \otimes \scal{\mfT})$
as well as $\Deltacut[\tau] \in \Hom(\scal{\tau},\scal{\mfT} \otimes \scal{\mfF})$.
\end{lemma}
\begin{proof}
Given any isomorphism $\phi$ of $\tau$, it suffices to note that, in $\Hom(\scal{\tau},\scal{\mfF} \otimes \scal{\mfT})$, we have the identity
\begin{equ}
(f \hookrightarrow \tau)_\mfe \circ \phi = (f\phi \hookrightarrow \tau)_{\mfe\phi}\;,
\end{equ}
where, if $f \hookrightarrow \tau$ is represented by $\iota$, then 
$f\phi \hookrightarrow \tau$ is represented by $\phi^{-1}\circ \iota$
and $\mfe\phi = \mfe \circ \phi$. It follows that $\Deltaex[\tau] \circ \phi = \Deltaex[\tau]$ as required.
The argument for $\Deltacut$ is virtually identical.
\end{proof}
It also follows from our construction that, given $\ttau \in \mfT$, 
$\Deltaex[\tau]$ and $\Deltacut[\tau]$ are independent of $\tau \in \ttau$, modulo canonical isomorphism as
in Remark~\ref{rem:canonicalIsomorphism} (see also \eqref{e:idenIso} above), so that
we can define $\Deltaex[\ttau] \in \Hom(\scal{\ttau},\scal{\mfF} \otimes \scal{\mfT})$ and similarly
for $\Deltacut$.
\subsection{Regularity structures generated by rules}\label{subsec: regstruct_gen_rules} 
We now show how regularity structures generated by rules as in \cite{BHZ19} can be recast
in this framework. 
This then allows us to easily formalise constructions of the type 
``attach a copy of $V$ to every noise / kernel'' as was done in a somewhat ad hoc fashion
in \cite[Sec.~3.1]{Mate2}. 

We will restrict ourselves to the setting of \emph{reduced} abstract regularity structures 
(as in the language of \cite[Section~6.4]{BHZ19}).
The extended label (as in \cite[Section~6.4]{BHZ19}) will not play an explicit role here 
but appears behind the scenes when we use the black box of \cite{BHZ19} to build a 
corresponding scalar reduced regularity structure, which is then identified,
via the natural transformation of Proposition~\ref{prop:nat_transform},
with the concrete regularity structure obtained by applying $\Func_{V}$ to our abstract regularity structure.  

We start by fixing a degree map $\deg \colon \Lab \to \R$ (where $\Lab$ was our previously fixed set of labels), 
a ``space'' dimension $d \in \N$, and a scaling\footnote{In \cite{Hairer14} scalings are integers, which was used in the proof of the reconstruction theorem using wavelets. More recent proofs \cite{FrizHairer} do not make use of this and allow for arbitrary scalings. Furthermore, the algebraic structure (which is the relevant aspect in this section) is impervious to
the scaling.} $\s \in [1,\infty)^{d+1}$.
Multiindices 
$k \in \N^{d+1}$ are given a scaled degree $|k|_\s = \sum_{i=0}^d k_i\s_i$.
(We use the convention that the $0$-component denotes the time direction.)

We then define, as in \cite[Eq.~(5.5)]{BHZ19}, the sets $\CE$ of edge labels and
$\CN$ of node types by
\begin{equ}
\CE = \Lab \times \N^{d+1}\;,\qquad \CN = \hat \CP(\CE)\;,
\end{equ} 
where $\hat \CP(A)$ denotes the set of all multisets with elements from $A$.
With this notation, we fix a ``rule'' $R \colon \Lab \to \CP(\CN) \setminus \{\emptyset\}$.\label{rule page ref}
We will only consider rules that are 
subcritical and complete in the sense of \cite[Def.~5.22]{BHZ19}.

Given $\tau \in \ttau\in\mfT$
with underlying tree $T = (V,E)$, every vertex $v \in V$ is naturally associated with a node type 
$\CN(v)\eqdef \big(o(e)\,:\, e_+ = v\big) \in \CN$, where we set $o(e) = (\mft(e), \mfn(e)) \in \CE$.
The degree of $\tau$ is given by 
\begin{equ}
\deg\tau = \sum_{v \in V} |\mfn(v)|_\s + \sum_{e \in E} \bigl(\deg \mft(e) - |\mfn(e)|_\s\bigr)\;.
\end{equ}
We say that $\tau$ \emph{strongly conforms} to $R$ if
\begin{enumerate}[label=(\roman*)]
\item \label{pt:above_root} for every $v \in V \setminus \{\rho\}$, one has
$\CN(v) \in R(\mft(v))$, and
\item \label{pt:root} there exists $\mft \in \Lab$ such that $\CN(\rho) \in R(\mft)$.
\end{enumerate}
We say that $\tau$ is \emph{planted} if $\sharp \CN(\rho) = 1$ and $\mfn(\rho) = 0$ and \emph{unplanted} otherwise.

The map $\deg$, the above properties,
and the label $\mfn(\rho)\in\N^{d+1}$ depend only on the isomorphism class $\ttau\ni\tau$,
and we shall use the same terminology for $\ttau$.
We write $\mfT(R) \subset \mfT$\label{mfT(R) page ref} for the set
of $\ttau$ that strongly conform to $R$.
We further write $\mfT_\star(R)\subset\mfT$ for the set of planted trees 
satisfying condition~\ref{pt:above_root}.

We now introduce the algebras of trees / forests that are used for negative and positive renormalisation.
We write $\mfF(R) \subset \mfF$ for the unital monoid generated (for the forest product) by $\mfT(R)$, $\mfF_-(R) \subset \mfF(R)$ for 
the unital monoid generated by
\begin{equ}[e:def-mfT-]
\mfT_-(R) \eqdef 
\{\ttau \in \mfT(R)\,:\, \deg \ttau < 0,\ \mfn(\rho) = 0,\ \ttau \textnormal{ unplanted}\}\;,
\end{equ}
and $\mfT_+(R)\subset \mfT$ 
for the unital monoid generated (for the tree product) by 
\begin{equ}
\{\mbX^k\,:\, k \in \N^d\} \cup \{\ttau \in \mfT_\star(R)\,:\, \deg \ttau > 0\}\;.
\end{equ}
We are now ready to construct our abstract regularity structure in the
category $\TStruc$ that was defined in Definition~\ref{def:typed_struct}.
We define $\ST, \STp, \STm, \SF, \SFm \in \ob(\TStruc)$ by  \label{STSF page ref}
\begin{equs}
\ST &\eqdef \scal{\mfT(R)} \;,
\quad \STp \eqdef \scal{\mfT_+(R)}\;,
\quad \STm \eqdef \scal{\mfT_-(R)}\;,
\\
\quad \SF&\eqdef \scal{\mfF(R)}\;,
\quad \SFm\eqdef \scal{\mfF_-(R)}\;.
\end{equs}
%
As mentioned earlier, we think of $\ST$ as an ``abstract'' regularity structure with ``characters on $\STp$'' forming its structure group
and ``characters on $\SFm$'' forming its renormalisation group.
Write
\begin{equ}
\pi_+ \in \Hom(\ST,\STp)\;,\qquad
\pi_- \in \Hom(\SF,\SFm)\;,
\end{equ}
for the natural projections. 
These allow us to define 
\begin{equs}
\Hom(\ST,\ST \otimes \STp) \ni \Deltap
&\eqdef \sum_{\ttau \in \mfT(R)} (\id \otimes \pi_+)\Deltacut[\ttau]\;,\\
\Hom(\ST,\SFm \otimes \ST) \ni \Deltam
&\eqdef \sum_{\ttau \in \mfT(R)} (\pi_- \otimes \id)\Deltaex[\ttau]\;.
\end{equs}
%
%

\begin{remark}
Here and below, it is not difficult to see that these expressions do indeed define morphisms of $\TStruc$.
Regarding $\Deltap$ for example, it suffices to note that, given any $\ttau \in \mfT(R)$ and $\ttau_+ \in \mfT_+(R)$,
there exist only finitely many pairs of trees $\ttau^{(2)} \subroot \ttau^{(1)}$ in $\mfT(R)$ 
and edge labels $\mfe$ in Section~\ref{subsec:coprod},
such that $\ttau_+ = \ttau^{(1)} / \ttau^{(2)}_\mfe$ and $\ttau = \ttau^{(2)}_\mfe$.
\end{remark}
In an analogous way, we also define
\begin{equ}
\Deltap_s \in \Hom(\STp,\STp\otimes \STp)\;,\qquad
\Deltam_s \in \Hom(\SFm,\SFm \otimes \SFm)\;,
\end{equ}
by
\begin{equ}
\Deltap_s
\eqdef \sum_{\ttau \in \mfT_+(R)}(\pi_+ \otimes \pi_+) \Deltacut[\ttau]\;,\qquad
\Deltam_s
\eqdef \sum_{\ff \in \mfF_-(R)}(\pi_- \otimes \pi_-) \Deltaex[\ff]\;.
\end{equ}
As in \cite{BHZ19}, one has the identities
\begin{equ}
\bigl(\Deltam_s \otimes \id\bigr)\circ \Deltam
= \bigl(\id \otimes \Deltam\bigr) \circ \Deltam\;,\qquad 
\bigl(\Deltam_s \otimes \id\bigr)\circ \Deltam
= \bigl(\id \otimes \Deltam\bigr) \circ \Deltam\;,
\end{equ}
as well as the coassociativity property for $\Deltam_s$, multiplicativity of $\Deltap$ and $\Deltap_s$
with respect to the tree product, and multiplicativity of $\Deltam$ and $\Deltam_s$
with respect to the forest product. 
\subsection{Concrete regularity structure}
\label{sec:concrete}
Recall that the label set $\Lab$ splits as $\Lab = \Lab_+ \cup \Lab_-$, where $\Lab_-$ indexes the set of
``noises'' while $\Lab_+$ indexes the set of kernels, which in the setting of \cite{BCCH21} equivalently
indexes the components of the class of SPDEs under consideration. 

It is then natural to introduce another space assignment called a target space assignment $(W_{\mft})_{\mft \in \mfL}$ where, for each $\mft \in \Lab$,
the vector space $W_\mft$ is the target space for the corresponding noise or component of the solution. 
As we already saw in the discussion at the start of Section~\ref{sec:motivation}, given a 
noise taking values in some space $W_\mft$ for some $\mft\in\mfL_-$, it is natural to
assign to it a subspace of the regularity structure that is isomorphic to the (algebraic) dual space $W_\mft^*$.

Similarly, there are circumstances in which our integration kernels will not be maps from our space-time into $\R$ but instead be vector valued.
Therefore it is natural to introduce a kernel space assignment $\mathcal{K} = (\mathcal{K}_{\mft})_{\mft \in \Lab_{+}}$, so that we can accommodate integral equations of the form
\begin{equ}\label{e:KSPDE}
 A_\mft = K_{\mft} \ast F_{\mft}(\mbA, \bxi)\;, \quad \mft\in\mfL_+\;,
\end{equ}
where $K_{\mft}$ takes values in $\mathcal{K}_{\mft}$ and the non-linearity $F_{\mft}(\mbA, \bxi)$ would take values in $\mathcal{K}_{\mft}^* \otimes W_{\mft}$ so that the right hand side takes values in $W_{\mft}$ as desired. 
Accordingly, edges that correspond to a type $\mft \in \Lab_{+}$ should be associated to $\mathcal{K}_{\mft}^*$. 
\begin{remark}\label{rem:vect_kernel_example}
For an example where a formulation such as \eqref{e:KSPDE} may be useful, suppose that we are looking at a system of equations where, for some fixed $\mft \in \Lab_{+}$, the equation for the $\mft$ component is of the form
\begin{equ}\label{eq:vect_kernel_example}
(\partial_{t} - \Delta)A_{\mft} = \sum_{j=1}^{d} \partial_{j} G_{j}(\mathbf{A},\zeta)\;.
\end{equ}
Such an equation does fit into our framework by just choosing $F_{\mft}(\mathbf{A},\zeta)$ to be the right hand side of the above equation after one expands the derivatives.  

However, this choice may not be optimal in some situations. 
For instance, let us now further assume that we already know that $G_{j}(\mathbf{A},\zeta)$ do not generate any renormalisation. 
We would also expect that the non-linear expression $F_{\mft}(\mathbf{A},\zeta)$ obtained after differentiation would also not generate any renormalisation but proving this implication can be non-trivial since the application of Leibniz rule may generate many new singular terms. 

A more lightweight approach is to formulate things as in \eqref{e:KSPDE} where we set 
\begin{equs}
\mathcal{K}_{\mft} 
&= \R^{d}\;,
\quad
K_{\mft}(x) 
= \big(\nabla(\partial_{t} - \Delta)^{-1}\big)(x) \in \mathcal{K}_{\mft} \;,\\
{}&
\textnormal{and} \quad 
F_{\mft}(\mathbf{A},\zeta) 
= 
\big(G_{j}(\mathbf{A},\zeta)\big)_{j=1}^{d} 
\in \mathcal{K}_{\mft}^{\ast} \otimes W_{\mft}\;.
\end{equs}
With this choice, our framework will, without any extra effort, allow us to infer that we do not see any renormalisation in \eqref{eq:vect_kernel_example}.  
\end{remark}
When we want to refer to a set of labels $\Lab$ (which comes with a split $\Lab = \Lab_{+} \sqcup \Lab_{-}$) along with an associated target space assignment $W = (W_{\mft})_{\mft \in \Lab}$ and kernel space assignment $\mathcal{K} = (\mathcal{K}_{\mft})_{\mft \in \Lab_{+}}$ we will often write them together as a triple $(\Lab,W,\mathcal{K})$
 
Then, given $(\Lab,W,\mathcal{K})$, for fixing the space assignment $(V_{\mft})_{\mft \in \mfL}$ used in the category theoretic constructions earlier in this section, our discussion above motivates space assignments of the form
\begin{equ}[e:defV]
V_\mft \eqdef 
\left\{\begin{array}{cl}
	W_\mft^* & \text{for $\mft \in \Lab_-$,} \\
	\mathcal{K}_{\mft}^* & \text{for $\mft \in \Lab_+$.}
\end{array}\right.
\end{equ} 
Given a space assignment $V$ of the form \eqref{e:defV}, we then use the functor $\Func_V$ defined in Section~\ref{subsec:sym_tensor_prod} to define the vector spaces
$\CT,\CT_+,\CT_-,\CF,\CF_-$ by
\begin{equ}[e:defCTtau]
\CU \eqdef \Func_V(\SU) = \prod_{\ttau \in \mfU(R)} \CU[\ttau] \;,
\qquad
\CU[\ttau] \eqdef \Func_V(\scal{\ttau})=V^{\otimes \scal{\ttau}}
\end{equ}
where, respectively,
\begin{itemize}
\item $\CU$ is one of $\CT,\CT_+,\CT_-,\CF,\CF_-$, 
\item $\SU$ is one of $\ST,\STp,\STm,\SF,\SFm$, and
\item $\mfU$ is one of $\mfT,\mfT_+,\mfT_-,\mfF,\mfF_-$ ($\ttau$ in \eqref{e:defCTtau} can denote either a tree or a forest).
\end{itemize}
We also adopt a similar notation for linear maps $h \colon \CU \to X$ (for any vector space $X$) by writing $h[\ttau]$ for the restriction of $h$ to $\CU[\ttau]$ for any $\ttau \in \mfU(R)$.
Note that the forest product turns $\CF$ and $\CF_-$ into algebras, while the tree product turns $\CT_+$ into an algebra.
We call $\CT$ the concrete regularity structure built from $\ST$. 
The notion of sectors of $\CT$ is defined as before in \cite[Def.~2.5]{Hairer14}.
\begin{remark}\label{rem:proj extend to E} 
Given a label decomposition $\bar{\Lab}$ of $\Lab$  under $\proj$,  we will naturally ``extend'' $\proj$ to a map $\proj\colon\CE \rightarrow \CP(\bar{\CE})$, where $\bar{\CE} = \bar{\Lab} \times \N^{d+1}$, by setting, for $o = (\mft, p) \in  \CE$, $\proj(o) = \proj(\mft) \times \{p\}$.

If we have a splitting of labels $\Lab = \Lab_{+} \sqcup \Lab_{-}$ then we implicitly work with a corresponding splitting $\bar{\Lab} = \bar{\Lab}_{+} \sqcup \bar{\Lab}_{-}$ given by $\bar{\Lab}_{\pm} = \bigsqcup_{\mft \in \Lab_{\pm}} \proj(\mft)$. 

Additionally, given a rule $R$ with respect to the labelling set $\Lab$, we obtain a corresponding rule $\bar{R}$ with respect to $\bar{\Lab}$ 
 by setting, for each $\bar{\mft} \in \bar{\Lab}$, 
\[
\bar{R}(\bar{\mft}) = \Big\{ \bar N \in \hat \CP(\CE)\,:\, \exists N \in R(\mft) \quad\text{with}\quad \bar N \tto N \Big\}\;,
\]
where $\mft$ is the unique element of $\Lab$ with $\bar{\mft} \in \proj(\mft)$ and we say that $\bar N \tto N$ if there is a bijection from $\bar{N}$ to $N$ respecting\footnote{Respecting $\proj$ means that if $\bar{N} \ni \bar{o} \mapsto o \in N$ then $\bar{o} \in \proj(o)$ where $\proj\colon\CE \rightarrow \CP(\bar{\CE})$ as above.} $\proj$. 
If we are also given a notion of degree $\deg\colon \mfL \rightarrow \R$ then we also have an induced degree $\deg\colon \bar{\mfL} \rightarrow \R$ by setting, for $\bar{\mft}$ and $\mft$ as above, $\deg(\bar{\mft}) = \deg(\mft)$. 
It then follows that subcriticality or completeness hold for $\bar{R}$ if and only if they hold for $R$

Linking back to Remark~\ref{rem:trees_remain_trees}, we also mention that $\{\proj(\ttau): \ttau \in \mfT_{\mfL}(R)\}$ is a partition of $\mfT_{\bar{\Lab}}(\bar{R})$. 
\end{remark}
\begin{remark}\label{rem:type_space_decomposition}
We say $(\bar{\Lab},\bar{W},\bar{\mathcal{K}})$ is a decomposition of $(\Lab,W,\mathcal{K})$ under $\proj$ if $\bar{\Lab}$ is a decomposition of $\Lab$ under $\proj$ and  
\begin{equs}\label{eq:space-decompositions}
W_{\mft} &= 
\begin{cases}
\bigoplus_{\bar{\mft} \in \proj(\mft)} W_{\bar{\mft}} & \text{ for all }\mft \in \Lab_{-}\;,\\
\bar{W}_{\bar{\mft}} & \text{ for all }\mft \in \Lab_{+} \text{ and } \bar{\mft} \in \proj(\mft)\;,
\end{cases}\\
\mathcal{K}_{\mft} &= \bigoplus_{\bar{\mft} \in \proj(\mft)} \bar{\mathcal{K}}_{\bar{\mft}} \text{ for all }\mft \in \Lab_{+}\;.
\end{equs}
 
The regularity structure $\CT$ defined in \eqref{e:defCTtau} is then fixed, up to natural transformation, under noise decompositions of $\Lab$.
In this case, the space assignment $\bar{V} = (\bar{V}_{\mfl})_{\mfl \in \bar{\Lab}}$ built from $(\bar{\Lab},\bar{W},\bar{\mathcal{K}})$ using \eqref{e:defV} is a decomposition of $(V_{\mft})_{\mft \in \Lab}$ built from $(\Lab,W,\mathcal{K})$. 

Note that our decompositions do not ``split-up'' the spaces $W_{\mft}$ for $\mft \in \Lab_{+}$, that is our decompositions for kernel types will decompose the spaces where our kernels take values but not the spaces where our solutions take values. 
\end{remark}
\begin{remark}\label{rem:scalar_decompositions}
We call $(\Lab, W,\mathcal{K})$ \emph{scalar}  if, for every $\mft \in \Lab_{-}$, $\dim(W_{\mft}) = 1$ and, for every $\mft \in \Lab_{+}$, $\dim(\mathcal{K}_{\mft}) = 1$.

Our construction of regularity structures and renormalisation groups in this section will match the constructions in \cite{Hairer14,BHZ19} when we have a scalar $(\Lab, W,\mathcal{K})$ and so we will have all the machinery developed in \cite{Hairer14,CH16,BHZ19,BCCH21} available. 

We say $(\bar{\Lab},\bar{W},\bar{\mathcal{K}})$ is a \emph{scalar decomposition} of $(\Lab,W,\mathcal{K})$ under $\proj$ if it is a decomposition under $\proj$ as described in Remark~\ref{rem:type_space_decomposition} and $(\bar{W},\bar{\mathcal{K}})$ is  scalar. 

In this situation natural transformations given by Proposition~\ref{prop:nat_transform} allow us to identify the regularity structure and renormalisation group built from $(\Lab,W,\mathcal{K})$ with those built from $(\bar{\Lab},\bar{W},\bar{\mathcal{K}})$ . 
Thanks to this we can leverage the machinery of \cite{Hairer14,CH16,BHZ19,BCCH21} for the regularity structure and renormalisation group built from $(\Lab,W,\mathcal{K})$. 
\end{remark} 
\begin{remark}
One remaining difference between the setting of a scalar  and the setting of \cite{Hairer14,BHZ19} is that in \cite{Hairer14,BHZ19} one also enforces $\dim(W_{\mft})=1$ for $\mft \in \Lab_{+}$ \dash this constraint enforces solutions to also be scalar-valued. 
However, while a scalar noise assignment allows $\dim(W_{\mft}) > 1$ for $\mft \in \Lab_{+}$, our decision to enforce $V_{\mft} = \R$ in \eqref{e:defV} means that we require that the ``integration'' encoded by edges of type $\mft$ acts diagonally on $W_{\mft}$, i.e., it doesn't mix components. 
In particular, with this constraint the difference between working with vector-valued solutions versus the corresponding system of equations with scalar solutions is completely cosmetic \dash the underlying regularity structures are the same and the only difference is how one organises the space of modelled distributions. 
\end{remark}
\begin{remark}
While the convention \eqref{e:defV} is natural in our setting,
an example where it must be discarded is the setting of \cite{Mate2}.
In \cite{Mate2} combinatorial trees also index subspaces of the regularity structure which generically are not one-dimensional. 
To start translating \cite{Mate2} into our framework one would want to take
$\mfL_{+} = \{\mft_{+}\}$ and set $V_{\mft_{+}} = \mcb{B}$
for $\mcb{B}$ an appropriate space of distributions. 
However, since $\mcb{B}$ is infinite-dimensional in this case, the machinery we develop in the remainder of this section 
does not immediately extend to this context. See however \cite{AndrisKonstantin} for a trick allowing to 
circumvent this in some cases.
\end{remark}
At this point we make the following assumption. 
\begin{assumption}\label{assump:finite_dim} 
Our target space assignments $W$ and kernel space assignments $\mathcal{K}$ are always finite-dimensional space assignments (which means the corresponding $V$ given by \eqref{e:defV} are finite-dimensional). 	
\end{assumption} 
\begin{remark}
There are several ways in which we use Assumption~\ref{assump:finite_dim} in the rest of this section. 
One key fact is that for vector spaces $X$ and $Y$ one has $L(X,Y) \simeq X^{\ast} \otimes Y$ provided that either $X$ or $Y$ is finite-dimensional \dash this is especially important in the context of Remark~\ref{rem:identity_is_special}. 
Another convenience of working with finite-dimensional space assignments is that we are then allowed to assume the existence of a scalar noise decomposition which lets us leverage the machinery of \cite{Hairer14,CH16,BHZ19,BCCH21}. 
\end{remark}
\subsection{The renormalisation group}
\label{sec:renorm}
Our construction also provides us with a ``renormalisation group'' that remains fixed under noise decompositions.
Recalling the set of forests $\mfF_{-}(R)$ and the associated algebra $\CF_{-}$ introduced in Section~\ref{sec:concrete} (in particular Eq.~\eqref{e:defCTtau}), we note that the map $\Deltam_s$ introduced in Section~\ref{subsec: regstruct_gen_rules} \dash or rather
its image under the functor $\Func_V$ \dash turns $\CF_{-}$ into a bialgebra. 
Moreover, we can use the number of edges of each element in $\mfF_{-}(R)$ to grade $\CF_{-}$. 
Since $\CF_{-}$ is connected by \eqref{e:def-mfT-} (i.e.\ its subspace of degree $0$ is generated by the unit),
it admits an antipode $\mcA_{-}$ turning it into a commutative
Hopf algebra and we denote by $\mathcal{G}_{-}$ the associated group of characters. 
It is immediate that, up to natural isomorphisms, the Hopf algebra $\CF_{-}$
and character group $\mathcal{G}_{-}$ remain fixed under noise decompositions. 
Given $\ell \in \mathcal{G}_{-}$, we define a corresponding renormalisation operator $M_{\ell}$, which is a linear operator
\begin{equ}\label{eq:renorm_op}
M_\ell \colon \CT \to \CT\;,\qquad
M_{\ell} 
\eqdef 
(\ell \otimes \id_{\CT})\Deltam \;.
\end{equ}

\begin{remark}
Note that the action of $M_{\ell}$ would not in general be well-defined on the direct product $\Func_{V}(\mfT)$ but it is well-defined on $\CT$ thanks to the assumption of subcriticality.
\end{remark}
Also note that there is a canonical isomorphism 
\begin{equ}\label{eq:character_group_iso}
\CG_{-} \simeq \bigoplus_{\ttau \in\mfT_{-}(R)}
\CT[\ttau]^{\ast}\;.
\end{equ}
In particular, given $\ell \in \CG_{-}$ and $\ttau \in\mfT_{-}(R)$, we write $\ell[\ttau]$ for the component of $\ell$ in $\CT[\ttau]^{\ast}$ above. 
\subsubsection{Canonical lifts}
For the remainder of this section we impose the following assumption.
\begin{assumption}\label{assump_noises-terminal}
The rule $R$ satisfies $R(\mfl) = \{()\}$ for every $\mfl \in \Lab_{-}$.
\end{assumption}
A kernel assignment is a collection of kernels $K = (K_{\mft}: \mft \in \Lab_{+})$ where each $K_{\mft}$ is a smooth compactly supported function from $\R^{d+1} \setminus \{0\}$ to $\mathcal{K}_{\mft}$. 
A smooth noise assignment is a tuple $\zeta = (\zeta_{\mft}: \mft \in \Lab_{-})$ where each $\zeta_{\mft}$ is a smooth function from $\R^{d+1}$ to $W_{\mft}$. 

Note that the set of kernel (or smooth noise) assignments for $\Lab$ and $W$ can be identified with the set of kernel (or smooth noise) assignments for any $\bar{\Lab}$ and $\bar{W}$ obtained via noise decomposition of the label set $\Lab$ and $W$. 
For smooth noise assignments this identification is given by the correspondence
\[
\big( 
\zeta_{\mfl}: \mfl \in \mfL_{-} 
\big) 
\leftrightarrow 
\big(
\zeta_{\bar{\mfl}} = \PP_{\bar{\mfl}} \zeta_{\mfl}: 
\mfl \in \Lab_{-},\ \bar{\mfl} \in \proj(\mfl) 
\big)\;.
\]
If we are working with scalar noises then, upon fixing kernel and smooth noise assignments $K$ and $\zeta$, \cite{Hairer14} introduces a map $\PPi_{\can}$ which takes trees $\ttau \in \mfT(R)$ into $\mcC^{\infty}(\R^{d+1})$. This map gives a correspondence between combinatorial trees and the space-time functions/distributions they represent (without incorporating any negative or positive renormalisation),  and $\PPi_{\can}$ is extended linearly to $\CT$. 
In the general case with vector valued noise we can appeal to any scalar noise decomposition $\bar\Lab$ of $\Lab$ and $W$ to again obtain a linear map $\PPi_{\can}\colon\CT_{\bar\Lab} \rightarrow \mcC^{\infty}(\R^{d+1})$ \dash this map is of course  independent of the particular scalar noise decomposition we appealed to for its definition. 

In order to make combinatorial arguments which use the structure of the trees of our abstract regularity structure, it is convenient to have an explicit vectorial formula for $\PPi_{\can}$. 

Given $\ttau \in \mfT(R)$, we have, by Assumption~\ref{assump:finite_dim} and~\eqref{eq:dual_tensors_finite_dim}, 
\begin{equ}\label{eq:dual_of_tree_space}
(V^{\ast})^{\otimes \scal{\ttau}}
\simeq
 \CT[\ttau]^{\ast}\;.
 \end{equ}
Writing, for any $\ttau \in \mfT(R)$, $\PPi_{\can}[\ttau]$ for the restriction of $\PPi_{\can}$ to $\CT[\ttau]$, we will realise $\PPi_{\can}[\ttau]$ as an element 
\[
\PPi_{\can}[\ttau] \in 
\mcC^{\infty}(\R^{d+1},(V^{\ast})^{\otimes \scal{\ttau}})\;,
\] 
where we remind the reader that \eqref{eq:dual_of_tree_space} gives us 
\[
\mcC^{\infty}(\R^{d+1},(V^{\ast})^{\otimes \scal{\ttau}})
\simeq 
L(\CT[\ttau],\mcC^{\infty}(\R^{d+1}))\;.
\] 
The explicit vectorial formula for $\PPi_{\can}$ mentioned above is then given by
\begin{equs}\label{eq:bold_faced_pi}
\PPi_{\can}[\ttau](z)
=&
\int_{(\R^{d+1})^{N(\tau)}} 
\mrd x_{N(\tau)}
\delta(x_{\rho} - z)
\Big(
\prod_{v \in N(\tau)} x_{v}^{\mfn(v)}
\Big)\\
{}&
\enskip
\Big(
\bigotimes_{
e \in K(\tau)}
D^{\mfn(e)}K_{\mft(e)}(x_{e_{+}} - x_{e_{-}})
\Big)
\Big(
\bigotimes_{e \in L(\tau)}
D^{\mfn(e)}\zeta_{\mft(e)}(x_{e_{+}})
\Big)
\end{equs}
where we have taken an arbitrary $\tau\in\ttau$ and set
\begin{equs}\label{eq: kernels and nodes}
L(\tau) &\eqdef \{e \in  E_{\tau}: \mft(e) \in \Lab_{-}\}\;,\\
K(\tau) &\eqdef E_{\tau} \setminus L(\tau)
\textnormal{ and }
N(\tau) \eqdef \big\{
v \in V_{\tau}: v \not = e_{-} 
\textnormal{ for any }e \in L(\tau)
\big\}\;.
\end{equs}
Thanks to Assumption~\ref{assump_noises-terminal}, all the integration variables $x_{v} \in \R^{d+1}$ appearing on the right-hand side of \eqref{eq:bold_faced_pi} satisfy $v \in N(\tau)$. 
Moreover, while the right-hand side of \eqref{eq:bold_faced_pi} is written as an element of $\big( \bigotimes_{e \in L(\tau)} W_{\mft(e)}\big) \otimes \big( \bigotimes_{e \in K(\tau)} \mathcal{K}_{\mft(e)} \big)$, due to the symmetry of the integrand it can canonically be identified with an element of $(V^{\ast})^{\otimes \scal{\ttau}}$. 

%
\subsubsection{The BPHZ character} 
In the scalar noise setting, upon fixing a kernel assignment $K$ and a \emph{random}\footnote{With appropriate finite moment conditions.} smooth noise assignment $\zeta$,~\cite[Sec.~6.3]{BHZ19} introduces a multiplicative linear functional $\bar{\PPi}_{\can}$
(denoted by $g^-(\PPi_{\can})$ therein)
on $\CF_{-}$ obtained by setting, for $\ttau \in \mfT(R)$, $\bar{\PPi}_{\can}[\ttau] = \E(\PPi_{\can}[\ttau](0))$ and then extending multiplicatively and linearly to the algebra $\CF_{-}$. 
\cite{BHZ19} also introduces a corresponding BPHZ renormalisation character $\ell_{\BPHZ} \in \CG_{-}$ given by
\begin{equ}[e:defBPHZ]
\ell_{\BPHZ} 
=
\bar{\PPi}_{\can} \circ \tilde\mcA_-\;,
\end{equ}
where $\tilde\mcA_-$ is the \emph{negative twisted antipode}, an algebra homomorphism from $\CF_{-}$ to $\CF$
determined by enforcing the condition that   
\begin{equ}\label{eq:recursive_antipode}
\mathcal{M}(\tilde\mcA_- \otimes \Id) \Deltam = 0
\end{equ}
on the subspace of $\CT$ generated by $\mfT_{-}(R)$. Here, $\mathcal{M}$ denotes the (forest) multiplication map from $\CF \otimes \CT$ into $\CF$.
We remind the reader that condition \eqref{eq:recursive_antipode}, combined with multiplicativity of $\tilde\mcA_-$, gives a 
recursive method for computing $\tilde\mcA_- \ttau$ where the recursion is in $|E_{\tau}|$.

In the general vector-valued case we note that, analogously to \eqref{eq:recursive_antipode}, we can consider
for a morphism $\boldsymbol{\tilde\CA_-}\in\Hom(\SFm,\SF)$
the identity
\begin{equ}\label{eq:recursive_antipode_vec}
\mathcal{M}(\boldsymbol{\tilde\CA_-} \otimes \Id) \Deltam = 0
\end{equ}
as an identity in $\Hom(\STm,\SF)$. If we furthermore impose that $\boldsymbol{\tilde\CA_-}$ is multiplicative in the sense
that $\boldsymbol{\tilde\CA_-} \circ \CM = \CM \circ (\boldsymbol{\tilde\CA_-} \otimes \boldsymbol{\tilde\CA_-})$
as morphisms in $\Hom(\SFm\otimes\SFm,\SF)$
and
$\boldsymbol{\tilde\CA_-}[\emptyset] = \id_{\SF[\emptyset]}$,\footnote{Recall that $\emptyset$ denotes the empty forest and is the unit of the algebras $\CF$ and  $\CF_-$, while $\bone$ is the tree with a single vertex and $0$ is just the zero of a vector space, so these three notations are completely different.}
 then we can proceed again by induction on $|E_{\tau}|$ to uniquely determine
$\boldsymbol{\tilde\CA_-} \in \Hom(\SFm,\SF)$.

Analogously to~\eqref{eq:dual_of_tree_space}, one has $\CT[\ff]^{\ast} \simeq W^{\otimes \scal{L(\ff)}}$ where $\scal{L(\ff)}$ is the symmetric set obtained by restricting the forest symmetries of every $f\in\ff$ to the set  of leaves $L(f) = \bigcup_{\tau \in f} L(\tau)$,
where the union runs over all the trees $\tau$ in $f$. 
This shows that, if we set again
\begin{equ}
\bar{\PPi}_{\can}[\ttau] = \E\big(\PPi_{\can}[\ttau](0)\big)\;,
\end{equ}
with $\PPi_{\can}$ given in \eqref{eq:bold_faced_pi}, we can view $\bar{\PPi}_{\can}[\ttau]$ as an element of $\CT[\ttau]^{\ast}$,
and, extending its definition multiplicatively, as an element of $\CT[\ff]^{\ast}$.
Hence~\eqref{e:defBPHZ} yields again an element of $\CG_-$, provided that we set $\tilde \CA_- = \Func_V(\boldsymbol{\tilde\CA_-})$.
\begin{remark}
This construction is consistent with \cite{BHZ19} in the sense that if we consider $\ell_\BPHZ$ as in \cite{BHZ19} for 
any scalar noise decomposition $\proj$ of $\Lab$, then this agrees with the construction we just described, provided that 
the corresponding spaces are identified via the functor $\proj^*$. 
\end{remark}

\subsection{Non-linearities, coherence, and the  map \texorpdfstring{$\Upsilon$}{Upsilon}}\label{subsec:nonlinearities} 
In this subsection we will use type decompositions to show that the formula for the $\Upsilon$ map which appears in the description of the renormalisation of systems of scalar equations in \cite{BCCH21} has an analogue in our setting of vector-valued regularity structures. 
We again fix a triple $(\Lab,W,\mathcal{K})$ which determines a space assignment $(V_\mft)_{\mft\in\Lab}$ by \eqref{e:defV}.
\begin{remark}
Up to now we were consistently working with isomorphism classes $\ttau\in\mfT$.
For brevity, we will henceforth work with concrete trees $\tau\in\ttau$,
all considerations for which will depend
only on the symmetric set $\scal{\ttau}$, which, by Remark~\ref{rem:canonical},
we canonically identify with $\scal{\tau}$.
We will correspondingly abuse notation and write $\tau\in\mfT$.
\end{remark}
\begin{remark}\label{remark:labels-edges}
In what follows we will often identify $\Lab$ with a subset of $\CE=\Lab \times \N^{d+1}$ by associating $\mft \mapsto (\mft,0)$.
\end{remark}
We define $\mcb{A} \eqdef \prod_{o\in\CE} W_o$ \label{def:mcbA}
where the $\big( W_{(\mft,k)}: k \in \N^{d+1} \big)$ are distinct copies of the space $W_\mft$. 
One should think of $\mathbf{A} \in \mcb{A}$ as describing the jet of both the noise
and the solution to a system of PDEs of the form~\eqref{e:SPDE}.
We equip $\mcb{A}$ with the product topology.

Given any two topological vector spaces $U$ and $B$, we write $\smooth(U,B)$
for the space of all maps $F \colon U \to B$ with the property that, for every
element $\ell \in B^*$, there exists a continuous linear map $\bar\ell\colon U\to \R^n$
and a smooth function $F_{\ell,\bar\ell} \colon \R^n \to \R$ such that, for every $u \in U$, 
\begin{equ}\label{def:smooth_funcs}
\scal{\ell, F(u)} = F_{\ell,\bar\ell} \bigl(\bar \ell(u)
\bigr)\;.
\end{equ}
When our domain is $U=\mcb{A}$ we often just write $\smooth(B)$\label{PB page ref} instead of $\smooth(\mcb{A},B)$.
Note that when $B$ is finite-dimensional then for each $F \in \smooth(B)$, $F(\mathbf{A})$ is a smooth function of $(A_{o}: o \in \CE_{F})$ for some \emph{finite} subset $\CE_{F} \subset \CE$.
%
%
\begin{remark}\label{rem:noises_are_upgraded}
One difference in the point of view of the present article versus that of \cite{BCCH21} is that here we will 
treat the solution and noise on a more equal footing. 
As an example, in \cite{BCCH21} the domain of our smooth functions would be a direct product indexed by 
$\Lab_{+} \times \N^{d+1}$ rather than one indexed by $\CE$. The fact that, in the case  of the stochastic 
Yang--Mills equations considered here, the dependence on the noise variables has to be affine is enforced
when assuming that the nonlinearity obeys our rule $R$, see Definition~\ref{def:conforming_nonlin} below.

In particular, when defining the $\Upsilon$ map in \cite{BCCH21} through an induction on trees $\tau$, the symbols associated to the noises (and derivatives and products thereof)\footnote{Here we are referring to the ``drivers'' of \cite{BCCH21}.} were treated as ``generators''\dash the base case of the $\Upsilon$ induction.
In our setting, however, the sole such generator will be the symbol $\bone$ and noises will be treated as branches \slash edges $\mcb{I}_{\mft}(\bone)$ for $\mft \in \Lab_{-}$. See also Remark~\ref{rem:noise is I-1} below.
\end{remark}
A specification of the right-hand side of our equation determines an element in
\begin{equ}[e:Qcirc]
\mathring{\mcb{Q}}
\eqdef  
\smooth\Big( \bigoplus_{\mft \in \Lab}\big(V_\mft \otimes W_{\mft}\big) \Big)
\simeq
\bigoplus_{\mft \in \Lab} 
\smooth\big(V_\mft \otimes W_{\mft} \big)\;.
\end{equ}
Writing an element $F \in \mathring{\mcb{Q}}$ as $F = \bigoplus_{\mft \in \Lab} F_{\mft}$ with $F_{\mft} \in \smooth(V_\mft \otimes W_{\mft})$, $F_{\mft}$ for $\mft \in \Lab_+$ plays the role of the function appearing on the right-hand 
side of~\eqref{e:KSPDE} namely a smooth function in the 
variable\footnote{The $(\mathbf{A},\bxi)$ written in \eqref{e:KSPDE} corresponds to the $\mathbf{A}$ here, see 
Remark~\ref{rem:noises_are_upgraded}.} $\mathbf{A} \in \mcb{A}$ taking values in $V_\mft \otimes W_{\mft} \simeq W_\mft$.
  
%
\begin{remark}\label{rem:identity_is_special}
Note that the vector space $\mcb{A}$ depends on $(\Lab,W,\mathcal{K})$.
However, if $(\bar{\Lab},\bar{W},\bar{\mathcal{K}})$ is a decomposition of $(\Lab,W,\mathcal{K})$ under $\proj$ and we define $\bar{\mcb{A}} \eqdef \prod_{\bar{o} \in\CE} \bar{W}_{\bar{o}}$ then there is a  linear natural surjection $\iota: \bar{\mcb{A}} \twoheadrightarrow \mcb{A}$
given by setting, for each $\bar{\mathbf{A}} = (\bar{A}_{\bar{o}})_{\bar{o} \in \bar{\CE}}$, $\iota(\bar{\mathbf{A}}) = \mathbf{A} = (A_{o})_{o \in \CE}$ to be given by
\begin{equ}\label{eq:jets_under_decomp}
A_{o} = \begin{cases}
\displaystyle\sum_{ \bar{o} \in \proj(o)} \bar{A}_{\bar{o}} & \text{ if }o \in \Lab_{+} \times \N^{d+1}\;,\\ \vspace{.2cm}
\displaystyle\bigoplus_{ \bar{o} \in \proj(o)} \bar{A}_{\bar{o}} & \text{ if }o \in \Lab_{-} \times \N^{d+1}\;,\\
 \end{cases}
\end{equ}
where $\proj(o)$ is understood as in Remark~\ref{rem:proj extend to E}.

Observe that there is an induced injection 
\[
 \smooth(B) = \smooth(\mcb{A},B) \hookrightarrow \smooth(\bar{\mcb{A}},B)\;,
 \]
 given by $G \mapsto G \circ \iota$ \dash we will often treat this as an inclusion and identify $\smooth(B)$ as a subset of $\smooth(\bar{\mcb{A}},B)$

It follows that, for any $\mft \in \Lab_{+}$ and $F_{\mft} \in \smooth(V_{\mft} \otimes W_{\mft})$ one has
\[
F_{\mft} = \bigoplus_{\bar{\mft} \in \proj(\mft)} F_{\bar{\mft}}
\in 
\bigoplus_{\bar{\mft} \in \proj(\mft)}
\smooth(V_{\bar{\mft}} \otimes W_{\bar{\mft}})
\subset
\bigoplus_{\bar{\mft} \in \proj(\mft)}
\smooth(\bar{\mcb{A}}, V_{\bar{\mft}} \otimes W_{\bar{\mft}})\;.
\]
In the first equality above we are using that 
\begin{equ}\label{eq:expansion_of_nonlinearity}
V_{\mft} \otimes W_{\mft} = \bigoplus_{\bar{\mft} \in \proj(\mft)} V_{\bar{\mft}} \otimes W_{\mft}
=
\bigoplus_{\bar{\mft} \in \proj(\mft)} V_{\bar{\mft}} \otimes W_{\bar{\mft}}\;.
\end{equ}
This shows that our definitions allow the $\Lab_{+}$-indexed non-linearities to make sense as $\bar{\Lab}_{+}$-indexed non-linearities after decomposition.
However, this does not hold for $\mfl \in \Lab_{-}$ since the \eqref{eq:expansion_of_nonlinearity} fails in this case. 
 
Indeed, given a finite-dimensional vector space $B$, the identity $\id_{B}$ is the unique (up to multiplication by a scalar) 
element of $L(B,B) \simeq B^{\ast} \otimes B$ such that, for every decomposition $B = \bigoplus_{i} B_{i}$ one has $\id_{B} \in \bigoplus_{i} L(B_{i},B_{i}) \simeq \bigoplus_{i} \big(B_{i}^{\ast} \otimes B_{i}\big)$.
This suggests that if we want to have nonlinearities that continue to make sense under decomposition (and constant
for $\mfl \in \Lab_-$ as enforced by Assumption~\ref{assump_noises-terminal}), we should
set $F_\mfl = \id_{W_\mfl}$ for $\mfl \in \Lab_-$. 

This is indeed the case and will be enforced in
Definition~\ref{def:CQ} below.
\end{remark} 
\subsubsection{Derivatives}
\label{sec:Monomials}
Just as in \cite{BCCH21} we introduce two families of differentiation operators, the first $\{D_{o}\}_{o \in \CE}$ corresponding to derivatives with respect to the components of the jet $\mcb{A}$ and the second $\{\partial_{j}\}_{j=0}^{d}$ corresponding to derivatives in the underlying space-time. 

Consider locally convex topological vector spaces $U$ and $B$.
Suppose that $B=\prod_{i\in I}B_i$, where each $B_i$ is finite-dimensional,
equipped with the product topology.
Let $F\in\smooth(U,B)$
and $\ell\in B^*$, $\bar\ell$, and $F_{\ell,\bar\ell}$ as in~\eqref{def:smooth_funcs}.
For $m \geq 0$ and $u\in U$,
consider the symmetric $m$-linear map
\begin{equ}\label{eq:dual_deriv}
U^m \ni(v_1,\ldots, v_m) \mapsto 
D^mF_{\ell,\bar\ell}(\bar\ell(u))(\bar\ell(v_1),\ldots, \bar\ell(v_m)) \in \R\;.
\end{equ}
For fixed $u,v_1,\ldots,v_m$,
the right-hand side of~\eqref{eq:dual_deriv}
defines a linear function of $\ell$
which one can verify is independent of the choice of $\bar\ell$.
Since $B^*=\bigoplus_{i\in I}B_i^*$, the algebraic dual of which is again $B$,
there exists an element $D^mF(u)(v_1,\ldots,v_m)\in B$
such that $\scal{\ell,D^mF(u)(v_1,\ldots,v_m)}$ agrees with the right-hand side of~\eqref{eq:dual_deriv}.
It is immediate that $U^m \ni(v_1,\ldots, v_m) \mapsto D^mF(u)(v_1,\ldots,v_m)\in B$ is symmetric and $m$-linear for every $u\in U$.

Turning to the case $U=\mcb{A}$,
for $o_1,\ldots,o_m\in\CE$ and $\mbA\in\mcb{A}$,
we define
\begin{equ}
D_{o_{1}} \cdots D_{o_{m}}F(\mbA) = D^m F(\mbA)\restr_{W_{o_1}\times\ldots\times W_{o_m}} \in L(W_{o_1},\ldots,W_{o_m}; B)\;.
\end{equ}
Due to the finite-dimensionality of $W_{o_i}$ by Assumption~\ref{assump:finite_dim}, the map \begin{equ}
D_{o_{1}} \cdots D_{o_{m}} F\colon \mbA\mapsto D_{o_{1}} \cdots D_{o_{m}} F(\mbA)
\end{equ}
is an element of $\smooth(L(W_{o_1},\ldots,W_{o_m}; B))$.
The operators $\{D_{o}\}_{o \in \CE}$ naturally commute, modulo reordering the corresponding factors.
\begin{remark}\label{rem:decom_of_deriv}
If $(\bar{\Lab},\bar{W},\bar{\mathcal{K}})$ is a decomposition of $(\Lab,W,\mathcal{K})$ under $\proj$ then our definitions give us another set of derivative operators $\{D_{\bar{o}}\}_{\bar{o} \in \bar{\CE}}$. 
Via the identifications $W_{o} = \bigoplus_{\bar{o} \in \proj(o)} W_{\bar{o}}$ for any $o \in \CE$, we have for any $F \in \smooth(B)$
\begin{equ}[e:decompDF]
D_{o}F =
\begin{cases}
\displaystyle\sum_{\bar{o} \in \proj(o)} 
D_{\bar{o}}F
& \text{ if }o \in \Lab_{+} \times \N^{d+1}\;,\\ \vspace{.2cm}
\displaystyle\bigoplus_{ \bar{o} \in \proj(o)} D_{\bar{o}}F 
& \text{ if }o \in \Lab_{-} \times \N^{d+1}\;.
\end{cases}
\end{equ}
For $o \in \Lab_{+} \times \N^{d+1}$ this equality follows from a simple computation using the chain rule to handle the composition with $\iota$ and the right hand side is a sum of elements of $\smooth(L(W_{o};B))$ \dash recall that $W_{\bar{o}} = W_{o}$ for all $\bar{o} \in \proj(o)$.
For $o \in \Lab_{-} \times \N^{d+1}$ the equality follows from the fact that $W_{o} = \sum_{\bar{o} \in \proj(o)} \bar{W}_{\bar{o}}$ so $\smooth(L(W_{o};B)) =\bigoplus_{\bar{o} \in \proj(o)} \smooth(L(W_{\bar{o}};B))$.  

Analogous identities involving iterated sums hold for iterated derivatives.
\end{remark}
\begin{remark}
Note that we will not always distinguish direct sums from regular sums in what follows \dash we chose to do so in Remarks~\ref{rem:identity_is_special} and \ref{rem:decom_of_deriv} simply to illustrate the slight difference in how types in $\Lab_{+}$ and $\Lab_{-}$ behave under decompositions.  
\end{remark}
For $j\in\{0,1,\ldots, d\}$ and $o = (\mft,p) \in \CE$, we first define $\d_j o \in \CE$ 
by $\d_j o = (\mft,p + e_{j})$.  
We then define operators $\d_j$ on $\smooth(B)$ by setting, for $\mathbf{A} = (\mathbf{A}_{o})_{o \in \CE} \in \mcb{A}$
and $F \in \smooth(B)$, 
\begin{equ}[e:def-partial]
\bigl(\d_j F\bigr)(\mathbf{A}) \eqdef \sum_{o \in \CE} \big(D_o F\big)(\mathbf{A})\,\mathbf{A}_{\d_j o}\;.
\end{equ}
Note that the operators $\{\partial_{j}\}_{j=0}^{d}$ commute amongst themselves and so $\partial^p$ is well-defined for any $p\in\N^{d+1}$.
\begin{remark}
Combining \eqref{e:decompDF} with \eqref{e:def-partial}, we see that the action of the derivatives $\{\partial_{j}\}_{j=0}^{d}$ 
remains unchanged under decompositions of $(\Lab,W,\mathcal{K})$, thus justifying our notation.   
\end{remark}
Just as in \cite{BCCH21}, we want to restrict ourselves to $F \in \mathring{\mcb{Q}}$ that \emph{obey}\footnote{The notion of obey (and the set $\mcb{Q}$) we choose here is analogous to item~(ii) of \cite[Prop.~3.13]{BCCH21} rather than \cite[Def.~3.10]{BCCH21}. In particular, our definition is not based on expanding a nonlinearity in terms of a polynomial in the rough components of $\mathbf{A}$ and a smooth function in the regular components of $\mathbf{A}$. This means we don't need to impose \cite[Assump.~3.12]{BCCH21} to use the main results of \cite{BCCH21}.}  the rule $R$ we use to construct our regularity structure.  
\begin{definition}\label{def:conforming_nonlin}
We say $F \in \mathring{\mcb{Q}}$ obeys a rule $R$ if, for each $\mft \in \Lab$ and $o_{1},\dots,o_{n} \in \CE$, 
\begin{equ}[e:conformRule]
(o_{1},\cdots,o_{n}) \not \in R(\mft) \quad\Rightarrow\quad 
D_{o_{1}} \cdots D_{o_{n}}F_{\mft} = 0\;.
\end{equ}
Note that, for any type decomposition $\bar{\Lab}$ of $\Lab$ under $\proj$, $F$ obeys a rule $R$ if and only if it obeys $\bar{R}$ as defined in Remark~\ref{rem:proj extend to E}.
\end{definition}
\begin{definition}\label{def:CQ}
Given a subcritical and complete rule $R$, define $\mcb{Q} \subset \mathring{\mcb{Q}}$ to be the set of 
$F$ obeying $R$ such that furthermore $F_\mfl(\mathbf{A}) = \id_{W_\mfl}$ for all $\mfl \in \Lab_-$.
\end{definition}
Recall that, by Assumption~\ref{assump_noises-terminal}, for any $F$ obeying the rule $R$,
$F_\mfl(\mathbf{A})$ is independent of $\mathbf{A}$ for $\mfl \in \Lab_-$. The reason for imposing the specific
choice $F_\mfl(\mathbf{A}) = \id_{W_\mfl}$ is further discussed in Remark~\ref{rem:identity_is_special} above 
and Remark~\ref{rem:GHM20} below.
%
%
%
%
\subsubsection{Coherence and the definition of $\Upsilon$} 
\label{sec:Coherent expansions}
In this subsection we formulate the notion of coherence from \cite[Sec.~3]{BCCH21} in the setting of vector regularity structures. 
In particular, in Theorem~\ref{thm:upsilon_summation}, we show that the coherence constraint is preserved under noise decompositions.  

We first introduce some useful notation.  
For $o = (\mft,p) \in \CE$
we set
\begin{equ}
\mcb{B}
\eqdef
\Func_V\big(\scal{\mfT}\big)
=\prod_{\tau\in\mfT}V^{\otimes \scal{\tau}} \;,\qquad
\mcb{B}_{o}
\eqdef
\Func_V\big(\scal{\mcb{I}_o\mfT}\big) \subset \mcb{B}\;,
\end{equ}
and equip $\mcb{B}$ with the product topology.
As usual, we use the notation $\mcb{B}_\mft = \mcb{B}_{(\mft,0)}$. 
Note that $\mcb{B}$ is an algebra when equipped with the tree product, or rather its image
under $\Func_V$.
The following remark, where we should have in mind the case $\mcb{B}_+
= \Func_V\big(\scal{\mfT \setminus \{\bone\}}\big) $, is crucial 
for the formulation of our construction.

\begin{remark}\label{rem:tensorPoly}
Let $\mcb{B}_+$ be an algebra such that $\mcb{B}_+ = \varprojlim_{n} \mcb{B}_+^{(n)}$ with each $\mcb{B}_+^{(n)}$ 
nilpotent, let $U$ and $B$ be locally convex spaces where $B$ is of the same form as in Section~\ref{sec:Monomials}, and
let $F\in \smooth (U, B)$. 
Write $\mcb{B} = \R\oplus \mcb{B}_+$, which is then a unital algebra. Then $F$ can be extended  to a 
map $\mcb{B} \otimes U \to \mcb{B} \otimes B$ as follows: for $u \in U$ and $\tilde u \in \mcb{B}_+ \otimes U$, we set
\begin{equ}[e:defFextension]
F(1 \otimes u + \tilde u) \eqdef \sum_{m \in \N} \frac{D^m F(u)}{m!}(\tilde u,\ldots, \tilde u)\;, 
\end{equ}
where $D^m F(u)\in L(U,\ldots, U; B)$ for each $u\in U$ 
naturally extends to a $m$-linear map $(\mcb{B} \otimes U)^m \to \mcb{B} \otimes B$ by imposing that 
\begin{equ}[e:defPextension]
D^m F(u)(b_1\otimes v_1,\ldots,b_m\otimes v_m) 
= (b_1\cdots b_m)\otimes  D^m F(u) (v_1,\ldots,v_m)\;.
\end{equ}
Note that the first term of this series~\eqref{e:defFextension} belongs to $\R \otimes B$
while all other terms belong to  $\mcb{B}_+ \otimes B$.
Since $\tilde u \in \mcb{B}_+ \otimes U$, the projection of this series onto any of 
the spaces $\mcb{B}_+^{(n)} \otimes B$ contains
only finitely many non-zero  terms by nilpotency, so that it is guaranteed to converge.
If $B$ and all the $\mcb{B}_+^{(n)}$ are finite-dimensional,
then the extension of $F$ defined in \eqref{e:defFextension} actually belongs to 
$\smooth(\mcb{B} \otimes U, \mcb{B} \otimes B)$,
where $\mcb{B}_+$ is equipped with the projective limit topology, under which it is nuclear, and $\mcb{B} \otimes U$, is equipped with the
projective tensor product.
Furthermore, in this case,
every element of $\smooth(\mcb{B} \otimes U, \mcb{B} \otimes B)$
extends to an element of $\smooth(\mcb{B}\hotimes U,\mcb{B}\otimes B)$,
where $\hotimes$ denotes the (completion of the) projective tensor product.
\end{remark}
We introduce a space\footnote{The space $\expan$ introduced here plays the role of the space $\mcb{H}^{\ex}$ in \cite[Sec.~3.7]{BCCH21}.} of expansions
$\expan = \bigoplus_{\mft \in \Lab} \expan_\mft$
with
\begin{equ}
\expan_\mft \eqdef   \big(
\mcb{B}_\mft\oplus \bar\CT
\big)\otimes W_\mft \subset \mcb{B} \otimes W_\mft\;,\qquad \bar \CT \eqdef \prod_{k \in \N^{d+1}}
\CT[\mbX^{k}]\;.
\end{equ}
Given $\mcA \in \expan$, we write $\mcA = \sum_{\mft \in \Lab} \mcA_{\mft}$ with
\begin{equ}\label{eq:coherent_jet}
\mcA_{\mft} = \mcA^{R}_{\mft} +
\Big(
\sum_{k \in \N^{d+1}}
\frac{\mathbf{X}^{k}}{k!}
\otimes
\mathbf{A}_{(\mft,k)}
\Big)
\in \expan_\mft\;,\qquad \mcA^{R}_{\mft} \in \mcb{B}_\mft\otimes W_\mft\;,
\end{equ}
and write $\mathbf{A}^{\mcA} = (\mathbf{A}_{o})_{o \in \CE} \in \mcb{A}$, where the coefficients $\mathbf{A}_{(\mft,k)}$ are as in \eqref{eq:coherent_jet}.\footnote{As a component of $\mathbf{A}^{\mcA} \in \mcb{A}$, $\mathbf{A}_{o}\in W_o$, while as a term of \eqref{eq:coherent_jet}, $\mathbf{A}_{o}\in W_\mft$. This is of course not a problem since $ W_o \simeq W_\mft$.}
Note that~\eqref{eq:coherent_jet} gives a natural inclusion $\mcb{A} \subset \expan$.
\begin{remark}
In \eqref{eq:coherent_jet} and several places to follow, we write $\sum$ to denote an element of a direct product.
This will simplify several expressions below, e.g.~\eqref{eq:upsilon_identity}.
\end{remark}
For $o = (\mft,p) \in \CE$ we also define 
\begin{equ}[e:def-cA_o]
\mcA_{o}
\eqdef
\mcA^{R}_{o} + 
\Big(
\sum_{k \in \N^{d+1}}
\frac{\mathbf{X}^{k}}{k!}
\otimes
\mathbf{A}_{(\mft,p+k)}
\Big) \in \mcb{B} \otimes W_o
\;,
\end{equ}
where $\mcA^{R}_{o} \in \mcb{B}_o \otimes W_o$
 is given by the image of $\mcA^{R}_{\mft}$ under the canonical 
 isomorphism  $\mcb{B}_o \otimes W_o \simeq \mcb{B}_\mft\otimes W_\mft$ 
 (but note that $\mcb{B}_o$ and $\mcb{B}_\mft$ are different subspaces of $\mcb{B}$
 when $p \neq 0$). 
 Collecting the $\mcA_{o}$ into one element $\hat{\mcA} \eqdef (\mcA_{o}: o \in \CE)$,
 we see that $\hat{\mcA}$ is naturally viewed as an element of $\mcb{B} \hotimes \mcb{A}$.

Note also that, for any $o =(\mft,p)\in \CE$, our construction of the morphism $\mcb{I}_{o}$ on symmetric sets built from trees in Section~\ref{sec:Products and coproducts} gives us, via the functor $\Func_{V}$, an isomorphism 
$\mcb{I}_{o} \colon \mcb{B} \otimes V_{\mft}
\to
\mcb{B}_o$.

We fix for the rest of this subsection a choice of $F \in \mcb{Q}$.  
Then the statement that $\mcA \in \expan$ algebraically solves \eqref{e:SPDE}  corresponds to
\begin{equ}\label{eq:fixedptpblm}
\mcA^{R}_{\mft} = (\mcb{I}_{\mft} \otimes \id_{W_\mft})\big( F_{\mft}(\hat\mcA) \big)\;.
\end{equ}
Here, we used Remark~\ref{rem:tensorPoly} to view $F_\mft \colon \mcb{A} \to V_\mft \otimes W_\mft$ as a map
from $\mcb{B} \hotimes \mcb{A}$ into $\mcb{B} \otimes V_\mft \otimes W_\mft$, which 
$\mcb{I}_{\mft}$ then maps into $\mcb{B}_\mft \otimes W_\mft$.

The coherence condition then encodes the constraint \eqref{eq:fixedptpblm} as a functional dependence 
of $\mcA^{R} = \bigoplus_{\mft \in \Lab} \mcA^{R}_{\mft}$  on $\mathbf{A}^{\mcA}$. 
This functional dependence will be formulated by defining a pair of (essentially equivalent) maps $\bUpsilon$ and $\bar{\bUpsilon}$ where 
\begin{equ}[e:Ups-Ups-intro]
\bar{\bUpsilon} \in 
\prod_{\mft \in \Lab}
\smooth(\mcb{B} \otimes V_\mft \otimes W_\mft)
\quad\textnormal{ and }\quad
\bUpsilon
\in 
\prod_{o \in \CE}
\smooth(\mcb{B}_{o} \otimes W_o)\;.
\end{equ}
The coherence condition on $\mcA \in \expan$ will be formulated as $\mcA^{R} = \bUpsilon(\mathbf{A}^{\mcA})$.
%

To define $\bar{\bUpsilon}$ and $\bUpsilon$, we first define corresponding maps $\bar{\Upsilon}$ 
and 
$\Upsilon$ (belonging to the same respective spaces) from which $\bar{\bUpsilon}$ and $\bUpsilon$ will be obtained by including some combinatorial factors (see \eqref{eq:Upsilon_with_sym}).
We will write, for $\mft \in \Lab$, $o \in \CE$, and  $\tau \in \mfT$, $\Upsilon_{o}[\tau]$ for the component of $\Upsilon$ in $\smooth(\mcb{B}[\mcb{I}_o\tau] \otimes W_o)$ and $\bar{\Upsilon}_{\mft}[\tau]$ for its component in $\smooth(\mcb{B}[\tau] \otimes V_\mft \otimes W_\mft)$.\footnote{This means that our notation for $\Upsilon[\tau]$ breaks the notational convention we've used so far for other
elements of spaces of this type (direct products of $\Func_V(\scal{\tau})$, possibly tensorised with  some fixed space). 
The reason we do this is to be compatible with the notations of~\cite{BCCH21}, and also to keep notations
in Sections~\ref{sec:BPHZ-YM2} and \ref{sec:renorm_for_system} cleaner.}
We will define $\Upsilon$ and $\bar{\Upsilon}$ by specifying the components $\Upsilon_{o}[\tau]$ and $\bar{\Upsilon}_{\mft}[\tau]$ through an induction in $\tau$.
Before describing this induction, we make another remark about notation.
\begin{remark} \label{rem:grafting_with_symmetry}
For $\mft\in\Lab$,
$G \in \smooth(V_{\mft} \otimes W_{\mft})$, $(o_{1},\tau_{1}),\dots,(o_{m},\tau_{m}) \in \CE \times \mfT$, and $k \in \N^{d+1}$, consider the map
\begin{equ}
\partial^{k}D_{o_{1}} \cdots D_{o_{m}}
G \in \smooth(L(W_{o_{1}},\dots,W_{o_{m}};V_{\mft} \otimes W_{\mft}))\;.
\end{equ}
It follows from Remark~\ref{rem:tensorPoly} that if we are given elements $\Theta_i \in \smooth(\mcb{B} \otimes W_{o_i})$  we 
have a canonical interpretation for
\begin{equ}[e:interpret]
\big( \partial^{k} D_{o_{1}} \cdots D_{o_{m}}
G(\mathbf{A})\big)(\Theta_{1}(\mathbf{A}),\ldots,\Theta_{m}(\mathbf{A})) \in \mcb{B}\otimes V_{\mft} \otimes W_{\mft}\;,
\end{equ}
which, as a function of $\mbA$, is an element of 
$\smooth(\mcb{B} \otimes V_{\mft} \otimes W_{\mft})$
which we denote by $\big( \partial^{k} D_{o_{1}} \cdots D_{o_{m}}G\big)(\Theta_{1},\dots,\Theta_{m})$.

We further note that if $\Theta_i \in \smooth(\mcb{B}_i \otimes W_{o_i})$ for some subspaces $\mcb{B}_i \subset \mcb{B}$,
then \eqref{e:interpret} belongs to $\hat{\mcb{B}}\otimes V_{\mft} \otimes W_{\mft}$, where $\hat{\mcb{B}} \subset \mcb{B}$
is the smallest closed linear space containing all products of the form $b_1\cdots b_m$ with $b_i \in \mcb{B}_i$.
\end{remark}
Now consider an isomorphism class of trees $\tau \in \mfT$.
Then $\tau$ can be written as
\begin{equ}\label{eq:general_tree_added}
\mbX^k \prod_{i=1}^m \mcb{I}_{o_i}(\tau_{i})\;,
\end{equ}
where $k = \mfn(\rho)$,
$m \geq 0$, $\tau_{i} \in \mfT$, and $o_i \in \CE$.
\begin{remark}\label{rem:noise is I-1}
Following up on Remark~\ref{rem:noises_are_upgraded}, in the analogous expression \cite[Eq.~(2.11)]{BCCH21} a tree $\tau$ 
may also contain a factor $\Xi$ representing a noise. 
However, in \eqref{eq:general_tree_added} a noise (or a derivative of a noise) is represented by 
$\mcb{I}_{(\mfl,p)}({\bf 1})$ with $ \mfl \in \mfL_-$.
\end{remark}
Given $\mft \in \Lab$ and $\tau$ of the form~\eqref{eq:general_tree_added}, 
$\bar{\Upsilon}_{\mft}[\tau]$ and 
$\Upsilon_{(\mft,p)}[\tau]$ are inductively defined   by first setting
\begin{equ}[eq:base_case_for_upsilon]
\bar{\Upsilon}_{\mft}[\bone] \eqdef
\bone \otimes F_{\mft}\;,
\end{equ}
which belongs to 
$\mcb{B}[\bone] \otimes  \smooth  (V_{\mft} \otimes W_\mft) 
\simeq  \smooth(\mcb{B}[\bone] \otimes V_{\mft} \otimes W_\mft)\subset \smooth(\mcb{B} \otimes V_{\mft} \otimes W_\mft)$
(the first isomorphism follows from the fact that $V_{\mft} \otimes W_\mft$ is finite dimensional by assumption) 
so this is indeed of the desired type. We then set
\begin{equs}[e:def-Upsilon-added]
\bar{\Upsilon}_{\mft} [\tau] 
&\eqdef \mbX^k 
\Big[ 
\partial^k D_{o_1}\cdots D_{o_m}
\bar{\Upsilon}_{\mft}[\bone]
\Big]
\big(\Upsilon_{o_1}[\tau_1],\ldots,\Upsilon_{o_m}[\tau_m]\big)
 \;,\\
\Upsilon_{(\mft,p)}[\tau]
&\eqdef
\big(\mcb{I}_{(\mft,p)} \otimes \id_{W_\mft}\big)(\bar{\Upsilon}_{\mft} [\tau] )\;.
\end{equs}
We explain some of the notation and conventions used in \eqref{e:def-Upsilon-added}.
By Remark~\ref{rem:grafting_with_symmetry}, the term following $\mbX^k$ in the right-hand side for
$\bar{\Upsilon}_{\mft}[\tau]$ is an element of
\begin{equ}
\smooth \Big(\mcb{B}\Big[\prod_{j=1}^{m}\mcb{I}_{o_{j}}(\tau_j) \Big] \otimes V_{\mft} \otimes W_{\mft}\Big)\;.
\end{equ}
We then interpret $\mbX^{k} \bullet$ as the canonical isomorphism 
\begin{equ}[e:Xk-multiplication]
\mcb{B}\Big[\prod_{j=1}^{m}\mcb{I}_{o_{j}}(\tau_j) \Big] \simeq
\mcb{B}\Big[\mbX^{k}\prod_{j=1}^{m}\mcb{I}_{o_{j}}(\tau_j) \Big]
= \mcb{B}[\tau]\;,
\end{equ} 
acting on the first factor  of the tensor product,
hence the right-hand side of the definition of $\bar\Upsilon_{\mft}[\tau]$
belongs to 
$\smooth(\mcb{B}[\tau] \otimes V_\mft \otimes W_{\mft})$, which is mapped to $\smooth(\mcb{B}[\mcb{I}_{(\mft,p)}\tau] \otimes W_{\mft})$
by $\mcb{I}_{(\mft,p)} \otimes \id_{W_\mft}$ as desired. 

\begin{remark}\label{rem:noise_edges_terminal}
We have two important consequences of \eqref{e:def-Upsilon-added} and \eqref{eq:base_case_for_upsilon}: 
\begin{enumerate}[label=(\roman*)]
\item\label{pt:conform} Since $F$ obeys $R$, we have, for any $o = (\mft,p) \in \CE$, $\Upsilon_{o}[\tau]=0$ and $\bar{\Upsilon}_{\mft}[\tau] = 0$ unless $\mcb{I}_{\mft}(\tau) \in \mfT(R)$. 
\item\label{pt:leaves} For any $\mft \in \mfL_{-}$, $p\in \N^{d+1}$ and $\tau \in \mfT \setminus \{\bone\}$, one has $\Upsilon_{(\mft,p)}[\tau]=0$ and $\bar{\Upsilon}_{\mft}[\tau] = 0$,
due to annihilation by the operators $\partial$ and $D$.
\end{enumerate}
In particular, our assumption that the rule $R$ is subcritical guarantees that for any given degree
$\gamma$, only finitely many of the components $\bar{\Upsilon}_{\mft}[\tau]$ with $\deg\tau < \gamma$
are non-vanishing.
\end{remark}
\begin{remark}\label{rem:point_of_view_upsilon}
Although it plays exactly the same role, the map $\Upsilon$ introduced in \cite{BCCH21} is of a slightly different type 
than the maps $\Upsilon$ and $\bar\Upsilon$ introduced here. 
More precisely, in \cite{BCCH21}, $\Upsilon_{o}[\tau](\mathbf{A}) \in \R$ played the role of a coefficient\footnote{The coefficient of $\mcb{I}_{o}(\tau)$ in a coherent jet and the coefficient of $\tau$ in the expansion of the non-linearities evaluated on a coherent jet.} of a basis vector in the regularity structure.  
In the present article on the other hand, $\Upsilon_{o}[\tau](\mathbf{A}) \in \mcb{B}[\mcb{I}_o\tau] \otimes W_o$.
In the setting of~\cite{BCCH21}, these spaces are canonically isomorphic to $\R$ and our definitions are consistent modulo this isomorphism.
\end{remark}
The $\bar{\Upsilon}$ and $\Upsilon$ defined in \eqref{e:def-Upsilon-added} are missing the combinatorial symmetry factors $S(\tau)$ associated to a tree $\tau \in \mfT$,
which we define in the same way as in~\cite{BCCH21}.
For this we represent $
\tau$ more explicitly than  \eqref{eq:general_tree_added}
by writing
\begin{equ}\label{eq:general_tree_withsym} 
\tau = \mbX^{k} \prod_{j=1}^{\ell}\mcb{I}_{o_{j}}(\tau_{j})^{\beta_{j}}\;,
\end{equ} 
with $\ell \ge 0 $, $\beta_{j} > 0$, and \emph{distinct} $(o_{1},\tau_{1}),\dots,(o_{\ell},\tau_{\ell}) \in \CE \times \mfT$,
and define
\begin{equ}\label{eq:sym_factor}
S(\tau) \eqdef
k! \Big( \prod_{j=1}^{\ell}S(\tau_{j})^{\beta_{j}} \beta_{j}! \Big)\;.
\end{equ}
We then set, for $\mft \in \Lab$, $o \in \CE$, and $\tau \in \mfT$,
\begin{equs}[eq:Upsilon_with_sym]
\bUpsilon &= \sum_{o \in \CE, 
\tau \in \mfT} \bUpsilon_{o}[\tau] \;,&\qquad \bUpsilon_{o}[\tau] &\eqdef
\Upsilon_{o}[\tau]/S(\tau)\;,\qquad\\
\bbUpsilon &= \sum_{\mft \in \Lab,
\tau \in \mfT} \bbUpsilon_{\mft}[\tau]\;,&\qquad \bbUpsilon_{\mft}[\tau] &\eqdef \bar{\Upsilon}_{\mft}[\tau]/S(\tau)\;.
\end{equs} 
%
%
%
We can now state the main theorem of this section.
\begin{theorem}\label{thm:upsilon_summation}
$\bbUpsilon$ and $\bUpsilon$ as defined in \eqref{eq:Upsilon_with_sym} are left unchanged under decompositions of $(\Lab,W,\mathcal{K})$. 
Precisely, given a decomposition $(\bar{\Lab},\bar{W},\bar{\mathcal{K}})$ of $(\Lab,W,\mathcal{K})$ under $\proj$, for any $\mft \in \Lab$, $\tau \in \mfT$ and $\mathbf{A} \in \mcb{A}$,  
\begin{equation}\label{eq:upsilon_identity}
\bUpsilon_{\mft}[\tau](\mathbf{A})
=
\sum_{\mfl \in \proj(\mft),
\bar{\tau} \in \proj(\tau)}
\bUpsilon_{\mfl}[\bar{\tau}](\mathbf{A})\;,
\end{equation} 
where $(\bUpsilon_\mfl)_{\mfl\in\bar\Lab}$ on the right-hand side is defined as above but with $(\bar{\Lab},\bar{W},\bar{\mathcal{K}})$ in its construction, and $\proj(\tau)$ is defined as in Remark~\ref{rem:trees_remain_trees}.
The equality \eqref{eq:upsilon_identity} also holds when $\bUpsilon$ is replaced by $\bbUpsilon$. 
\end{theorem}
\begin{proof}
We prove \eqref{eq:upsilon_identity} inductively in the number of edges of $\tau$. 
Writing $\tau$ in the form \eqref{eq:general_tree_added}, our base case corresponds to $m = 0$, i.e.\ $\tau = \mbX^{k}$ for $k \in \N^{d+1}$,
so that $\proj(\mbX^{k}) = \{\mbX^{k}\}$ and $S(\mbX^{k}) = k!$.
This case is covered by Remark~\ref{rem:identity_is_special}.
For our inductive step, we may assume that $m \ge 1$ in \eqref{eq:general_tree_added}. 
 
Then, inserting our inductive hypothesis in \eqref{e:def-Upsilon-added} (including the special case $\Upsilon_{\mft}[\bone] = \sum_{\mfl \in \proj(\mft)} \Upsilon_{\mfl}[\bone]$) and also applying  Remark~\ref{rem:decom_of_deriv} we see that \eqref{eq:upsilon_identity} follows if we can show that, for any $\mfl \in \proj(\mft)$, 
\begin{equs}
\sum_{(l,\sigma) \in d(\tau)}
(\mcb{I}_{\mft} \otimes \id_{W_\mft})
\Big[
\mbX^{k}&
\Big(\partial^{k}
\big(D_{l_{1}} \cdots  
D_{l_{m}}
\bar{\Upsilon}_{\mft}[\bone]\big)
\big(\Upsilon_{l_1}[\sigma_1],\ldots,\Upsilon_{l_m}[\sigma_m]\big)
\Big)
\Big]\\
&=
\sum_{
\bar{\tau} \in \proj(\tau)}
\frac{S(\tau)}{S(\bar{\tau})}
\Upsilon_{\mfl}[\bar{\tau}]\;,\label{eq:main_claim_for_upsilonident}
\end{equs}
where $d(\tau)$ consists of all pairs of tuples $(l,\sigma)$ with $l = (l_{i})_{i=1}^{m}$, $\sigma = (\sigma_{i})_{i=1}^{m}$, $l _{i} \in \proj(o_{i})$ and $\sigma_{i} \in \proj(\tau_{i})$. 
Given $(l,\sigma) \in d(\tau)$ we write $\tau(l,\sigma) \eqdef \mathbf{X}^{k}\prod_{i=1}^{m} \mcb{I}_{l_{i}}(\sigma_{i}) \in \mfT_{\bar{\Lab}}$.
Clearly one has $\tau(l,\sigma) \in \proj(\tau)$ and for fixed $(l,\sigma)$ the corresponding summand on the left-hand side of \eqref{eq:main_claim_for_upsilonident} is simply $\Upsilon_{\mfl}[\tau(l,\sigma)]$. 
Finally, for any $\bar{\tau} \in \proj(\tau)$, it is straightforward to prove, using a simple induction and manipulations of multinomial coefficients, that
\begin{equ}
\frac{S(\tau)}{S(\bar{\tau})}
=
|\{(l,\sigma) \in d(\tau): \tau(l,\sigma) = \bar{\tau}\}|\;,
\end{equ}
which shows~\eqref{eq:main_claim_for_upsilonident}.
Using natural isomorphisms between the spaces where $\bbUpsilon$ and $\bUpsilon$ live, it follows that \eqref{eq:upsilon_identity} also holds for $\bbUpsilon$. 
\end{proof} 

\begin{remark}\label{rem:GHM20}
We used Remark~\ref{rem:identity_is_special} in a crucial way to start the induction, which shows that
since we consider rules with $R(\mfl) = \{()\}$ for $\mfl \in \Lab_-$, the choice $F_\mfl = \id_{W_\mfl}$ is
the only one that complies with \eqref{e:conformRule} and also guarantees invariance under noise decompositions.
See however \cite{AndrisKonstantin} for an example where $R(\mfl) \neq \{()\}$ and it is natural to make a different choice for $F_\mfl$. 
\end{remark}

We now precisely define coherence in our setting.
For $L \in \N \cup \{ \infty \}$, we denote by $p_{\le L}$  the projection map on $\bigoplus_{\mft \in \Lab}\mcb{B} \otimes W_{\mft}$ which vanishes on any subspace of the form $\CT[\tau] \otimes W_{\mft}$ if 
\[
|E_{\tau}| + \sum_{v \in V_{\tau}} |\mfn(v)| > L
\]
and is the identity otherwise. Above 
we write $E_{\tau}$ for the set of edges of $\tau$, $V_{\tau}$ for the set of nodes of $\tau$, and $\mfn$ for the label on $\tau$.
Note that $p_{\le \infty}$ is just the identity operator.
\begin{definition}\label{def:coherence}
We say $\mcA \in \expan$ is coherent to order $L \in \N \cup \{\infty\}$
with $F$ if
 \begin{equ}\label{eq:coherence_def}
p_{\le L} \mcA^{R} = p_{\le L} \bUpsilon(\mathbf{A}^{\mcA})\;,
\end{equ} 
where $\mcA^{R} = \bigoplus_{\mft \in \Lab} \mcA^{R}_{\mft}$ with $\mcA^{R}_{\mft}$ 
determined from $\mcA$ as in \eqref{eq:coherent_jet}. 
\end{definition}
%
%
Note that, by Theorem~\ref{thm:upsilon_summation}, coherence to any order $L$ is preserved under noise decompositions.
Thanks to Theorem~\ref{thm:upsilon_summation} we can reformulate \cite[Lem.~3.21]{BCCH21} to show that our definition of $\bUpsilon$ encodes the condition \eqref{eq:fixedptpblm}; we state this as a lemma. 
\begin{lemma}\label{lem:analog_of_lemma3.16}
$\mcA \in \expan$ is coherent to order $L\in\N\cup\{\infty\}$ 
with $F$ if and only if, for each $\mft \in \Lab$, 
\begin{equ}\label{eq:abs_fixedpt}
p_{\le L}
\mcA^{R}_{\mft}	
=
p_{\le L} 
(\mcb{I}_{\mft} \otimes \id_{W_\mft}) 
F_\mft(\mcA)\;.
\end{equ} 
\end{lemma}
\begin{remark}\label{rem:alt_way_calc_ups}
Combining Lemma~\ref{lem:analog_of_lemma3.16} with Definition~\ref{def:coherence} shows that 
$\bUpsilon$ does indeed have the advertised property, namely it yields a formula for the 
``non-standard part'' of the expansion of any solution to the algebraic counterpart \eqref{eq:fixedptpblm}
of the mild formulation of the original problem \eqref{e:SPDE}.

Conversely, this
provides us with an alternative method for computing $\bUpsilon(\mathbf{A})[\tau]$ for any $\tau \in \mfT$.
Given $\mathbf{A} \in \mcb{A}$, set $\mcA^{(0)}  = \mathbf{A} \in \expan$ (recall \eqref{eq:coherent_jet} for the identification of $\mcb{A}$ as a subspace of $\expan$) and then proceed iteratively by setting
\begin{equ}
\mcA^{(n+1)}_\mft = \mathbf{A}_\mft + (\mcb{I}_{\mft} \otimes \id_{W_\mft}) F_\mft(\mcA^{(n)})\;.
\end{equ}
Subcriticality then guarantees that any of the projections $\mcA^{(n)}[\tau]$ stabilises 
after a finite number of steps, and one has $\bUpsilon_{\mft}[\tau](\mathbf{A}) = \mcA^{(\infty)}[\tau]$. 
\end{remark}
\begin{remark}\label{rem:analytic-fix-pt}
The material discussed in Section~\ref{subsec:nonlinearities} up to this point has been devoted to treating \eqref{e:SPDE} as an \emph{algebraic} 
fixed point problem in the space $\expan$. 
We also want to solve an \emph{analytic} fixed point problem in a space of modelled distributions, namely in a space of $\expan$-valued functions 
over some space-time domain. 

Posing the analytic fixed point problem requires us to start with more input than we needed for the algebraic one. 
After fixing $F \in \mcb{Q}$ one also needs to fix suitable\footnote{Here, ``suitable'' means sufficiently large so that the fixed point
problem is well-posed. Subcriticality guarantees that setting $\gamma_\mft = \gamma$ for all $\mft$ is a suitable choice provided that $\gamma \in \R_+$ is sufficiently large.} regularity exponents $(\gamma_\mft : \mft \in \Lab)$ for the modelled distribution spaces involved and initial data $(u_{\mft}: \mft \in \Lab_{+})$ for the problem.
Moreover, one prescribes a modelled distribution expansion for each noise, namely for every $\mfl \in \Lab_-$,
we fix a modelled distribution $\mathcal{O}_{\mfl}$ of regularity $\cD^{\gamma_\mfl}$ of the form 
\begin{equ}\label{eq:input_for_noise}
\mathcal{O}_{\mfl}(z)
=
\sum_{k \in \N^{d+1}}
O_{(\mfl,k)}(z)\mbX^{k}
+
\mcb{I}_{\mfl}(\bone)\;.
\end{equ}
The corresponding analytic fixed point problem \cite[Eq.~(5.6)]{BCCH21} is then posed on a space of modelled distributions 
$\mcU = (\mcU_{\mft}: \mft \in \Lab_{+})$ such that $\mcU_{\mft} \in \cD^{\gamma_\mft}$ (at least locally).
On some space-time domain $D$ (typically of the form $[0,T] \times \R^d$), the fixed point problem is of the form
\begin{equ}
\mcU_{\mft}
=
\CP_\mft \bone_{t > 0} F_\mft\big((\mcU\sqcup \mcO)(\bigcdot)\big)
+
G_{\mft}u_{\mft}\;.
\end{equ}
In this identity, $\CP_\mft$ is an operator of the form
\begin{equ}
\big(\CP_\mft \mathcal{F}\big)(z) =  \mathfrak{p}_{\le \gamma_{\mft}} \mcb{I}_{\mft} \mathcal{F}(z) + (\ldots)\;,
\end{equ}
where $(\ldots)$ takes values in $\bar \CT_\mft \eqdef \bigoplus_{ k \in \N^{d+1}} \CT[\mbX^{k}] \otimes W_{\mft}$,
and $G_\mft$ is the ``harmonic extension map'' as in \cite[(7.13)]{Hairer14} associated to $(\partial_{t} - \mathscr{L}_{\mft})^{-1}$
(possibly with suitable boundary conditions). Here, $\mathfrak{p}_{\le \gamma_{\mft}}$ is the projection onto components of degree 
less than $\gamma_{\mft}$. 
Since $G_\mft$ also takes values in $\bar \CT_\mft$, it follows that for any solution $\mcU$ to such a fixed 
point problem and any space-time point $z \in D$, $\mcU(z) \sqcup \mcO(z)$ is coherent with $F$ to 
some order $L$ which depends on the exponents $(\gamma_\mft : \mft \in \Lab)$; see~\cite[Thm.~5.7]{BCCH21} for a precise statement.

Note that, depending on the degrees of our noises, there can be some freedom in our choice of \eqref{eq:input_for_noise} depending on how we choose to have our model act on symbols $\mcb{I}_{\mfl}(\bone)$ for $\mfl \in \Lab_{-}$ \dash the key fact is that $\mathcal{O}_{\mfl}$
represents the corresponding driving noise in our problem, not necessarily $\mcb{I}_{\mfl}(\bone)$. 
However, when $\deg(\mfl) < 0$, a natural choice for the input \eqref{eq:input_for_noise} is to simply set $O_{(\mfl,k)}(z) = 0$ for all 
$k \in \N^{d+1}$ and this is the convention we use in Sections~\ref{sec:solution_theory} and \ref{sec:gauge_equivar}. 
\end{remark}

\subsubsection{Renormalised equations}\label{subsec:renorm_eq}
We now describe the action of the renormalisation group $\CG_{-}$ on nonlinearities, which is how it produces counterterms in equations.
We no longer treat $F \in \mcb{Q}$ as fixed and when we want to make the dependence of $\bar{\bUpsilon}$ on $F \in \mcb{Q}$  explicit we write  $\bar{\bUpsilon}^{F}$.  
We re-formulate the main algebraic results of \cite{BCCH21} in the following proposition;
the proof is obtained by using Theorem~\ref{thm:upsilon_summation} to restate \cite[Lem.~3.22, Lem.~3.23, and Prop.~3.24]{BCCH21}.
\begin{proposition}\label{prop:renormalisation_preserves_coherence}
Fix $\ell \in \mathcal{G}_{-}$.
There is a map $F \mapsto M_{\ell}F$, taking $\mcb{Q}$ to itself, defined by, for $\mft \in \Lab$ and $\mathbf{A} \in \mcb{A}$,
\begin{equs}[e:MellF]
(M_{\ell}F)_{\mft}(\mathbf{A})
& \eqdef 
(p_{\bone,\mft} M_{\ell}\otimes \id_{V_\mft \otimes W_{\mft}}) \bar{\bUpsilon}_{\mft}^{F}(\mathbf{A})
\\
&=
F_{\mft}(\mathbf{A})
+
\sum_{
\tau \in \mfT_{-}(R)}
(\ell \otimes \id_{V_\mft \otimes W_{\mft}})
\bar{\bUpsilon}^{F}_{\mft}[\tau](\mathbf{A})\;,
\end{equs}
where $p_{\bone,\mft}$ denotes the projection onto $\CT[\bone]$ and the operator $M_{\ell}$ on the right-hand side is given by \eqref{eq:renorm_op}.

Moreover, for any $L \in \N \cup \{\infty\}$,
\begin{equ}\label{eq:renormalised_coherence}
p_{\le L}\Big(\bigoplus_{\mft \in \Lab} M_{\ell} \otimes \id_{V_\mft \otimes W_{\mft}}\Big)
 \bar{\bUpsilon}^{F} = p_{\le L} \bar{\bUpsilon}^{M_{\ell}F}\;,
\end{equ}
and there exists $\bar L \in \N \cup \{\infty\}$,
which depends only on $L$ and the rule $R$ and which can be taken finite if $L$ is finite, such that if $\mcA \in \expan$ is coherent to order $\bar L$ with $F \in \mcb{Q}$ then $(M_{\ell} \otimes \id)\mcA$ is coherent to order $L$ with $M_{\ell}F$. 
\end{proposition}

\begin{remark}
Note that for $\mft \in \Lab$, $\bar{\bUpsilon}^{F}_{\mft}[\tau](\mathbf{A})\in \CT[\tau]\otimes V_\mft \otimes W_\mft$, so every term on the right-hand side of \eqref{e:MellF} is an element of $V_\mft \otimes W_\mft$.
\end{remark}

%

\section{Solution theory of the SYM equation}
\label{sec:solution_theory}
In this section we make rigorous the solution theory for \eqref{eq:SYM} and provide the proof of Theorem~\ref{thm:local_exist}. 
In particular, we explicitly identify the counterterms appearing in the renormalised equation as this will be needed for the proof of gauge covariance in Section~\ref{sec:gauge_equivar}.
Recalling Remark~\ref{rem:simple_vs_reductive} we make the following assumption. 
\begin{assumption}\label{as:simple}
The Lie algebra $\mfg $ is simple.
\end{assumption}
The current section is split into two parts.
In Section~\ref{subsec:Recasting vec reg struct} we recast \eqref{eq:SYM} into the framework of regularity structures with vector-valued noise using Section~\ref{sec:vector-valued_noises}.  
In Section~\ref{sec:BPHZ-YM2} we invoke the black box theory of \cite{Hairer14,CH16,BHZ19,BCCH21} to prove convergence of our mollified \slash renormalised solutions and then explicitly compute our renormalised equation (using Proposition~\ref{prop:renormalisation_preserves_coherence}) in order to show that, when $d=2$, the one counterterm appearing converges to a finite value as $\eps \downarrow 0$. 
\subsection{Regularity structure for the SYM equation}
\label{subsec:Recasting vec reg struct}
We set up our regularity structure for formulating \eqref{eq:SYM} in $d=2$ or $d=3$ space dimensions. 
However, when, computing counterterms, we will again fix $d = 2$. We write $[d]\eqdef \{1,\cdots,d\}$.

Our space-time scaling $\s \in [1,\infty)^{d+1}$ is given by setting $\s_{0} = 2$ and $\s_{i} = 1$ for $i \in [d]$. 
We define $\mfL_+ \eqdef \{\mfa_i\}_{i=1}^{d}$ and $\mfL_- \eqdef \{\mfl_i\}_{i=1}^{d}$.
We define a degree $\deg\colon \mfL \rightarrow \R$ on our label set by setting 
\begin{equ}[e:setting-deg]
\deg( \mft) 
\eqdef
\begin{cases}
2 & \mft \in \mfL_{+}\;,\\
-d/2 - 1 - \kappa & \mft \in \mfL_{-}	
\end{cases}
\end{equ}
where we fix $\kappa \in (0,1/4)$.
 
Looking at equation \eqref{eq:renormalised_SYM} leads us to consider the rule $\mathring{R}$ given by setting, for each $i \in [d]$, $\mathring{R}(\mfl_i) \eqdef \{\emptyset\}$ and
\begin{equ}\label{eq:sym_rule}
\mathring{R}(\mfa_i)
\eqdef
\left\{
\begin{array}{c}
\{(\mfl_i,0)\},\ \{(\mfa_i,0),(\mfa_j,0),(\mfa_j,0)\}\\
\{(\mfa_j,0),(\mfa_j,e_i)\},\ \{(\mfa_j,0),(\mfa_i,e_j)\}
\end{array}
: j \in [d] \right\}
\;.
\end{equ} 
It is straightforward to verify that $\mathring{R}$ is subcritical. 
The rule $\mathring{R}$ has a smallest normal \cite[Definition~5.22]{BHZ19} extension and this extension admits a completion $R$ as constructed in \cite[Proposition~5.21]{BHZ19} \dash $R$ is also subcritical and will be the rule used to define our regularity structure. 

We fix our target space assignment 
$(W_{\mft})_{\mft \in \mfL}$ and kernel space assignment by setting 
\begin{equ}[e:YM2-target space assignment]
W_{\mft} \eqdef \mfg
\quad \forall\mft \in \mfL 
\qquad 
\text{and}
\qquad 
\mathcal{K}_{\mft} = \R 
\quad
\forall\mft \in \mfL_{+}\;.
\end{equ}
The space assignment $(V_{\mft})_{\mft \in \mfL}$ used in the construction of our concrete regularity structure via the functor  $\Func_{V}$ is then given by \eqref{e:defV}. 
\begin{remark}\label{rem:switching_to_standard_notation}
While the notation $\mathbf{A} =(\mathbf{A}_{(\mft,p)}: (\mft,p) \in \CE) \in \mcb{A}$ was convenient for the formulation and proof of the statements of Section~\ref{subsec:nonlinearities}, it would make the computations of this section and Section~\ref{sec:gauge_equivar} harder to follow.
We thus go back to using the symbol $A$ for the components  
$\mathbf{A}_{(\mft,p)}$ of $\mathbf{A}$ with $\mft \in \Lab_{+}$ and
the symbol $\xi$ for the components $\mathbf{A}_{(\mfl,0)}$ with $\mfl \in \Lab_-$.
To streamline notations, we also write the subscript $p$ as a derivative, namely we write
\begin{equ}\label{eq:switching_to_standard_notation}
\xi_{i} = \mathbf{A}_{(\mfl_i,0)}\;,\qquad
A_{i} = \mathbf{A}_{(\mfa_i,0)}
\qquad
 \mbox{and}
 \qquad
\d_j A_{i} = \mathbf{A}_{(\mfa_i,e_{j})}\;.
\end{equ}
\end{remark} 

Regarding the specification of the right-hand side $F = \bigoplus_{\mft \in \Lab} F_{\mft} \in \mcb{Q}$, we set, for each $i \in [d]$ and $\mathbf{A}  \in \mcb{A}$,  $F_{\mfl_i}(\mathbf{A}) = \id_{\mfg}$ and
\begin{equs}
F_{\mfa_i}(\mathbf{A}) 
&=  \mathbf{A}_{(\mfl_i,0)}+
\sum_{j=1}^{d}
[\mathbf{A}_{(\mfa_j,0)},2 \mathbf{A}_{(\mfa_i,e_j)} - \mathbf{A}_{(\mfa_j,e_i)} + [\mathbf{A}_{(\mfa_j,0)},\mathbf{A}_{(\mfa_i,0)}]] \\
&=  \xi_{i}+
\sum_{j=1}^{d}
[A_{j},2\d_j A_{i} - \d_i A_{j} + [A_{j},A_{i}]]\;,\label{e:YM-Fi}
\end{equs}
where the identification of the two lines uses the notations of Remark~\ref{rem:switching_to_standard_notation}.

The right-hand side of \eqref{e:YM-Fi} is clearly a polynomial in a finite number of components of $\mathbf{A}$ taking values in $W_{\mfa_i} = \mfg$, 
so indeed $F \in \mathring{\mcb{Q}}$.
The derivatives $D_{o_{1}}\cdots D_{o_{m}}F_{\mfa_i}(\mathbf{A})$, for $o_{1},\dots,o_{m} \in \CE$, are not difficult to compute. 
For instance, for fixed $\mathbf{A} \in \mcb{A}$, $D_{(\mfa_j,0)}F_{\mfa_i}(\mathbf{A}) \in L(W_{(\mfa_j,0)},W_{(\mfa_i,0)}) = L(\mfg,\mfg)$ is given by
\begin{equs}
\big(
D_{(\mfa_j,0)}F_{\mfa_i}(\mathbf{A})
\big)(\bullet)
&= [\bullet,2\d_j A_{i} - \d_i A_{j}
+
[A_{j},A_{i}]] +
[A_{j}, [\bullet,A_{i}]] \\
&\qquad +
\delta_{i,j} \sum_{k=1}^d [A_{k}, [A_k,\bullet]]\;.
\end{equs}
It is then straightforward to see that $F \in \mcb{Q}$, namely, it obeys the rule $R$ in the sense of Definition~\ref{def:conforming_nonlin}. 
\subsection{The BPHZ model \texorpdfstring{\slash}{/} counterterms for the SYM equation in \texorpdfstring{$d=2$}{d=2}}
\label{sec:BPHZ-YM2}

We remind the reader that we now restrict to the case $d=2$ and $|\mfl|_\s = -2-\kappa$ for every $\mfl \in \mfL_{-}$.

As mentioned in Remark~\ref{rem:noise is I-1},
we will use the symbol $\Xi_i \eqdef \mcb{I}_{(\mfl_i,0)}({\bf 1})$ for the noise.\label{Xi page ref}
Similarly to Remark~\ref{rem:switching_to_standard_notation}, we also use the notations $\mcb{I}_{i}$ and $\mcb{I}_{i,j}$
as shorthands for $\mcb{I}_{(\mfa_i,0)}$ and $\mcb{I}_{(\mfa_i,e_j)}$ respectively.

Below we introduce a graphical notation to describe forms of relevant trees. 
The noises $\Xi_{i}$ are circles $\<Xi>$, noises with polynomials 
$\mbX^{e_j}\Xi$ with $j \in [d]$ are crossed circles $\<XXi>$, and edges $\mcb{I}_{i}$ and $\mcb{I}_{i,j}$ 
are thin and thick grey lines respectively. 
It is always assumed that the indices $i$ and $j$ appearing on occurrences of $\Xi_{i}$, $\mcb{I}_{i}$, and $\mcb{I}_{i,j}$ throughout the tree are constrained by the requirement that our trees conform to the rule $R$. 

We now give a complete list of the forms of trees in $\mfT(R)$ with negative degree, the form is listed on the top and the degree below it. 
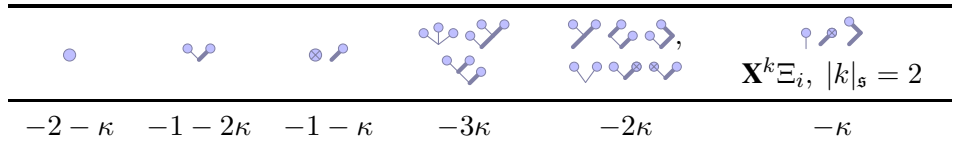
\begin{center}
\begin{tabular}{ccccccc}
\toprule
$\<Xi>$
&
$\<IXiI'Xi_notriangle>$
&
$\<XXi>$ $\<I'Xi_notriangle>$
&
$\begin{array}{c}
\<IXi^3_notriangle>\ \<IXiI'[IXiI'Xi]_notriangle>\\ 
\<I[IXiI'Xi]I'Xi_notriangle>
\end{array}$
&
$\begin{array}{c}
\<I'[IXiI'Xi]_notriangle>\ \<I[I'Xi]I'Xi_notriangle>\ \<IXiI'[I'Xi]_notriangle>,\\
\<IXi^2>\ \<IXiI'XXi_notriangle>\ \<IXXiI'Xi_notriangle>
\end{array}$
&
$\begin{array}{c}
\<IXi>\ \<I'XXi_notriangle>\ \<I'[I'Xi]_notriangle>\\
 \mbX^k\Xi_{i},\ |k|_\s = 2
\end{array}$

\\
\midrule
$-2-\kappa$
&
$-1-2\kappa$
&
$-1-\kappa$
&
$-3\kappa$
&
$-2\kappa$
&
$-\kappa$
\\
\bottomrule
\end{tabular}
\end{center}
Note that each symbol above actually corresponds to a family of trees, determined by assigning indices in a way that conforms to the rule $R$. 
For instance, when we say that $\tau$ is of the form $\<I[IXiI'Xi]I'Xi_notriangle> $, then $\tau$ could be any tree of the type
\[
\mcb{I}_{i_{1}}\big(\mcb{I}_{i_{2}}(\Xi_{i_2}) \mcb{I}_{i_{3},j_{3}}(\Xi_{i_3})\big)
\mcb{I}_{i_4,j_4}(\Xi_{i_4})
\]
for any $i_{1},i_{2},i_{3},i_{4},j_{3},j_{4} \in [d]$
satisfying both of the following two constraints:
first, one must have either $i_{1} = i_{4}$ \emph{or} $j_{4} = i_{1}$,
and, second, one must have either $i_{2} = i_{3}$ and $j_{3} = i_{1}$
\emph{or} $i_{2} = j_{3}$ and $i_{3} = i_{1}$. 

Note that a circle $\<Xi>$ or a crossed circle $\<XXi>$ actually represents an \textit{edge} when we 
think of any of the corresponding typed combinatorial trees; for instance, in the sense of Section~\ref{sec:symtrees}, 
$\<IXiI'Xi_notriangle>$ has \textit{four} edges and not two.
In Section~\ref{subsec:comp_of_ups}, we will further colour our graphical symbols to encode constraints on indices.

\subsubsection{Kernel and noise assignments for \eqref{eq:SYM}}\label{subsubsec:kernelandnoiseforsym}
We fix a kernel assignment by setting, for every $\mft \in \Lab_{+}$, $K_{\mft} = K$ where we fix $K$ to be a truncation of the Green's function $G(z)$ of the heat operator which satisfies the following properties: 
\begin{enumerate}
\item $K(z)$ is smooth on $\R^{d+1} \setminus \{0\}$.
\item $K(z) = G(z)$ for $0 \le |z|_{\s} \le 1/2$.
\item $K(z) = 0$ for $|z|_{\s} > 1$
\item\label{pt:K_non-ant} Writing $z =(t,x)$ with $x \in \T^{2}$, $K(0,x) = 0$ for $x \not = 0$, and $K(t,x) = 0$ for $t < 0$.  
\item Writing $z = (t,x_{1},x_{2})$ with $x_{1},x_{2}$ the spatial components, $K(t,x_{1},x_{2}) = K(t,-x_{1},x_{2}) = K(t,x_{1},-x_{2})$ and $K(t,x_{1},x_{2}) = K(t,x_{2},x_{1})$. 
\end{enumerate}
We will also use the shorthand $K^{\eps} = K \ast \moll^{\eps}$.
\begin{remark}\label{rem:kernel_assump}
Property~\ref{pt:K_non-ant} is not strictly necessary for the proof of Theorem~\ref{thm:state_space} but will be convenient for proving item \ref{pt:unique_Markov} of Theorem~\ref{thm:gauge_covar} in Section~\ref{sec:gauge_equivar} so we include it here for convenience.  Property 5 is also not strictly necessary, but convenient if we want certain BPHZ renormalisation constants to vanish rather than just being finite.

Note that we do not assume a moment vanishing condition here as in \cite[Assumption~5.4]{Hairer14} \dash the only real change from the framework of \cite{Hairer14} that dropping this assumption entails is that, for $ p,k \in \N^{d+1}$, we can have presence of expressions such as $\mcb{I}_{(\mfm,p)}(\mbX^{k})$ when we write out trees in $\mfT(R)$. 
Works such as \cite{CH16}, \cite{BHZ19}, \cite{BCCH21} already assume trees containing such expressions are allowed to be present.	
\end{remark}
Next, we overload notation and introduce a random smooth noise assignment $\zeta^{\eps} = (\zeta_{\mfl})_{\mfl \in \Lab_{-}}$ by setting, for $i \in [d]$, $\zeta_{\mfl(i)} = \xi_{i}^{\eps}$ where we recall that $\xi_{i}^{\eps} = \xi_{i} \ast \moll^{\eps}$ and  $(\xi_{i})_{i=1}^{d}$ are the i.i.d. $\mfg$-valued space-time white noises introduced as the beginning of the paper. 
With this fixed choice of kernel assignment and random smooth noise assignment $\zeta^{\eps}$ for $\eps > 0$ we have a corresponding BPHZ renormalised model $Z^{\eps}_{\BPHZ}$. 
We also write $\ell_{\BPHZ}^{\eps} \in \CG_{-}$ for the corresponding BPHZ character. 
\subsubsection{Convergence of models for \eqref{eq:SYM}}
We now apply \cite[Theorem~2.15]{CH16} to prove the following.
\begin{lemma}\label{lem:conv_of_models}
The random models $Z^{\eps}_{\BPHZ}$ converge in probability, as $\eps \downarrow 0$, to a limiting random model $Z_{\BPHZ}$. 
\end{lemma}
\begin{proof}
We note that \cite[Theorem~2.15]{CH16} is stated for the scalar noise setting so to be precise one must verify its conditions after applying some choice of scalar noise decomposition. 
However, it is not hard to see that the conditions of the theorem are completely insensitive to the choice of scalar noise decomposition. 
Let $\zeta = (\zeta_{\mfl})_{\mfl \in \Lab_{-}}$ be the unmollified random noise assignment, that is,
$\zeta_{\mfl(i)} = \xi_{i}$. 

 For any scalar noise decomposition, it is straightforward to verify the condition that the random smooth noise assignments $\zeta^{\eps}$ are a uniformly compatible family of Gaussian noises that converge to the Gaussian noise $\zeta$ (again, seen as a rough, random, noise assignment for scalar noise decomposition in the natural way). 

The first three listed conditions of \cite[Theorem~2.15]{CH16} refer to power-counting considerations written in terms of the degrees of the combinatorial trees spanning the scalar regularity structure and the degrees of the noises. 
Since this power-counting is not affected by decompositions, they can be checked directly on the trees of $\mfT(R)$. 
We note that 
\[
\min \{ \deg(\tau): \tau \in \mfT(R), |N(\tau)| > 1\}
=
-1 - 2\kappa > -2
\]
is achieved for $\tau$ of the form $\<IXiI'Xi_notriangle>$ \dash here $N(\tau)$ is as defined in \eqref{eq: kernels and nodes}.
This is greater than $-|\s|/2$, so the third criterion is satisfied.
Combining this with the fact that $\deg(\mfl) = -|\s|/2 - \kappa$ for every $\mfl \in \Lab_{-}$ guarantees that the second criterion is satisfied. 
Finally, the worst case scenario for the first condition is for $\tau$ of the form $\<IXiI'Xi_notriangle>$ 
with the subtree $S$ therein being $\<I'Xi_notriangle>$
and the set  $A$ therein being a single noise,
 so 
 \[
 |S|_{\s}+|\mfl|_{\s}+|A||\s| = -1-\kappa-2-\kappa+4=1-2\kappa>0
 \]
as required. 
\end{proof}
\subsubsection{The BPHZ renormalisation constants}
The set of trees $\mfT_{-}(R)$ is given by all trees of the form 
\begin{equ}
\<IXiI'Xi_notriangle>, \enskip
\<IXi^3_notriangle>, \enskip
\<IXiI'[IXiI'Xi]_notriangle>, \enskip
\<I[IXiI'Xi]I'Xi_notriangle>, \enskip
\<I[I'Xi]I'Xi_notriangle>, \enskip
\<IXiI'[I'Xi]_notriangle>, \enskip
\<IXi^2>, \enskip
\<IXiI'XXi_notriangle>, \enskip
\textnormal{or }
\<IXXiI'Xi_notriangle> \enskip.
 \end{equ}
Our remaining objective for this section is to compute the counterterms
\begin{equ}[e:Main-Renorm-todo]
\sum_{
\tau \in \mfT_{-}(R)}
(\ell_{\BPHZ}^{\eps}[\tau] \otimes \id_{W_{\mft}})
\bar\bUpsilon^{F}_{\mft}[\tau](\mathbf{A})\;
\end{equ}
for each $\mft \in \Lab_{+}$. 
In what follows, we perform separate computations for the character $\ell_{\BPHZ}^{\eps}[\tau]$ and for
$\bar\bUpsilon^{F}$, before combining them to compute \eqref{e:Main-Renorm-todo}. 
The following lemma identifies some cases where $\ell_{\BPHZ}^{\eps}[\tau] = 0$. 
\begin{lemma}\label{lem:symbols_vanish}
\begin{enumerate}[label=(\roman*)]
\item\label{item:odd}
$\ell_{\BPHZ}^{\eps}[\tau]=0$
 for each $ \tau$ consisting of an odd number of noises, that is any $\tau$ of the form
 $\<IXi^3_notriangle>$, $\<IXiI'[IXiI'Xi]_notriangle>$, and $\<I[IXiI'Xi]I'Xi_notriangle>$.
 \item\label{item:antipode_minus} On every subspace $\CT[\tau]$ of the regularity structure with  $\tau$ of the form $\<IXiI'Xi_notriangle>, \<IXi^2>$, $\<I[I'Xi]I'Xi_notriangle>$ or $\<IXiI'[I'Xi]_notriangle>$, one has $\tilde\mcA_-   = -\id$, so that $\ell_{\BPHZ}^{\eps}[\tau] = -\bar{\PPi}_{\can}[\tau]$.
 \item\label{item:BPHZ_zero}
 $\ell_{\BPHZ}^{\eps}[\tau] = 0$
 for every $\tau$ of the form $\<IXiI'Xi_notriangle>$. 
\item\label{item:noises_same} For $\tau$ of the form $\<IXi^2>$, $\<I[I'Xi]I'Xi_notriangle>$ or $\<IXiI'[I'Xi]_notriangle>$, one has $\ell_{\BPHZ}^{\eps}[\tau]=0$ unless the two noises $\Xi_{i_{1}}$ and $\Xi_{i_{2}}$ appearing in $\tau$ carry the same index, that is $i_{1}=i_{2}$.
\item\label{item:derivs_same} For $\tau$ of the form $\<I[I'Xi]I'Xi_notriangle>$ or $\<IXiI'[I'Xi]_notriangle>$, one has $\ell_{\BPHZ}^{\eps}[\tau] = 0$ unless the two spatial derivatives appearing on the two thick edges in $\tau$ carry the same index.
\end{enumerate}
\end{lemma}
\begin{proof}
Item~\ref{item:odd} is true for every Gaussian noise.
For item~\ref{item:antipode_minus}, the statement about the abstract regularity structure is a direct consequence of the definition \eqref{eq:recursive_antipode_vec} of the twisted antipode (see also \cite[Prop.~6.6]{BHZ19}) and the statement about $\ell^{\eps}_{\BPHZ}[\tau]$ then follows from \eqref{e:defBPHZ}. 

For item~\ref{item:BPHZ_zero} if we write $\tau =  \mcb{I}_{i}(\Xi_{i}) \mcb{I}_{j,l}(\Xi_{j})$ then
 \[
 \bar{\PPi}_{\can}[\tau]
 =
 \int \mrd u \mrd v\ K(-u)\partial_{l}K(-v) \E[ \xi_{i}^{\eps}(u) \otimes \xi_{j}^{\eps}(v)]\;.
 \]
Performing a change of variable by flipping the sign of the $l$-component of $v$, followed by exploiting the equality in law of $\xi^{\eps}$ and the change in sign of $\partial_{l}K$ under such a reflection, shows that the integral above vanishes.

For item~\ref{item:noises_same}, the fact that $\E[\xi_{i}^{\eps}(u) \otimes \xi_{j}^{\eps}(v)]=0$ if $i \not = j$ enforces the desired constraint. 

For item~\ref{item:derivs_same}, the argument is similar to that of item~\ref{item:BPHZ_zero} - namely the presence of precisely one spatial derivative in a given direction allows one to argue that $\bar{\PPi}_{\can}[\tau]$ vanishes by performing a reflection in the appropriate integration variable in that direction. 
\end{proof}
\begin{remark}
We now start to use splotches of colour such as $\<green>$ or $\<red>$ to represent indices in $[d]$, since they will allow us to work with expressions that would become unwieldy when using letters. 
We also use Kronecker notation to enforce the equality of indices represented by colours, for instance writing $\delta_{\<green>,\<red>}$. 

We can use colours to include indices in our graphical notation for trees in an unobtrusive way, for instance writing $\<IXigreen> = \mcb{I}_{\<green>}(\<Xigreen>) = \mcb{I}_{\<green>}(\Xi_{\<green>})$. 
Note that the splotch of $\<green>$ in the symbol $\<IXigreen>$ fixes the two indices in $\mcb{I}_{\<green>}(\Xi_{\<green>})$ which have to be equal for any tree conforming to our rule $R$. 
The edges corresponding to integration can be decorated by derivatives which introduce a new index, so we introduce 
notation such as $\<I'Xigreen-red> = \mcb{I}_{\<green>,\<red>}(\Xi_{\<green>})$, where the colour of a thick edge 
determines the index of its derivative. 

For drawing a tree like $\mcb{I}_{\<green>}(\Xi_{\<green>})\mcb{I}_{\<orange>,\<red>}(\mcb{I}_{\<green>,\<red>}(\Xi_{\<green>}))$, our earlier way of drawing $\<IXiI'[I'Xi]_notriangle>$ didn't give us a node to colour $\<orange>$, so we add small triangular nodes to our drawings to allow us to display the colour 
determining the type of the edge incident to that node from below, for example
$\<IXiI'[I'Xi]_typed>
=
\mcb{I}_{\<green>}(\Xi_{\<green>})\mcb{I}_{\<orange>,\<red>}(\mcb{I}_{\<green>,\<red>}(\Xi_{\<green>}))$.
\end{remark}

We will see 
by Lemma~\ref{lem:bUps-shortlist}
 that $\Upsilon_{i}[\tau] = 0$ for any $\tau$ of the form $\<IXiI'XXi_notriangle>$ or  $\<IXXiI'Xi_notriangle>$. 
Therefore, \eqref{e:Main-Renorm-todo} will only have contributions from trees of the form
\begin{equ}\label{eq: trees for renorm constants}
\<I[I'Xi]I'Xi_typed>\;,\quad
\<IXiI'[I'Xi]_typed>\;,\quad\text{or}\quad 
\<IXi^2green>\;.
\end{equ}
Thanks to the invariance of our driving noises under the action of the Lie group we will see in Lemma~\ref{lem:tree_calcs} below that $\ell_{\BPHZ}^{\eps}[\tau]$ has to be a scalar multiple of the Casimir element (in particular, it belongs to the subspace $\mfg \otimes_{s} \mfg \subset \mfg \otimes \mfg$). 
This is an immediate consequence of using noise that is white with respect to our inner product $\langle \cdot, \cdot \rangle$ on $\mfg$. 
In particular, note that this inner product on $\mfg$ induces an inner product $\langle \cdot, \cdot \rangle_{2}$ on $\mfg \otimes \mfg$ and that there is a unique element $\Cas \in \mfg \otimes_{s} \mfg \subset \mfg \otimes \mfg$ with the property that, for any $h_{1},h_{2} \in \mfg$, 
\begin{equ}\label{eq:def_of_cas}
\langle \Cas, h_{1} \otimes h_{2} \rangle_{2} 
=
\langle h_{1},h_{2} \rangle\;.
\end{equ}
One can write $\Cas$ explicitly as $\Cas = \sum_{i} e_{i} \otimes e_{i}$ for any orthonormal basis of $\mfg$ but we will refrain from doing so since we want to perform computations without fixing a basis. 
$\Cas$ should be thought of as the covariance of the $\mfg$-valued white noise, in particular for $i,j \in \{1,2\}$ we have 
\begin{equ}\label{eq:explicit_covar_noise}
\E[\xi_{i}(t,x) \otimes \xi_{j}(s,y)] = \delta_{i,j} \delta(t-s)\delta(x-y)\, \Cas\;.
\end{equ}
Thanks to \eqref{eq:def_of_cas}, $\Cas$ is invariant under the action of the Lie group $G$ in the sense that
\begin{equ}[e:propC]
(\Ad_{g} \otimes \Ad_{g})\, \Cas = \Cas\;, \qquad \forall g \in G\,.
\end{equ}
The identity \eqref{e:propC} is of course just a statement about the rotation invariance of our noise. 
Alternatively, we can interpret $\Cas$ as an element of $U(\mfg)$, the universal enveloping algebra of $\mfg$.
The following standard fact will be crucial in the sequel.
\begin{lemma}\label{lem:centre}
$\Cas$ belongs to the centre of $U(\mfg)$.
\end{lemma}
\begin{proof}
Let $h \in \mfg$ and let $\theta$ be a random element of $\mfg$ with  $\E(\theta \otimes \theta) = \Cas$.
Differentiating $\E[\Ad_{g}\theta \otimes \Ad_{g}\theta]$ at $g = e$ in the direction of $h$ yields
\begin{equ}[e:asym]
\E \bigl([h,\theta] \otimes \theta\bigr) = -\E \bigl(\theta\otimes [h,\theta]\bigr)\;.
\end{equ}
We conclude that
\begin{equ}{}
[h,\Cas] = \bigl[h,\E \bigl(\theta \otimes \theta\bigr)\bigr]
= \E \bigl(h \otimes \theta \otimes \theta - \theta \otimes \theta \otimes h\bigr)
= \E \bigl([h,\theta] \otimes \theta + \theta \otimes [h,\theta]\bigr)= 0\;,
\end{equ}
as claimed, where we used \eqref{e:asym} in the last step.
\end{proof}
\begin{remark}\label{rem:action_of_casimir}
We note that $\Cas$ is of course just the quadratic Casimir. 
Moreover, recall that every element $h \in U(\mfg)$ yields a linear operator $\ad_h \colon \mfg \to \mfg$
by setting
\begin{equ}
\ad_{h_1 \otimes \cdots \otimes h_k} X = \bigl[h_1,\cdots[h_k,X]\cdots\bigr]\;.
\end{equ}
With this notation, Lemma~\ref{lem:centre} implies that $\ad_{\Cas}$ commutes with every other operator
of the form $\ad_h$ for $h \in \mfg$ (and therefore also $h \in U(\mfg)$). 
If $\mfg$ is simple, then this implies that $\ad_{\Cas} = \lambda \id_{\mfg}$.
\end{remark}
We now describe $\ell_{\BPHZ}^{\eps}[\tau]$ for $\tau$ of the form \eqref{eq: trees for renorm constants}. 
We define  
\begin{equ}[e:defConstants]
\bar C^\eps \eqdef \int \mrd z\ K^\eps(z)^{2},
\quad
\hat{C}^\eps
\eqdef \int \mrd z\ \partial_{j}K^\eps(z)(\partial_{j}K*K^\eps)(z),
\quad 
C_{\sym}^{\eps} \eqdef   4\hat{C}^\eps - \bar C^{\eps} 
\end{equ}
where, on the right-hand side of the second equation one can choose either $j =1$ or $2$ \dash they both give the same value.
\begin{lemma}\label{lem:tree_calcs}
For $\hat{C}^{\eps}$ and $\bar{C}^{\eps}$ as in \eqref{e:defConstants}, one has 
\begin{equ}[e:valueRenorm]
\ell_{\BPHZ}^{\eps}[\<I[I'Xi]I'Xi_typed>] = 
-\ell_{\BPHZ}^{\eps}[\<IXiI'[I'Xi]_typed>] = - \hat{C}^{\eps}\Cas\;,\qquad
\ell_{\BPHZ}^{\eps}[\<IXi^2green>] =- \bar{C}^\eps \Cas\;.
\end{equ}
Furthermore,
$C_{\sym}^{\eps}$ as defined in \eqref{e:defConstants} converges to a finite value $C_{\sym}$ as $\eps \to 0$.
\end{lemma}
\begin{remark}
The last statement of Lemma~\ref{lem:tree_calcs} is the special feature of working in two spatial dimensions; see Remark~\ref{rem:2d_no_renorm}.
\end{remark}
\begin{remark}
The first identity of \eqref{e:valueRenorm} makes sense since, even though there are two natural
isomorphisms $\CT[\<I[I'Xi]I'Xi_typed>]^{\ast} \approx \mfg \otimes \mfg$ and
$\CT[\<IXiI'[I'Xi]_typed>]^{\ast}\approx \mfg \otimes \mfg$ (corresponding to the two
ways of matching the two noises), $\Cas$ is invariant under that transposition.
For the second identity, note that $\CT[\<IXi^2green>]^*  \simeq \mfg \otimes_{s} \mfg$. 
\end{remark}

\begin{proof}
The identities \eqref{e:valueRenorm} readily follow from direct computation once one uses that in 
all cases $\ell_{\BPHZ}^{\eps}[\tau] = -\bar{\PPi}_{\can}[\tau]$ (this is item~\ref{item:antipode_minus} of 
Lemma~\ref{lem:symbols_vanish}), writes down the corresponding expectation\slash integral, moves the 
mollification from the noises to the kernels, and uses \eqref{eq:explicit_covar_noise}.

Regarding the last claim of the lemma, since $K$ is a truncation of the heat kernel, observe that
\begin{equ}[e:KdeltaQ]
(\partial_t - \Delta) K = \delta_0 + Q\;,
\end{equ}
where $Q$ is smooth and supported away from  the origin. 
Using the shorthand $\mathrm{Int}[F] = \int \mrd z\ F(z)$
it follows that
\begin{equ}\label{eq:Delta_K_1}
\mathrm{Int}[(\Delta K * K^\eps) K^\eps]
= -\mathrm{Int}[(K^\eps)^2] + \mathrm{Int}[(\partial_t K * K^\eps)K^\eps] - \mathrm{Int}[(Q*K^\eps)K^\eps]\;.
\end{equ}
On the other hand, we also have
\begin{equs}\label{eq:Delta_K_2}
\mathrm{Int}[&(\Delta K * K^\eps) K^\eps]
= \mathrm{Int}[ (K * K^\eps) \Delta K^\eps]
\\
&= -\mathrm{Int}[(K*K^\eps)\moll^\eps] - \mathrm{Int}[(\partial_t K * K^\eps) K^\eps] - \mathrm{Int}[(K*K^\eps)(Q*\moll^\eps)]\;.
\end{equs}
Observe that $\mathrm{Int}[(K*K^\eps)\moll^\eps]$, $\mathrm{Int}[(Q*K^\eps)K^\eps]$, and $\mathrm{Int}[(K*K^\eps)(Q*\moll^\eps)]$ all converge\footnote{These facts follow easily from the fact that $K \ast K$ is well-defined and bounded and also continuous away from the origin. To see this note that the semigroup property gives $(G \ast G)(t,x) = tG(t,x)$  and $G-K$ is smooth and supported away from the origin.} as $\eps \to 0$.
Hence, adding~\eqref{eq:Delta_K_1} and~\eqref{eq:Delta_K_2}, we obtain that 
\begin{equs}\label{eq:integrals_for_csym}
{}& - 2\mathrm{Int}[(\Delta K * K^\eps)K^\eps] - \mathrm{Int}[(K^\eps)^2]\\
{}&= \mathrm{Int}[(K*K^\eps)\moll^\eps] + \mathrm{Int}[(K*K^\eps)(Q*\moll^\eps)] +
 \mathrm{Int}[(Q*K^\eps)K^\eps]
\end{equs}
converges as $\eps\to0$. We now note that the quantity above equals $C_{\sym}^{\eps}$ since $\bar C^\eps = \mathrm{Int}[(K^\eps)^2]$ and, by integration by parts, 
\[
\hat{C}^\eps = -\frac{1}{2}\mathrm{Int}[(\Delta K * K^\eps)K^\eps]
\;,
\] 
which completes the proof.
\end{proof}
\subsubsection{Computation of $\bar{\bUpsilon}^{F}$}\label{subsec:comp_of_ups}
Before continuing, we introduce some notational conventions that will be convenient when we calculate $\bar{\bUpsilon}^{F}$.
 
Recall that, for $\<green> \in [d]$, the symbol $\Xi_{\<green>}$ is a tree that indexes a subspace $\CT[\Xi_{\<green>}]$ of our concrete regularity structure $\CT$.
We introduce a corresponding notation $\boldsymbol{\Xi}_{\<green>} \in \CT[\Xi_{\<green>}] \otimes \mfg$ which, under the isomorphism $\CT[\Xi_{\<green>}]\otimes \mfg \simeq \mfg^{\ast} \otimes \mfg \simeq L(\mfg,\mfg)$, is given by $\boldsymbol{\Xi}_{\<green>} = \id_{\mfg}$. 
The expression $\boldsymbol{\Xi}_{\<green>}$ really represents the corresponding noise in the sense that $(\PPi_{\can} \boldsymbol{\Xi}_{\<green>})(\cdot) = \xi^{\eps}_{\<green>}(\cdot)$, where we are abusing notation by having $\PPi_{\can}$ only act on the left factor of the tensor product. 

Continuing to develop this notation, we also define 
\[
\Psi_{\<green>} 
= \mcb{I}_{\<green>}\boldsymbol{\Xi}_{\<green>}
\in
\CT[\<IXigreen>] \otimes \mfg\;.
\] 
where we continue the same notation abuse, namely $\mcb{I}_{\<green>}$ acts only on the left factors appearing in $\boldsymbol{\Xi}_{\<green>}$.
In particular, we have $\PPi_{\can} \Psi_{\<green>}(\cdot) = (K \ast \xi^{\eps}_{\<green>})(\cdot)$.
We also have a corresponding notation 
\[
\Psi_{\<green>,\<red>} = \mcb{I}_{\<green>,\<red>}\boldsymbol{\Xi}_{\<green>}\in \CT[\<I'Xigreen-red>] \otimes \mfg\;.
\]
We now show how this notation is used for products\slash non-linear expressions. 
Given some $h \in \mfg$, we may write an expression such as
\begin{equ}\label{eq:nonlinear}
[\Psi_{\<green>},[h,\Psi_{\<green>,\<red>}]] \in 
\CT[\<IXiI'Xi_typed2_notriangle>] \otimes \mfg\;.
\end{equ}
In an expression like this, we apply the multiplication $\CT[\<IXigreen>] \otimes \CT[\<I'Xigreen-red>] \rightarrow \CT[\<IXiI'Xi_typed2_notriangle>]$ to combine the left factors of $\Psi_{\<green>}$ and $\Psi_{\<green>,\<red>}$. 
The right $\mfg$-factors of $\Psi_{\<green>}$ and $\Psi_{\<green>,\<red>}$ are used as the actual arguments of the brackets above, yielding the new $\mfg$-factor on the right. 
\begin{remark}
For what follows, given $i,j \in [d]$, we write $\Upsilon_{i}$ and $\Upsilon_{i,j}$ for $\Upsilon_{\mfa_{i}}$ and $\Upsilon_{(\mfa_{i},e_{j})}$ respectively. In particular, we will use notation such as $\Upsilon_{\<green>}$ and $\Upsilon_{\<green>,\<red>}$. We extend this convention, also writing $\bUpsilon_{\<green>}$ and $\bar{\bUpsilon}_{\<green>}$ along with $W_{\<green>}$ and $W_{\<green>,\<red>}$.
\end{remark}
With these conventions in place the following computations follow quite easily from our definitions:
\begin{equ}[e:barUpsilonXi]
\bar\Upsilon_{\<green>}^F[\<Xired>]  = \delta_{\<green>,\<red>}
\boldsymbol{\Xi}_{\<green>}\;,\quad
\Upsilon_{\<green>}^F[\<Xired>] = \delta_{\<green>,\<red>} \Psi_{\<green>}\;,\quad
\Upsilon_{\<green>,\<orange>}^F[\<Xired>] 
 = \delta_{\<green>,\<red>} \Psi_{\<green>,\<orange>}\;.
\end{equ}
(Since $S(\tau) = 1$ for these trees, the corresponding $\bUpsilon$ are identical.)
Moreover, for $k \in \N^{d+1}$ with $k \not = 0$, 
\begin{equ}[e:barUpsilonXXi]
\Upsilon_{\<green>}^F[\mbX^{k}\Xi_{\<red>}]  =
\Upsilon_{\<green>,\<orange>}^F[\mbX^{k}\Xi_{\<red>}] = 0\;.
\end{equ}
Note that the left-hand sides of \eqref{e:barUpsilonXi} are in principle allowed to depend on an argument $\mathbf{A} \in \mcb{A}$, but here they are constant in $\mathbf{A}$, so we are using a canonical identification of constants with constant functions here.\footnote{In \eqref{e:barUpsilonXi} we are also exploiting canonical isomorphisms between $W_{\<green>}$ and $W_{\<green>,\<orange>}$ and $\mfg$. 
For instance, one also has $\bar\Upsilon_{\<green>,\<orange>}^F[\<Xired>]  = \delta_{\<green>,\<red>}
\boldsymbol{\Xi}_{\<green>}$ but here the last $\mfg$ factor on the right-hand side should be interpreted via the isomorphism with $W_{\<green>,\<orange>}$ rather than $W_{\<green>}$ as in the first equality of \eqref{e:barUpsilonXi}.}
We now compute $\bar\Upsilon^F$ for all the trees appearing in \eqref{eq: trees for renorm constants}.
\begin{lemma}\label{lem:bUps-shortlist}
\begin{equs}\label{e:bUps-shortlist}
\bar\bUpsilon^F_{\<dblue>} [\<IXi^2green>](\mathbf A)
&= \bone_{\<dblue>\neq \<green>} 
[ \Psi_{\<green>},[\Psi_{\<green>}, A_{\<dblue>}]] \\[.2em]
\bar\bUpsilon^F_{\<dblue>}[\<I[I'Xi]I'Xi_typed>](\mathbf A)
& =
 (2\delta_{\<green>,\<dblue>}\delta_{\<orange>,\<red>}  -\delta_{\<dblue>,\<red>}\delta_{\<orange>,\<green>})\,[[2\delta_{\<green>,\<orange>}A_{\<red>} - \delta_{\<red>,\<orange>}A_{\<green>} \,,\, 
\mcb{I}_{\<orange>}(\Psi_{\<green>,\<red>})]\,,\,
\Psi_{\<green>,\<red>}]\\[.2em]
\bar\bUpsilon^F_{\<dblue>}[\<IXiI'[I'Xi]_typed>](\mathbf A)
& =(2\delta_{\<orange>,\<dblue>}\delta_{\<green>,\<red>}  -\delta_{\<dblue>,\<red>}\delta_{\<green>,\<orange>})
[\Psi_{\<green>},[2\delta_{\<green>,\<orange>}A_{\<red>} - \delta_{\<red>,\<orange>}A_{\<green>}, \mcb{I}_{\<orange>,\<red>}( \Psi_{\<green>,\<red>})]] 
\end{equs}
Moreover, for any $\tau$ of the form $\<IXiI'XXi_notriangle>$ or $\<IXXiI'Xi_notriangle>$, $\bar{\bUpsilon}^{F}_\mft[\tau]=0$ for every $\mft\in\Lab_+$.
\end{lemma}
\begin{proof} 
Let $\tau = \mcb{I}_{\<red>}(\tau_{1}) \mcb{I}_{\<orange>}(\tau_{2})$ for 
trees $\tau_1$ and $\tau_2$.
Then, by~\eqref{e:def-Upsilon-added},
\begin{equs}
\bar\Upsilon^F_{\<dblue>} [\tau](\mathbf A)
& = \bone_{\<red>=\<orange>\neq \<dblue>} \Big([\Upsilon^F_{\<red>}[\tau_{1}](\mathbf A),[\Upsilon^F_{\<orange>}[\tau_{2}](\mathbf A), A_{\<dblue>}]] 
\\
&\qquad \qquad\qquad+ [\Upsilon^F_{\<orange>}[\tau_{2}](\mathbf A),[\Upsilon^F_{\<red>}[\tau_{1}](\mathbf A), A_{\<dblue>}]]\Big)
\\
&  + \bone_{\<red>=\<dblue>\neq\<orange>}
\Big([\Upsilon^F_{\<orange>}[\tau_{2}](\mathbf A),[A_{\<orange>},\Upsilon^F_{\<red>}[\tau_{1}](\mathbf A)]] 
\\
&\qquad \qquad\qquad+ [A_{\<orange>},[\Upsilon^F_{\<orange>}[\tau_{2}](\mathbf A),\Upsilon^F_{\<red>}[\tau_{1}](\mathbf A)]]\Big)
\\
&+ \bone_{\<dblue>=\<orange>\neq\<red>} \Big([\Upsilon^F_{\<red>}[\tau_{1}](\mathbf A),[A_{\<red>},\Upsilon^F_{\<orange>}[\tau_{2}](\mathbf A)]] 
\\
&\qquad \qquad\qquad+ [A_{\<red>},[\Upsilon^F_{\<red>}[\tau_{1}](\mathbf A),\Upsilon^F_{\<orange>}[\tau_{2}](\mathbf A)]]\Big)\;.
\end{equs}
Specifying to
 $\tau=\<IXi^2_typed>$,
 using \eqref{e:barUpsilonXi} in the above identity gives
\begin{equs}[e:Upsilon-cherry-asym]
\bar\Upsilon^F_{\<dblue>} [\<IXi^2_typed>](\mathbf A)
& = \bone_{\<orange>=\<green>\neq \<dblue>} \Big(
[ \Psi_{\<orange>},[\Psi_{\<green>}, A_{\<dblue>}]] 
+ [\Psi_{\<green>},[\Psi_{\<orange>}, A_{\<dblue>}]]\Big)
\\
& \qquad + \bone_{\<orange>=\<dblue>\neq\<green>}
\Big([\Psi_{\<green>},[A_{\<green>},\Psi_{\<orange>}]] 
+ [A_{\<green>},[\Psi_{\<green>},\Psi_{\<orange>}]]\Big)
\\
&\qquad + \bone_{\<dblue>=\<green>\neq\<orange>}
\Big([\Psi_{\<orange>},[A_{\<orange>},\Psi_{\<green>}]] 
+ [A_{\<orange>},[\Psi_{\<orange>},\Psi_{\<green>}]]\Big)\;.
\end{equs}
Therefore,\footnote{Note that $\bone_{\<dblue>\neq \<green>}$ here is necessary, because different colours only means ``not necessarily identical''!}
\begin{equ}[e:Ups-cherry]
\bar\Upsilon^F_{\<dblue>} [\<IXi^2green>](\mathbf A)
 = 2 \, \bone_{\<dblue>\neq \<green>} 
[ \Psi_{\<green>},[\Psi_{\<green>}, A_{\<dblue>}]]  \;.
\end{equ}
By \eqref{eq:sym_factor} we have $S(\<IXi^2green>) = (2!) \,S(\<Xigreen>)^2=2$ and so the first identity of \eqref{e:bUps-shortlist} follows.

Before moving onto the second identity we recall that, again using \eqref{e:def-Upsilon-added}, 
\begin{equs}[e:calculationSimple]
\bar\Upsilon^F_{\<orange>} [\<I'Xigreen-red>](\mathbf A)
&= [2\delta_{\<green>,\<orange>}A_{\<red>}-\delta_{\<red>,\<orange>}A_{\<green>}, 
\Upsilon^F_{\<green>,\<red>}[\<Xigreen>](\mathbf A)]
\\
&= [2\delta_{\<green>,\<orange>}A_{\<red>}-\delta_{\<red>,\<orange>}A_{\<green>},\Psi_{\<green>,\<red>}]\;.
\end{equs}
It follows that 
 \begin{equs}
\bar\Upsilon^F_{\<dblue>}[\<I[I'Xi]I'Xi_typed>](\mathbf A)
&=
(2\delta_{\<green>,\<dblue>}\delta_{\<orange>,\<red>}-\delta_{\<dblue>,\<red>}\delta_{\<orange>,\<green>})
[\Upsilon^F_{\<orange>}[\<I'Xigreen-red>](\mathbf A),
\Upsilon^F_{\<green>,\<red>}[\<Xigreen>](\mathbf A)] \\
&=
(2\delta_{\<green>,\<dblue>}\delta_{\<orange>,\<red>}-\delta_{\<dblue>,\<red>}\delta_{\<orange>,\<green>})\,[[2\delta_{\<green>,\<orange>}A_{\<red>} - \delta_{\<red>,\<orange>}A_{\<green>} \,,\, 
\mcb{I}_{\<orange>}(\Psi_{\<green>,\<red>})]\,,\,
\Psi_{\<green>,\<red>}]\;,
\end{equs}
where again, in $\mcb{I}_{\<orange>}(\Psi_{\<green>,\<red>}) \in\CT[\<I[I'Xi]_typed>] \otimes \mfg$,
 the operator $\mcb{I}_{\<orange>}$ is acting only on the left factor of $\Psi_{\<green>,\<red>} \in \CT[\<I'Xigreen-red>] \otimes \mfg$.
We then obtain the second identity of \eqref{e:bUps-shortlist} since $S(\<I[I'Xi]I'Xi_typed>) = 1$. 

For the third identity we recall that, by \eqref{e:def-Upsilon-added} and \eqref{e:calculationSimple},
\begin{equ}
\Upsilon^F_{\<orange>,\<red>}[\<I'Xigreen-red>](\mathbf A)
= 
\mcb{I}_{\<orange>,\<red>} (\bar\Upsilon^F_{\<orange>}[\<I'Xigreen-red>])(\mathbf A)
= \big[2\delta_{\<green>,\<orange>}A_{\<red>}-\delta_{\<red>,\<orange>}A_{\<green>},\mcb{I}_{\<orange>,\<red>}(\Psi_{\<green>,\<red>})\big]\;,
\end{equ}
so that 
\begin{equs}
\bar\Upsilon^F_{\<dblue>}[\<IXiI'[I'Xi]_typed>](\mathbf A)
&=
(2\delta_{\<orange>,\<dblue>}\delta_{\<green>,\<red>}-\delta_{\<dblue>,\<red>}\delta_{\<green>,\<orange>})
[\Upsilon^F_{\<green>}[\<Xigreen>](\mathbf A),
\Upsilon^F_{\<orange>,\<red>}[\<I'Xigreen-red>](\mathbf A)] \\
&=
(2\delta_{\<orange>,\<dblue>}\delta_{\<green>,\<red>}-\delta_{\<dblue>,\<red>}\delta_{\<green>,\<orange>})
\big[\Psi_{\<green>},\big[2\delta_{\<green>,\<orange>}A_{\<red>} - \delta_{\<red>,\<orange>}A_{\<green>}, \mcb{I}_{\<orange>,\<red>} (\Psi_{\<green>,\<red>})\big]\big]\;.
\end{equs}
Since $S(\<IXiI'[I'Xi]_typed>) = 1$ we obtain the desired result.

The final claim of the lemma regarding trees of the form $\<IXiI'XXi_notriangle>$ or $\<IXXiI'Xi_notriangle>$ follows immediately from the induction \eqref{e:def-Upsilon-added} combined with \eqref{e:barUpsilonXXi}.
\end{proof}

\subsubsection{Putting things together and proving Theorem~\ref{thm:local_exist}}\label{subsec:proof_of_localexist}
Before proceeding, we give more detail on how to use our notation for computations. 
We note that, given any $w \in \mfg \otimes_{s} \mfg$, one can use the isomorphism\footnote{Again, this is only canonical up to permutation of the factors, but doesn't matter since $w$ is symmetric.}  
$\CT[\<IXiI'Xi_typed2_notriangle>]^{\ast} \simeq \mfg \otimes \mfg$ to view $w$ as acting on the expression \eqref{eq:nonlinear} via an adjoint action, namely, 
\begin{equ}[e:adjointAction]
(w \otimes \id_{\mfg})
\big[\Psi_{\<green>},[h,\Psi_{\<green>,\<red>}]\big] 
=
-(w \otimes \id_{\mfg})
\big[\Psi_{\<green>},[\Psi_{\<green>,\<red>},h]\big]
=
-\ad_{w}h
\;.
\end{equ}
\begin{lemma}\label{lemma:final_ups_calc}
\begin{equ}\label{eq:final_ups_calc}
\sum_{\tau \in \mfT_{-}(R)}
(\ell_{\BPHZ}^{\eps}[\tau] \otimes \id_{\mfg})
\bar\bUpsilon^{F}_{\<dblue>}[\tau](\mathbf{A})
 = \lambda C_{\sym}^{\eps} A_{\<dblue>}\;,
\end{equ}
where $\lambda$ is the constant given in Remark~\ref{rem:action_of_casimir} and $C_{\sym}^{\eps}$ is as in \eqref{e:defConstants}. 
\end{lemma}
\begin{proof}
By \eqref{e:bUps-shortlist} and Lemma~\ref{lem:tree_calcs}, 
\begin{equs}
{}&\sum_{\<green>,\<red>,\<orange>} 
(\ell_{\BPHZ}^{\eps}[\<I[I'Xi]I'Xi_typed>] \otimes \id_{\mfg})
\bar{\bUpsilon}_{\<dblue>}[\<I[I'Xi]I'Xi_typed>](\mathbf{A})
\\
& = - (\hat C^\eps \Cas \otimes \id_{\mfg})\Big(
4[[A_{\<dblue>},\mcb{I}_{\<dblue>} \Psi_{{\<dblue>},{\<dblue>}}], \Psi_{{\<dblue>},{\<dblue>}}]
- \sum_{{\<red>}} 2 [[A_{\<dblue>},\mcb{I}_{\<red>} \Psi_{{\<dblue>},{\<red>}}], \Psi_{{\<dblue>},{\<red>}}]
\\
& \qquad\qquad \qquad - \sum_{\<orange>} 2 [[A_{\<dblue>},\mcb{I}_{\<orange>} \Psi_{\<orange>,{\<dblue>}}], \Psi_{\<orange>,{\<dblue>}}]
+[[A_{\<dblue>},\mcb{I}_{\<dblue>} \Psi_{{\<dblue>},{\<dblue>}}], \Psi_{{\<dblue>},{\<dblue>}}]
\Big)
\\
&= 3 \hat C^\eps \ad_{\Cas} A_{{\<dblue>}}\;,
\end{equs}
where  we used $d=2$ to sum over the free indices ${\<red>}$ and $\<orange>$, as well as 
\eqref{e:adjointAction} in the last step.

Also, by \eqref{e:bUps-shortlist} and Lemma~\ref{lem:tree_calcs} 
\begin{equs}
{}&\sum_{\<green>,\<orange>,\<red>} 
(\ell_{\BPHZ}^{\eps}[\<IXiI'[I'Xi]_typed>] \otimes \id_{\mfg})
\bar{\bUpsilon}_{\<dblue>}[\<IXiI'[I'Xi]_typed>](\mathbf{A})
\\
& = (\hat C^\eps \Cas  \otimes \id_{\mfg})\Big(
4[\Psi_{\<dblue>},[A_{\<dblue>},\mcb{I}_{{\<dblue>},{\<dblue>}} \Psi_{{\<dblue>},{\<dblue>}}]]
-  2 [\Psi_{{\<dblue>}},[A_{\<dblue>},\mcb{I}_{\<dblue>,\<dblue>} \Psi_{\<dblue>,\<dblue>}]]
\\
& \qquad\qquad 
- \sum_{\<green>} 2[\Psi_{\<green>},[A_{\<dblue>},\mcb{I}_{\<green>,{\<dblue>}} \Psi_{\<green>,{\<dblue>}}]]
+ [\Psi_{\<dblue>},[A_{\<dblue>},\mcb{I}_{\<dblue>,\<dblue>} \Psi_{\<dblue>,\<dblue>}]]
\Big)
\\
&= \hat C^\eps \ad_{\Cas} A_{\<dblue>}\;,
\end{equs}
where $d=2$ and \eqref{e:adjointAction} are used again.
Finally by \eqref{e:bUps-shortlist} and Lemma~\ref{lem:tree_calcs}, 
\begin{equ}
\sum_{\<green>} (\ell_{\BPHZ}^{\eps}[\<IXi^2green>] \otimes \id_{\mfg}) 
 \bar\bUpsilon^F_{\<red>} [\<IXi^2green>](\mathbf A) 
= -\bar{C}^\eps \ad_{\Cas} A_{\<red>}\;.
\end{equ}
Adding these three terms and recalling Remark~\ref{rem:action_of_casimir} gives \eqref{eq:final_ups_calc}.
\end{proof}
With these calculation, we are ready for the proof of Theorem~\ref{thm:local_exist}.
\begin{remark}
It would be desirable to apply the black box convergence theorem~\cite[Thm.~2.21]{BCCH21} directly.
However, we are slightly outside its scope since we are 
working with non-standard spaces $\Omega^1_\alpha$ and 
are required to show continuity at time $t=0$ for the solution $A^\eps\colon[0,T]\to \Omega^1_\alpha$.
Nonetheless, we can instead use several more general results from~\cite{BCCH21, CH16, Hairer14}.
\end{remark}
\begin{proof}[of Theorem~\ref{thm:local_exist}]
Consider the lifted equation associated to~\eqref{eq:SPDE_for_A} in the bundle of modelled distributions $\cD^{\gamma,\eta}_{-\kappa} \ltimes \mathscr{M}$ for $\gamma>1+\kappa$ and $\eta\in(\frac\alpha4-\frac12,\alpha-1]$, where $\alpha\in(\frac23,1)$ as before.
We further take $\kappa\in(0,\frac14)$ sufficiently small such that $-2\kappa>\eta$.
Note that $\gamma>1+\kappa$ and $\eta>-\frac12$ ensure that, by the same argument as in~\cite[Sec.~9.4]{Hairer14}, the lifted equation admits a unique fixed point $\mcA \in \cD^{\gamma,\eta}_{-\kappa}$ and is locally Lipschitz in $(a, Z)\in\Omega\mcC^{\eta} \times \mathscr{M}$, where $\mathscr{M}$ is the space of models on the associated regularity structure equipped a metric analogous to that in~\cite[Prop.~9.8]{Hairer14}.
Specialising to $Z = Z^{\eps}_{\BPHZ}$, the computation of Lemma~\ref{lemma:final_ups_calc} along with \cite[Thm.~5.7]{BCCH21} (and its partial reformulation in the vector case via Proposition~\ref{prop:renormalisation_preserves_coherence}) show that the reconstruction $A^\eps \eqdef \mcR\mcA(a,Z^{\eps}_{\BPHZ})$ is the maximal solution in $\Omega\mcC^\eta$ to the PDE~\eqref{eq:SPDE_for_A} starting from $a$ with $C$ replaced by $C^{\eps}$ given by 
\begin{equ}\label{eq:BPHZ_mass}
C^{\eps} = \lambda C^{\eps}_{\sym} \;,
\end{equ}
where $C^{\eps}_{\sym}$ is as in Lemma~\ref{lem:tree_calcs}.

We now show that $A^\eps$ converges in the space $(\Omega^1_\alpha)^\sol$.
To this end, let us decompose $A^\eps = \Psi^\eps + B^\eps$, where $\Psi^\eps$ solves $\partial_t \Psi^\eps = \Delta\Psi^\eps + \xi^\eps$ on $\R_+\times \T^2$ with initial condition $a\in\Omega^1_\alpha$.
Write also $\Psi$ for the solution to $\partial_t\Psi = \Delta\Psi + \xi$ on $\R_+\times \T^2$ with initial condition $a\in\Omega^1_\alpha$.
Combining Theorem~\ref{thm:SHE_regularity} and Proposition~\ref{prop:strongly_continuous}, we see that $\Psi\in\mcC(\R_+,\Omega^1_\alpha)$,
and, by Corollary~\ref{cor:SHE_mollif_converge}, $\Psi^\eps \to \Psi$ in $\mcC(\R_+,\Omega^1_\alpha)$.

Observe that $B^\eps= \mcR(\mcP\bone_+ H(\mcA))$, where $H$ is the part of the non-linearity $F$ in~\eqref{e:YM-Fi} involving the Lie brackets, and that $H(\mcA) \in \cD^{\gamma-1-\kappa,2\eta-1}_{-1-2\kappa}$.
It follows from the convergence of models given by Lemma~\ref{lem:conv_of_models}, the continuity of the reconstruction map,  and~\cite[Thm.~7.1]{Hairer14}, that $B^\eps$ converges in $(\Omega\CC^{0,\alpha/2})^\sol$ with $B^\eps(0)=0$,
where we used the condition $\eta>\frac\alpha4-\frac12\Leftrightarrow 2\eta+1>\frac\alpha2$.
Due to the embedding $\Omega\mcC^{0,\alpha/2} \hookrightarrow \Omega^1_\alpha$ (Remark~\ref{rem:Holder_Omega_embedding})
and the fact that we can start the equation from any element of $\Omega^1_\alpha$,
it readily follows that $A^\eps = \Psi^\eps+B^\eps$ converges in $(\Omega^1_\alpha)^\sol$ as $\eps\to 0$.

Finally, observe that we can perturb the constants $C^{\eps}$ in \eqref{eq:SPDE_for_A} by any bounded quantity while retaining convergence of maximal solutions to \eqref{eq:SPDE_for_A} and so, thanks to the convergence of \eqref{eq:BPHZ_mass} promised by Lemma~\ref{lem:tree_calcs}, we obtain the desired convergence for \emph{any} family $(C^\eps)_{\eps\in(0,1]}$ such that $\lim_{\eps\to 0}C^\eps$ exists and is finite.
\end{proof}
\begin{remark}\label{rem:sym_indep_of_mollifier}
The proof of Theorem~\ref{thm:local_exist} allows us to make the following important observation about dependence of the solution on the mollifier. Namely, given any fixed constant $\delta C \in \R$, the limiting maximal solution to \eqref{eq:SPDE_for_A} obtained as one takes $\eps \downarrow 0 $ with $C= \delta C + \lambda C_{\sym}$ is independent of the choice of mollifier, namely all dependence of the solution on the mollifier is cancelled by $C_{\sym}$'s dependence on the mollifier.\footnote{This is because $\lambda C_{\sym}$ is the renormalisation arising from limiting BPHZ model $Z_{\BPHZ}$ which is independent of the mollifier.}    

Moreover, recall that $C_{\sym}$ is the $\eps \downarrow 0$ limit of the right-hand side of \eqref{eq:integrals_for_csym} and 
\[
\lim_{\eps \downarrow 0} \int \mrd z\ (K*K^\eps)(z)(Q*\moll^\eps)(z)
\quad\textnormal{and}\quad
\lim_{\eps \downarrow 0} \int \mrd z\ (Q*K^\eps)(z)K^\eps(z)
\]
are both independent of $\moll$,
where $Q$ is as in \eqref{e:KdeltaQ}.
In particular, if one chooses 
\[
C =
\lambda \lim_{\eps \downarrow 0} 
\int \mrd z\ \moll^\eps(z) (K*K^\eps)(z)
\]
then the limiting solution to \eqref{eq:SPDE_for_A} is independent of the mollifier $\moll$.  
\end{remark}

\section{Gauge covariance}\label{sec:gauge_equivar}
The aim of this section is to show that the projected process $[A(t)]$ on the 
orbit space is again a Markov process, which is a very strong form of gauge invariance.
The first three subsections will be devoted to proving Theorem~\ref{thm:gauge_covar} \dash 
most of our work will be devoted to part~\ref{pt:gauge_covar} and we will obtain part~\ref{pt:indep_moll} afterwards 
by a short computation with renormalisation constants. 
We close with Section~\ref{subsec:Markov_process} where we construct the desired Markov process. 
\subsection{The full gauge transformed system of equations}\label{subsec:preprocessing_system}

One obstruction we encounter when trying to directly treat the systems \eqref{eq:SPDE_for_B} and  \eqref{eq:SPDE_for_bar_A} using currently available tools 
is that the evolution for $g$ takes place in the non-linear space $G$. 
Fortunately, it is possible to rewrite the equations in such a way that the role of $g$ is 
played by variables that both live in linear spaces. 

Recall that $\mathring{\mfG}^{0,\alpha}$ was defined as a quotient of $\mfG^{0,\alpha}$ but we embed it into a linear space by appealing to the inclusion
\begin{equ}\label{eq:inclusion_of_reduced_group}
\mathring{\mfG}^{0,\alpha} \hookrightarrow 
\hat{\mathfrak{G}}^{0,\alpha}
=
 \CC^{0,\alpha} \big(\T^2,L(\mfg,\mfg) 
\big) 
\times
\Omega^1_{\gr\alpha}\;,
\end{equ}
given by $[g] \mapsto (U,h)$ where, for $g \in \hat{\mathfrak{G}}^{0,\alpha}$, 
one defines
\begin{equ}\label{eq:h_and_U_def}
h \eqdef (\mrd g )g^{-1} = -0^{g}
\qquad\textnormal{and}\qquad
U
\eqdef
\Ad_{g}\;.
\end{equ} 
\begin{remark}\label{rem:Uh_vs_g}
Note that the inclusion in \eqref{eq:inclusion_of_reduced_group} is a homeomorphism onto its image, so we often identify $\mathring{\mfG}^{0,\alpha}$ with its image under this inclusion \eqref{eq:inclusion_of_reduced_group}, and write $(U,h)$ rather than $[g]$ to denote an element of $\mathring{\mfG}^{0,\alpha}$ where it is understood that $U$ and $h$ are of the form \eqref{eq:h_and_U_def} for some $g \in \mfG^{0,\alpha}$. 
With this identification, $\mathring{\mfG}^{0,\alpha}$ is a closed, nonlinear subset of $\hat{\mfG}^{0,\alpha}$. 
\end{remark}

We can rewrite the first equation in~\eqref{eq:SPDE_for_B} and~\eqref{eq:SPDE_for_bar_A} respectively as
\begin{equ}[eq:B_final]
\d_t B_i = \Delta B_i  + [B_j, 2\d_j B_i - \d_i B_j + [B_j,B_i]]
+ U \chi^{\eps} \ast \xi_i
+ C B_i + C h_{i}
\end{equ}
and
\begin{equ}[eq:bar_A_final]
\d_t \bar{A}_i = \Delta \bar{A}_i  + [\bar{A}_j, 2\d_j \bar{A}_i - \d_i \bar{A}_j + [\bar{A}_j,\bar{A}_i]]
+ \chi^{\eps} \ast (\bar U \xi_i) + C \bar{A}_i + (C-\bar C) \bar{h}_{i}
\end{equ}
where $U$ and $h$ are defined from the $g$ of  \eqref{eq:SPDE_for_B} via \eqref{eq:h_and_U_def} and $\bar{U}$ and $\bar{h}$ are defined analogously from the $\bar{g}$ of \eqref{eq:SPDE_for_bar_A}. 

The following lemma identifies the equations solved by $U,h$ and $\bar U,\bar h$.
\begin{lemma}\label{lemma:linear_g_system}
Consider any smooth $B\colon(0,T]\to\Omega\mcC^\infty$ and suppose $g$ solves the second equation in~\eqref{eq:SPDE_for_B}.
Then $h$ and $U$ defined by \eqref{eq:h_and_U_def} satisfy 
\begin{equs}[e:final_system]
\d_t h_{i} &= 
\Delta h_i - [h_j,\d_j h_i] + [[B_j, h_j],h_i] + \d_i [B_j, h_j]\;,
\\
\d_{t} U &=
\Delta U - [h_j,[h_j, \cdot]] \circ U+ [[B_j, h_j],\cdot] \circ U \;.
\end{equs}
In particular $(h,U)$ (resp. $(\bar h,\bar U)$)
defined below \eqref{eq:B_final}-\eqref{eq:bar_A_final}
satisfies the above equation with $B$ solving \eqref{eq:B_final}
(resp. with $B$ replaced by $\bar A$ which solves \eqref{eq:bar_A_final}).
\end{lemma}
\begin{proof}
By definition \eqref{eq:h_and_U_def} of $h$ and the equation for $g$ in \eqref{eq:SPDE_for_B}, one has the following identities
\begin{equs}[e:convenient-identities]
 (\d_t g)g^{-1} &= \div h + [B_j, h_j]\;,\\
\d_j h_i - \d_i h_j &= [h_j, h_i]\;, \\
\Delta h_i - \d_i \div h &= \d_j [h_j,h_i]\;,
\end{equs}
where the last identity follows from the second.	
One then obtains
\begin{equs}
\d_t h_i &= [(\d_t g)g^{-1},h_i] + \d_i ((\d_t g)g^{-1}) \\
&= [\div h + [B_j, h_j],h_{i}] + \d_i \div h + \d_i [ B_j, h_j] \\
&= \Delta h_i - [ h_j,\d_j h_i] + [[B_j, h_j],h_i] + \d_i [B_j, h_j]\;.
\end{equs}
For the $U$ equation, we start by noting that 
\begin{equ}[e:partialU]
\d_i U =  [h_i, \cdot] \circ U
\end{equ}
and therefore
\begin{equ}[e:DeltaU]
\Delta U = [\div h, \cdot] \circ U + [h_i,[h_i, \cdot]] \circ U \;.
\end{equ}
By the first identity in \eqref{e:convenient-identities}
\begin{equ}
\d_t U = [(\d_t g)g^{-1}, \cdot] \circ U
= [\div h + [B_j, h_j],  \cdot] \circ U\;,
\label{eq:Ad}
\end{equ}
and the claim follows from \eqref{e:DeltaU}.
\end{proof}
%
%
%
\subsection{Regularity structure for the gauge transformed system}
\label{subsec:regstruct_for_system}
To recast~\eqref{eq:B_final},~\eqref{eq:bar_A_final}, and~\eqref{e:final_system} in the language of regularity structures, we use the label sets 
\[
\mfL_+ \eqdef \{\mfa_{i},\mfh_i, \mfm_{i}\}_{i=1}^{d} \cup \{\mfu\}  \quad\textnormal{and}\quad \mfL_- \eqdef \{\bar{\mfl}_{i},\mfl_i\}_{i=1}^{d}\;.
\]
Our approach is to work with {\it one single} regularity structure to study the systems $(B,U,h)$ and $(\bar A, \bar U, \bar h)$, allowing us to compare their solutions at the abstract level of modelled distributions. 

Our particular choice of label sets and abstract non-linearities also involves some pre-processing to allow us to use the machinery of Section~\ref{subsec:renorm_eq} to obtain the form of our renormalised equation.  
  The label $\mfh_{i}$ indexes the solutions $h_{i}$ or $\bar{h}_{i}$, $\mfu$ indexes the solutions $U$ or $\bar{U}$, and $\mfl_{i}$ indexes the noise $\xi_{i}$ (while $\bar\mfl_{i}$ indexes a noise mollified at scale $\eps$, see below). 
 
The other labels are used to describe the $B_{i}$ and $\bar A_i$ equations~\eqref{eq:B_final} and~\eqref{eq:bar_A_final}. 
To explain our strategy, we first note that (ignoring for the moment the contribution coming from the initial condition) 
the equation for $\bar{A}$ can be written as the integral fixed point equation 
\begin{equ}
\bar{A}_i =  G \ast
\big( [\bar{A}_j, 2\d_j \bar{A}_i - \d_i \bar{A}_j + [\bar{A}_j,\bar{A}_i]]
+
C\bar{A}_{i} + (C - \bar{C}) \bar{h}_{i}
\big) 
+
G_{\eps} \ast (\bar{U} \xi_{i})\;,
\end{equ}
where $G_{\eps} = \chi^{\eps} \ast G$. 
While this can be cast as an abstract fixed point problem at the level of jets\slash modelled distributions, it does not quite fit into the framework of Section~\ref{subsec:renorm_eq} since it involves multiple kernels on the right-hand side. 
We can deal with this problem by introducing a component $\mfm_{i}$ to index a new component of our solution that is only used to represent the term $G_{\eps} \ast (\bar{U} \xi_{i})$. 
The label $\mfa_{i}$ then represents the remainder $\bar A_i - G_{\eps} \ast (\bar{U} \xi_{i})$; see~\eqref{e:convention-bfA} and the non-linearity $\bar F$ below.

Turning to the equation for $B$, the corresponding fixed point problem is
\[
B_{i}
=
G \ast
\big( 
[B_j, 2\d_j B_i - \d_i B_j + [B_j,B_i]]
+
CB_{i}
+
Ch_{i}
+
U \chi^{\eps} \ast \xi_i
\big)\;.
\]
Note that we cannot combine the mollification by $\chi^{\eps} $ with a kernel that acts on the whole RHS above, so we instead use the 
label $\bar{\mfl}_{i}$ to represent $\chi^{\eps} \ast \xi_i$ which we treat, at a purely algebraic level, as a 
completely separate noise from $\xi_{i}$.   

Turning to our kernel space assignment $\mathcal{K} = (\mathcal{K}_{\mft})_{\mft \in \Lab_{+}}$ and target space assignment  $(W_{\mft})_{\mft \in \mfL}$ we set
\begin{equ}[e:system-target space assignment]
\mathcal{K}_{\mft} = \R \;\; 
\forall \mft \in \Lab_{+} 
\quad \enskip
\text{and}
\quad \enskip
W_{\mft} \eqdef 
\begin{cases}
\mfg
& \mft = \mfa_{i}, \mfh_{i}, \mfl_{i}, \bar{\mfl}_{i}, \textnormal{ or }  \mfm_{i}\;,\\
L(\mfg,\mfg) 
&\mft = \mfu\;.
\end{cases} 
\end{equ} 
The space assignment $(V_{\mft})_{\mft \in \mfL}$ used to build our regularity structure is then again given by \eqref{e:defV}.
We also define $\deg\colon \mfL \rightarrow \R$ by setting
\[
\deg(\mft) \eqdef 
\begin{cases}
2 - \kappa 
& \mft  = \mfa_{i} \textnormal{ or } \mfm_{i}\;,\\
2
& \mft = \mfh_{i} \textnormal{ or } \mfu\;,\\
-d/2 - 1 - \kappa 
&\mft = \mfl_{i} \textnormal{ or } \bar{\mfl}_{i}\;,
\end{cases}
\] 
where $\kappa \in (0,\frac{1}{12})$ is such that $2\kappa<1-\alpha$ for $\alpha\in(\frac23,1)$ as before.

The systems of equations~\eqref{eq:B_final},~\eqref{eq:bar_A_final}, and~\eqref{e:final_system} and earlier discussion about the roles of our labels lead us to the rule $\mathring{R}$ given by setting\footnote{The node type $\{\mfu \mfl_{i}\} \in \mathring{R}(\mfa_{i})$ does not appear in the expansion of~\eqref{eq:B_final} or~\eqref{eq:bar_A_final} individually but naturally appears when we want to compare their fixed point problems, see \eqref{e:GUhatXi-compare}.}  
\begin{equs}[e:rule-gauged]
\mathring{R}(\mfl_{i}) &= \mathring{R}(\bar{\mfl}_{i}) = \{\emptyset\}\;,\quad 
\mathring{R}(\mfm_{i}) = \big\{\mfu \mfl_{i} \big\}\;,
\\
\mathring{R}(\mfu) &= \{\mfu\mfh_j^2,\; \mfu\mfq_j\mfh_j\,:\, \mfq \in \{\mfa,\mfm\},\, j \in [d]\}\;,
\\
\mathring{R}(\mfh_i) &= \{\mfh_j \d_j \mfh_i,\; \mfq_j\mfh_j\mfh_i,\; \mfh_j \d_i \mfq_j,\; \mfq_j \d_i\mfh_j\,:\, \mfq \in \{\mfa,\mfm\},\, j \in [d]\}\;,
\\
\mathring{R}(\mfa_{i}) &= 
\{\mfq_{i},\; \mfh_{i},\; \mfq_i \hat{\mfq}_{j} \tilde{\mfq}_{j},\; \mfq_j \d_i \hat{\mfq}_j,\; \mfq_j \d_j \hat{\mfq}_i,\; \mfu \mfl_{i},\; \mfu \bar{\mfl}_{i}:\,  \mfq,\hat{\mfq},\tilde{\mfq} \in \{\mfa,\mfm\},\, j \in [d]\}\;.
\end{equs}
Here we are using monomial notation for node types: a type $\mft \in \Lab$ should be associated with $(\mft,0)$ and the symbol $\d_{j}\mft$ represents $(\mft,e_{j})$. 
We write products to represent multisets, for instance $ \mfa_j^2\d_{k}\mfa_i = \{ (\mfa_{i},e_{k}), (\mfa_{j},0), (\mfa_{j},0)\}$. 
We write $\mfq$, $\hat{\mfq}$, and $\tilde{\mfq}$ as dummy symbols since any occurrence of $B_{i}$ or $\bar{A}_{i}$ can correspond to an occurrence of $\mfa_{i}$ or $\mfm_{i}$. 

It is straightforward to check that $\mathring{R}$ is subcritical and as in Section~\ref{subsec:Recasting vec reg struct}
the rule $\mathring{R}$ has a smallest normal extension which admits a completion $R$ which is also subcritical. This is the rule
that is used to define the set of trees $\mfT(R)$ which is used to build our regularity structure.

We adopt conventions analogous to those of Remark~\ref{rem:switching_to_standard_notation} and \eqref{eq:switching_to_standard_notation}, writing (using our monomial notation)
\begin{equs}[e:convention-bfA]
{}&\pr{\bar{A}_{i}} = \pr{\mathbf{A}_{\mfa_{i}}} + \pr{\mathbf{A}_{\mfm_{i}}}\,,\, 
\pr{\d_{j}\bar{A}_{i}} = \pr{\mathbf{A}_{\d_{j}\mfa_{i}}} + \pr{\mathbf{A}_{\d_{j}\mfm_{i}}}\,,\,\\ 
{}&\pr{B_{i}} = \pr{\mathbf{A}_{\mfa_{i}}} \,,\,
\pr{\d_{j}B_{i}} = \pr{\mathbf{A}_{\d_{j} \mfa_{i}}}\,,\,
 \pr{U} = \pr{\bar{U}} =\pr{\mathbf{A}_{\mfu}} \,,\, \pr{\partial_{i}U} = \pr{\partial_{i}\bar{U}} = \pr{\mathbf{A}_{\partial_{i}\mfu}}\;,\\ 
{}&\chi^{\eps} \ast \xi_{i}
=
\xi_{\bar{\mfl}_{i}}\,,\ 
\xi_{i}
=
\xi_{\mfl_{i}}\,,\, 
\pr{h_{i}} = \pr{\bar{h}_{i}}  =\pr{\mathbf{A}_{\mfh_{i}}}   \,,\ 
 \pr{\d_{j}h_{i}} = \pr{\d_{j} \bar{h}_{i}} =\pr{\mathbf{A}_{\d_{j}\mfh_{i}}} \,. 
\end{equs} 
Here, we choose to typeset components of $\pr{\mathbf A}$ in \pr{purple} in order to be able to identify them at a glance
as a solution\footnote{Components of $\pr{\mathbf A}$ corresponding to the noise such as $\xi_{\mfl_{i}}$ and $\xi_{\bar{\mfl}_{i}}$ will be left in black since their values are not solution-dependent.}-dependent element. This will be convenient later on when we manipulate expressions belonging to 
$\CT \otimes W$ for some vector space $W$ (typically $W = \mfg$ or $W = L(\mfg,\mfg)$), in which case purple variables
are always elements of the second factor $W$. 
Note that, when referring to components
of $\pr{\mathbf{A}} = (\pr{\mathbf{A}_{o}})_{ o \in \CE}$ the symbols $\pr{\bar{U}}$, $\pr{\bar{h}}$, and $\pr{\d_{j} \bar{h}_{i}}$ are identical to their unbarred versions but we still use both notations depending on which system of equations we are working with. 

We now fix two non-linearities $F = \bigoplus_{\mft \in \Lab} F_{\mft},\ \bar{F} = \bigoplus_{\mft \in \Lab} \bar{F}_{\mft} \in \mcb{Q}$, which encode our systems \eqref{eq:B_final} and~\eqref{eq:bar_A_final} respectively. 
For some constants $\mathring{C}_{1}$, and $\mathring{C}_{2}$ to be fixed later\footnote{One will see that these constants are shifts of the constant $C$ by some finite constants that depend only on our truncation of the heat kernel $K$.} we set  $F_\mft$ and $\bar F_\mft$ to be $\id_\mfg$ for $\mft\in \Lab_-$ and 
\begin{equs}
F_{\mft}(\pr{\mathbf{A}})
\eqdef&
\begin{cases}
[\pr{B_j}, 2\pr{\d_j B_i} - \pr{\d_i B_j} + [\pr{B_j},\pr{B_i}]] + \mathring{C}_{1} \pr{B_{i}} + \mathring{C}_{2} \pr{h_{i}} + \pr{U} \chi^{\eps} \ast \xi_{i}& \textnormal{ if }
\mft = \mfa_{i}\;,\\
- [\pr{h_j},\pr{\d_jh_i}] + [[\pr{B_j}, \pr{h_j}],\pr{h_i}] + \pr{\d_i} [\pr{B_j}, \pr{h_j}] & \textnormal{ if }
\mft = \mfh_{i}\;,\\
- [\pr{h_j},[\pr{h_j}, \cdot]] \circ \pr{U} + [[\pr{B_j}, \pr{h_j}],\cdot] \circ \pr{U} & \textnormal{ if }
\mft = \mfu\;,\\
0 & \textnormal{ if }
\mft = \mfm_{i}\;.
\end{cases}
\end{equs}
(The term $\pr{\d_i} [\pr{B_j}, \pr{h_j}]$ should be interpreted by formally applying the Leibniz rule.)
For $\bar{F}_{\mft}(\pr{\mathbf{A}})$ we set
\begin{equs}[eq:bar_F_def]
\bar{F}_{\mft}(\pr{\mathbf{A}})
\eqdef&
\begin{cases}
[\pr{\bar{A}_j}, 2\pr{\d_j \bar{A}_i} - \pr{\d_i \bar{A}_j} + [\pr{\bar{A}_j},\pr{\bar{A}_i}]] + \mathring{C}_{1} \pr{\bar{A}_{i}} + \mathring{C}_{2} \pr{\bar{h}_{i}}& \textnormal{ if }
\mft = \mfa_{i}\;,\\
- [\pr{\bar{h}_j},\pr{\d_j \bar{h}_i}] + [[\pr{\bar{A}_j}, \pr{\bar{h}_j}],\pr{\bar{h}_i}] + \pr{\d_i} [\pr{\bar{A}_j}, \pr{\bar{h}_j}] & \textnormal{ if }
\mft = \mfh_{i}\;,\\
- [\pr{\bar{h}_j},[\pr{\bar{h}_j}, \cdot]] \circ \pr{\bar{U}}+ [[\pr{\bar{A}_j}, \pr{\bar{h}_j}],\cdot] \circ \pr{\bar{U}} & \textnormal{ if }
\mft = \mfu\;,\\
\pr{\bar{U}} \xi_{i} & \textnormal{ if }
\mft = \mfm_{i}\;.
\end{cases}
\end{equs}
For $j\in \{0\}\sqcup [d]$ we also introduce the shorthands
\begin{equs}
{}&\Xi_{i} = \mcb{I}_{(\mfl_{i},0)}(\bone),\ 
\overline{\Xi}_{i} = \mcb{I}_{(\bar{\mfl}_{i},0)}(\bone),\ 
\mcb{I}_{i,j}(\cdot) = \mcb{I}_{(\mfa_{i},j)}(\cdot),\ 
\bar{\mcb{I}}_{i,j}(\cdot) = \mcb{I}_{(\mfm_{i},j)}(\cdot),\\
{}&
\mcb{I}^{\mfh}_{i,j}(\cdot) = \mcb{I}_{(\mfh_{i},j)}(\cdot),\ 
\mcb{I}^{\mfu}(\cdot) = \mcb{I}_{(\mfu,0)}(\cdot)\;.
\end{equs}
When  $j=0$ in the above notation we sometimes suppress this index, 
for instance writing $\mcb{I}_{i}(\cdot)$ instead of $\mcb{I}_{i,0}(\cdot)$. 
\subsubsection{Kernel \slash noise assignments and BPHZ models} 
We write $K^{(\eps)} = (K_{\mft}^{(\eps)}: \mft \in \Lab_{+})$ for the kernel assignment given by setting
\begin{equ}[e:K-Keps-assign]
K_\mft^{(\eps)} = 
\left\{\begin{array}{cl}
	K & \text{for $\mft = \mfa_{i}, \mfh_{i}$, or $\mfu$,} \\
	K^\eps = K \ast \chi^{\eps} & \text{for $\mft = \mfm_{i}$.}
\end{array}\right.
\end{equ}
We also write $\mathscr{M}$ for the space of all models and, for $\eps \in [0,1]$, we write $\mathscr{M}_{\eps} \subset \mathscr{M}$ for the family of $K^{(\eps)}$-admissible models.\label{model page ref}

Note that in our choice of degrees we enforced $\deg(\mfa_{i}) = \deg(\mfm_{i}) = 2 - \kappa$ rather than $2$.
The reason is that this allows us to extract a factor $\eps^{\kappa}$ from any occurrence of $K - K^{\eps}$,
which is crucial for the estimates in Section~\ref{sec:Comparing fixed point problems} below.

We make this more precise now. 
Recall first the notion of a $\beta$-regularising kernel from \cite[Assumption~5.1]{Hairer14}.
We introduce some terminology so that we can use that notion in a slightly more quantitative sense. 
For $\beta, R > 0$, $r \ge 0$ we say that a kernel $J$ is $(r,R,\beta)$-regularising, if one can find a decomposition of the form \cite[(5.3)]{Hairer14} such that the estimates \cite[(5.4), (5.5)]{Hairer14} hold with the same choice of $C=R$ for all multi-indices $k,l$ with $|k|_{\s}, |l|_{\s} < r$.
We use the norms $\$ \act \$_{\alpha,m}$ on functions with prescribed singularities at the origin that were defined in \cite[Definition~10.12]{Hairer14}. 
If $J$ is a smooth function (except possibly at the origin), satisfies \cite[Assumption~5.4]{Hairer14} with for some $r \ge 0$, and is supported on the ball $|x| \le 1$, then it is straightforward to show that $J$ is $(r,2\$J\$_{\beta-|\s|,r},\beta)$-regularising. 
We then have the following key estimate. 
\begin{lemma}\label{lemma:control_of_kernels}
For any $m \in \N$ one has $\$K\$_{2,m} < \infty$ and there exists $R$ such that $K - K^{\eps}$ is $(m,\eps^{\kappa} R,2-\kappa)$-regularising for all $\eps \in [0,1]$. 
\end{lemma}
\begin{proof}
The first statement is standard.
The second statement follows from combining the first statement, our conditions on the kernel $K$, \cite[Lemma~10.17]{Hairer14}, and the observations made above. 
\end{proof}
We now turn to our random noise assignments.
In \eqref{eq:B_final} and \eqref{eq:bar_A_final}
both a mollified noise $\chi^{\eps} \ast \xi_{i} = \xi_{i}^{\eps}$ and an un-mollified noise $\xi_i$ appear. 
In order to start our analysis with smooth models, we replace the un-mollified noise with one mollified at scale $\delta$. 
In particular, given $\eps,\delta \in (0,1]$ we define a random noise assignment $\zeta^{\delta,\eps} = (\zeta_{\mfl}: \mfl \in \mfL_{-})$ by setting 
\begin{equ}
\zeta_{\mfl} = 
\left\{\begin{array}{cl}
	\chi^{\delta} \ast \xi_{i} = \xi_{i}^{\delta} & \text{for $\mfl = \mfl_{i}$} \\
	\chi^{\eps} \ast \xi_{i} = \xi_{i}^{\eps} & \text{for $\mfl = \bar{\mfl}_{i}$}\;.
\end{array}\right.
\end{equ}
We also define $Z^{\delta,\eps}_{\BPHZ} = (\Pi^{\delta,\eps},\Gamma^{\delta,\eps}) \in \mathscr{M}_{\eps}$ to be the BPHZ lift associated to the kernel assignment $K^{(\eps)}$ and random noise assignment $\zeta^{\delta,\eps}$.
We will first take $\delta \downarrow 0$ followed by $\eps \downarrow 0$ \dash the first limit is a minor technical point while the second limit is the limit referenced in part \ref{pt:gauge_covar} of Theorem~\ref{thm:gauge_covar}. 

Note that we have ``doubled'' our noises in our noise assignment by having two sets of noise labels $\{ \mfl_{i}\}_{i=1}^{d}$ 
and $\{ \bar{\mfl}_{i}\}_{i=1}^{d}$ \dash we will want to use the fact that these two sets of noises take values in the 
same space $\mfg$ (and in practice, differ only by mollification).  This is formalised by noting that there are canonical 
isomorphisms $\CT[\Xi_{i}] \simeq \mfg^{\ast} \simeq \CT[\bar{\Xi}_{i}]$ for each $i \in [d]$, which we combine into an isomorphism
\begin{equ}\label{eq:noise_iso}
\sigma\colon \bigoplus_{i=1}^{d} \CT[\Xi_{i}] 
\rightarrow \bigoplus_{i=1}^{d} \CT[\overline{\Xi}_{i}]\;.
\end{equ}
\subsubsection{$\eps$-dependent regularity structures}\label{subsec:eps_reg_structs}
In the framework of regularity structures,
 analytic statements regarding models and modelled distributions reference norms $\|\act \|_{\ell}$ on
 the vector space   $T_{\ell}$ of all elements of degree $\ell \in \deg(R) = \{ \deg(\tau): \tau \in \mfT(R)\}$
   \dash in our setting this is given by
\[
T_{\ell} = 
\bigoplus_{
\tau \in \mfT(\ell,R)}
\CT[\tau]\;,
\]
where $\mfT(\ell,R) = \{ \tau \in \mfT(R): \deg(\tau) = \ell\}$. 
In many applications the spaces $T_{\ell}$ are finite-dimensional and there is no need to 
specify the norm $\| \act \|_{\ell}$ on $T_{\ell}$ (since they are all equivalent).

While the spaces $T_{\ell}$ are also finite-dimensional in our setting, we want to encode the fact that $K^{\eps}$ is converging to $K$ and $\xi^{\eps}$ is converging to $\xi$ as $\eps \downarrow 0$ in a way that allows us to treat discrepancies between these quantities as small at the level of our abstract formulation of the fixed point problem. 
We achieve this by defining, for each $\ell \in \deg(R)$, a family of norms $\{ \| \act \|_{\ell, \eps}: \eps \in (0,1]\}$ on $T_{\ell}$.
Our definition will depend on a small parameter $\theta \in  (0,\kappa]$ which we treat as fixed in what follows. 

Heuristically and pretending for a moment that we are in the scalar noise setting,  we  define these $\| \act \|_{\ell, \eps}$ norms by performing a ``change of basis'' and writing out trees in terms of the noises $\Xi_{i}$, $\bar{\Xi}_{i} - \Xi_{i}$, operators $\mcb{I}_{i,p}$, $\bar{\mcb{I}}_{i,p} -\mcb{I}_{i,p}$, $\mcb{I}^{\mfh}_{i,p}$ and $\mcb{I}^{\mfu}$ 
instead of $\Xi_{i}$, $\bar{\Xi}_{i}$, $\mcb{I}_{i,p}$, $\bar{\mcb{I}}_{i,p}$, $\mcb{I}^{\mfh}_{i,p}$ and $\mcb{I}^{\mfu}$, respectively.  For instance, we rewrite
\[
\bar{\mcb{I}}_i (\Xi_i) =  (\bar{\mcb{I}}_i - \mcb{I}_i)(\Xi_i) + \mcb{I}_i(\Xi_i)  
\quad
\text{and}
\quad
\mcb{I}_i (\bar{\Xi}_i) = \mcb{I}_i (\bar{\Xi}_i - \Xi_{i}) 
+
\mcb{I}_i(\Xi_{i})\;.
\]
We then define, for any $\ell \in \deg(R)$ and $v = \sum_{\tau \in \mfT(\ell,R)} v_{\tau} \tau \in T_{\ell}$,
\[
\| v \|_{\ell,\eps} = \max \{ \eps^{ m(\tau) \theta}|v_{\tau}| : \tau \in \mfT(\ell,R)\}\;,
\]
where $m(\tau)$ counts the number of occurrence of $\bar{\mcb{I}}_{i,p} -\mcb{I}_{i,p}$ and $\bar{\Xi}_{i} - \Xi_{i}$ in $\tau$. 

We now make this idea more precise and formulate it our setting of vector-valued noise. 
Recall that in our new setting the trees serve as indices for subspaces of our regularity structure, instead of  basis vectors, so we do not really ``change basis''.
We note that there is a (unique) isomorphism $\Theta\colon \CT \rightarrow \CT$ with the following properties.
\begin{itemize}
\item $\Theta$ preserves the domain of $\mcb{I}_{i,p}$, $\mcb{I}_{i,p}^{\mfh}$ and $\mcb{I}^{\mfu}$ and commutes with these operators on their domain.
\item For any $\tau$ with $\bar{\mcb{I}}_{i,j}(\tau) \in \mfT(R)$, one has $\Theta \circ \bar{\mcb{I}}_{i,j} (\tau)= (\bar{\mcb{I}}_{i,j} + \mcb{I}_{i,j}) \circ \Theta (\tau)$.
\item For any $u,v \in \CT$ with $uv \in \CT$ one has $\Theta(u)\Theta(v) = \Theta(uv)$ \dash here we are referencing the partially defined product on $\CT$ induced by the partially defined tree product on $\mfT(R)$. 
\item The restriction of $\Theta$ to $\CT[\bar{\Xi}_{i}]$ is given by $\id + \sigma^{-1}$ where $\sigma^{-1}$ is the inverse of the map $\sigma$ given in \eqref{eq:noise_iso}.
\item  $\Theta$ restricts to the corresponding identity map on both $\CT[\mbX^{k}]$ and $\CT[\Xi_{i}]$. 
\end{itemize} 
It is immediate that $\Theta$ furthermore preserves $T_{\ell}$ for every $\ell \in \deg(R)$. 

We now fix, for every $\tau \in \mfT(R)$, some norm $\| \bigcdot \|_{\tau}$ on $\CT[\tau]$.
Since each $\CT[\tau]$ is isomorphic to a subspace of $(\mfg^*)^{\otimes n}$ and the isomorphism
is furthermore canonical up to permutation of the factors, this can be done by choosing a
norm on $\mfg^*$ as well as a choice of uniform crossnorm (for example the projective crossnorm).

We then define a norm $\wnorm{\act}_{\ell, \eps}$ on $T_{\ell}$ by setting, for any $v \in T_{\ell}$, 
\[
\wnorm{v}_{\ell,\eps}
=
\max\{ \eps^{m(\tau) \theta} \|P_{\tau} v\|_{\tau}: \tau \in \mfT(R,\ell)\}\;,
\]
where $P_{\tau}$ is the projection from $T_{\ell}$ to $\CT[\tau]$ and now $m(\tau)$ counts the number of occurrences of the labels $\{ \bar{\mfl}_{i}, \mfm_{i}\}_{i=1}^{d}$ appearing in $\tau$. 
Finally, the norm $\| \act \|_{\ell, \eps}$ is given by setting $\|v\|_{\ell,\eps} = \wnorm{ \Theta v}_{\ell,\eps}$.  \label{norm ell eps page ref}

The following lemma, which is straightforward to prove, states that these norms have the desired qualities.
\begin{lemma}\label{lem:compare on T}
\ 
\begin{itemize}
\item Let  $\ell \in \deg(R)$ and $v \in T_{\ell}$ with  $v$  in the domain of the operator $\bar{\mcb{I}}_{i,p} - \mcb{I}_{i,p}$. Then one has, uniform in $\eps$, 
\begin{equ}\label{eq:diff_of_integral_est}
\| (\bar{\mcb{I}}_{i,p} - \mcb{I}_{i,p})(v) \|_{\ell + 2 - \kappa,\eps} \lesssim \eps^{\theta} \| v \|_{\ell,\eps}\;.
\end{equ}
\item For any $u \in \CT[\Xi_{i}]$ one has, uniformly in $\eps$, 
\begin{equ}\label{eq:diff_of_noise_est}
\|\sigma(u)-u\|_{-d/2 -1 - \kappa,\eps}
\lesssim
\eps^{\theta}
\|u\|_{-d/2- 1 -\kappa,\eps}\;.
\end{equ}
\end{itemize}
\end{lemma}

Once we fix these $\eps$-dependent norms on our regularity structure we also obtain corresponding \label{eps norms page ref}
\begin{itemize}
\item $\eps$-dependent seminorms and pseudo-metrics on models which we  denote by $\$\act\$_{\eps}$ and $d_{\eps}(\act,\act)$ respectively; and
\item $\eps$-dependent seminorms on  $\cD^{\gamma,\eta} \ltimes \mathscr{M}_\eps$  which we  denote by $| \act |_{\gamma,\eta,\eps}$.  
\end{itemize}
The seminorms $\$\act\$_{\eps}$ and pseudo-metrics $d_\eps\equiv \$\act;\act\$_\eps$ are defined as in Appendix~\ref{app:reconstruct}.
We recall that they and $| \act |_{\gamma,\eta,\eps}$ are indexed by compact sets.
When the compact set is $O_\tau \eqdef [-1,\tau]\times\T^d$, we will simply write $| \act |_{\gamma,\eta,\eps;\tau}$ for $| \act |_{\gamma,\eta,\eps;O_\tau}$, and likewise for $\$\act\$_{\eps}$ and $d_\eps$.

\begin{remark}\label{rem:norms_on_tensor_products}
Recall that modelled distributions in the scalar setting take values in the regularity structure $\CT$, so in the definition of a norm on modelled distributions we reference norms $\| \act \|_{\ell}$ on the spaces $T_{\ell}$. 
When our noises\slash solutions  live in finite-dimensional vector spaces, our modelled distributions will take values in $\CT \otimes W$ for some finite-dimensional vector space $W$, so when specifying a norm on such modelled distributions we will need to reference norms on $T_{\ell} \otimes W$.
We thus assume that we have already fixed
an $\eps$-independent norm $\| \act \|_{W}$  on the space $W$. Then we view our norm  on $T_{\ell} \otimes W$ as induced by the norm on $T_{\ell}$ by taking some choice of crossnorm (the particular choice does not matter). 
\end{remark}
\begin{remark}
Clearly, all of our $\eps$-dependent seminorms\slash metrics on models are equivalent for different values of $\eps \in (0,1]$, but not uniformly so as $\eps \downarrow 0$. The distances for controlling models (resp.\ modelled distributions) become stronger (resp.\ weaker) as one takes $\eps$ smaller.  
\end{remark} 
\begin{remark}
In general, one would not expect the estimates of the extension theorem \cite[Thm.~5.14]{Hairer14} to hold uniformly as we take $\eps \downarrow 0$. 
However, it is straightforward to see from the proof of \cite[Thm.~5.14]{Hairer14} that they do hold uniformly in $\eps$ for models in $\mathscr{M}_{\eps}$ (and, more trivially, $\mathscr{M}_{0}$) thanks to Lemma~\ref{lemma:control_of_kernels} and the fact that $\theta \le \kappa$. 
\end{remark}
\subsubsection{Comparing fixed point problems}\label{sec:Comparing fixed point problems}
For sufficiently small $\tilde{\theta} > 0$, one has a classical Schauder estimate 
\begin{equ}
|G \ast f- G_{\eps} \ast f|_{\mcC^{2+\ell}} \lesssim \eps^{\tilde{\theta}}|f|_{\mcC^{\ell+\tilde{\theta}}}\;,
\end{equ}
which holds for all distributions $f$ and non-integer regularity exponents. 
Our conditions on our $\eps$-dependent norms let us prove an analogous estimate at 
the level of modelled distributions.
In what follows we write $\mcb{K}_{i}$, $\bar{\mcb{K}}_{i}$ for the abstract integration operators on modelled distributions associated to $\mfa_{i}$, and $\mfm_{i}$, respectively.

We will be careful in this section to ensure that convolution estimates on a time interval $[-1,\tau]$ depend only on the size of the model on $[-1,\tau]$.
This is used in the proof of Proposition~\ref{prop:SPDEs_conv_zero}
to ensure that the existence time $\tau>0$ of the fixed points~\eqref{eq:abstract_fixed_point_eq}-\eqref{eq:abstract_fixed_point_eq_2}
is a stopping time.

We will assume throughout this subsection that $\moll$ is non-anticipative.
This implies, in particular, that $K^\eps$ is non-anticipative.

\begin{lemma}\label{lem:J^eps_diff_bound-2}
Fix $i \in [d]$ and let $\CV$
 be a sector of regularity $\alpha$ in  our regularity structure which is in the domain of both $\mcb{I}_{i}$ and $\bar{\mcb{I}}_{i}$. 
Fix $\gamma  > 0$ and $\eta < \gamma$  such that $\gamma + 2 - \kappa \not \in \N$, $\eta + 2 - \kappa \not \in \N$, and $\eta \wedge \alpha > -2$.  

Then, for fixed $M > 0$, one has for all $\tau\in(0,1)$
\begin{equ}
|\mcb{K}_{i} \bone_+ f - \bar{\mcb{K}}_{i} \bone_+ f|_{\gamma + 2 - \kappa,\bar{\eta},\eps;\tau} \lesssim \eps^{\theta} |f|_{\gamma,\eta,\eps;\tau}
\end{equ}
uniformly in $\eps \in (0,1]$, $Z \in \mathscr{M}_{\eps}$ with $\$Z\$_{\eps;\tau} \le M$, and $f \in \cD^{\gamma,\eta}(\CV) \ltimes Z$,  
where $\theta \in  (0,\kappa]$ is the fixed small parameter as above and $\bar{\eta}  = (\eta \wedge \alpha) + 2 - \kappa$.
\end{lemma}
\begin{proof}
This result follows from the proof of \cite[Thm.~5.12 and Prop.~6.16]{Hairer14} together with the improved reconstruction theorem Lemma~\ref{lem:reconstruct}. 
Indeed, in the context of this reference, and working with some fixed norm on the given regularity structure, if the abstract integrator $\mcb{I}(\cdot)$ of order $\beta$ in question has an operator norm (as an operator on the regularity structure) bounded by $\tilde{M}$, and the kernel $\mcb{I}$ realises is $(\gamma + \beta,\tilde{M},\beta)$-regularising, then as long as  $\gamma +\beta \not \in \N$ and $\eta + \beta \not \in \N$,
one has
\[
|\mcb{K}\bone_+f|_{\gamma + \beta,(\eta \wedge \alpha) + \beta;\tau} \lesssim \tilde{M} |f|_{\gamma,\eta;\tau}\;.
\]
Here, $\mcb{K}$ is the corresponding  integration  on modelled distributions and the  proportionality constant only depends on the size of the model in the model norm (which corresponds to the fixed norm on the regularity structure).  
The fact that the compact set for the model norm can be taken as $O_\tau$ (rather than $[-1,2]\times \T^d$ as in~\cite{Hairer14})
is due to Lemma~\ref{lem:reconstruct} substituting the role of~\cite[Lem.~6.7]{Hairer14}
together with the fact that $K$ is non-anticipative.

Our result then follows by combining this observation with the fact that we can view $\mcb{I}_{i} - \bar{\mcb{I}}_{i}$ as an abstract integrator of order $2 - \kappa$ on our regularity structure realising the kernel $K - K^{\eps}$ which is  $(m,\eps^{\kappa} R,2-\kappa)$-regularising by Lemma~\ref{lemma:control_of_kernels}, and the fact that $\mcb{I}_{i} - \bar{\mcb{I}}_{i}$ has norm bounded by $\eps^{\theta}$ by \eqref{eq:diff_of_integral_est}.
\end{proof}

To state the next two lemmas, we refer to Appendix~\ref{app:Singular modelled distributions}
for definitions of integration operators of the form $\mcb{K}^\omega$
as well as spaces $\hat\cD$. The $\eps$-dependent norms on $\cD^{\gamma,\eta} \ltimes \mathscr{M}_\eps$
 in particular induces 
$\eps$-dependent norms 
on $\hat\cD^{\gamma,\eta} \ltimes \mathscr{M}_\eps$
denoted by $|\cdot |_{\hat\cD^{\gamma,\eta,\eps}}$.

\begin{lemma}\label{lem:Schauder-input-KK}
Under the same assumptions as Lemma~\ref{lem:Schauder-input}, and assuming $\moll$ and $K$ are non-anticipative,
one has for fixed $M>0$ and all $\tau\in(0,1)$
\begin{equ}[e:K-barK-zeta]
| \mcb{K}^\omega \bone_+f -\bar{\mcb{K}}^\omega \bone_+f |_{\bar\gamma-\kappa,\bar\eta,\eps;\tau}
\lesssim \eps^\theta
| f |_{\gamma,\eta,\eps;\tau}\;,
\end{equ} 
uniformly in $\eps \in (0,1]$, $Z \in \mathscr{M}_{\eps}$ with $\$Z\$_{\eps;\tau} \le M$, and $f \in \cD^{\gamma,\eta}(\CV) \ltimes Z$.
\end{lemma}

\begin{proof}
This follows verbatim as in the proofs of 
Lemma~\ref{lem:J^eps_diff_bound-2} and \cite[Lem.~4.12]{MateBoundary}.
\end{proof}

\begin{lemma}\label{lem:K-barK-hat}
Fix $\gamma  > 0$ and $\eta < \gamma$  such that $\gamma + 2 - \kappa \not \in \N$, $\eta + 2 - \kappa \not \in \N$, and $\eta>-2$.
Suppose that $K^{\eps}$ is non-anticipative.
Then, for fixed $M > 0$ and all $\tau\in(0,1)$
\[
|\mcb{K} f - \bar{\mcb{K}} f|_{\hat\cD^{\gamma + 2 - \kappa, \eta  + 2 - \kappa,\eps};\tau} \lesssim \eps^{\theta} |f|_{\hat\cD^{\gamma,\eta,\eps};\tau}
\]
uniformly in $\eps \in (0,1]$, $Z \in \mathscr{M}_{\eps}$ with $\$Z\$_{\eps;\tau} \le M$, and $f \in \hat\cD^{\gamma,\eta}(\CV) \ltimes Z$.
\end{lemma}

\begin{proof}
This follows in the same way as Lemma~\ref{lem:J^eps_diff_bound-2}
(with an obvious change that $(\eta \wedge \alpha)$ is replaced by $\eta $)
combined with Theorem~\ref{thm:integration} (for non-anticipative kernels).
\end{proof}

We also remark that the  standard multiplication bound \cite[Prop.~6.12]{Hairer14}
and Lemma~\ref{lem:multiply-hatD} also hold for these $\eps$-dependent norms.

Now we  define, as in Section~\ref{sec:solution_theory}, $\boldsymbol{\Xi}_{i}  \in \CT[\Xi_{i}] \otimes \mfg$ and $\boldsymbol{\bar{\Xi}}_{i}  \in \CT[\bar{\Xi}_{i}] \otimes \mfg$ to be given by ``$\id_{\mfg}$'' via the canonical isomorphisms 
 $\CT[\bar{\Xi}_{i}] \otimes \mfg \simeq \CT[\Xi_{i}] \otimes \mfg \simeq \mfg^{\ast} \otimes \mfg \simeq L(\mfg,\mfg)$. 
Note that we have $\sigma \boldsymbol{\Xi}_{i}  = \boldsymbol{\bar{\Xi}}_{i}$
where we continue our abuse of notation with $\sigma$ acting only on the left factor. 
We also remark that, as $\mfg$-valued modelled distributions, $\boldsymbol{\Xi}_{i}$, $\boldsymbol{\bar{\Xi}}_{i} \in \cD^{\infty,\infty}_{-d/2 - 1 - \kappa}$. 
Then, thanks to \eqref{eq:diff_of_noise_est}
and Remark~\ref{rem:norms_on_tensor_products} one has, uniform in $\eps \in (0,1]$, 
\begin{equ}\label{eq:diff_of_noises_tensor}
|\boldsymbol{\Xi}_{i} - \boldsymbol{\bar{\Xi}}_{i}|_{\infty,\infty,\eps} \lesssim \eps^{\theta}\;.
\end{equ}
Note that 
 $\hat{\cD}^{\infty,-2-\kappa}_{-2-\kappa}$  
 coincides with the space of elements in $\cD^{\infty,-2-\kappa}_{-2-\kappa}$ which vanish on $\{t\le 0\}$. 
 One also has
\begin{equ}\label{eq:diff_of_noises}
|\mathbf{1}_+ \boldsymbol{\Xi} - \mathbf{1}_+ \boldsymbol{\bar{\Xi}}|_{\hat\cD^{\infty,-2-\kappa,\eps}}
 \lesssim \eps^{\theta}\;,
\end{equ}
where $\mathbf{1}_{+}$\label{one_+_page_ref} is the map that restricts modelled distributions to non-negative times.

We now write out the analytic fixed point problems for~\eqref{eq:B_final},~\eqref{eq:bar_A_final}, and~\eqref{e:final_system}. 
We introduced the labels $\mfm_{i}$ just to assist with deriving the renormalised equation and so when we pose our analytic fixed point problem we stray from the formulation given in Remark~\ref{rem:analytic-fix-pt} and instead eliminate the components $\mfm_{i}$ appearing in \eqref{eq:bar_A_final} by performing a substitution. 

In what follows, we write $\mathcal{R}$ for the reconstruction operator. 
Recall that $\mcb{K}_{i}$, $\bar{\mcb{K}}_{i}$ are the abstract integration operators associated to $\mfa_{i}$ and $\mfm_{i}$; we also write $\mcb{K}_{\mfh_i}$ and $\mcb{K}_{\mfu}$ for the abstract integration operators on modelled distributions corresponding to $\mcb{I}^{\mfh}_{i}$ and $\mcb{I}_{\mfu}$, and $R$  the operator realising convolution with $G - K$ as a map from appropriate H{\"o}lder--Besov functions into modelled distributions as in \cite[(7.7)]{Hairer14}.

We assume henceforth that $d=2$.
Recall that we have fixed $\alpha\in(\frac23,1)$ and $\eta\in (\frac\alpha4-\frac12,\alpha-1]$.
Given initial data 
\begin{equ}\label{eq:initial_data_for_gauge_sys}
(B^{(0)},U^{(0)},h^{(0)}) \in \Omega^\init \eqdef \Omega\mcC^{\eta} \times \mcC^{\alpha}(\T^{2},L(\mfg,\mfg)) \times \Omega\mcC^{\alpha - 1} \;, 
\end{equ}
the fixed point problem associated with \eqref{eq:B_final} and~\eqref{e:final_system} for the $\mfg$-valued modelled distributions $(\mathcal{B}_{i})_{i=1}^{2}$,$(\mathcal{H}_{i})_{i=1}^{2}$ and $L(\mfg,\mfg)$-valued modelled distribution $\mathcal{U}$ is
\begin{equs}
\mathcal{B}_{i}
&=
\CG_i
\mathbf{1}_{+}\Big(
[\mathcal{B}_j, 2\d_j \mathcal{B}_i - \d_i \mathcal{B}_j + [\mathcal{B}_j,\mathcal{B}_i]] 
+ \mathring{C}_{1} \mathcal{B}_{i} + \mathring{C}_{2} \mathcal{H}_{i}
+\hat{\mathcal{U}} \boldsymbol{\bar{\Xi}}_{i}\Big)
\\
 &\qquad \qquad\qquad\qquad
  + \CG_i^{\bar\omega_i} (GU^{(0)} \boldsymbol{\bar{\Xi}}_{i})
  + GB^{(0)}_{i}\;,
 \label{eq:abstract_fixed_point_eq}
 \\
\mathcal{H}_{i}
&=
\CG_{\mfh_{i}}
\mathbf{1}_{+}
\Big(
 [\mathcal{H}_j,\d_j \mathcal{H}_i] + [[\mathcal{B}_j, \mathcal{H}_j],\mathcal{H}_i]  
  + \d_i  [\mathcal{B}_j, \mathcal{H}_j]\Big) + Gh^{(0)}_{i}\;,
  \\
\hat{\mathcal{U}} &=
\CG_{\mfu}
\mathbf{1}_{+} \Big( - [\mathcal{H}_j,[\mathcal{H}_j, \cdot]] \circ \mathcal{U}
+ [[\mathcal{B}_j, \mathcal{H}_j],\cdot] \circ \mathcal{U} \Big) \;,
\\
\mathcal{U} &=  GU^{(0)}+\hat{\mathcal{U}} 
\end{equs}
where
$\CG_\mft \eqdef \mcb{K}_{\mft} + R \mathcal{R}$
and $G\act$ refers to the ``harmonic extension'' map of \cite[(7.13)]{Hairer14}.
Here, $\bar\omega_i$ is compatible with $GU^{(0)}\bar\bXi_i$ and we refer to Appendix~\ref{subapp:input_distr} for the operator $\CG_i^{\bar\omega_i}$.
We remark that the input to the fixed point problem is a model $Z$, the initial conditions, and the distributions $\bar\omega_i$.

\begin{remark}
The reason that we introduced the ``intermediate'' object $\hat{\mathcal{U}}$
is that  it will be in a $\hat\cD$ space (see Appendix~\ref{app:Singular modelled distributions}) which has improved behaviour 
near $t=0$, so that we can apply integration operator to $\hat{\mathcal{U}} \boldsymbol{\bar{\Xi}}_{i}$ using Theorem~\ref{thm:integration}.
Note that the standard integration result \cite[Prop.~6.16]{Hairer14} would require 
that the lowest degree of $\hat{\mathcal{U}} \boldsymbol{\bar{\Xi}}_{i}$ is larger than $-2$ (for purpose of reconstruction) which is not true here;
but  Theorem~\ref{thm:integration}  does not require this assumption.
The distribution $\bar\omega_i$ in our case will simply be $GU^{(0)}\xi^\eps_i$.
\end{remark}

The modelled distribution fixed point problem for the $(\bar A,\bar U,\bar h)$ system~\eqref{eq:bar_A_final} and~\eqref{e:final_system} is the same as \eqref{eq:abstract_fixed_point_eq} except that the first equation is replaced by
\begin{equs}[eq:abstract_fixed_point_eq_2]
\mathcal{B}_{i}
&=
\CG_i
\mathbf{1}_{+}
\Big(
[\mathcal{B}_j, 2\d_j \mathcal{B}_i - \d_i \mathcal{B}_j + [\mathcal{B}_j,\mathcal{B}_i]] + \mathring{C}_{1} \mathcal{B}_{i} + \mathring{C}_{2}\mathcal{H}_{i} 
\Big)\\
{}&\qquad
+ \bar\CG_i^{\omega_i} (GU^{(0)} \boldsymbol{\Xi}_{i})
+ \CW_i
+  
\bar{\CG}_i
\mathbf{1}_{+}
(\hat{\mathcal{U}} \boldsymbol{\Xi}_{i})
+
GB^{(0)}_{i}\;,
\end{equs}
where $\bar{\CG}_i\eqdef \bar{\mcb{K}}_{i} + \bar{R} \mathcal{R}$ with
$\bar{R}$  defined just like $R$ but with $G - K$ replaced by $G^{\eps} - K^{\eps}$.
As before, $\omega_i$ is compatible with $GU^{(0)}\bXi_i$ and is part of the input which we will later take as $\omega_i=GU^{(0)}\xi^\delta_i$.
The modelled distribution $\CW_i$ takes values in the polynomial sector and is likewise part of the input -- we will later take $\CW_i$ as the
canonical lift of the (smooth) function
$G*(\bone_+ \moll^\eps*(\xi_i\bone_{-}))$ for $\eps>0$,
and simply as $\CW_i=0$ for $\eps=0$.
Here $\bone_-$ is the indicator function of set $\{(t,x)\in\R\times \T^2\,:\,t<0\}$.

In \eqref{eq:abstract_fixed_point_eq_2} we have written $\mathcal{B}_{i}$ instead of something like $\bar{\mcA}_{i}$ to make it clearer that we are comparing two fixed point problems which have ``almost'' the same form 
\dash only the terms involving the noises $\boldsymbol{\bar{\Xi}}_{i}$ or $ \boldsymbol{\Xi}_{i}$ and the term $\CW_i$ are different.
We can now make precise what we mean by the two problems being ``close''.

For $r,\tau>0$, we call an \emph{$\eps$-input of size $r$ over time $\tau$} a collection of the following objects:
a model $Z\in\mathscr{M}_\eps$, distributions $\bar\omega_i,\omega_i$ compatible with $GU^{(0)}\bar\bXi_i,GU^{(0)}\bXi_i$,
initial data $(B^{(0)},U^{(0)},h^{(0)})\in\Omega^\init$,
and a modelled distribution $\CW=(\CW_i)_{i=1}^2$
such that
\begin{multline}\label{eq:def_inputs_size}
\$Z\$_{\eps;\tau}
+ |(B^{(0)},U^{(0)},h^{(0)})|_{\Omega^\init}
+|\omega_i|_{\CC^{-2-\kappa}_\tau}
+|\bar\omega_i|_{\CC^{-2-\kappa}_\tau} + \eps^{-\theta}|\omega_i-\bar\omega_i|_{\CC^{-2-\kappa}_\tau}
\\
+\eps^{-\kappa}|\CW|_{\frac32,-2\kappa;\tau}
+|\PPi^Z\mcb{I}[\bar\bXi]|_{\CC([0,\tau],\Omega^1_\alpha)}
+
|\PPi^Z\bar{\mcb{I}}[\bXi]|_{\CC([0,\tau],\Omega^1_\alpha)}
\leq r\;,
\end{multline}
where $|\act|_{\CC^{-2-\kappa}_\tau}$ is shorthand for $|\act|_{\CC^\alpha(O_\tau)}$
which is defined in Appendix~\ref{subapp:input_distr} and $\PPi^Z\eqdef \Pi_0$ where we denote as usual $Z=(\Pi,\Gamma)$.
(The choice `$0$' in $\Pi_0$ is of course arbitrary since $\Pi_x \mcb{I}[\bar\bXi]=\Pi_0\mcb{I}[\bar\bXi]$ for all $x\in\R\times\T^2$.)
\begin{remark}
The left-hand side of~\eqref{eq:def_inputs_size} is increasing in $\tau$, hence
an $\eps$-input of size $r$ over time $\tau$ is also an $\eps$-input of size $\bar r$ over time $\bar \tau$ for all $\bar\tau\in (0,\tau)$ and $\bar r>r$.
\end{remark}
\begin{remark}
The final two terms on the left-hand side of~\eqref{eq:def_inputs_size} play no role in Lemma~\ref{lemma:fixedptpblmclose},
but will be important in Lemma~\ref{lem:fixed_close_enhanced}.
\end{remark}
\begin{lemma}\label{lemma:fixedptpblmclose} 
Consider the bundle of modelled distributions
\begin{equ}\label{eq:fixed_pt_space}
\Big( \bigoplus_{\mft = \mfa_{i},\mfh_{i},\mfu} \cD^{\gamma_{\mft},\eta_{\mft}}_{\alpha_{\mft}} \Big) \ltimes \mathscr{M}
\textnormal{ where }
(\gamma_{\mft},\alpha_{\mft}, \eta_{\mft}) = 
\begin{cases}
(1+3\kappa, - 2\kappa, \eta) & \textnormal{ if }\mft = \mfa_{i}\,,\\
(2 + 2 \kappa, 0, \eta + 1) & \textnormal{ if }\mft = \mfu\,,\\
(1 + 3\kappa, 0, \eta ) & \textnormal{ if }\mft = \mfh_{i}\,.
\end{cases}
\end{equ}
Consider further $r>0$ and $\eps\in[0,1]$.
Then there exists $\tau\in(0,1)$, depending only on $r$,
such that, for every $\eps$-input of size $r$ over time $\tau$,
there exist solutions $\CS$ and $\bar\CS$
to \eqref{eq:abstract_fixed_point_eq} and \eqref{eq:abstract_fixed_point_eq_2} respectively
in the space of modelled distributions defined by~\eqref{eq:fixed_pt_space}
on the time interval $(0,\tau)$.
Furthermore, uniformly in $\eps\in[0,1]$ and all $\eps$-inputs of size $r$,
\begin{equ}\label{eq:abstract_soln_plus_reconstruction_estimate}
|\mathcal{S} - \bar{\mathcal{S}}|_{\vec{\gamma},\vec{\eta},\eps;\tau} 
\lesssim 
\eps^{\theta}\;.
\end{equ}
Here $| \act |_{\vec{\gamma},\vec{\eta},\eps;\tau} $ is the multi-component modelled distribution seminorm for \eqref{eq:fixed_pt_space} over the time interval $(0,\tau)$.
Finally, $\CS$ and $\bar\CS$ are each locally uniformly continuous with respect to the input when the space of models is equipped with $d_{1;\tau}$.
\end{lemma}

\begin{proof}
We first prove well-posedness of \eqref{eq:abstract_fixed_point_eq}
in the space \eqref{eq:fixed_pt_space}, where
\[
\hat{\mathcal U}\in \hat\cD^{\gamma_\mfu,\eta_\mfu}_{0} = \hat\cD^{2+2\kappa,\eta+1}_{0}\;.
\]
Since $\mathbf{1}_{+}   \bar{\boldsymbol{\Xi}}_{i} \in  \hat\cD^{\infty,-2-\kappa}_{-2-\kappa}$,
by Lemma~\ref{lem:multiply-hatD} one has
$\hat{\mathcal U}  \bar{\boldsymbol{\Xi}}_{i} \in \hat\cD^{\kappa,\eta-1-\kappa}_{-2-\kappa}$.
Since $\eta-1-\kappa > -2$, by Theorem~\ref{thm:integration}, one has
\[
\CG_i (\hat{\mathcal U}  \bar{\boldsymbol{\Xi}}_{i} ) \in \hat\cD^{2-\kappa,\eta+1-2\kappa}_{-\kappa}
\subset \cD^{\gamma_{\mfa_i},\eta_{\mfa_i}}_{\alpha_{\mfa_i}}\;.
\]
Also, since $GU^{(0)}\in \cD^{\infty,\alpha}_0$,
by standard multiplication bounds \cite[Prop.~6.12]{Hairer14}, followed by Lemma~\ref{lem:Schauder-input}, one has
\[
GU^{(0)} \boldsymbol{\bar{\Xi}}_{i} \in \cD^{\infty,\alpha-2-\kappa}_{-2-\kappa}
\; \Rightarrow \;
\CG_i^{\bar\omega_i} (GU^{(0)} \boldsymbol{\bar{\Xi}}_{i})
\in \cD^{\infty,-2\kappa}_{-2\kappa}
\subset \cD^{\gamma_{\mfa_i},\eta_{\mfa_i}}_{\alpha_{\mfa_i}}\;,
\]
where we used $-2\kappa>1-\alpha\geq \eta$.
Here, note that the compatibility condition of Lemma~\ref{lem:Schauder-input} is satisfied by assumption.

It is standard to check that the other terms in  \eqref{eq:abstract_fixed_point_eq} belong to the desired spaces.
The only subtlety is to check that 
the right-hand side of the equation for $\hat{\mathcal U}$
indeed belongs to $\hat\cD^{\gamma_\mfu,\eta_\mfu}_{0}$.
For this we note that 
 $\mathcal U \in \cD^{\gamma_\mfu,\eta_\mfu}_0$, so 
 one can check that the two terms
$ [\mathcal{H}_j,[\mathcal{H}_j, \cdot]] \circ \mathcal{U}$
and $ [[\mathcal{B}_j, \mathcal{H}_j],\cdot] \circ \mathcal{U}$
both belong to $\cD^{1+2\kappa, 2\eta}_{-2\kappa}$
which is identical with
$ \hat\cD^{1+2\kappa, 2\eta}_{-2\kappa}$, since these modelled distributions are supported on $\{t\ge 0\}$ and
$2\eta \leq -2\kappa$.
Integration using Theorem~\ref{thm:integration} then gives elements in  $\hat\cD^{\gamma_\mfu,\eta_\mfu}_{0}$.

The fact that, for sufficiently small $\tau>0$, the corresponding map stabilises and is a contraction on a ball in~\eqref{eq:fixed_pt_space}
follows from the short-time convolution estimates~\eqref{eq:short_time} and Proposition~\ref{prop:short_time}.
It follows that the solution map $\CS$ for~\eqref{eq:abstract_fixed_point_eq} is locally well-posed.
The local uniform continuity of $\CS$ follows from stability of the fixed point established in the proof of \cite[Thm~7.8]{Hairer14}
together with~\eqref{eq:short_time} and Lemma~\ref{lem:Schauder-input}.
Local well-posedness and local uniform continuity of the solution map $\bar\CS$ for~\eqref{eq:abstract_fixed_point_eq_2} follows in the same way.

It remains to prove~\eqref{eq:abstract_soln_plus_reconstruction_estimate}.
Note that the two fixed point problems differ in the $\mathcal{B}_{i}$ components.
By standard multiplication bounds and Lemma~\ref{lem:Schauder-input},
\begin{equs}
|\CG_i^{\bar\omega_i} & (GU^{(0)} \boldsymbol{\bar{\Xi}}_{i})
-
\CG_i^{\omega_i} (GU^{(0)} \boldsymbol{\Xi}_{i})|_{\infty,\alpha-\kappa,\eps} 
\\
&\lesssim
|GU^{(0)} (\boldsymbol{\bar{\Xi}}_{i} -  \boldsymbol{\Xi}_{i})|_{\infty,\alpha-2-\kappa,\eps}
+ |\bar\omega_i-\omega_i|_{\CC^{-2-\kappa}} \lesssim \eps^\theta\;.
\end{equs}
Moreover, by Lemma~\ref{lem:Schauder-input-KK},
\begin{equ}
|(\CG_i^{\omega_i} 
-
\bar\CG_i^{\omega_i} )(GU^{(0)} \boldsymbol{\Xi}_{i})|_{\infty,\alpha-\kappa,\eps} 
\lesssim
\eps^\theta\;.
\end{equ}
Also, by similar arguments as above using Lemma~\ref{lem:multiply-hatD} and Theorem~\ref{thm:integration}, 
but combined with the difference estimates \eqref{eq:diff_of_noises}
and Lemma~\ref{lem:K-barK-hat},  one has
\begin{equs}[e:GUhatXi-compare]
\big|\CG_i (\hat{\mathcal U}  \bar{\boldsymbol{\Xi}}_{i} ) 
 & -
\bar\CG_i (\hat{\mathcal U}  \boldsymbol{\Xi}_{i} ) 
\big|_{\hat\cD^{\kappa+2,\eta+1-\kappa,\eps}_{-\kappa}}
\\
&\lesssim
\big|\CG_i (\hat{\mathcal U}  \bar{\boldsymbol{\Xi}}_{i} 
- \hat{\mathcal U}  \boldsymbol{\Xi}_{i}) 
+
\big(\CG_i -\bar\CG_i \big)(\hat{\mathcal U}  \boldsymbol{\Xi}_{i} ) 
|_{\hat\cD^{\kappa+2,\eta+1-\kappa,\eps}_{-\kappa}}
\lesssim
\eps^\theta\;.
\end{equs}
Finally, by assumption, $|\CW_i|_{\frac32,-2\kappa}\leq r\eps^{\kappa} \leq r\eps^{\theta}$.
Summarising these bounds, one has that the difference
between the  $\mathcal{B}_{i}$ components in the  two fixed point problems 
is bounded by a multiple of $\eps^\theta$.

We remark that the integration results from Appendix~\ref{app:Singular modelled distributions}
which we apply above are stated in terms of operators $\mcb{K}_{i}$ and $\bar{\mcb{K}}_{i} $.
But their differences with  $\CG_{i}$ and $\bar{\CG}_{i} $ are bounded much more easily since they are smooth.
For instance, regarding
\[
R \mathcal{R}
(\hat{\mathcal{U}} \boldsymbol{\bar{\Xi}}_{i})
-
\bar{R} \mathcal{R}
(\hat{\mathcal{U}}  \boldsymbol{\Xi}_{i})\;,
\]
we write it as $R \mathcal{R}( \hat{\mathcal{U}} (\boldsymbol{\bar{\Xi}}_{i} - \boldsymbol{\Xi}_{i} ))
 + (R-\bar{R})\mathcal{R} (\hat{\mathcal{U}} \boldsymbol{\Xi}_{i})$. 
The first piece can be estimated as in  \cite[Lemma~7.3]{Hairer14} together with Theorem~\ref{thm:reconstructDomain}. 
By definition \cite[(7.7)]{Hairer14} of $R - \bar{R}$, the second piece is of order $\eps$ since, for any $\alpha > 0$, one has $\| (G - K) - (G^{\eps} - K^{\eps})\|_{\mcC^{\alpha}} \lesssim \eps$. 

The estimate~\eqref{eq:abstract_soln_plus_reconstruction_estimate} then follows in a routine way from the these bounds together with the short-time convolution estimates for modelled distributions.
\end{proof}

In the situation we consider, we further have convergence of the most singular tree $\<IXi>$ in the space $\CC(\R,\Omega^1_\alpha)$ (Corollary~\ref{cor:SHE_mollif_converge}).
The next lemma shows that this suffices to also obtain convergence of solutions in $\CC((0,\tau],\Omega^1_\alpha)$.

\begin{lemma}\label{lem:fixed_close_enhanced}
Suppose we are in the setting of Lemma~\ref{lemma:fixedptpblmclose}.
For every $\tau>0$, we equip $\mathscr{M}_1$ with the pseudo-metric\footnote{More precisely,
we redefine  $\mathscr{M}_1$ as the closure of smooth models under this pseudo-metric with $\tau=1$.}
\begin{equ}
d_{1;\tau}(Z, \bar Z) + |(\PPi^Z - \PPi^{\bar Z}) \mcb{I}[\bar\bXi]|_{\CC([0,\tau],\Omega^1_\alpha)}\;.
\end{equ}
Let $\CB$ denote the corresponding component of $\CS$
and let $\lambda\in(0,1)$, $r>0$.
Then there exists $\tau\in(0,1)$, depending only on $r$, such that $\CR\CB$
is a uniformly continuous function into $\CC([\lambda\tau,\tau],\Omega^1_\alpha)$ on the set
\begin{multline*}
\CI_{r,\tau} \eqdef \Big\{
I = (Z,\bar \omega,(B^{(0)},U^{(0)},h^{(0)}))
\in \mathscr{M}_1\times \CC^{-2-\kappa}_\tau \times \Omega^\init
\\
:\,
I \textnormal{ is a $1$-input of size $r$ over time $\tau$}
\Big\}\;.
\end{multline*}
The same statement holds with $\CB$ replaced by $\bar\CB$, the $\CB$-component of $\bar\CS$,
$\mcb{I}[\bar\bXi]$ replaced by $\bar{\mcb{I}}[\bXi]$, and $\bar\omega$ replaced by $\omega$.
Furthermore, as $\eps\downarrow 0$,
\begin{equ}
|\CR\CB-\CR\bar\CB|_{\CC([\lambda\tau,\tau],\Omega^1_\alpha)}\to 0
\end{equ}
uniformly over all $\eps$-inputs of size $r$ over time $\tau$
which satisfy
\begin{equ}\label{eq:consistency_eps_models}
\PPi^Z \mcb{I}[\bar\bXi] = \PPi^Z \bar{\mcb{I}}[\bXi]\;.
\end{equ}
\end{lemma}

\begin{proof}
We choose $\tau$ as half the `$\tau$' associated to $2r+1$ appearing in Lemma~\ref{lemma:fixedptpblmclose}.
Let $\mcU$ and $\bar\mcU$ denote the $\mcU$-component of $\CS$ and $\bar\CS$ respectively.
The idea is to use the uniform continuity on  $\CI_{r,\tau}$ of $\CR\mcU$ in a space of good regularity.
Truncating at level $\gamma=1-6\kappa$, we obtain
\begin{equ}
\CB (t,x)= u(t,x) \mcb{I}[\bar\bXi] + b(t,x) \bone\;.
\end{equ}
Note that the structure group acts trivially on $\mcb{I}[\bar\bXi]$ and $\bone$.

We know that $u(t,x) = \CR \mcU(t,x)\in \CC([\lambda\tau,\tau],\CC^{0,\alpha})$
and is a uniformly continuous function on $\CI_{r,\tau}$ due to Lemma~\ref{lemma:fixedptpblmclose} and continuity of the reconstruction map.
Recalling the embedding $\Omega\CC^{0,\alpha/2}\hookrightarrow\Omega^1_\alpha$ (Remark~\ref{rem:Holder_Omega_embedding})
and that $1-6\kappa>\alpha/2$
we also see that $b\in\CC([\lambda\tau,\tau],\CC^{1-6\kappa})\hookrightarrow \CC([\lambda\tau,\tau],\Omega^1_\alpha)$
is a uniformly continuous function on $\CI_{r,\tau}$.

Furthermore, on $[\lambda\tau,\tau]\times\T^2$,
\begin{equ}
\CR\CB = u \PPi^Z \mcb{I}[\bar\bXi] + b\;.
\end{equ}
It follows from the continuity of $\Omega^1_\alpha\times \CC^{0,\alpha}\ni(A,u)\mapsto u A$ (Lemmas~\ref{lem:group_action_gr} and~\ref{lem:group_action_vee}, which we note only use that $u\in\CC^{0,\alpha}$
and not that $u=\Ad_g$ for some $g\in\mfG^{0,\alpha}$),
that $u \PPi^Z \mcb{I}[\bar\bXi]\in
\CC([\lambda\tau,\tau],\Omega^1_\alpha)$
is a uniformly continuous function on $\CI_{r,\tau}$,
which proves the first claim.
The second claim concerning $\bar\CB$ follows in an identical manner. 

Finally, due to~\eqref{eq:abstract_soln_plus_reconstruction_estimate} and the continuity of the reconstruction map,
$|\CR\mcU-\CR\bar\mcU|_{\CC([\lambda\tau,\tau],\CC^{0,\alpha})} \to 0$ as $\eps\downarrow0$ uniformly over all $\eps$-inputs of size $r$.
Therefore the final claim follows from~\eqref{eq:consistency_eps_models} by the same argument as above.
\end{proof}

\begin{remark}\label{rem:time-zero}
With more technical effort, one should be able to set $\lambda=0$ in Lemma~\ref{lem:fixed_close_enhanced} once the initial conditions are taken in the appropriate spaces
and extra assumptions of the type $G*\omega,G*\bar\omega\in\CC([0,1],\Omega^1_\alpha)$
are made to handle behaviour at $t=0$. 
\end{remark}

\subsubsection{Maximal solutions}\label{subsubsec:maximal}

To prove convergence of maximal solutions for the $(\bar A,\bar U,\bar h)$ system~\eqref{eq:bar_A_final} and~\eqref{e:final_system},
we require a slightly different fixed point problem than~\eqref{eq:abstract_fixed_point_eq_2} which includes knowledge of the modelled distribution $\mcU\bXi$ on an earlier time interval
(this is to substitute the role of $\CW$ in~\eqref{eq:abstract_fixed_point_eq_2}).

\begin{remark}
It is much easier to show convergence of the maximal solutions of the $(B,U, h)$ system
by using the probabilistic bounds in Section~\ref{subsubsec:prob_bounds} below and stopping time arguments to patch together local solutions. 
Alternatively, one can take advantage of the fact that $(B,U,h)$ is gauge equivalent to an equation with additive noise
\dash see the argument in the proof of Proposition~\ref{prop:SPDEs_conv_zero} on pg. \pageref{page:conv_maximal_B}. 
\end{remark}

\begin{remark}
The results of this subsection will \textit{not} be used to control the difference
between the maximal solution of the $(B,U, h)$ and $(\bar A,\bar U,\bar h)$ systems.
This will be done instead in the proof of Proposition~\ref{prop:SPDEs_conv_zero} using the strong Markov property.
\end{remark}

We assume again in this subsection that $\moll$ is non-anticipative. 
The initial data taken as input for our fixed point problem will be richer than \eqref{eq:initial_data_for_gauge_sys}, consisting of $(B^{(0)},h^{(0)}) \in \Omega\mcC^{\eta}  \times \Omega\mcC^{\alpha - 1}$ along with a modelled distribution $\tilde\mcU\in \cD^{2+2\kappa}(\CV)$ with support on an interval $[-T,0]\times\T^2$ for some $T\in(0,1)$ and
taking values in the same sector $\CV$ as $\mcU$
in~\eqref{eq:abstract_fixed_point_eq}.

We extend $\tilde\mcU$ to positive times $t>0$ by $\tilde \mcU(t) = \big(G U^{(0)} \big)(t)$,
where $U^{(0)}=\CR\tilde \mcU(0)$ is the reconstruction of $\tilde\mcU$ at time zero
and $G U^{(0)}$ is the lift to the polynomial sector of its harmonic extension.
Since $\CR\colon \cD^{2+2\kappa}(\CV) \to \CC^{2-3\kappa}$ is continuous,
it holds that $\tilde \mcU \in \bar\cD^{2+2\kappa, 2-3\kappa}$ as a modelled distribution on $[-T,\infty)\times\T^2$,
where we recall the definition of $\bar \cD$ from Appendix~\ref{subapp:model_vanish}.

Then the fixed point problem we pose for $(\bar A,\bar U,\bar h)$ system is the same
as~\eqref{eq:abstract_fixed_point_eq} except that the first equation is replaced by
\begin{equs}[eq:abstract_fixed_point_eq_3]
\mathcal{B}_{i}
&=
\CG_i
\mathbf{1}_{+}
\Big(
[\mathcal{B}_j, 2\d_j \mathcal{B}_i - \d_i \mathcal{B}_j + [\mathcal{B}_j,\mathcal{B}_i]] + \mathring{C}_{1} \mathcal{B}_{i} + \mathring{C}_{2}\mathcal{H}_{i} 
\Big)\\
{}&\qquad
+ \bar\CG_i (\tilde \mcU \boldsymbol{\Xi}_{i}) - G v_i
+  
\bar{\CG}_i
\mathbf{1}_{+}
(\hat{\mathcal{U}} \boldsymbol{\Xi}_{i})
+
G B^{(0)}_{i}\;,	
\end{equs}
and the definition of $\mcU$ is now $\tilde \mcU + \hat\mcU$.
Here $Gv_i$ is the harmonic extension of
\begin{equ}\label{eq:v_i_def}
v_i \eqdef  \CR\bar\CG_i (\tilde \mcU \boldsymbol{\Xi}_{i})(0)
\;.
\end{equ}
Remark that $\mcB_i$ is now supported on times $t\geq-T$, not just $t\geq 0$,
and for $t \le 0$ we have $\CB_i=\bar\CG_i (\tilde \mcU \boldsymbol{\Xi}_{i})$.
 
\begin{lemma}\label{lem:restarted}
The fixed point problem specified above is well-posed with $\mcB_i \in \bar\cD^{1+3\kappa,\eta}_{-2\kappa}$,
$\hat\mcU\in\hat\cD^{2+2\kappa,\eta+1}_0$,
and $\CH \in \cD^{1+2\kappa,\eta}_0$
on an interval $[-T,\tau]$,
where $\tau>0$ depends only on the size of the tuple
\begin{equ}
\big(
\tilde \mcU, Z, B^{(0)},h^{(0)}, v 
\big) \in \cD^{2+2\kappa}_{[-T,0]} \times \mathscr{M}_1 \times \Omega\CC^\eta\times\Omega\CC^{\alpha-1}\times \Omega\CC^\eta\;,
\end{equ}
where $Z$ is the underlying model.
Furthermore $(\CB,\mcU,\CH)\restr_{(0,\tau)}$, as an element of the space~\eqref{eq:fixed_pt_space},
is a locally uniformly continuous function of the
same tuple.
\end{lemma}

\begin{proof}
Remark that $\bXi \in \bar\cD^{\infty,\infty}_{-2-\kappa}$, and therefore, by Lemma~\ref{lem:multiply-barD}, $\tilde \mcU \bXi \in \bar\cD^{\kappa,-4\kappa}_{-2-\kappa}$.
Therefore, by Theorem~\ref{thm:integration},
$\bar\CG_i (\tilde \mcU \boldsymbol{\Xi}_{i})\in \bar\cD^{2-\kappa,2-5\kappa}_{-2\kappa}$.
The rest of the proof is essentially identical to the proof of Lemma~\ref{lemma:fixedptpblmclose}.
\end{proof}

\begin{lemma}\label{lem:restart_improve}
Equip $\mathscr{M}_1$ with the same pseudo-metric as in Lemma~\ref{lem:fixed_close_enhanced}.
Then, for fixed $\lambda>0$, $\CR\CB$ and $\{\CR\bar\CG_i(\mcU\bXi_i)\}_{i=1}^2$ are both
uniformly continuous functions into $\CC([\lambda\tau,\tau],\Omega^1_\alpha)$
of the tuple as in Lemma~\ref{lem:restarted}.
\end{lemma}

\begin{proof}
Identical to Lemma~\ref{lem:fixed_close_enhanced}.
\end{proof}

We now combine Lemmas~\ref{lemma:fixedptpblmclose},~\ref{lem:fixed_close_enhanced},~\ref{lem:restarted} and~\ref{lem:restart_improve} to define \textit{maximal solutions} of the $(\bar A,\bar U,\bar h)$ system and show they are continuous functions of the inputs.

Consider, for the rest of this subsection, a family of inputs $\{I^\eps\}_{\eps\in[0,1]}$ with $I^\eps\in\CI_{r,1}$ and such that $I^\eps \to I^0$ as $\eps\downarrow0$
\dash here $\CI_{r,1}$ and the metric on $\mathscr{M}_1$ are as in Lemma~\ref{lem:fixed_close_enhanced}.
Suppose further that, for every $\eps\in[0,1]$, 
$\PPi^{Z^\eps}\bar{\mcb{I}}[\bXi] \in \CC(\R,\Omega^1_\alpha)$ and, for all $\tau>0$,
\begin{equ}
\lim_{\eps\downarrow0} d_{1;\tau}(Z^\eps,Z^0) + |(\PPi^{Z^\eps} - \PPi^{Z^0}) \bar{\mcb{I}}[\bXi]|_{\CC([0,\tau],\Omega^1_\alpha)} = 0\;.
\end{equ}
Recall that the kernel $K^\eps=K*\moll^\eps$ corresponding to $\bar{\mcb{K}}_{i} = \bar{\mcG}_i - \bar R\mcR$ is non-anticipative
and comes with the same parameter $\eps$
and the model component of $I^\eps$ realises $K^\eps$ for the associated abstract integration operator $\bar{\mcb{I}}$.

\begin{definition}\label{def:maximal}
The \textit{maximal solution} of the system~\eqref{eq:abstract_fixed_point_eq_2} and~\eqref{eq:abstract_fixed_point_eq_3} associated to $I^\eps$ is defined as follows.
We first solve the fixed point problem~\eqref{eq:abstract_fixed_point_eq_2} for $(\CB,\mcU,\CH)$ in the space~\eqref{eq:fixed_pt_space} over some interval $(0,\sigma_1]\eqdef (0,2\tau_1]$ using the fact that $I^\eps$ is a $1$-input
(see Remark~\ref{rem:tau_nonunique} on the choice of $\tau_1$).
By
Lemma~\ref{lem:fixed_close_enhanced}, we can furthermore assume that
$\mcR\CB(\sigma_1)$ takes values in $\Omega^1_\alpha \hookrightarrow \CC^\eta$.
By the proof of the same lemma, we can likewise assume $\mcR\bar\CG(\mcU\bXi)(\sigma_1) \in \Omega^1_\alpha$. 

We then solve the fixed point problem~\eqref{eq:abstract_fixed_point_eq_3}
with time centred around $\sigma_1$ instead of $0$, and with $T=\tau_1$ and $\mcU\restr_{[\sigma_1-T,\sigma_1]}$ playing the role of $\tilde \mcU$.
The initial condition is chosen as $B^{(0)}=\mcR\CB(\sigma_1)$ and $h^{(0)}=\mcR\CH(\sigma_1)$.
We therefore extend the solution to an interval  $(0,\sigma_2]\eqdef (0,\sigma_1+2\tau_2]$,
and, by Lemma~\ref{lem:restart_improve} (and its proof), we can again assume that
$\mcR\CB(\sigma_2)$ and $\mcR\bar\CG(\mcU\bXi)(\sigma_2)$ both take values in $\Omega^1_\alpha \hookrightarrow \CC^\eta$.

We then again solve~\eqref{eq:abstract_fixed_point_eq_3}
with time centred around $\sigma_2$, and with $T=\tau_2$ and $\mcU\restr_{[\sigma_2-T,\sigma_2]}$ playing the role of $\tilde \mcU$.
The initial condition is chosen as $B^{(0)}=\mcR\CB(\sigma_2)$ and $h^{(0)}=\mcR\CH(\sigma_2)$.
Proceeding inductively, we define $(\CB,\mcU,\mcH)$ on the interval $(0,\tau^\star_\eps)$ where $\tau^\star_\eps = \sum_{i=1}^\infty \sigma_i$.
\end{definition}

Let $Q^\eps\eqdef (\bar A^\eps,\bar U^\eps,\bar h^\eps)$ denote the reconstruction on each interval $(\sigma_i,\sigma_{i+1}]$ of the corresponding modelled distribution.
Observe that, by Lemma~\ref{lem:restarted}, if $\tau^\star_\eps<\infty$ then 
\begin{equ}[eq:blow_up_criteria]
\lim_{k\to\infty} |\bar A^\eps(\sigma_k)|_{\CC^\eta}+
|\bar h^\eps(\sigma_k)|_{\CC^{\alpha-1}}+|\mcU|_{\cD^{2+2\kappa};[\sigma_k-\tau_{k},\sigma_k]}+ |v(\sigma_k)|_{\CC^\eta} = \infty
\end{equ}
where $v=(v_i)_{i=1}^2$ is defined on $\sigma_1,\sigma_2,\ldots$ as in~\eqref{eq:v_i_def}.

\begin{remark}\label{rem:tau_nonunique}
The $\tau_i$ taken in Definition~\ref{def:maximal}
is not unique -- it can be chosen as any value $\tau_i\leq \tau$, with $\tau$ as in Lemmas~\ref{lem:restarted} and~\ref{lem:restart_improve}.
Furthermore, if we choose two different $\tau_n,\bar\tau_n$ with $\tau_n\wedge \bar\tau_n \leq \eps^2$ for some $n\geq 1$,
then we will, in general, obtain \textit{different} reconstructions $Q^\eps$ on $(0,\sigma_n\wedge\bar\sigma_n]$.

However, if we make two choices $\tau_i,\bar\tau_i$ with $\tau_i\wedge\bar\tau_i > \eps^2$ for all $i=1,\ldots, n$,
then $Q^\eps$ remains unaffected on $(0,\sigma_n\wedge\bar\sigma_n]$ -- this is because the term $Gv_i$ in~\eqref{eq:abstract_fixed_point_eq_3} removes the dependence of $-\bar\CG_i (\tilde \mcU \boldsymbol{\Xi}_{i})$ on times in $(-\infty,\eps^2)$ upon reconstruction.
\end{remark}

\begin{remark}\label{rem:discrepancy}
In a similar vein to Remark~\ref{rem:tau_nonunique}, the reconstructions $Q^\eps$ will \textit{not} in general solve an equation of the form~\eqref{eq:bar_A_final}+\eqref{e:final_system}.
They will only solve such an equation on the interval $(0,\sigma_n]$ provided that $\eps^2<\tau_i$ for all $i\leq n$.
However, since $I^\eps\to I^0$ as $\eps\downarrow0$,
for every $n\geq1$ there exists $\eps_n>0$ such that we can take the same $\{\tau_i\}_{i=1}^n$ for
all $i=1,\ldots, n$ and $\eps \in [0,\eps_n]$
and such that $\eps^2<\tau_i$.
In particular, this discrepancy won't matter in the $\eps\downarrow0$ limit.
\end{remark}

\begin{proposition}\label{prop:conv_max_sols}
It holds that $Q^\eps\in\CC((0,\tau^\star_\eps),\Omega^1_\alpha\times\hat\mfG^{0,\alpha})$.
Furthermore, define $Q^\eps(t)= \skull$ for $t\geq \tau^\star_\eps$, and, for $\lambda>0$, denote $Q^{\eps;\lambda}(t) = Q^\eps(t -\lambda)$.
If $Q^{0;\lambda} \in (\Omega^1_\alpha\times\hat\mfG^{0,\alpha})^\sol$, then
$Q^{\eps;\lambda}\to Q^{0;\lambda}$ as $\eps\downarrow0$ in same sense as in Lemma~\ref{lem:alt_conv}\ref{pt:D_conv_alt}
where the underlying metric space is $\Omega^1_\alpha\times\hat\mfG^{0,\alpha}$.
\end{proposition}

\begin{proof}
The fact that $Q^\eps\in\CC((0,\tau^\star_\eps),\Omega^1_\alpha\times\hat\mfG^{0,\alpha})$ follows from Lemma~\ref{lem:restart_improve} and the fact that we can also restart the fixed point~\eqref{eq:abstract_fixed_point_eq_3}
at times $\sigma_k-\frac12\tau_k$ instead of $\sigma_k$ but with the same $\tilde \mcU$ on $[\sigma_k-\tau_k,\sigma_k-\frac12\tau_k]$
(which gives the same reconstruction).

The convergence $Q^{\eps;\lambda} \to Q^{0;\lambda}$ in the sense of Lemma~\ref{lem:alt_conv}\ref{pt:D_conv_alt}
readily follows from the stability results in Lemmas~\ref{lemma:fixedptpblmclose},~\ref{lem:fixed_close_enhanced},~\ref{lem:restarted} and~\ref{lem:restart_improve}.
The key remark is that the modelled distributions $(\CB^\eps,\mcU^\eps,\mcH^\eps)$ 
stay close to $(\CB^0,\mcU^0,\mcH^0)$ until either the modelled distribution norm of $\mcU^0$
becomes large or $|\CR\bar\CG_i ( \mcU^0 \boldsymbol{\Xi})|_{\Omega^1_\alpha}+ |Q^0|_{\Omega^1_\alpha\times\hat\mfG^{0,\alpha}}$ becomes large,
and this is enough to control the difference $|Q^\eps-Q^0|_{\Omega^1_\alpha\times\hat\mfG^{0,\alpha}}$ (as well as the difference $|\CR\bar\CG_i ( \mcU^\eps \boldsymbol{\Xi})-\CR\bar\CG_i ( \mcU^0 \boldsymbol{\Xi})|_{\Omega^1_\alpha}$)
by Lemmas~\ref{lem:fixed_close_enhanced} and~\ref{lem:restart_improve}.
The assumption $Q^{0;\lambda}\in (\Omega^1_\alpha\times\hat\mfG^{0,\alpha})^\sol$
ensures that $Q^{\eps;\lambda}$ can only be far from $Q^{0;\lambda}$ when \textit{both} are large,
which implies the condition stated in Lemma~\ref{lem:alt_conv}\ref{pt:D_conv_alt}.
\end{proof}


%
\subsubsection{Probabilistic bounds}\label{subsubsec:prob_bounds}
The key input in our argument regarding  stochastic control of models is given by the following lemma. 
\begin{lemma}\label{lem:conv_of_models2}
One has, for any $p \ge 1$, 
\begin{equ}\label{eq:eps_control_of_model}
\sup_{\eps \in (0,1]}
\sup_{\delta \in (0,\eps)}
\E[ \|Z^{\delta ,\eps}_{\BPHZ}\|_{\eps}^{p} ] 
< \infty\;.
\end{equ}
Moreover, there exist models $Z^{0,\eps}_{\BPHZ} \in \mathscr{M}_{\eps}$ for $\eps \in (0,1]$ such that, for any such $\eps$,  
\begin{equ}\label{eq:delta_conv_of_models}
\lim_{\delta \downarrow 0} Z^{\delta,\eps}_{\BPHZ}
= Z^{0,\eps}_{\BPHZ}
\end{equ}
in probability with respect to the topology of $d_{\eps}(\act,\act)$.
 
Finally, there exists a model $Z^{0,0}_{\BPHZ} \in \mathscr{M}_{0}$ such that
\begin{equ}\label{eq:eps_conv_of_models}
\lim_{\eps \downarrow 0}
Z^{0,\eps}_{\BPHZ}
= Z^{0,0}_{\BPHZ}
\end{equ}
in probability with respect to the topology of $d_{1}(\act,\act)$.
\end{lemma}
\begin{proof}
As in Lemma~\ref{lem:conv_of_models} we proceed by using the results of \cite{CH16}. 
We start by proving \eqref{eq:delta_conv_of_models} and here we appeal to \cite[Theorem~2.15]{CH16}. 
We first note that for any scalar noise decomposition, it is straightforward to verify that the random smooth noise assignments $\zeta^{\delta,\eps}$ are a uniformly compatible family of Gaussian noises that converge to the Gaussian noise $\zeta^{0,\eps}$. 
The verification of the first three listed power-counting conditions of \cite[Theorem~2.15]{CH16} is analogous to how they were checked for Lemma~\ref{lem:conv_of_models}.
 This gives the existence of the limiting models $Z^{0,\eps}_{\BPHZ}$ and the desired convergence statement (note that for fixed $\eps > 0$, the metric $d_{\eps}(\act,\act)$ is equivalent to $d_{1}(\act,\act)$).  

To prove \eqref{eq:eps_conv_of_models} we will show
\begin{equ}\label{eq:double_limit_model}
\lim_{\eps \downarrow 0}
\sup_{
\delta, \tilde{\eps} \in (0,\eps)}
\E[ d_{1}(Z^{\delta,\eps}_{\BPHZ},Z^{\delta,\tilde{\eps}}_{\BPHZ})^{2}]
=
0\;.
\end{equ}
By using Fatou to take the limit $\delta \downarrow 0$, this gives us that $Z_{\BPHZ}^{0,\eps}$ is Cauchy in $L^{2}$ as $\eps \downarrow 0$ and so we obtain the desired limiting model $Z^{0,0}_{\BPHZ}$ and the desired convergence statement.  

To prove \eqref{eq:double_limit_model} we will use  the more quantitative \cite[Theorem~2.31]{CH16}.
Here we take $\mfL_{\mathrm{cum}}$ to be the set of all pairings of $\Lab_{-}$ and so the three power-counting conditions we verified for \cite[Theorem~2.15]{CH16} also imply the super-regularity assumption of \cite[Theorem~2.31]{CH16}. Since we only work with pairings, the cumulant homogeneity $\mfc$ is determined by our degree assignment on our noises. 
After rewriting the difference of the action of models as a telescoping sum which allows one to factor the corresponding difference in the kernel assignment $K - K^{\eps}$ or noise assignment $\xi_{i}^{\tilde{\eps}} - \xi_{i}^{\eps}$, one is guaranteed at least one factor of order $\eps^{\kappa}$ the right-hand side of the bound \cite[(2.15)]{CH16} \dash coming from $\|K - K^{\eps}\|_{2 - \kappa,k}$ in the first case or the contraction $\xi_{i}^{\tilde{\eps}} - \xi_{i}^{\eps}$ with another noise measured in the $\| \act \|_{-4-2\kappa,k}$ kernel norm in the second case. 
This gives us the estimate \eqref{eq:double_limit_model}. 

The above argument for obtaining \eqref{eq:double_limit_model} can also be applied to obtain \eqref{eq:eps_control_of_model}, namely, with the constraint  that $\delta \in (0,\eps)$, any occurrence of $\mcb{I}_{\mft,p} - \bar{\mcb{I}}_{\mft,p}$ gives a factor of $\eps^{\kappa}$ through the difference $K - K^{\eps}$ and any occurrence of $\Xi_{\mfl_{i}} - \Xi_{\bar{\mfl}_{i}}$ gives a factor $\eps^{\kappa}$ through the difference $ \xi_{i}^{\delta} - \xi_{i}^{\eps}$ and since $\theta \in (0,\kappa]$ this gives the suitable uniform in $\eps$ bounds on the moments of the model norm $\| \act \|_{\eps}$.  
\end{proof}

The next bound helps to control the difference between two solutions started from different (possibly random) initial conditions.
This is important when restarting the equation to obtain convergence of the maximal solutions.
For $\sigma\in\R$ and $U\in\CC^\kappa(\T^2)$,
we denote
$G_{(\sigma)} U (t,\cdot) \eqdef G U(t-\sigma,\cdot)$
the heat flow started at time $\sigma$ with initial condition $U$.
In particular $G_{(0)} U = G U$.
Let $\mbF=(\mcF_t)_{t\geq 0}$\label{pageref:mbF} denote the filtration generated by the white noise $\xi\eqdef(\xi_1,\xi_2)$.

\begin{lemma}\label{lem:Uxi_conv}
Let $\kappa>0$ and $\K\subset\R\times\T^2$ be compact.
For all $0<\bar\eps\leq\eps\leq1$,
$\mbF$-stopping times $\sigma\in[0,1]$,
and $\mcF_\sigma$-measurable $U\in \CC^{\kappa}(\T^2)$,
there exists an $\mcF_\sigma$-measurable random variable $M>0$ such that
\begin{equs}
\E
\big[
|G_{(\sigma)} U \xi^\eps |_{\CC^{-2-\kappa }(\K)}
\mid \mcF_\sigma \big]
&\leq |U|_{\CC^\kappa} M
\;, \label{eq:diff_init_cond}
\\
\E
\big[
|G_{(\sigma)} U (\xi^\eps - \xi^{\bar\eps}) |_{\CC^{-2-\kappa }(\K)}
\mid \mcF_\sigma \big]
&\leq |U|_{\CC^\kappa}M \eps^\varsigma\;,\label{eq:same_init_cond}
\end{equs}
where $\varsigma>0$ depends only on $\kappa$.
Furthermore every moment of $M$ is bounded uniformly in $\sigma$, $U$, $\eps$, and $\bar\eps$.
In particular, for every $\mcF_\sigma$-measurable $U\in \CC^\kappa(\T^2)$,
the limit
\begin{equ}\label{eq:Uxi_conv}
G_{(\sigma)} U \xi \eqdef \lim_{\eps\to 0} G_{(\sigma)} U \xi^\eps
\end{equ}
exists in $\CC^{-2-\kappa}$.
\end{lemma}

\begin{proof}
We prove the statement for $\sigma=0$; the general case follows in an identical manner.
It suffices to prove the statement for $\bar\eps\in[\frac\eps2,\eps]$ with the general case following from telescoping.
Observe that $\xi^\eps \restr_{[2\eps^2,\infty)}$ and $\xi^{\bar\eps} \restr_{[2\eps^2,\infty)}$
are independent of $\mcF_0$ and therefore of $GU$.

It follows that, uniformly in $\psi\in\mcB^3$, $\lambda,\eps\in(0,1)$,
\begin{equ}
\E\big[
\scal{GU\xi^\eps\restr_{[2\eps^2,\infty)},\psi^\lambda_z}^2 \mid \mcF_0
\big]
\leq |\moll^\eps *(GU\psi^\lambda_z)|^2_{L^2}\lesssim |GU|_{L^\infty}^2 \lambda^{-4}\;.
\end{equ}
Likewise, uniformly in $\psi\in\mcB^3$, $\lambda,\eps\in(0,1)$, and  $\bar\eps\in(0,\eps)$, 
\begin{equs}
\E\big[\scal{GU(\xi^\eps - \xi^{\bar\eps})\restr_{[2\eps^2,\infty)},\psi^\lambda_z}^2 \mid \mcF_0
\big]
&\leq
|(\moll^\eps-\moll^{\bar\eps}) *(GU\psi_z^\lambda)|^2_{L^2}
\\
&\lesssim \eps^{\kappa}\lambda^{-4-\kappa}|U|^2_{\CC^{\kappa/2}}\;.
\end{equs}
Hence the required estimates hold for $GU\xi^\eps\restr_{[2\eps^2,\infty)}$ and $GU(\xi^\eps-\xi^{\bar\eps})\restr_{[2\eps^2,\infty)}$
by the equivalence of Gaussian moments and a Kolmogorov argument.

For the contribution on $[0,2\eps^2]$, define the random variable
\begin{equ}
M = \eps^{2-\frac\kappa4}\sup_{t \in [0,8]} |\xi^\eps(t,\cdot)|_{\CC^{-\frac\kappa2}}\;.
\end{equ}
Note that every moment of $M$ is bounded uniformly in $\eps$.
Then
\begin{equs}
|\scal{G U \xi^\eps \restr_{[0,2\eps^2]}, \psi^\lambda_z}|
&\lesssim  \lambda^{-2-\frac\kappa2} \eps^2 \sup_{t \le 2\eps^2} 
|(G U)(t,\cdot) \xi^\eps(t,\cdot)|_{\CC^{-\frac\kappa2}}
\\
&\lesssim
M\lambda^{-2-\frac\kappa2}\eps^{\frac\kappa4} |U|_{\CC^{\kappa}}\;.
\end{equs}
uniformly in $\psi\in\CB^3$, $\lambda\in (0,1)$, and $\eps\in(0,1)$.
The same bound holds for $|\scal{G U \xi^{\bar\eps} \restr_{[0,2\eps^2]}, \psi^\lambda_z}|$, which is where we use $\frac\eps2\leq\bar\eps \leq\eps$.
%
%
\end{proof}

Our final probabilistic estimates controls the term $\CW$ in~\eqref{eq:abstract_fixed_point_eq_2}.

\begin{lemma}\label{lem:CW}
Let $\zeta=G*\bone_+\moll^\eps*(\xi \bone_{-})\in \CC^\infty(\R_+)$
and let $\CW$ be its lift to the polynomial sector defined for positive times.
Then, for all $\kappa\in (0,1)$, there exists a random variable $M$ with moments bounded of all orders such that, for all $\tau\in(0,1)$,
\begin{equ}
|\CW|_{\frac32,-2\kappa; \tau} \leq M\eps^\kappa\;.
\end{equ}
\end{lemma}

\begin{proof}
Recall that $\sup_{t\in[0,\eps^2]}|\moll^\eps*(\xi \bone_{-})(t)|_{\CC^{-2\kappa}} \leq M \eps^{-2+\kappa}$.
Therefore, for all $\gamma \in [0,2-2\kappa)\setminus\{1\}$
and $t\in [0,2\eps^2]$
\begin{equs}
|\zeta(t)|_{\CC^\gamma} \leq
\int_0^t | G_{t-s}(\moll^\eps*(\xi \bone_{-})(s))|_{\CC^\gamma}\mrd s
&\lesssim \int_0^t (t-s)^{-\frac\gamma2-\kappa} M \eps^{-2+\kappa}\mrd s
\\
&\lesssim
Mt^{1-\frac\gamma2-\kappa} \eps^{-2+\kappa}
\lesssim Mt^{-\kappa-\frac\gamma2}\eps^\kappa\;.
\end{equs}
For $t\in(2\eps^2,\tau)$, we simply have
\begin{equs}
|\zeta(t)|_{\CC^\gamma} \leq
\int_0^t | G_{t-s}(\moll^\eps*(\xi \bone_{-})(s))|_{\CC^\gamma}\mrd s
&\lesssim \int_0^{\eps^2} (t-s)^{-\frac\gamma2-\kappa} M \eps^{-2+\kappa}\mrd s
\\
&\lesssim
\eps^2(t-\eps^2)^{-\frac\gamma2-\kappa}
M\eps^{-2+\kappa}
\\
&\lesssim Mt^{-\frac\gamma2-\kappa}\eps^\kappa\;,
\end{equs}
from which the conclusion follows.
\end{proof}

\subsection{Renormalised equations for the gauge transformed system}
\label{sec:renorm_for_system}

In this section we derive the renormalised equations for the $B$ system and the $\bar A$ system 
and prove that they converge to the same limit, i.e. Proposition~\ref{prop:SPDEs_conv_zero}.

Given $\delta \in (0,1]$ and $\eps \in (0,1]$, we write $\ell^{\delta,\eps}_{\BPHZ}[\act]$ for the BPHZ renormalisation group character that goes between the canonical lift and $Z_{\BPHZ}^{\delta,\eps}$.
The rule given below \eqref{e:rule-gauged} determines the set $\mfT_{-}(R)$ of trees as in \eqref{e:def-mfT-}
and we only  list the trees in $\mfT_{-}(R)$ that are relevant  to deriving  the renormalised equations in the following two tables (for the $F$ system and the $\bar F$ system respectively). 
The reason that we will only need to be concerned with these trees will be clear by  Lemma~\ref{lemma:list_of_trees_gsym} below, which follows easily from the definition of $\Upsilon^{\act}_{\mft}[\act]$  and the parity constraints on the noises and spatial derivatives that are necessary for $\ell^{\delta,\eps}_{\BPHZ}[\act]$ not to vanish. 

Here the graphic notation 
is similarly as in Section~\ref{sec:solution_theory}:
(thick) lines denote (derivatives of) $\mcb{I}$,
colours denote spatial indices, and
the colour of a tiny triangle labels
the spatial index for the kernel immediately below it. Moreover, we draw a circle (resp. crossed circle) for $\bar\Xi$ (resp. $\mbX\bar\Xi$),
with a convention that the line immediately below
it understood as $\mcb{I}$,
and a square (resp. crossed square) for $\Xi$ (resp. $\mbX\Xi$),
with a convention that the line immediately below
it understood as $\bar{\mcb{I}}$.
We also draw a zigzag line $\tikz[baseline=-3] \draw[snake=snake , segment length=1pt, segment amplitude=1pt] (0,0) -- (0.3,0) ;$
for $\mcb{I}^\mfu$ and
a wavy line $\tikz[baseline=-3] \draw[snake=snake , segment length=3pt, segment amplitude=1pt] (0,0) -- (0.3,0) ;$
for $\mcb{I}^\mfh$. Their thick versions and tiny triangles above them are understood as before. 
\begin{center}
\begin{minipage}{.45\linewidth}
\centering
{\setlength{\extrarowheight}{5pt}
\begin{tabular}{cc}
\toprule\\[-1.7\normalbaselineskip] 
\multicolumn{2}{c}{Table~\refstepcounter{tables}\thetables\label{table:TableB}}\\
\midrule 
$\<IXi^2green>$ & $\mcb{I}_{\<green>}(\bar{\Xi}_{\<green>})^{2}$ 
\\
$\<IXiI'[I'Xi]_typed>$ & $\mcb{I}_{\<green>}(\bar{\Xi}_{\<green>})\mcb{I}_{\<orange>,\<red>}(\mcb{I}_{\<green>,\<red>}(\bar{\Xi}_{\<green>}))$ \\
$\<I[I'Xi]I'Xi_typed>$ & $\mcb{I}_{\<green>,\<red>}(\bar{\Xi}_{\<green>})\mcb{I}_{\<orange>}(\mcb{I}_{\<green>,\<red>}(\bar{\Xi}_{\<green>}))$ 
\\  
$\<IXiI'XXi_typed>$ & $\mcb{I}_{\<green>}(\bar{\Xi}_{\<green>})\mcb{I}_{\<green>,\<red>}(\mbX_{\<red>}\bar{\Xi}_{\<green>}) $
\\ 
$\<IXXiI'Xi_typed>$ & $\mcb{I}_{\<green>}(\mbX_{\<red>}\bar{\Xi}_{\<green>})\mcb{I}_{\<green>,\<red>}(\bar{\Xi}_{\<green>})$  
\\
$\<I[I[Xi]]Xi_typed>$ & $\mcb{I}^{\mfu}(\mcb{I}_{\<green>}(\bar{\Xi}_{\<green>}))\bar{\Xi}_{\<green>}$
\\
$\<I[I'Xi]I'Xi_typed-h>$ & $\mcb{I}_{\<green>,\<green>}(\bar{\Xi})\mcb{I}_{\<green>}^{\mfh}(\mcb{I}_{\<green>,\<green>}(\bar{\Xi}))$
\\
$\<IXiI'[I'Xi]_typed-h>$ & $\mcb{I}_{\<green>}(\bar{\Xi})\mcb{I}_{\<green>,\<green>}^{\mfh}(\mcb{I}_{\<green>,\<green>}(\bar{\Xi}))$\\
\bottomrule
\end{tabular}
}
\end{minipage}
\qquad
\begin{minipage}{.45\linewidth}
\centering
{\setlength{\extrarowheight}{5pt}
\begin{tabular}{cc}
\toprule\\[-1.7\normalbaselineskip]
\multicolumn{2}{c}{Table~\refstepcounter{tables}\thetables\label{table:TablebarA}}\\
\midrule 
$\<IXi^2green*>$ & $\bar{\mcb{I}}_{\<green>}(\Xi_{\<green>})^{2}$ 
\\ 
$\<IXiI'[I'Xi]_typed*>$ & $\bar{\mcb{I}}_{\<green>}(\Xi_{\<green>})\mcb{I}_{\<orange>,\<red>}(\bar{\mcb{I}}_{\<green>,\<red>}(\Xi_{\<green>}))$ 
\\  
$\<I[I'Xi]I'Xi_typed*>$ & $\bar{\mcb{I}}_{\<green>,\<red>}(\Xi_{\<green>})\mcb{I}_{\<orange>}(\bar{\mcb{I}}_{\<green>,\<red>}(\Xi_{\<green>}))$
 \\   
$\<IXiI'XXi_typed*>$ & 
$\bar{\mcb{I}}_{\<green>}(\Xi_{\<green>})\bar{\mcb{I}}_{\<green>,\<red>}(\mbX_{\<red>}\Xi_{\<green>})$
\\ 
$\<IXXiI'Xi_typed*>$ & $\bar{\mcb{I}}_{\<green>}(\mbX_{\<red>}\Xi_{\<green>})\bar{\mcb{I}}_{\<green>,\<red>}(\Xi_{\<green>})$   \\
$\<I[I[Xi]]Xi_typed*>$ & $\mcb{I}^{\mfu}(\bar{\mcb{I}}_{\<green>}(\Xi_{\<green>}))\Xi_{\<green>}$\\
$\<I[I'Xi]I'Xi_typed-h*>$ & $\bar{\mcb{I}}_{\<green>,\<green>}(\Xi)\mcb{I}_{\<green>}^{\mfh}(\bar{\mcb{I}}_{\<green>,\<green>}(\Xi))$\\
$\<IXiI'[I'Xi]_typed-h*>$ & $\bar{\mcb{I}}_{\<green>}(\Xi)\mcb{I}_{\<green>,\<green>}^{\mfh}(\bar{\mcb{I}}_{\<green>,\<green>}(\Xi))$\\
\bottomrule
\end{tabular}
}
\end{minipage}
\end{center}

The first five trees in each of the two tables have the same structure as the ones 
that appeared in Section~\ref{sec:solution_theory},
except that now the noises are understood as $\Xi$ or $\bar\Xi$, and edges understood as $\mcb{I}$ or $\bar{\mcb{I}}$. 
An important difference from Section~\ref{sec:solution_theory} is that 
the trees of the type
$\<IXiI'XXi_notriangle>\ $ and $\<IXXiI'Xi_notriangle>$  had vanishing $\Upsilon$ in Section~\ref{sec:solution_theory} 
and therefore no effect on the renormalised equation, but  this is not  the case now, as we will see below, 
due to the term $U \chi^{\eps} \ast  \xi$ (or $\chi^{\eps} \ast (\bar U \xi)$) in our equation.
Moreover, the tables also show trees in 
 $\mfT_{-}(R)$ such as those of the form $\<I[I[Xi]]Xi_notriangle>$
 which do not have any counterpart in Section~\ref{sec:solution_theory}. 

\begin{lemma}\label{lemma:list_of_trees_gsym}
If $\tau \in \mfT_{-}(R)$ is not of any of the forms listed in Table~\ref{table:TableB} (resp. Table~\ref{table:TablebarA}) then either $\ell^{\delta,\eps}_{\BPHZ}[\tau] = 0$ or $\Upsilon_{\mft}^{F}[\tau] = 0$
(resp. either $\ell^{\delta,\eps}_{\BPHZ}[\tau] = 0$ or $\Upsilon_{\mft}^{\bar F}[\tau] = 0$) for every $\mft \in \Lab_{+}$.
\end{lemma}
\begin{proof}
The proof of this lemma follows similar lines as Lemma~\ref{lem:symbols_vanish}, so we do not repeat the details. We only remark that 
for trees with 
a ``polynomial'' $\mathbf{X}$, namely
\begin{equ}
\mcb{I}_{\<red>,\<dblue>}  (\bar\Xi_{\<red>}) \mcb{I}_{\<green>} (\mathbf{X}_{\<orange>} \bar{\Xi}_{\<green>}),
\quad
\mcb{I}_{\<dblue>}(\bar\Xi_{\<dblue>})\mcb{I}_{\<red>,\<green>} (\mathbf{X}_{\<orange>} \bar{\Xi}_{\<red>}) ,
\quad
\bar{\mcb{I}}_{\<red>,\<dblue>}  (\Xi_{\<red>}) \bar{\mcb{I}}_{\<green>} (\mathbf{X}_{\<orange>} \Xi_{\<green>}),
\quad
\bar{\mcb{I}}_{\<dblue>}(\Xi_{\<dblue>})\bar{\mcb{I}}_{\<red>,\<green>} (\mathbf{X}_{\<orange>} \Xi_{\<red>}) ,
\end{equ}
the polynomial can be dealt with in the same way as for the derivative in Lemma~\ref{lem:symbols_vanish}; for instance for  the first tree,
if $\<dblue>\neq \<orange>$,
then flipping the sign of the $\<dblue>$-component (or, $ \<orange>$-component) of the appropriate  integration variable shows that $ \bar{\PPi}_{\can}[\tau]=0$.
\end{proof}

We now state a sequence of lemmas with identities for $\bar\Upsilon^F$ and $\bar\Upsilon^{\bar F}$, but we will not give the 
detailed calculations within the proof of each lemma, since these are straightforward (for instance they follow  similarly 
as in Section~\ref{sec:solution_theory}). 
We first show that in both $F$ and $\bar{F}$ systems we don't see any renormalisation of the $\mfu$ or $\mfh_{i}$ equations. 
\begin{lemma}\label{lemma:renorm-for-u-h}
For any of the $\tau$ of the form listed in Table~\ref{table:TableB} (resp. Table~\ref{table:TablebarA}) one has  $\bar{\bUpsilon}^{F}_{\mfu}[\tau] = 0$ (resp. $\bar{\bUpsilon}_{\mfu}^{\bar{F}}[\tau] = 0$).
Moreover, we have 
\[
\sum_{\tau \in \mfT_{-}(R)}
(\ell_{\BPHZ}^{\delta,\eps}[\tau] \otimes \id)
\bar{\bUpsilon}_{\mfh,\<dblue>}^{F}[\tau](\pr{\mathbf{A}})
=
\sum_{\tau \in \mfT_{-}(R)}
(\ell_{\BPHZ}^{\delta,\eps}[\tau] \otimes \id)
\bar{\bUpsilon}_{\mfh,\<dblue>}^{\bar{F}}[\tau](\pr{\mathbf{A}}) = 0\;.
\]
\end{lemma}

\begin{proof}
The fact that 
$\bar{\bUpsilon}^{F}_{\mfu}[\tau] = \bar{\bUpsilon}_{\mfu}^{\bar{F}}[\tau] = 0$ for $\tau$ appearing in the tables
follows from direct computation.
One has $\bar{\bUpsilon}^{F}_{\mfh,\<dblue>}[\tau] = 0$ (resp. $\bar{\bUpsilon}_{\mfh,\<dblue>}^{\bar{F}}[\tau] = 0$) for any $\tau$ in Table~\ref{table:TableB} (resp. Table~\ref{table:TablebarA}) of the first six shapes. 
For the other trees one has, by integration by parts, 
\begin{equs}\label{eq:renorm_h}
{}\ell_{\BPHZ}^{\delta,\eps} [\<I[I'Xi]I'Xi_typed-h>]
= - \ell_{\BPHZ}^{\delta,\eps}[\<IXiI'[I'Xi]_typed-h>]\;,
\qquad 
\ell_{\BPHZ}^{\delta,\eps} [\<I[I'Xi]I'Xi_typed-h*>]
= - \ell_{\BPHZ}^{\delta,\eps}[\<IXiI'[I'Xi]_typed-h*>]\;.
\end{equs}
Additionally, one has
\begin{equs}\label{eq:coherence_h}
{}\Upsilon_{\mfh, \<dblue>}[\<I[I'Xi]I'Xi_typed-h>]
=\Upsilon_{\mfh, \<dblue>}[\<IXiI'[I'Xi]_typed-h>]\;,
\qquad
\Upsilon_{\mfh, \<dblue>}[\<I[I'Xi]I'Xi_typed-h*>]
=\Upsilon_{\mfh, \<dblue>}[\<IXiI'[I'Xi]_typed-h*>]\;.
\end{equs}
Above we are exploiting the canonical isomorphisms between the spaces where the objects above live \dash namely for any 
two trees $\tau$, $\bar \tau$ of any of the four forms appearing above, one has a canonical isomorphism 
$\CT[\tau] \simeq \CT[\bar \tau]$ by using Remark~\ref{rem:canonicalIsomorphism} and the canonical isomorphisms
between these trees obtained by only keeping their tree structure.
Combining \eqref{eq:renorm_h} with \eqref{eq:coherence_h} then yields the last claim.
\end{proof}
We define a subset $\bar{\mcb{A}} \subset \mcb{A}$ that encodes additional constraints on the jet 
of our solutions which comes from~\eqref{eq:h_and_U_def}. 
These constraints will help us simplify the counterterms for the $\mfa_{i}$ and $\mfm_{i}$ equations.
%
\begin{definition}\label{def:special_jets}
We define $\bar{\mcb{A}}$ to be the collection of all $\pr{\mathbf{A}} = (\pr{\mathbf{A}_{o}})_{ o \in \CE} \in \mcb{A}$ such that 
\begin{itemize}
\item $\pr{\mathbf{A}_{\mfu}}$ is unitary.
\item For all $a,b \in \mfg$, $\pr{\mathbf{A}_{\mfu}}[a,b] = [\pr{\mathbf{A}_{\mfu}}a,\pr{\mathbf{A}_{\mfu}}b]$.
\end{itemize} 
\end{definition}
We now turn to explicitly identifying the renormalisation counterterms for the $\mfa_{i}$ and $\mfm_{i}$ equations in the $\bar{F}$ system. 

We start by collecting formulae for the the renormalisation constants.
Write $K^{\delta,\eps} = K^{\eps} \ast \chi^{\delta} = K \ast \chi^{\eps} \ast \chi^{\delta}$
and recall the constants $\hat{C}^{\eps}$ and $\bar{C}^{\eps}$ defined in \eqref{e:defConstants}.
We then define the variants
\begin{equ}[e:defConstants_2]
\bar{C}^{\delta,\eps} \eqdef \int \mrd z\ K^{\delta,\eps}(z)^{2}\;,\qquad
\hat{C}^{\delta,\eps}
\eqdef \int \mrd z\ \partial_{j}K^{\delta,\eps}(z)(\partial_{j}K*K^{\delta,\eps})(z)\;,
\end{equ}
where one can choose any $j  \in \{1,2\}$ as in \eqref{e:defConstants}. 
We then have the following lemma.
\begin{lemma}\label{lem:tree_calcs2}
For $\hat{C}^{\eps}$ and $\bar{C}^{\eps}$ as in \eqref{e:defConstants}, one has  
\begin{equs}[e:valueRenorm2b]
{}&\ell_{\BPHZ}^{\delta,\eps}[\<I[I'Xi]I'Xi_typed>] = 
-\ell_{\BPHZ}^{\delta,\eps}[\<IXiI'[I'Xi]_typed>] = -\hat{C}^{\eps}\Cas\;, \quad \ell_{\BPHZ}^{\delta,\eps}[\<IXi^2green>] = -\bar{C}^\eps \Cas\;.
\end{equs} 
For $\hat{C}^{\delta,\eps}$ and $\bar{C}^{\delta,\eps}$ defined as in \eqref{e:defConstants_2} one has
\begin{equs}[e:valueRenorm2a]
{}&\ell_{\BPHZ}^{\delta,\eps}[\<I[I'Xi]I'Xi_typed*>] = 
-\ell_{\BPHZ}^{\delta,\eps}[\<IXiI'[I'Xi]_typed*>] =- \hat{C}^{\delta,\eps}\Cas\;, 
\quad 
\ell_{\BPHZ}^{\delta,\eps}[\<IXi^2green*>] =- \bar{C}^{\delta,\eps} \Cas\;.
\end{equs}
Finally, for any $\eps>0$, one has $\lim_{\delta \downarrow 0} \hat{C}^{\delta,\eps} = \hat{C}^{\eps}$ and $\lim_{\delta \downarrow 0} \bar{C}^{\delta,\eps} = \bar{C}^{\eps}$.
\end{lemma}
\begin{proof}
The statements \eqref{e:valueRenorm2b} and \eqref{e:valueRenorm2a} follow in the same way as Lemma~\ref{lem:tree_calcs}. 
The final statement about convergence as $\delta \downarrow 0$ of the renormalisation constants is obvious. 
\end{proof}
We introduce additional renormalisation constants 
\begin{equ}[e:defConstantsTildeC]
\tilde{C}^{\eps}
\eqdef
\int \mrd z\ \chi^{\eps}(z) (K \ast K^{\eps})(z)\;,
\quad
\tilde{C}^{\delta,\eps}
\eqdef
\int \mrd z\ \chi^{\delta}(z) (K \ast K^{\delta,\eps})(z)\;.
\end{equ}
The following lemma is straightforward to prove. 
\begin{lemma}\label{lemma:renorm_constants_gsym}
One has 
\begin{equs}
\ell_{\BPHZ}^{\delta,\eps}
[\<I[I[Xi]]Xi_typed>]
&=-\tilde{C}^{\eps}\Cas\;, 
&\quad
\ell_{\BPHZ}^{\delta,\eps}
[\<I[I[Xi]]Xi_typed*>]
&=-\tilde{C}^{\delta,\eps} \Cas\;,
\end{equs}
and furthermore $\lim_{\delta \downarrow 0} \tilde{C}^{\delta,\eps} =  (K \ast K^{\eps})(0) \eqdef \tilde{C}^{0,\eps}$. Additionally, there are finite constants $C_{\gsym}$ and $\bar{C}_{\gsym}$ such that
\[
\lim_{\eps \downarrow 0}\tilde{C}^{\eps}
=C_{\gsym}
\quad\textnormal{and}\quad
\lim_{\eps \downarrow 0} \tilde{C}^{0,\eps} = \bar{C}_{\gsym}\;.
\]
Finally, we have that $\ell_{\BPHZ}^{\delta,\eps}
[\<IXXiI'Xi_typed*>]$, $\ell_{\BPHZ}^{\delta,\eps}
[\<IXiI'XXi_typed>]$, $\ell_{\BPHZ}^{\delta,\eps} 
[\<IXXiI'Xi_typed>]$, and $
\ell_{\BPHZ}^{\delta,\eps}
[\<IXiI'XXi_typed*>]$ are each given by a multiple of $\Cas$ where the prefactor only depends on $\delta,\eps$ and the form\footnote{That is, they do not depend on the specific colours\slash spatial indices appearing in the tree as long as they obey the constraints given in Tables~\ref{table:TableB} and~\ref{table:TablebarA}.} of the tree.  
\end{lemma}
The rest of our computation of the renormalised equation is summarised in the following lemmas. 
In what follows we refer to the constant $\lambda$ fixed by Remark~\ref{rem:action_of_casimir}.
We also introduce the shorthand\footnote{Note that our use of the notations $\Psi_{\<green>} $ and $\Psi_{\<green>,\<red>}$ differs slightly from
Section~\ref{sec:solution_theory}.}
\[
\Psi_{\<green>} 
= \mcb{I}_{\<green>}\bar{\boldsymbol{\Xi}}_{\<green>} \;,
\qquad
\Psi_{\<green>,\<red>} = \mcb{I}_{\<green>,\<red>}\bar{\boldsymbol{\Xi}}_{\<green>}\;,
\qquad
\bar{\Psi}_{\<green>} = \bar{\mcb{I}}_{\<green>}\boldsymbol{\Xi}_{\<green>}\;,
\qquad
\bar{\Psi}_{\<green>,\<red>} = \bar{\mcb{I}}_{\<green>,\<red>}\boldsymbol{\Xi}_{\<green>}\;.
\]
We now walk through the computation of renormalisation counterterms for the system of equations given by $\bar{F}$.
We will directly give the expressions for $\bar{\bUpsilon}^{\bar F}$ such as \eqref{e:UpsbFxiIIxi} and \eqref{e:bUps-shortlist-2a} below, which follow by straightforward calculations from the definitions.

Recall the convention \eqref{e:convention-bfA}
for writing components of $\pr{\mathbf{A}}$ as $\pr{B}$, $\pr{\bar A}$, $\pr{U}$, etc.
The following lemma gives the renormalisation for the $\mfm_{i}$ equation in this system.
\begin{lemma}\label{lemma:renorm-for-m}
$\bar{\bUpsilon}^{\bar{F}}_{\mfm,\<dblue>}[\tau] = 0$ for all $\tau$ of the form in Table~\ref{table:TablebarA} except for 
$\tau=\<I[I[Xi]]Xi_typed*>$ where
\begin{equ}[e:UpsbFxiIIxi]
\bar{\bUpsilon}^{\bar F}_{\mfm,\<dblue>}
[\<I[I[Xi]]Xi_typed*>]
(\pr{\mathbf{A}}) = 
\delta_{\<dblue>,\<green>}
\big[[\pr{\bar{U}}\mcb{I}^{\mfu}\bar{\Psi}_{\<green>} , \pr{\bar{h}}_{\<green>}],\pr{\bar{U}}\boldsymbol{\Xi}_{\<green>}\big]  \;.
\end{equ}
In particular, for $\pr{\mathbf{A}} \in \bar{\mcb{A}}$,  
\begin{equs}\label{eq:renorm_moll_eq}
\sum_{\tau \in \mfT_{-}(R)}
(\ell_{\BPHZ}^{\delta,\eps}[\tau] \otimes \id)
\bar{\bUpsilon}_{\mfm,\<dblue>}^{\bar{F}}[\tau](\pr{\mathbf{A}})
=   \lambda \tilde{C}^{\delta,\eps}\pr{\bar{h}}_{\<dblue>}\;.
\end{equs}
\end{lemma}
\begin{proof}
Using the assumption that $\pr{\mathbf{A}} \in \bar{\mcb{A}}$ we have
\[
\big[[\pr{\bar{U}}\mcb{I}^{\mfu}(\bar{\Psi}_{\<dblue>}), \pr{\bar{h}}_{\<dblue>}],\pr{\bar{U} }\boldsymbol{\Xi}_{\<dblue>}\big]
=
-\pr{\bar{U}}[\boldsymbol{\Xi}_{\<dblue>},
\big[ \mcb{I}^{\mfu}(\bar{\Psi}_{\<dblue>}),\pr{\bar{U}}^{-1} \pr{\bar{h}}_{\<dblue>}]\big]\;.
\]
Inserting this into the left-hand side of \eqref{eq:renorm_moll_eq} and combining it with Lemma~\ref{lemma:renorm_constants_gsym}, we see that it is equal to 
\begin{equs}
\tilde{C}^{\delta,\eps}( \Cas \otimes \id)\,
\pr{\bar{U}}[\boldsymbol{\Xi}_{\<dblue>},
[\mcb{I}^{\mfu}(\bar{\Psi}_{\<dblue>}),\pr{\bar{U}}^{-1} \pr{\bar{h}}_{\<dblue>}]] 
=
 \tilde{C}^{\delta,\eps} \pr{\bar U} \ad_{\Cas} \pr{\bar{U}}^{-1}\pr{\bar{h}}_{\<dblue>}
=
\lambda \tilde{C}^{\delta,\eps} \pr{\bar{h}}_{\<dblue>}
\end{equs}
since $\ad_{\Cas} = \lambda \id_{\mfg}$.  
\end{proof}
For the $\mfa_{i}$ components we have the following lemma. 
\begin{lemma}\label{lemma:renorm_of_bar_a}
$\bar{\bUpsilon}^{\bar{F}}_{\mfa,\<dblue>}[\<I[I[Xi]]Xi_typed*>] = 0$, 
and
\begin{equs}\label{e:bUps-shortlist-2}
\bar\bUpsilon^{\bar F}_{\mfa,\<dblue>} [\<IXi^2green*>](\pr{\mathbf A})
&= \bone_{\<dblue>\neq \<green>} 
[ \pr{\bar{U}}\bar{\Psi}_{\<green>},[\pr{\bar{U}}\bar{\Psi}_{\<green>}, \pr{\bar{A}}_{\<dblue>}]]\;, \\[.2em]
\bar\bUpsilon^{\bar F}_{\mfa,\<dblue>}[\<I[I'Xi]I'Xi_typed*>](\pr{\mathbf A}) &=
 (2\delta_{\<green>,\<dblue>}\delta_{\<orange>,\<red>}  -\delta_{\<dblue>,\<red>}\delta_{\<orange>,\<green>})\,[[2\delta_{\<green>,\<orange>}\pr{\bar{A}}_{\<red>} - \delta_{\<red>,\<orange>}\pr{\bar{A}}_{\<green>} \,,\, 
\pr{\bar{U}}\mcb{I}_{\<orange>}(\bar{\Psi}_{\<green>,\<red>})]\,,\,
\pr{\bar{U}}\bar{\Psi}_{\<green>,\<red>}]\;, \\[.2em]
\bar\bUpsilon^{\bar F}_{\mfa,\<dblue>}[\<IXiI'[I'Xi]_typed*>](\pr{\mathbf A})
&=(2\delta_{\<orange>,\<dblue>}\delta_{\<green>,\<red>}  -\delta_{\<dblue>,\<red>}\delta_{\<green>,\<orange>})
[\pr{\bar{U}}\bar{\Psi}_{\<green>},[2\delta_{\<green>,\<orange>}\pr{\bar{A}}_{\<red>} - \delta_{\<red>,\<orange>}\pr{\bar{A}}_{\<green>}, \pr{\bar{U}}\mcb{I}_{\<orange>,\<red>} (\bar{\Psi}_{\<green>,\<red>})]] \;,\\
{}\bar\bUpsilon^{\bar F}_{\mfa,\<dblue>}[\<IXiI'XXi_typed*>](\pr{\mathbf A})  
& = \delta_{\<dblue>, \<red>}(2 \delta_{\<green>, \<red> } -1) [\pr{\bar{U}}\bar{\Psi}_{\<green>}, \pr{\partial}_{\<red>}\pr{\bar{U}}\bar{\mcb{I}}_{\<green>,\<red>}(\mbX_{\<red>}\boldsymbol{\Xi}_{\<green>})]\;,\label{e:bUps-shortlist-2a}\\
\bar\bUpsilon^{\bar F}_{\mfa,\<dblue>}[\<IXXiI'Xi_typed*>](\pr{\mathbf A}) 
& = \delta_{\<dblue>, \<red>} (2 \delta_{\<green>, \<red> } -1)[\pr{\partial}_{\<red>}\pr{\bar{U}}\bar{\mcb{I}}_{\<green>}(\mbX_{\<red>}\boldsymbol{\Xi}_{\<green>}),\pr{\bar{U}}\bar{\Psi}_{\<green>,\<red>}]\;.\label{e:bUps-shortlist-2b}
\end{equs}
In particular, for $\pr{\mathbf{A}} \in \bar{\mcb{A}}$,   
\begin{equ}\label{eq:renorm_for_a_equation}
\sum_{\tau \in \mfT_{-}(R)}
(\ell_{\BPHZ}^{\delta,\eps}[\tau] \otimes \id)
\bar{\bUpsilon}_{\mfa,\<dblue>}^{\bar{F}}[\tau](\pr{\mathbf{A}})
=  (  4\hat{C}^{\delta,\eps}-\bar C^{\delta,\eps} ) \lambda \pr{\bar{A}}_{\<dblue>}\;.
\end{equ} 
\end{lemma} 
\begin{proof}
The right-hand side of \eqref{eq:renorm_for_a_equation} comes from the contribution of trees  of
the form $\<IXi^2green*>$, $\<I[I'Xi]I'Xi_typed*>$, and $\<IXiI'[I'Xi]_typed*>$, which can be shown as in Lemma~\ref{lemma:final_ups_calc}, combined with the condition that $\pr{\mathbf{A}} \in \bar{\mcb{A}}$ (namely, the second relation of Definition~\ref{def:special_jets}) to cancel the factors of $\pr{\bar{U}}$.
The total contributions from the trees of the form $\<IXiI'XXi_typed*>$ and those of the form $\<IXXiI'Xi_typed*>$ each vanish. 
For the case of trees of form $\<IXiI'XXi_typed*>$ this total contribution is given by
\[
\check{C}
\sum_{\<green> =1,2}
(2 \delta_{\<green>, \<dblue> } -1) ( \Cas \otimes \id) [\pr{\bar{U}}\bar{\Psi}_{\<green>}, \pr{\partial}_{\<dblue>}\pr{\bar{U}}\bar{\mcb{I}}_{\<green>,\<dblue>}(\mbX_{\<dblue>}\boldsymbol{\Xi}_{\<green>})]\;,
\]
for some constant $\check{C}$. 
Other than the factor $(2 \delta_{\<green>, \<dblue> } -1)$, the summand above does not depend on $\<green>$ and since $\sum_{\<green> =1,2}
(2 \delta_{\<green>, \<dblue> } -1) = 0$ it follows that the sum above vanishes as claimed. 
A similar argument takes care of the case of $\<IXXiI'Xi_typed*>$. 
\end{proof}
The computation of the renormalisation of the $\mfa_{i}$ components in the $F$ system of equations mirrors the computations we have just done for the $\bar{F}$ system with the one difference that the term $\pr{U} \chi^{\eps} \ast \xi_{i}$, which is the analogue of the term $\pr{U}\xi_{i}$ that was part of $\bar{F}_{\mfm,i}$, is included in $F_{\mfa,i}$. In particular,
$\bar{\bUpsilon}^{F}_{\mfa}[\tau]$
for $\tau \in \{\<IXi^2green>,\<IXiI'[I'Xi]_typed>,\<IXiI'XXi_typed>,\<IXXiI'Xi_typed>,\<I[I[Xi]]Xi_typed>\}$
are given by formulas as in \eqref{e:bUps-shortlist-2} and \eqref{e:bUps-shortlist-2a}
with the following replacement
\[
\pr{\bar A} \mapsto \pr{B},\quad \pr{\bar U} \mapsto \pr{U} , \quad
\bar \Psi \mapsto \Psi ,\quad
\bar{\mcb{I}}(\mbX \boldsymbol{\Xi})
\mapsto 
\mcb{I}(\mbX \bar{\boldsymbol{\Xi}})
\]
and
$\bar{\bUpsilon}^{F}_{\mfa,\<dblue>}
[\<I[I[Xi]]Xi_typed>]
(\pr{\mathbf{A}}) = 
\delta_{\<dblue>,\<green>}
[[\pr{U}\mcb{I}^{\mfu}\Psi_{\<green>}, \pr{h}_{\<green>}],\pr{U}\bar{\boldsymbol{\Xi}}_{\<green>}] $. 

By using the renormalisation constants given in Lemma~\ref{lemma:renorm_constants_gsym} and performing again computations of the type found in Lemmas~\ref{lemma:renorm-for-m} and \ref{lemma:renorm_of_bar_a}, one obtains the following lemma. 
\begin{lemma}\label{lemma:renorm_of_b_equation}
For $\pr{\mathbf{A}} \in \bar{\mcb{A}}$,  
\begin{equs}
\sum_{\tau \in \mfT_{-}(R)}
(\ell_{\BPHZ}^{\delta,\eps}[\tau] \otimes \id)
\bar{\bUpsilon}_{\mfa,\<dblue>}^{F}[\tau](\pr{\mathbf{A}})
=
  \lambda C_{\sym}^{\eps} 
\pr{B}_{\<dblue>}
+ \lambda \tilde{C}^{\eps} \pr{h}_{\<dblue>}
\;,
\end{equs}
where $C_{\sym}^{\eps} $  
is as in \eqref{e:defConstants}.
\end{lemma}

\begin{remark}
Lemmas~\ref{lemma:renorm-for-m}, \ref{lemma:renorm_of_bar_a}, and  \ref{lemma:renorm_of_b_equation}, still hold if one replaces the first condition of Definition~\ref{def:special_jets} by only requiring the invertibility of $\pr{\mathbf{A}_{\mfu}}$. 
\end{remark}
The main result of this section is the following proposition.
\begin{proposition}\label{prop:SPDEs_conv_zero} 
Suppose that $\moll$ is non-anticipative.
Fix any constants $\mathring{C}_{1}$ and $\mathring{C}_{2}$ and initial data $\bar{a} \in \Omega_\alpha^1$ and $g(0) \in \mfG^{\alpha}$. 
Consider the system of equations 
\begin{equs}[eq:renorm_bar_a]
\partial_t \bar{A}_i &= 
\Delta \bar{A}_i + \moll^\eps * (\bar{g}\xi_i \bar{g}^{-1}) + \bar{C}^{\eps}_{1} \bar{A}_i + \bar{C}^{\eps}_{2} (\partial_i \bar{g}) \bar{g}^{-1}
\\
&\quad + [\bar A_j,2\partial_j \bar A_i - \partial_i \bar A_j + [\bar A_j,\bar A_i]]\;, \qquad 
\bar A(0) = \bar a\;,
\end{equs}
and
\begin{equs}[eq:renorm_b]
\partial_t B_i &= 
\Delta B_i  + g \xi^\eps_i g^{-1} + C^{\eps}_{1} B_i 
+ C^{\eps}_{2} (\partial_i g) g^{-1}\\
& \quad +[B_j,2\partial_j B_i - \partial_i B_j + [B_j,B_i]]\;, \qquad 
B(0) = \bar a\;,
\end{equs}
where $\bar{g}$ and $g$ are given by the corresponding equations in \eqref{eq:SPDE_for_bar_A} and \eqref{eq:SPDE_for_B} started with the same initial data $\bar{g}(0) = g(0)$, and where we set $\bar{g}\equiv 1$ on $(-\infty,0)$ in~\eqref{eq:renorm_bar_a}.
The constants are defined by
\begin{equs}\label{eq:translated_renorm_constants}
\bar{C}^{\eps}_{1} &= \mathring{C}_{1} + \lambda C_{\sym}^{\eps} \;,\quad  
\bar{C}^{\eps}_{2} = \mathring{C}_{2} + \lambda\tilde{C}^{0,\eps} \;,\\
C^{\eps}_{1} &= \mathring{C}_{1} + \lambda C_{\sym}^{\eps} \;,\quad  
C^{\eps}_{2} = \mathring{C}_{2} + \lambda  \tilde{C}^{\eps}\;.
\end{equs}
Then, for every $\eps>0$,~\eqref{eq:renorm_bar_a}
is well-posed in the sense that, replacing $\xi_i$ by $\bone_{t<0}\xi_i + \bone_{t\geq0}\xi^\delta_i$, the (smooth) maximal solution $(\bar A,\bar g)$ converges as $\delta\downarrow0$ to a limiting (smooth) maximal solution in $(\Omega^1_\alpha\times \mfG^{0,\alpha})^\sol$.
Furthermore $(\bar{A},[\bar g])$ and $(B,[g])$ 
converge in probability in $(\Omega^1_\alpha\times\mathring \mfG^{0,\alpha})^\sol$
to the same limit as $\eps \downarrow 0$.
\end{proposition}

\begin{remark}\label{rem:nonAnt}
It follows from the explicit expression \eqref{e:defConstantsTildeC} that $\moll$ being 
non-anticipative implies $ \tilde{C}^{0,\eps} = 0$. (The convolution of non-anticipative
functions is non-anticipative and any continuous non-anticipative function vanishes at the origin.)
\end{remark}

\begin{proof}
We first claim that, for $\eps,\delta>0$ the solution to~\eqref{eq:renorm_bar_a}, with $\xi_i$ replaced by $\bone_{t<0}\xi_i + \bone_{t\geq0}\xi_i^\delta$, is equal on a small interval $(0,\tau)$
to the reconstruction
of the solution to the fixed point problem \eqref{eq:abstract_fixed_point_eq_2}
with model $Z^{\delta,\eps}_{\BPHZ}$, and with $\omega_i\eqdef GU^{(0)}\xi^\delta_i$ and $\CW_i$ the canonical lift of $\zeta_i\eqdef K*\bone_+\moll^\eps*(\xi_i \bone_{-})$ therein.
Here $\tau>0$ is random but independent of $\delta$ (and $\eps$).
The fact that~\eqref{eq:renorm_bar_a} is well-posed in the sense of maximal solutions converging as $\delta\downarrow0$ follows from this claim
due to the convergence of $Z^{\delta,\eps}_{\BPHZ}$ and $\omega_i$ as $\delta \downarrow 0$ (\eqref{eq:delta_conv_of_models} and~\eqref{eq:Uxi_conv})
and standard PDE arguments since $\bar g$ is smooth (and thus $\bar g\xi_i\bar g^{-1}$ is well-defined) for positive times.


We deploy \cite[Thm.~5.7]{BCCH21} and Proposition~\ref{prop:renormalisation_preserves_coherence} to get the renormalised reconstruction $(\bar A, \bar U, \bar h)$ of the equation \eqref{eq:abstract_fixed_point_eq_2}. 
For $\bar A$, in terms of the indeterminates $\pr{\mathbf{A}} = (\pr{\mathbf{A}_{o}})_{o \in \CE}$ and nonlinearity $\bar{F}$, this amounts to \emph{summing} the renormalised and reconstructed integral fixed point equations for the indeterminates $\pr{\mathbf{A}_{\mfa_{i}}}$ and $\pr{\mathbf{A}_{\mfm_{i}}}$ with nonlinearity $\bar{F}$, and recalling \eqref{e:K-Keps-assign}.

We first argue that $\bar U$ satisfies the conditions of Definition~\ref{def:special_jets}.
Lemma~\ref{lemma:renorm-for-u-h} implies that $\bar U$ and $\bar h$ satisfy the same PDE as the corresponding components in~\eqref{e:final_system} with $B$ replaced by $\bar A$.
Consider the process $\bar g$ which solves~\eqref{eq:SPDE_for_bar_A} for the given $\bar A$,
and define $f_i = (\partial_i \bar g) \bar g^{-1}$ and $V=\Ad_{\bar g}$.
Then Lemma~\ref{lemma:linear_g_system} applied to $(\bar A, V, f)$ implies that $(V,f)$ also solves~\eqref{e:final_system} with $B$ replaced by $\bar A$ and with the same initial condition as $(\bar U, \bar h)$.
Hence $(V,f)=(\bar U,\bar h)$ and thus $\bar U$ satisfies the conditions of Definition~\ref{def:special_jets}.

We can now apply Lemmas~\ref{lemma:renorm-for-m} and \ref{lemma:renorm_of_bar_a}, and take the limit $\delta \downarrow 0$ of the renormalisation constants as given in Lemmas~\ref{lem:tree_calcs2} and \ref{lemma:renorm_constants_gsym}
to conclude that
$\bar A$ solves~\eqref{eq:renorm_bar_a} but with $\bar g\xi_i \bar g^{-1}$ replaced by $\bar U\xi_i$ and $(\partial_i \bar g) \bar g^{-1}$ replaced by $\bar h_i$.
It remains to observe that if $(D, \bar g)$ is the solution to~\eqref{eq:renorm_bar_a} and the second equation in~\eqref{eq:SPDE_for_bar_A} with $\bar A$ replaced by $D$,
then a similar argument as above using Lemma~\ref{lemma:linear_g_system} implies $\bar A=D$, which proves the claim.



A similar argument shows that \eqref{eq:renorm_b} is the renormalised equation obtained via the reconstruction (with respect to $Z^{0,\eps}_{\BPHZ}$) of the fixed point problem  \eqref{eq:abstract_fixed_point_eq} with $\bar\omega_i \eqdef GU^{(0)}\xi_i^\eps$.
The minor difference is that one is aiming for the renormalised and reconstructed integral fixed point equations for just the indeterminates $\pr{\mathbf{A}_{\mfa_{i}}}$ with non-linearity $F$ so the computations of Lemma~\ref{lemma:renorm-for-m} and~\ref{lemma:renorm_of_bar_a} are replaced by that of Lemma~\ref{lemma:renorm_of_b_equation}.

We now turn to proving the statements concerning convergence in probability as $\eps \downarrow 0$.
Let us denote
\begin{equ}
S\eqdef \Omega^1_\alpha\times\mathring\mfG^{0,\alpha}\;.
\end{equ}
\label{page:conv_maximal_B}We first claim that $Y \eqdef (B,[g])$ converges in probability in $S^\sol$.
Indeed, recall from Theorem~\ref{thm:grp_action} that $(A,g)\mapsto A^g$ is uniformly continuous from every ball in $S$
into $\Omega^1_\alpha$.
In particular, the map $(A,g)\mapsto (A^g,[g])$ is continuous and both $(A,g)\mapsto (A^g,[g])$ and its inverse (which is well-defined up to a constant multiple of $g$)
are bounded on every ball.

Next, a similar calculation to that at the start of Section~\ref{subsec:gauge_covar_results}
shows that $B=A^g$ where $A$ solves~\eqref{eq:SPDE_for_A} with initial condition $a=\bar a^{g^{-1}}$ and $C\eqdef C^\eps_1$, and with an addition drift term on the right-hand side of the form $\hat C^\eps (\partial_i g^{-1}) g$ for some $\hat C^\eps$ converging to a finite value as $\eps\downarrow0$, and where $g$ solves
\begin{equ}\label{eq:SPDE_for_g_from_A}
g^{-1}(\partial_t g)
= \partial_i(g^{-1}\partial_i g)+ [A_i,g^{-1}\partial_i g]\;,\qquad g(0)\in\mfG^{0,\alpha}\;.
\end{equ}
A proof similar to that of Theorem~\ref{thm:local_exist},
using in particular that~\eqref{eq:SPDE_for_g_from_A} is classically well-posed for any $A\in(\Omega^1_\alpha)^\sol$, 
shows that $(A,[g])$ converges in $S^\sol$.
An application of Lemma~\ref{lem:alt_conv} then shows that $Y=(A^g,[g])$ converges in probability in $S^\sol$ as claimed.

It remains to prove that $X \eqdef (\bar A,[\bar g])$ converges to the limit of $Y$ in probability in $S^\sol$.
We first note that,
because $\moll$ is non-anticipative, $X$ is equal in law to $(A,[\check g])$
where $A$ solves~\eqref{eq:SPDE_for_A} with initial condition $\bar a$ and $C\eqdef \bar C^\eps_1$, and with the additional drift term $\bar C^\eps_2 (\partial_i \check g) \check g^{-1}$ on the right-hand.
Here $\check g$ solves the corresponding equation in~\eqref{eq:SPDE_for_bar_A} with $\bar A$ replaced by $A$.
Again, a proof similar to that of Theorem~\ref{thm:local_exist}
show that $(A,[\check g])$ converges in $S^\sol$.
It follows that $X$ converges \textit{in law} in $S^\sol$.

Recall from Remark~\ref{rem:Uh_vs_g} that we identify $\mathring\mfG^{0,\alpha}$ with a subset of $\hat\mfG^{0,\alpha}$.
Observe that for every $\lambda\in (0,1)$, $|X-Y|_{\CC([\lambda\tau,\tau],S)} \to 0$ in probability
due to Corollary~\ref{cor:SHE_mollif_converge}, Lemmas~\ref{lem:fixed_close_enhanced},~\ref{lem:conv_of_models2}, and~\ref{lem:CW},
and the bound~\eqref{eq:same_init_cond} in Lemma~\ref{lem:Uxi_conv} with $\sigma=0$ therein.
Here $\tau\in(0,1)$ is the time appearing in Lemma~\ref{lem:fixed_close_enhanced}.
Observe that $\tau>0$ can be taken as an $\mbF$-stopping time, where $\mbF$ is the filtration generated by the noise, since $\moll$ is non-anticipative and therefore the necessary conditions on the models $Z^{0,\eps}_{\BPHZ}$ and input distributions $\omega,\bar\omega$
which ensure existence and convergence of the solutions on an interval $[0,t]$ can be checked on $[-1,t]\times\T^2$.
%

To handle time $t=0$,
observe that for all $\lambda,\delta>0$,
the set
\begin{equ}
\Big\{f\in S^\sol \,:\, f(0)=(\bar a,[\bar g])\,,\, \sup_{t\in [0,\lambda]} |f(t)-f(0)|_S\geq \delta\Big\}
\end{equ}
is closed in $S^\sol$.
Therefore, the convergence in law of $X$ and the Portmanteau Lemma imply that, for all $\delta>0$, 
\begin{equ}
\lim_{\lambda\to0}\limsup_{\eps\to0} \P
\Big[
\sup_{t\in[0,\lambda\tau]} |X(t)-X(0)|_{S} \geq \delta
\Big]
= 0\;.
\end{equ}
Since $Y(0)=X(0)$, it follows that $|Y-X|_{\CC([0,\tau],S)}\to0$ in probability.

To conclude the proof, let $Q$ denote the reconstruction of maximal solutions as defined in Section~\ref{subsubsec:maximal} with the model taken as $Z^{0,\eps}_{\BPHZ}$,
along with the blow-up time $\tau^\star$ (and extended to all positive times by $Q(t)=\skull$ for $t\geq\tau^\star$).
By Proposition~\ref{prop:conv_max_sols}, $Q$ converges in probability in the sense of Lemma~\ref{lem:alt_conv}\ref{pt:D_conv_alt} as $\eps \downarrow0$.

We now claim that $Q\in S^\sol$ for $\eps=0$.
Indeed, for $\eps=0$, the fixed point problem~\eqref{eq:abstract_fixed_point_eq_3} is easily seen (using Lemmas~\ref{lem:Uxi_conv} and~\ref{lem:CW}) to give the same reconstruction
as~\eqref{eq:abstract_fixed_point_eq_2}
with $\CW=0$
and the input distribution $\omega$ defined through~\eqref{eq:Uxi_conv} (we use here that each $\sigma_k$ can be taken as an $\mbF$-stopping time).
Therefore, when $\eps=0$, we obtain a maximal solution $\bar Q\in S^\sol$
directly from Lemmas~\ref{lemma:fixedptpblmclose} and~\ref{lem:fixed_close_enhanced},
and $Q=\bar Q$ on $(0,\tau^\star)$.

On the other hand, $(\CB,\mcU,\mcH)$ is coherent with $\bar F$ on each interval $(\sigma_k,\sigma_{k+1})$ (more precisely, it is coherent after we add the component $\mfm_i$, see~\eqref{eq:bar_F_def}).
It readily follows that $\mcU$ is determined by the reconstruction of the tuple.
Because $\mcU$ in~\eqref{eq:abstract_fixed_point_eq_2} is controlled purely by the input as in Lemma~\ref{lem:fixed_close_enhanced}, it follows as in the proof of that lemma
that the final two terms on the left-hand side of~\eqref{eq:blow_up_criteria}
can only blow up once $\bar Q$ blows up, i.e.\ $\tau^\star$ is the blow-up time of $\bar Q$ in $S$.
Therefore $Q=\bar Q \in S^\sol$ as claimed.

Observe now that, for $\eps>0$ and with notation as in Definition~\ref{def:maximal}, $X$ is equal to $Q$ on the interval $[0,2\tau_1+\ldots+2\tau_n]$
provided that $\eps^2<\tau_i$ for all $i=1,\ldots,n$.
Following Remark~\ref{rem:discrepancy}
and using that, for $\eps>0$, $X$ is an element of $S^\sol$, we obtain that
$X$ converges in probability in $S^\sol$ as $\eps\downarrow0$ (to the same limit as $Q$).

Finally, remark that the $\eps\downarrow0$
limits of $X$ and $Y$ agree almost surely on an interval $[0,\tau]$, where $\tau$ is an $\mbF$-stopping time which is stochastically bounded from below by a distribution depending only on the size of the initial condition.
Both limits are furthermore strong Markov with respect to $\mbF$, 
from which it follows that the limits are almost surely equal as elements of $S^\sol$.
\end{proof}

\begin{proof}[of Theorem~\ref{thm:gauge_covar}]
We first prove statement \ref{pt:gauge_covar}. Given $C \in \R$, which is assumed to be a real constant by Remark~\ref{rem:simple_vs_reductive} and Assumption~\ref{as:simple}, we fix $\mathring{C}_{1} = C - \lambda C_{\sym} $ and $\mathring{C}_{2} = C - \lambda C_{\gsym} $. 
We then take the $\bar{C}$ as claimed in the theorem as  
\[
\bar{C} \eqdef \lambda (  C_{\gsym} -  \bar{C}_{\gsym})\;.
\] 
With these choices and the definitions of \eqref{eq:translated_renorm_constants},
together with
\[
\tilde{C}^{\eps} - C_{\gsym}=o(1)\;,
\quad
C_{\sym}-C_{\sym}^\e = o(1)\;,
\quad
\tilde C^{0,\eps}-\bar C_{\gsym}=o(1)
\]
it follows that, as $\eps \downarrow 0$,  
\[
C = C_{1}^{\eps} + o(1) = C_{2}^{\eps} + o(1) = \bar{C}_{1}^{\eps} + o(1)\;,
\qquad
 C-\bar{C} = \bar{C}_{2}^{\eps} + o(1)\;.
\]
The desired statement then follows from Proposition~\ref{prop:SPDEs_conv_zero}. 
Note that if $\moll$ is non-anticipative then
 $\tilde{C}^{0,\eps} = 0$ by Remark~\ref{rem:nonAnt}, but we don't use 
this here in the proof of \ref{pt:gauge_covar}.

We now prove~\ref{pt:indep_moll}. 
Since $\moll$ is non-anticipative, $ \tilde{C}^{0,\eps} = 0$ for every $\eps > 0$ and so $\bar{C}_{\gsym} = 0$. 
It follows that $\bar{C} =  \lambda C_{\gsym} =  \lambda \lim_{\eps \downarrow 0} \int \mrd z\ \moll^{\eps}(z)(K \ast K^{\eps})(z)$ and so the desired statement follows from Remark~\ref{rem:sym_indep_of_mollifier}. 
\end{proof}

\subsection{Construction of the Markov process}
\label{subsec:Markov_process}

In this subsection, we prove Theorem~\ref{thm:Markov_process}.
We begin with several lemmas.

\begin{lemma}\label{lem:inf_continuous}
Let $\alpha\in (\frac23,1]$
and $A,B\in\Omega^1_\alpha$.
Then
\begin{equ}\label{eq:inf_Lip_cont}
\Big| \inf_{g\in\mfG^{0,\alpha}} |B^g|_\alpha
-\inf_{g\in\mfG^{0,\alpha}} |A^g|_\alpha \Big| \lesssim (1+|A|_\alpha+|B|_\alpha)|A-B|_\alpha\;,
\end{equ}
where the proportionality constant depends only on $\alpha$.
\end{lemma}

\begin{proof}
As in the proof of Theorem~\ref{thm:grp_action}, for $g\in\mfG^{0,\alpha}$ we can write
\begin{equ}
A^g-B^g = \big((A-B)^g-0^g - (A-B)\big) + (A-B)\;,
\end{equ}
from which it follows by Lemmas~\ref{lem:group_action_gr} and~\ref{lem:group_action_vee} that
\begin{equ}\label{eq:A_B_g_bound}
|A^g-B^g|_\alpha \lesssim (1+|g|_{\Hol\alpha})|A-B|_\alpha\;,
\end{equ}
where the proportionality constant depends only on $\alpha$.
Consider a minimising sequence $g_n\in\mfG^{0,\alpha}$ for $A$.
Then, by Proposition~\ref{prop:gauge_equiv},
$\limsup_{n\to\infty} |g_n|_{\Hol\alpha}\lesssim |A|_\alpha$,
and thus by~\eqref{eq:A_B_g_bound}
\begin{equ}
\inf_{g\in\mfG^{0,\alpha}} |B^g|_\alpha
-\inf_{g\in\mfG^{0,\alpha}} |A^g|_\alpha
\leq \limsup_{n\to\infty}|B^{g_n}|_\alpha - |A^{g_n}|_\alpha
\lesssim (1+|A|_\alpha)|A-B|_\alpha\;.
\end{equ}
Swapping $A$ and $B$ and applying the same argument, we obtain~\eqref{eq:inf_Lip_cont}.
\end{proof}

\begin{lemma}\label{lem:selection}
Let $\alpha\in (\frac23,1)$.
There exists a measurable (Borel) selection $S\colon\mfO_\alpha\to\Omega^1_\alpha$
such that $|S(x)|_\alpha < 1+\inf_{A\in x}|A|_\alpha$
for all $x\in\mfO_\alpha$.
\end{lemma}

\begin{proof}
Consider the subset $Y \eqdef \{A\in\Omega^1_\alpha \ssep |A|_\alpha < 1+\inf_{g\in\mfG^{0,\alpha}}|A^g|_\alpha\}$,
which is open due to Lemma~\ref{lem:inf_continuous}.
In particular, $Y$ is Polish (Alexandrov’s theorem) and, by Lemma~\ref{lem:closed}, the gauge equivalence classes in $Y$ are closed.
Note that $\pi\colon \Omega^1_\alpha\to \mfO_\alpha$ is an open map since $\pi^{-1}(\pi(U))=\cup_{g\in\mfG^{0,\alpha}}U^g$
is open for every open subset $U\subset\Omega^{1}_\alpha$.
Therefore $\pi\colon Y \to \mfO_\alpha$ is also open and
the conclusion follows by the Rokhlin--Kuratowski--Ryll-Nardzewski
selection theorem~\cite[Thm.~6.9.3]{Bogachev07}.
\end{proof}

For the rest of the section, let us fix a non-anticipative mollifier $\moll$ 
and set $C=\bar C$, the 
constant from part~\ref{pt:gauge_covar} of Theorem~\ref{thm:gauge_covar}.
By a ``white noise'' we again mean a pair of i.i.d.\ $\mfg$-valued
white noises $\xi=(\xi_1,\xi_2)$ on $\R\times \T^2$.

\begin{proof}[of Theorem~\ref{thm:Markov_process}]
\ref{pt:generative_exists}
By Lemma~\ref{lem:selection}, there exists a measurable selection $S\colon\hat\mfO_\alpha \to \hat\Omega^1_\alpha$ such that for all $x\in\mfO_\alpha$
\begin{equ}\label{eq:S_bound}
|S(x)|_\alpha < 1+\inf_{a\in x} |a|_\alpha
\end{equ}
and $S(\skull)=\skull$.
Let $\xi$ be white noise and let $(\mcF_t)_{t \geq 0}$ be the filtration generated by $\xi$.

Consider any $a \in \hat\Omega^1_{\alpha}$.
We define a Markov process $A \colon \mcO \to D_0(\R_+, \hat\Omega^1_{\alpha})$ with $A(0)=a$
and a sequence of stopping times $(\sigma_j)_{j=0}^\infty$ as follows.
For $j=0$, set $\sigma_0 = 0$ and $A(0)=a$.
Define further
\begin{equ}
b \eqdef
\begin{cases}
S([a]) \quad &\text{ if } |a|_\alpha \geq 2+\inf_{g\in\mfG^{0,\alpha}}|a^g|_\alpha\;,
\\
a \quad &\text{ otherwise.}
\end{cases}
\end{equ}
Consider now $j \geq 0$.
If $\sigma_j=\infty$, then we set $\sigma_{j+1}=\infty$.
Otherwise, if $\sigma_j<\infty$, 
suppose that $A$ is defined on $[0,\sigma_{j}]$.
If $A(\sigma_j)=\skull$, then define $\sigma_j=\sigma_{j+1}$.
Otherwise, define $\Theta \in \mcC([\sigma_{j},\infty),\hat\Omega^1_\alpha)$
by 
\begin{equ}
\Theta(t)=
\begin{cases}
\Phi_{\sigma_{j},t}(A(\sigma_{j})) \quad &\text{ if } j \geq 1\;,
\\
\Phi_{0,t}(b) \quad &\text{ if } j = 0\;,
\end{cases}
\end{equ}
where we used the notation $\Phi_{s,t}$ as in Definition~\ref{def:generative}.
We then set
\begin{equ}
\sigma_{j+1} = \inf\{t > \sigma_j \ssep |\Theta(t)|_\alpha \geq 2+\inf_{g\in\mfG^{0,\alpha}}|\Theta(t)^g|_\alpha\}\;,
\end{equ}
and define $A(t) = \Theta(t)$ for all $t\in(\sigma_{j},\sigma_{j+1})$
and $A(\sigma_{j+1}) = S([\Theta(\sigma_{j+1})])$.
Observe that $\sigma_{j}<\sigma_{j+1}$ a.s.\ due to Lemma~\ref{lem:inf_continuous}, the condition~\eqref{eq:S_bound}, and the continuity of $\Theta$ at $\sigma_j$.
In fact, defining $M(t) = \inf_{g\in\mfG^{0,\alpha}}|A(t+)^g|_\alpha$,
then by decomposing $\Theta$ into the SHE with initial 
condition $A(\sigma_j+)$ and a remainder as in the proof of Theorem~\ref{thm:local_exist},
we see that the law of $\sigma_{j+1}-\sigma_j$ depends only on $A(\sigma_j+)$ and can be stochastically bounded from below by a strictly positive random variable depending only on $M(\sigma_j)$.
In particular, if the quantity $T^* \eqdef \lim_{j\to\infty}\sigma_j$ is finite, then a.s. $\lim_{t\nearrow T^*}M(t) = \infty$.
In this case, we have defined $A$ on $[0,T^*)$, and then set $A\equiv \skull$ on $[T^*,\infty)$.
If $T^* = \infty$, then we have defined $A$ on $\R_+$ and the construction is complete.
Note that, in either case, a.s. $T^* = \inf\{t \geq 0\ssep A(t) = \skull\}$.

To complete the proof of~\ref{pt:generative_exists}, we 
need only remark that items~\ref{pt:dynamics} and~\ref{pt:honest_blow_up}
of Definition~\ref{def:generative} are satisfied and that $A$ is Markov, which all follow from the 
construction of $(\sigma_j)_{j=0}^\infty$, the definition of $b$, and the above discussion.

\ref{pt:unique_Markov}
The idea of the proof is to couple any generative 
probability measure $\bar \mu$ to the law of the 
process $A$ constructed in part~\ref{pt:generative_exists}.
Consider a white noise $\bar\xi$ with an admissible 
filtration $(\bar\mcF_t)_{t\geq0}$, a $\bar\mcF$-stopping time $\sigma$,
a solution $\bar A\in\mcC([s,\sigma),\Omega^1_\alpha)$
to the SYM driven by $\bar \xi$,
and a gauge equivalent initial condition $A(s) = \bar A(s)^{g(s)}$.
Remark that, by part~\ref{pt:gauge_covar} of
Theorem~\ref{thm:gauge_covar}, we can construct 
on the same probability space a stopping time $\tau$,
a time-dependent gauge 
transformation $g\in\mcC([s,\tau),\mfG^{0,\alpha})$ 
(namely $g^{-1} = \bar{g}$, the solution to that component of ~\eqref{eq:SPDE_for_bar_A} driven 
by $\bar A$ started with initial data $\bar{g}(s) = g^{-1}(s)$) and a solution $A\in\mcC([s,\tau),\Omega^1_\alpha)$ to the SYM driven by the
white noise $\xi \eqdef \Ad_g \bar\xi$ such that $\bar A^{g} = A$ on $[s,\tau)$. 
Moreover, by the bound~\eqref{eq:g_Hol_bound} in Proposition~\ref{prop:gauge_equiv},
$|g|_{\Hol\alpha}$ cannot blow-up before $|\bar A|_{\gr\alpha}+|A|_{\gr\alpha}$ does.
Since $\Omega^1_{\gr\alpha} \hookrightarrow \Omega\mcC^{0,\alpha-1}$ (see Section~\ref{subsec:smooth_one_forms}),
and since by Theorem~\ref{thm:local_exist} we can start the SYM from any initial condition in $\Omega\mcC^{\eta}$, $\eta\in(-\frac12,0)$,
it follows that we can take $\tau=\sigma\wedge T^*$ where $T^*$
is the blow-up time of $|A|_{\alpha}$.
Note also that $g$ and $\xi$ are adapted to the 
filtration generated by $\bar \xi$, and $A$ is adapted to the filtration generated by $\xi$.

Consider $a,\bar a\in\hat\Omega^1_\alpha$ with $[a]=[\bar a]$ and a generative 
probability measure $\bar\mu$ on $D_0(\R_+,\hat\Omega^1_\alpha)$ with initial condition $\bar a$.
Let $\bar A\in D_0(\R_+,\hat\Omega^1_\alpha)$ denote 
the corresponding process with filtration $(\bar\mcF_t)_{t\geq 0}$, white noise $\bar\xi$, and blow-up 
time $\bar T^*$ as in Definition~\ref{def:generative}.
It readily follows from the above remark and the
conditions in Definition~\ref{def:generative} that
there exist, on the same probability space,
\begin{itemize}
\item a process $g\colon \R_+ \to \mfG^{0,\alpha}$
adapted to $(\bar\mcF_t)_{t \geq 0}$,
which is c{\`a}dl{\`a}g on the interval $(0,\bar T^*)$, for which $\lim_{t\downarrow0}g(t)$ exists, and which
remains constant $g\equiv 1$ on $[\bar T^*,\infty)$, and
\item a Markov process $A\in D_0(\R_+,\hat\Omega^1_\alpha)$ constructed 
as in part~\ref{pt:generative_exists} using the white 
noise $\xi\eqdef\Ad_g\bar\xi$ such that $A=\bar A^{g}$ and $A(0)=a$.
\end{itemize}
(Specifically, the process $g$ is constructed to have 
jumps in $[0,\bar T^*)$ only at the jump times of $\bar A$
and $A$, and $\bar{g} = g^{-1}$ solves~\eqref{eq:SPDE_for_bar_A} driven by $\bar A$ on its intervals of continuity.)
In particular, the pushforwards $\pi_*\bar\mu$ and $\pi_*\mu$ coincide,
where $\mu$ is the law of $A$.

To complete the proof, it remains only to show that 
for the process $A$ from part~\ref{pt:generative_exists}
with any initial condition $a\in\hat\Omega^1_\alpha$, the projected process
$\pi A$ takes values in $\mfO_\alpha^\sol$ and is Markov.
The Markov property of $\pi A$ follows from that of $A$ and from
taking $\bar \mu$ in the above argument as the law of $A$ with initial condition $\bar a\sim a$.
The fact that $\pi A$ takes values in $\mfO_\alpha^\sol$
is due to items~\ref{pt:fix_at_zero},~\ref{pt:solves_SYM},~\ref{pt:gauge_at_jumps}, and~\ref{pt:honest_blow_up} of Definition~\ref{def:generative}
together with the fact that,
due to Lemma~\ref{lem:k_lower_bound},
for any sequence $x_n\in\mfO_\alpha$,
$\inf_{a\in x_n} |a|_\alpha \to \infty \Rightarrow x_n\to\skull$ in $\hat\mfO_\alpha$.
\footnote{The converse $x_n\to\skull$ in $\hat\mfO_\alpha \Rightarrow\inf_{a\in x_n} |a|_\alpha \to \infty $ is also true due to~\eqref{eq:K_upper_bound} and
Lemma~\ref{lem:points_on_spheres}
as in the proof of Lemma~\ref{lem:k_to_D_conv},
but is not needed here.}
\end{proof}

\appendix

\section{Singular modelled distributions}
\label{app:Singular modelled distributions}

We collect in this appendix several results on singular modelled distributions used in Section~\ref{sec:gauge_equivar}.

\subsection{Short-time convolution estimates}
\label{app:reconstruct}

In Section~\ref{sec:gauge_equivar}, it is important that
the short time $\tau>0$ for which the fixed point maps~\eqref{eq:abstract_fixed_point_eq} and~\eqref{eq:abstract_fixed_point_eq_2} are contractions is a stopping time with respect to the white noise.
This requires short-time convolution estimates which only depend on the size of the model up to time $\tau>0$ (see Proposition~\ref{prop:short_time}).
In this appendix, we prove a refined version of the reconstruction theorem, Lemma~\ref{lem:reconstruct}, which is used to prove this result.

Fix a regularity structure $\CT$ on $\R^{d+1}$,
an arbitrary scaling $\s$, as well as an integer 
$r \ge 1$ such that $-r$ is
smaller than the lowest homogeneity appearing in $\CT$ and such that furthermore
$r \ge \sup_i \s_i$.
Given  model $Z = (\Pi,\Gamma)$ and a compact set $\K\subset\R^{d+1}$, we denote by $\$Z\$_\K$
the smallest constant $C$ such that, for all homogeneous elements $\tau \in \CT$, 
\begin{equ}
\bigl| \bigl(\Pi_x \tau\bigr)(\phi_x^\lambda)\bigr| \le C \|\tau\| \lambda^{\deg\tau}\;,
\end{equ}
for all $\phi \in \CB^r_{\s,0}$ and all $x \in \K$, $\lambda \in (0,1]$ such that
$B_\s(x,2\lambda) \subset \K$, as well as
\begin{equ}
\|\Gamma_{x,y} \tau\|_\alpha \le C \|\tau\| \, \|x-y\|_\s^{\deg \tau - \alpha}\;,
\end{equ}
for all $x,y \in \K$. Note that $\K \mapsto \$Z\$_\K$ is increasing by definition.
We likewise define the pseudo-metric $\$Z;\bar Z\$_\K$ with $\Pi_x$ replaced by $\Pi_x-\bar\Pi_x$ and $\Gamma_{x,y}$ replaced by $\Gamma_{x,y}-\bar\Gamma_{x,y}$.

\begin{remark}
The definition of $\$\act\$_{\K}$ and $\$\act;\act\$_{\K}$ in~\cite[Eq.~(2.16)-(2.17)]{Hairer14} is \textit{not} equivalent to our definition above, the difference being that we enforce $B_\s(x,2\lambda)\subset\K$ instead of just $\phi\in\mcB^r_{s,0}$ and $x\in\K$.
\end{remark}

With this definition at hand, we then have the following sharpening of 
\cite[Prop.~7.2]{Hairer14}.

\begin{lemma}\label{lem:reconstruct}
Let $\gamma>0$, $F\in \cD^\gamma$, $\psi \in \CB^r_{\s,0}$ and $\K \eqdef \supp \psi_x^\lambda$. Then
\begin{equ}
 \bigl| \bigl(\CR F - \Pi_x F(x)\bigr)(\psi_x^\lambda)\bigr|
 \lesssim \lambda^\gamma \|F\|_{\cD^\gamma_\K} \bigl(1 + \$Z\$_\K\bigr)\;.
\end{equ}
\end{lemma}

\begin{proof}
We claim that one can construct a countable collection 
of compactly supported functions $\phi_k$ with the following properties:
\begin{claim}
\item For every $k$ there exists $\lambda_k \in 2^{-\N}$ such that the support of $\phi_k$ 
belongs to the (scaled) ball $B_\s(x_k, \lambda_k)$ for some
$x_k \in \K$ such that also $B_\s(x_k, 2\lambda_k) \subset \K$ but $B_\s(x_k, 3\lambda_k) \not\subset \K$. 
\item One has $\sum_{k \ge 0} \phi_k(x) = 1$ for all $x \in \mathring \K$.
\item There exists $C>0$ such that 
$\sup |D^\ell \phi_k| \le C \lambda_k^{-|\ell|_\s}$ for all $|\ell| \le r$.
\item There exist $C>0$ such that the number of values $k$
such that $\lambda_k = 2^{-n}$ is bounded by $C (\lambda / \lambda_k)^{|\s|}$.
\end{claim}
The construction of this partition of unity is virtually identical to
the one given in \cite[\S II.9]{Whitney}, except that the axes are scaled 
according to $\s$. 

Since $\psi_x^\lambda \in \CC^r$ for some $r \ge \sup_i |\s_i|$ and since all of its derivatives vanish outside of $\K$, 
it follows that $|\psi_x^\lambda(y)| \lesssim \lambda^{-|\s|}  (d_\s(y,\d \K)/\lambda)$
uniformly over all $k\ge 0$.
By the first property above, points in the support of $\phi_k$ are at scaled distance 
at most $\CO(\lambda_k)$ of the boundary of $\K$.
Since we furthermore have $\lambda_k \lesssim \lambda$ for every $k$ by the first property and
since $\|D^\ell \psi_x^\lambda\|_{L^\infty} \lesssim \lambda^{-|\s|-|\ell|_\s}$ for every $\ell$ with
$|\ell| \le r$, it follows that 
\begin{equ}[e:boundDpsi]
\sup_{y \in \supp \phi_k} |D^\ell \psi_x^\lambda(y)| \lesssim \lambda^{-|\s|} (\lambda/\lambda_k)^{|\ell|_\s -1}\;,
\end{equ}
uniformly over $k \ge 1$ and $|\ell| \le r$.

We then write $\psi_x^\lambda = \sum_{k \ge 0} \psi_k$
with $\psi_k = \phi_k \psi_x^\lambda$
and note that one has the bound
\begin{equs}
\|D^\ell \psi_k\|_{L^\infty} &\lesssim \sum_{\ell_1 + \ell_2 = \ell}
\|D^{\ell_1} \phi_k\|_{L^\infty}\|D^{\ell_2} \psi_x^\lambda\|_{L^\infty(\supp \phi_k)} \\
&\lesssim \lambda_k^{-|\ell_1|_\s} \lambda^{-|\s|-|\ell_2|_\s} (\lambda / \lambda_k)^{-1+|\ell_2|_\s} \\
&= \lambda_k^{-|\ell|_\s-|\s|}(\lambda / \lambda_k)^{-1-|\s|}\;,
\end{equs}
provided that $|\ell| \le r$.

It follows that $\psi_k$ is of the form 
$\alpha_k \bar \psi_{x_k}^{\lambda_k}$ for some $\bar \psi \in \CB^r_{\s,0}$
and for $\alpha_k \lesssim (\lambda/\lambda_k)^{-|\s|-1}$. It then follows
from \cite[Lem.~6.7]{Hairer14} that 
\begin{equ}
\bigl| \bigl(\CR F - \Pi_{x_k} F(x_k)\bigr)(\psi_k)\bigr|
 \lesssim \alpha_k \lambda_k^\gamma \|F\|_{\CD^\gamma_\K} \bigl(1 + \$Z\$_{B(x_k, 2\lambda_k)}\bigr)\;.
\end{equ}
and in particular that
\begin{equ}
\bigl| \bigl(\CR F - \Pi_{x_k} F(x_k)\bigr)(\psi_k)\bigr|
 \lesssim \alpha_k \lambda^\gamma \|F\|_{\CD^\gamma_\K} \bigl(1 + \$Z\$_{\K}\bigr)\;.
\end{equ}
It then remains to sum over $k$, using the last property above to show
that $\sum_k \alpha_k$ is of order $1$.
\end{proof}

\begin{remark}
The quantity $ \bigl| \bigl(\Pi_x F(x)\bigr)(\psi_x^\lambda)\bigr|$
itself can similarly be bounded by
$\lambda^{\bar \alpha} \|F(x)\| \bigl(1 + \$Z\$_\K\bigr)$
for $\bar\alpha \le 0$ the lowest degree appearing in the regularity structure, provided
that $r$ is sufficiently large as a function of $\bar \alpha$.
The reason is that in order to be able to retrace the exact same proof as
above, we need to replace the bound \eqref{e:boundDpsi} by
a bound of order $\lambda^{-|\s|} (\lambda/\lambda_k)^{|\ell|_\s -\beta}$
for some $\beta > -\bar \alpha$.
\end{remark}

We assume henceforth that we are in the periodic setting, i.e.\ we work over $\R\times \T^d$.
For $T>0$, let $O_T\eqdef [-1,T]\times \T^d$.\label{O_tau_page_ref}

\begin{proposition}\label{prop:short_time}
The statement of~\cite[Thm~7.1]{Hairer14} holds with the improvement that $\$\act\$_{O}$ and $\$\act;\act\$_{O}$ therein are replaced by $\$\act\$_{O_T}$ and $\$\act;\act\$_{O_T}$ respectively.
\end{proposition}

\begin{proof}
The proof is very similar to that of~\cite[Thm~7.1]{Hairer14}; the only difference is that the application of~\cite[Prop.~7.2]{Hairer14}
is replaced by Lemma~\ref{lem:reconstruct}.
\end{proof}

\subsection{Modelled distributions with prescribed behaviour at \texorpdfstring{$t=0$}{t=0}}
\label{subapp:model_vanish}

Assume henceforth that we are given a regularity structure $\CT$ as
well as a model on $\R^{d+1}$ endowed with the parabolic scaling. (Other scalings can be
dealt with in exactly the same way.) We write $P = \{(t,x) \,:\, t=0\}$ for the time $0$ hyperplane
and consider the corresponding spaces $\cD^{\gamma,\eta}$ defined as in \cite[Sec.~6]{Hairer14}. 

We also recall that the reconstruction operator defined in \cite{Hairer14} is local, so that there
exists a continuous reconstruction operator $\tilde \CR \colon \cD^{\gamma,\eta} \to \CD'(\R^{d+1}\setminus P)$\label{def:tildeR} satisfying the bound above \cite[Lem.~6.7]{Hairer14}. 
One problem is that there is in 
full generality no way of canonically extending $\tilde \CR$ to an operator $\CR \colon \cD^{\gamma,\eta} \to \CD'(\R^{d+1})$. 
 \cite[Prop.~6.9]{Hairer14} provides such an extension under the assumption $\alpha\wedge \eta >-2$ (where $\alpha$ is the lowest degree of our modelled distributions) which is insufficient for our purposes in Section~\ref{sec:gauge_equivar}.
However, below we show that assuming improved behavior near $P$, one only needs $ \eta >-2$ for the extension to be unique.

We write $\bar \cD^{\gamma,\eta} = \cD^{\gamma,\eta} \cap \cD^\eta$ 
as well as $\hat \cD^{\gamma,\eta} \subset \bar \cD^{\gamma,\eta}$ for the subspace of those  \label{def:Dhat-space}
functions $f\in \bar \cD^{\gamma,\eta}$ such that $f(t,x) = 0$ for $t \le 0$.
Similarly to \cite[Lem.~6.5]{Hairer14} one can show that these are closed subspaces
of $\cD^{\gamma,\eta}$, so that we endow them with the usual norms $\$f\$_{\gamma,\eta}$.
We also note that $\bone_+ f= f$ for all $f\in\hat\cD^{\gamma,\eta}$.

Note that elements of $\hat \cD^{\gamma,\eta}$
have the improved behaviour
that its component of degree $\ell$ vanishes at the rate $|z|_P^{\eta-\ell}$ when $z$ is near $P$
if $\ell<\eta$ (rather than just being bounded as for generic elements of $\cD^{\gamma,\eta}$).

\begin{definition}
Given $f\in \cD^{\gamma,\eta}_\alpha$ for $\gamma>0$ and $\omega\in \CC^{\eta\wedge \alpha}$, we say that $\omega$ is \emph{compatible} with $f$ if $\omega(\phi) = (\tilde \CR f )(\phi)$ for all $\phi\in \CC^\infty_c (\R^{d+1}\backslash P)$.
\end{definition}

Recall as in \cite{Hairer14} that $ \CQ_{< \eta}$ is the projection to the subspace of the regularity structure of degree less than $\eta$.

\begin{theorem}\label{thm:reconstructDomain}
Let $\gamma > 0$ and $\eta \in (-2,\gamma]$. 
There exists a unique continuous linear operator
$\CR:\bar\cD^{\gamma,\eta}_\alpha\to\CC^{\eta\wedge \alpha}$ such that
$\CR f$ is compatible with $f$
and such that
\begin{equ}[e:eta-bound]
\big(\CR f-\Pi_x \CQ_{<\eta} f(x)\big)(\psi_x^\lambda)\lesssim \lambda^\eta\;,
\end{equ}
uniformly over $\lambda \le 1$ and over $\psi\in\CB$. 
\end{theorem}

\begin{proof}
The proof is virtually identical to that of \cite[Thm.~C.5]{EtienneCLT} with the boundary of the domain
$D$ there playing the role of the time-$0$ hyperplane $P$. In particular, the exponent $\eta$ in
our statement should be compared to that of $\sigma$ in \cite[Thm.~C.5]{EtienneCLT} and the exponent $-2$ 
appearing here is analogous to the exponent $-1$ there due to parabolic scaling. Finally,
the quantity $\CR_+ f_+ + \CR_- f_-$ appearing there should be replaced throughout by
$\hat \CR \CQ_{<\eta} f$ where $\hat \CR$ denotes the (continuous) reconstruction
operator for $\gamma < 0$ given in the second part of \cite[Thm.~3.10]{Hairer14}.
\end{proof}

\begin{lemma}\label{lem:multiply-barD}
For $F_i \in \bar \cD^{\gamma_i,\eta_i}_{\alpha_i}$ with $\alpha_i \le 0 < \gamma_i$ and $\eta_i \le \gamma_i$, one has
$F_1 \cdot F_2 \in \bar\cD^{\gamma, \eta}_{\alpha_1+\alpha_2}$ with $\gamma = (\alpha_1 + \gamma_2)\wedge (\alpha_2 + \gamma_1)$ and $\eta = (\alpha_1 + \eta_2)\wedge (\alpha_2 + \eta_1) \wedge (\eta_1 + \eta_2)$.
\end{lemma}

\begin{proof}
This follows from~\cite[Proposition~6.12]{Hairer14} and the definition of $\bar \cD^{\gamma,\eta}$.
\end{proof}

\begin{lemma}\label{lem:multiply-hatD}
For $F_i \in \hat \cD^{\gamma_i,\eta_i}_{\alpha_i}$ with $\alpha_i \le 0 < \gamma_i$ and $\eta_i \le \gamma_i$, one has
$F_1 \cdot F_2 \in \hat\cD^{\gamma, \eta_1+\eta_2}_{\alpha_1+\alpha_2}$.
\end{lemma}

\begin{proof}
The proof is virtually identical to that of \cite[Proposition~6.12]{Hairer14}.
\end{proof}

%

In the rest of this appendix we work on $\R\times\T^d$.
For $T>0$, we denote by $\$\act\$_{\gamma,\eta;T}$ the modelled distribution norm as defined earlier for the set $O_T\eqdef [-1,T]\times\T^d$.
As in \cite{Hairer14}  we assume henceforth that 
we have an abstract integration map $\CI$ of order $\beta$ defined on a sector $\CV$
and admissible models $Z,\bar Z$ realising a kernel $K$ for $\CI$.
We furthermore assume that $\zeta+\beta\notin\N$ for all homogeneities $\zeta$ of $\CV$,
except for the polynomial sector.

For the spaces $\bar \cD^{\gamma,\eta}$, we have the following version of Schauder estimate.
A typical situation to apply this result is
$\alpha <-2< \eta$. Note the improved exponent
$\eta+\beta$ (rather than $(\eta\wedge \alpha)+\beta$)
as well as the improved dependence on $Z$ only on the interval $[-1,T]$ (rather than $[-1,T+1]$)
in the non-anticipative case.

\begin{theorem}\label{thm:integration}
Let $\gamma>0$, and $\eta > -2$ such that $\gamma+\beta,\eta+\beta\notin \N$.
Then, there exists an operator 
$\mcb{K} \colon \bar \cD^{\gamma,\eta}(\CV) \to \bar \cD^{\gamma+\beta,\eta+\beta}$ such that 
$\CR \mcb{K} f = K* \CR f$ with the reconstruction $\CR$ from Theorem~\ref{thm:reconstructDomain}.  

If furthermore
$K$ is non-anticipative in the sense that $K(t,x) = 0$ for $t < 0$, 
then $\mcb{K}$ maps $ \hat \cD^{\gamma,\eta} (\CV)$ to $\hat\cD^{\gamma+\beta,\eta+\beta}$
and for $T\in(0,1)$ and $\kappa\geq 0$
\begin{equs}[eq:short_time]
\$\mcb K f \$_{\gamma+\beta-\kappa,\eta+\beta-\kappa;T}
&\lesssim T^{\kappa/2}\$f\$_{\gamma,\eta;T}\;,
\\
\$\mcb K f; \mcb K \bar f \$_{\gamma+\beta-\kappa,\eta+\beta-\kappa;T}
&\lesssim T^{\kappa/2}(\$f;\bar f\$_{\gamma,\eta;T} + \$Z;\bar Z\$_{O_T})\;,
\end{equs}
where $\bar f\in \hat \cD^{\gamma,\eta}_\alpha(\CV)$ is a modelled distribution
with respect to $\bar Z$.
The first proportionality constant above depends only on $\$Z\$_{O_T}$
and the second depends on $\$Z\$_{O_T}+\$\bar Z\$_{O_T}+\$f \$_{\gamma,\eta;T}+\$\bar f\$_{\gamma,\eta;T}$.
\end{theorem}

\begin{proof}
The proof of the first statement is very similar to that of \cite[Prop.~6.16]{Hairer14},
so we only point out where it differs.
As usual, one writes $K = \sum_{n \ge 0} K_n$,
where we further assume that $K_n$ annihilates polynomials of sufficiently high order for $n\geq 1$.
Regarding the bound on $\mcb{K} f$, the only point where the proof differs is when one considers
the $\mbX^\ell$-component of $\mcb{K} f$ in the regime when $2^{-n} \gtrsim |x|_P$. What is required there
is a bound on
\begin{equ}
\bigl| \bigl(\CR f - \Pi_x \CQ_{\le (|\ell|_\s - \beta)} f(x)\bigr)\bigl(D_1^\ell K_n(x,\cdot)\bigr)\bigr|\;.
\end{equ}
Instead of simply bounding the two terms separately as in \cite{Hairer14}, we bound it by
\begin{equs}
\bigl| \bigl(\CR f &- \Pi_x \CQ_{< \eta} f(x)\bigr)\bigl(D_1^\ell K_n(x,\cdot)\bigr)\bigr|
+ \sum_{\eta \le \zeta \le |\ell|_\s - \beta} 
\bigl| \bigl(\Pi_x \CQ_{\zeta} f(x)\bigr)\bigl(D_1^\ell K_n(x,\cdot)\bigr)\bigr| \\
&\lesssim 2^{(|\ell|_\s - \beta - \eta)n} + \sum_{\eta \le \zeta \le |\ell|_\s - \beta} 2^{(|\ell|_\s - \beta - \zeta)n} |x|_P^{\eta - \zeta}\;,  \label{e:use-rec-dom}
\end{equs}
where we used Theorem~\ref{thm:reconstructDomain} to bound the first term
and where the sum in the second line has a strict inequality $\zeta < |\ell|_\s - \beta$ if $n\geq 1$ due to our assumption that $K_n$ annihilates polynomials of sufficiently high degree.
After summing over $n$,
this is then bounded by $|x|_P^{(\eta + \beta - |\ell|_\s)\wedge 0}$ as required.

If we further assume that $K$ is non-anticipative, 
then we can improve the bound to $|x|_P^{\eta + \beta - |\ell|_\s}$,
which is required for $\mcb{K} f$ to take values in  $\hat\cD^{\gamma+\beta,\eta+\beta}$. To see this, 
we redo the above bound by
\begin{equs}[e:non-anti-integrate]
\bigl| \bigl(\CR f & - \bone_+ \Pi_x \CQ_{< \eta} f(x)\bigr)\bigl(D_1^\ell K_n(x,\cdot)\bigr)\bigr|
\\
&+\bigl|
 \bigl( \bone_+ \Pi_x \CQ_{< \eta} f(x) +\Pi_x \CQ_{\le |\ell|_\s - \beta} f(x)\bigr)\bigl(D_1^\ell K_n(x,\cdot)\bigr)\bigr| 
\end{equs}
where $\bone_+$ is the indicator function for positive time.
The second term is bounded by
$ \sum_{ \zeta \le |\ell|_\s - \beta} 2^{(|\ell|_\s - \beta - \zeta)n} |x|_P^{\eta - \zeta}$, and summing over the relevant values of $n$
indeed yields a bound by $|x|_P^{\eta + \beta - |\ell|_\s}$, since $|\ell|_\s - \beta - \zeta>0$ if $n\geq 1$.

Since $f\in  \hat \cD^{\gamma,\eta} $,
and $K$ is non-anticipative,
the contribution to the first term in \eqref{e:non-anti-integrate}
only comes from
$t\in (0,|x|_P) $ where $t$ is the time variable of the implicit 
argument in $D_1^\ell K_n(x,\cdot)$.
Let $m$ be such that $2^{-m} \le |x|_P \le 2^{-m+1}$. 
One can find a partition of unity $\{\phi_z\}_{z\in \Lambda}$ for some index set $\Lambda$, such that each function $\phi_z$ has a support of diameter $2^{-m}$,
and the  first term in \eqref{e:non-anti-integrate} is equal to
\[
\sum_{y\in \Lambda}\, 
\bigl| \bigl(\CR f - \bone_+ \Pi_x \CQ_{< \eta} f(x)\bigr)\bigl(\phi_y D_1^\ell K_n(x,\cdot)\bigr)\bigr|
\]
and finally the number of terms contributing to the above sum
is $(2^{-n}/|x|_P)^d$.
Since each term
here can be bounded by
$2^{(d+ |\ell|_\s)n} 2^{-m(d+\beta+\eta)} $,
the above expression is bounded by 
$ 2^{ |\ell|_\s n}2^{-m(\beta+\eta)}$.
Since $|\ell|_\s \ge 0$, summing over the relevant $n$
we get a bound by $|x|_P^{\eta + \beta - |\ell|_\s}$.

Regarding the bound on $(\mcb{K} f)(x) - \Gamma_{xy} (\mcb{K} f)(y)$ again, the only regime in which the proof differs
is when considering the $\mbX^\ell$-component in the regime $2^{-n} \gtrsim |x|_P$. 
As shown in \cite[Eq.~(5.48)]{Hairer14}, the term that needs to be bounded can be written as
\begin{equs}\label{e:relevantTerm}
\bigl(\Pi_y f(y) - \CR \CQ_{<\eta} f&\bigr)(K_{n;xy}^{\ell,\gamma}) + \sum_{\eta \le \zeta < \gamma} \bigl(\Pi_y\CQ_{\zeta} f(y)\bigr)(K_{n;xy}^{\ell,\gamma}) \\ 
&- \sum_{\zeta \le |\ell|_\s - \beta} \bigl(\Pi_x \CQ_\zeta \bigl(\Gamma_{xy} f(y) - f(x)\bigr)\bigr)\bigl(D_1^\ell K_n(x,\cdot)\bigr)\;.
\end{equs}
It follows furthermore from the Taylor remainder formula \cite[Eq.~(5.28)]{Hairer14} that in this regime 
one has $K_{n;xy}^{\ell,\gamma} = \|x-y\|_\s^{\gamma + \beta - |\ell|_\s} 2^{\gamma n} \phi_y^{2^{-n}}$ 
for some test function $\phi \in \CB$.

Theorem~\ref{thm:reconstructDomain} then implies that the first term in \eqref{e:relevantTerm} is bounded
by some multiple of  $\|x-y\|_\s^{\gamma + \beta - |\ell|_\s} 2^{(\gamma - \eta)n}$, which sums up to
$\|x-y\|_\s^{\gamma + \beta - |\ell|_\s} |x|_P^{(\eta - \gamma) \wedge 0}$, as required.
The second term in \eqref{e:relevantTerm} is bounded by some multiple of
\begin{equ}
\sum_{\eta \le \zeta < \gamma} \|x-y\|_\s^{\gamma + \beta - |\ell|_\s} 2^{(\gamma - \zeta)n} |y|_P^{\eta-\zeta}\;,
\end{equ}
which again leads to an analogous bound after summing over $n$.

The $\zeta$-summand of the last term in \eqref{e:relevantTerm} is bounded by some multiple of 
\begin{equ}
\|x-y\|_\s^{\gamma-\zeta} |x|_P^{\eta-\gamma} 2^{(|\ell|_\s - \beta - \zeta)n}\;,
\end{equ} 
which then sums up to (recalling again that $\zeta<|\ell|_\s-\beta$ by assumption if $n\geq 1$)
\begin{equ}
\|x-y\|_\s^{\gamma-\zeta} |x|_P^{\eta + \beta + \zeta -\gamma - |\ell|_\s} 
\lesssim 
\|x-y\|_\s^{\gamma + \beta -|\ell|_\s} |x|_P^{\eta -\gamma}\;, 
\end{equ}
where we used that $\zeta \le |\ell|_\s- \beta$ and $\|x-y\|_\s \lesssim |x|_P$ for the inequality.
This is again of the required form, thus concluding the proof that $\mcb{K} f$ takes values in  $\hat\cD^{\gamma+\beta,\eta+\beta}$.

It remains to show~\eqref{eq:short_time}.
The proof of this is similar to that of~\cite[Thm.~7.1]{Hairer14}
upon using~\cite[Lem.~6.5]{Hairer14} and using Lemma~\ref{lem:reconstruct} in place of~\cite[Prop.~7.2]{Hairer14}.
\end{proof}

\subsection{Schauder estimates with input distributions}
\label{subapp:input_distr}

Assume we are in the setting of Section~\ref{subapp:model_vanish}.
As in \cite[Sec.~4.5]{MateBoundary}, given a space-time distribution $\omega$
and a modelled distribution $f$,
we write $\mcb{K}^\omega f$\label{K_omega_page_ref} for 
the modelled distribution
defined as in  \cite[Sec.~5]{Hairer14} with $\CR f$ replaced by $\omega$.

Given a distribution $\omega\in\CD'(\R^{d+1})$, a compact set $\K\subset \R^{d+1}$, and $\alpha\in\R$, we let $|\omega|_{\CC^\alpha(\K)}$ be the smallest constant $C$ such that
\begin{equ}
|\omega(\phi^\lambda_x)| \leq C\lambda^{\alpha}
\end{equ}
for all $\phi \in \CB^r_{\s,0}$ and all $x \in \K$, $\lambda \in (0,1]$ such that
$B_\s(x,2\lambda) \subset \K$.

\begin{lemma}\label{lem:Schauder-input}
Fix $\gamma>0$. Let $f\in \cD^{\gamma,\eta}_\alpha(\CV)$, and $\omega\in \CC^{\eta\wedge \alpha}$ which is compatible with $f$. Set $\bar\gamma=\gamma+\beta$, $\bar\eta=(\eta\wedge \alpha)+\beta$,
which are assumed to be non-integers,
 $\bar\alpha=(\alpha+\beta)\wedge 0$ 
and $\bar\eta\wedge\bar\alpha >-2$. Then $\mcb{K}^\omega f \in \cD^{\bar\gamma,\bar\eta}_{\bar \alpha}$,
and one has $\CR\mcb{K}^\omega f = K* \omega$.

Furthermore, if $\bar f\in \cD^{\gamma,\eta}_\alpha(\CV)$ 
with respect to $\bar Z$, and $\bar\omega\in \CC^{\eta\wedge \alpha}$  is compatible with $\bar f$, then, for every compact $\K\subset\R^{d+1}$,
 \begin{equ}
 |\mcb{K}^\omega f ; \mcb{K}^{\bar\omega}\bar f|_{ \bar\gamma,\bar\eta;\K} 
 \lesssim
   |f;\bar f|_{\gamma,\eta;\bar\K}
  + \$Z;\bar Z\$_{\bar\K}
+  |\omega-\bar\omega|_{\CC^{\eta\wedge \alpha}(\bar\K)}
 \end{equ}
 locally uniformly in models, modelled distributions and space-time distributions $\omega$,
where $\bar\K$ is the $2$-fattening of $\K$.
The above bound also holds uniformly in $\eps$
for the $\eps$-dependent norms 
on models and modelled distributions defined in Section~\ref{subsec:eps_reg_structs}.

Finally, if $K$ is non-anticipative and we are in the spatially periodic setting, then the same bound holds with $f,\bar f$
replaced by $\bone_+ f,\bone_+\bar f$,
and with $\K=\bar\K = O_T$ for any $T\in(0,1)$.
\end{lemma}

\begin{proof}
All the statements except for the last one follow from~\cite[Lem.~4.12(i)]{MateBoundary}.
The last statement follows from a similar proof as that of~\cite[Lem.~5.2]{MateBoundary} upon using the improved reconstruction bound Lemma~\ref{lem:reconstruct}.
We remark that the input distributions $\omega,\bar\omega$ need only be evaluated on $D^\ell_1 K_n(z,\cdot)$ for $z=(t,x)$ with $t<\tau$.
\end{proof}

\section{Symbolic index}

We collect in this appendix commonly used symbols of the article, together
with their meaning and, if relevant, the page where they first occur.

 \begin{center}
\renewcommand{\arraystretch}{1.1}
\begin{longtable}{lll}
\toprule
Symbol & Meaning & Page\\
\midrule
\endfirsthead
\toprule
Symbol & Meaning & Page\\
\midrule
\endhead
\bottomrule
\endfoot
\bottomrule
\endlastfoot
$|\cdot|_\alpha$ & 
Extended norm on $\Omega$ & 
\pageref{norm_alpha page ref}\\
$\| \act \|_{\ell, \eps}$&
$\eps$-dependent norms on regularity structure of degree $\ell$ &
\pageref{norm ell eps page ref}\\
$\$\act\$_{\eps}$, $d_{\eps}$ &
 $\eps$-dependent seminorms and metrics on models&
\pageref{eps norms page ref}\\
$| \act |_{\gamma,\eta,\eps}$ &
 $\eps$-dependent norms  on  modelled distributions  &
\pageref{eps norms page ref}\\
$\bone_+$ &
 Indicator function of $\{(t,x)\,:\,t\geq 0\}$  &
\pageref{one_+_page_ref}\\
$\mcb{A}$ &
Target space of the jet of the noise
and the solution &
\pageref{def:mcbA}\\
$\tilde\mcA_-$, $\boldsymbol{\tilde\CA_-}$ &
Negative twisted antipode and its abstract version &
\pageref{eq:recursive_antipode_vec}\\
$\mathbf{A}^{\mcA}$ &
 Element of $\mcb{A}$ describing the
polynomial part of 
 $\mcA \in \expan$&
\pageref{eq:coherent_jet}\\
$\Cas$ &
Covariance of $\mfg$-valued white noise = quadratic Casimir&
\pageref{eq:def_of_cas}\\
$\bar C^\eps,\hat{C}^\eps$ & 
Renormalisation constants for stochastic YM equation&
\pageref{e:defConstants}\\
$C_{\sym}^{\eps},C_{\sym}$ & 
Combination of  renormalisation constants
and its limit &
\pageref{e:defConstants}\\
$E$ & A generic Banach space & \pageref{E page ref}\\
$\mbF$ & Natural filtration $\mbF=(\mcF_t)_{t\geq0}$
of the noise $\xi$ & \pageref{pageref:mbF}\\
$\mfF$&
Isomorphism classes of labelled forests
& \pageref{mfF page ref} \\
$\Func_V$ & 
The monoidal functor between   $\SSet$  and  $\Vec$  &  \pageref{def:CF-V} \\
$G$ & Compact Lie group & \pageref{G page ref}\\
$\CG_-$&
Renormalisation group&
\pageref{eq:character_group_iso}\\
$\mfg$ & Lie algebra of $G$ & \pageref{mfg page ref}\\
$\mfG^\alpha$  & $\alpha$-H{\"o}lder continuous gauge transformations & \pageref{mfG^alpha page ref}\\
$\mfG^{0,\alpha}$  & Closure of smooth functions in $\mfG^{\alpha}$ & \pageref{mfG^0,alpha page ref}\\
$\mathring{\mfG}^{0,\alpha}$  & Quotient of $\mfG^{0,\alpha}$ by kernel of action & \pageref{mrmfG^0,alpha page ref}\\
 $\expan$&
 Set of expansions with polynomial part and tree part&
\pageref{eq:coherent_jet}\\
$\Hom(\symset,\bar{\symset})$ &
Morphisms between two symmetric sets $\symset$ and $\bar{\symset}$ & \pageref{def:morphism}\\
$K^{(\eps)}$ &
Kernel assignment for gauge transformed system&
\pageref{model page ref}\\
$\mcb{K}^{\omega}$ &
Integration operator with input distribution $\omega$ &
\pageref{K_omega_page_ref}\\
$\ell_{\BPHZ}$&
BPHZ renormalisation character&
\pageref{e:defBPHZ}\\
$\mathscr{M}_{\eps}$ &
 The family of $K^{(\eps)}$-admissible models&
\pageref{model page ref}\\
$O_\tau$ &
 $[-1,\tau]\times\T^d$ for $\tau>0$ &
\pageref{O_tau_page_ref}\\
$\Omega$ & Space of additive $E$-valued functions on $\mcX$ & \pageref{Omega page ref}\\
$\Omega_\alpha$ & Banach space $\{A \in \Omega \ssep |A|_\alpha < \infty\}$ & \pageref{Omega_alpha page ref}\\
$\Omega\CB$ & $E$-valued $1$-forms with components in $\mcB$ & \pageref{Omega_CB page ref}\\
$\Omega^1_\alpha$ & Closure of smooth $E$-valued $1$-forms in $\Omega_\alpha$& \pageref{def:closure_smooth_1_forms}\\
$\mfO_\alpha$  & Space of orbits $\Omega^1_\alpha/\mfG^{0,\alpha}$ & \pageref{mfO_alpha page ref}\\
 $\CP(A)$ & Powerset of a set $A$ & \pageref{powerset page ref}\\
 $\proj^*$ & 
 Functor 
 from $\SSet_\Lab$ to $\TStruc_{\bar\Lab}$ &
 \pageref{e:linkpi3} \\
$\smooth(B) $&
Space of smooth functions from $\mcb{A}$ to $B$&
 \pageref{PB page ref}\\
$\mathring{\mcb{Q}}$ (resp. $\mcb{Q}$) &
The set of choices of RHS of SPDE (resp. obeying $R$) &
\pageref{e:Qcirc} \\
$\rho$ & Distance function on $\mcX$ & \pageref{rho page ref}\\
$R$ & Subcritical, complete rule & \pageref{rule page ref}\\
$\symset$ & A generic symmetric set & \pageref{def:sym-typed-set}\\
$\SSet_{\mfL}$ & The category of symmetric sets with types $\mfL$ & \pageref{def:SSet} \\
 $\TStruc$ & 
Category of typed structures,
 objects are 
 $\prod_{\alpha \in \CA} \symset_{\alpha}$ &
 \pageref{def:typed_struct} \\
 $\scal{\tau}$ &
 The symmetric set for a  labelled rooted tree $\tau$
 & \pageref{def:scal-tau}\\
$\mfT$&
Isomorphism classes of labelled trees
& \pageref{mfT page ref} \\
$\mfT(R)$&
Trees strongly conforming to $R$
& \pageref{mfT(R) page ref} \\
$\mfT_{-}(R)$ &
Negative degree unplanted trees in $\mfT(R)$ with $\mfn(\rho)=0$
& \pageref{e:def-mfT-}\\
$\ST,\SF$&
Our abstract regularity structures&
\pageref{STSF page ref}\\
$\CT,\CF$  &
Vector spaces for concrete regularity structure
& \pageref{e:defCTtau}\\
$V^{\otimes \symset} $ &
Tensor product determined by the symmetric set $\symset$&
\pageref{e:def-V-tensor-symset}\\
$\mcX$  & Set of line segments & \pageref{mcX page ref}\\
$\Xi_i$ &
Symbol for noise, defined as $\mcb{I}_{(\mfl_i,0)}({\bf 1})$  for $\mfl_i \in \mfL_-$ &
\pageref{Xi page ref}\\
$\bUpsilon,\bbUpsilon$ &
Maps describing coherence of expansions &
\pageref{eq:Upsilon_with_sym}\\
$\Upsilon,\bar \Upsilon$ &
Unnormalised coherence maps &
\pageref{e:def-Upsilon-added}
\end{longtable}
 \end{center}

\endappendix
\bibliographystyle{./Martin}
\bibliography{./refs}

\def\cprime{$'$} \def\polhk#1{\setbox0=\hbox{#1}{\ooalign{\hidewidth
  \lower1.5ex\hbox{`}\hidewidth\crcr\unhbox0}}}
\begin{thebibliography}{BCCH21}
\def\myhref#1#2{\href{#2}{\nolinkurl{#1}}}

\bibitem[AB83]{AB83}
\textsc{M.~F. Atiyah} and \textsc{R.~Bott}.
\newblock The {Y}ang-{M}ills equations over {R}iemann surfaces.
\newblock \emph{Philos. Trans. Roy. Soc. London Ser. A} \textbf{308}, no. 1505,
  (1983), 523--615.
\newblock
  \myhref{doi:10.1098/rsta.1983.0017}{https://dx.doi.org/10.1098/rsta.1983.0017}.

\bibitem[AK20]{AK20}
\textsc{S.~Albeverio} and \textsc{S.~Kusuoka}.
\newblock The invariant measure and the flow associated to the
  {$\Phi^4_3$}-quantum field model.
\newblock \emph{Ann. Sc. Norm. Super. Pisa Cl. Sci. (5)} \textbf{20}, no.~4,
  (2020), 1359--1427.
\newblock
  \myhref{doi:10.2422/2036-2145.201809_008}{https://dx.doi.org/10.2422/2036-2145.201809_008}.

\bibitem[BCCH21]{BCCH21}
\textsc{Y.~Bruned}, \textsc{A.~Chandra}, \textsc{I.~Chevyrev}, and
  \textsc{M.~Hairer}.
\newblock Renormalising {SPDE}s in regularity structures.
\newblock \emph{J. Eur. Math. Soc. (JEMS)} \textbf{23}, no.~3, (2021),
  869--947.
\newblock \myhref{doi:10.4171/jems/1025}{https://dx.doi.org/10.4171/jems/1025}.

\bibitem[BCFP19]{BCFP19}
\textsc{Y.~Bruned}, \textsc{I.~Chevyrev}, \textsc{P.~K. Friz}, and
  \textsc{R.~Prei\ss}.
\newblock A rough path perspective on renormalization.
\newblock \emph{J. Funct. Anal.} \textbf{277}, no.~11, (2019), 108283, 60.
\newblock
  \myhref{doi:10.1016/j.jfa.2019.108283}{https://dx.doi.org/10.1016/j.jfa.2019.108283}.

\bibitem[BG20]{BG18}
\textsc{N.~Barashkov} and \textsc{M.~Gubinelli}.
\newblock A variational method for {$\Phi^4_3$}.
\newblock \emph{Duke Math. J.} \textbf{169}, no.~17, (2020), 3339--3415.
\newblock
  \myhref{doi:10.1215/00127094-2020-0029}{https://dx.doi.org/10.1215/00127094-2020-0029}.

\bibitem[BHST87]{BHST87II}
\textsc{Z.~Bern}, \textsc{M.~B. Halpern}, \textsc{L.~Sadun}, and
  \textsc{C.~Taubes}.
\newblock Continuum regularization of quantum field theory. {II}. {G}auge
  theory.
\newblock \emph{Nuclear Phys. B} \textbf{284}, no.~1, (1987), 35--91.
\newblock
  \myhref{doi:10.1016/0550-3213(87)90026-5}{https://dx.doi.org/10.1016/0550-3213(87)90026-5}.

\bibitem[BHZ19]{BHZ19}
\textsc{Y.~Bruned}, \textsc{M.~Hairer}, and \textsc{L.~Zambotti}.
\newblock Algebraic renormalisation of regularity structures.
\newblock \emph{Invent. Math.} \textbf{215}, no.~3, (2019), 1039--1156.
\newblock
  \myhref{doi:10.1007/s00222-018-0841-x}{https://dx.doi.org/10.1007/s00222-018-0841-x}.

\bibitem[Bog07]{Bogachev07}
\textsc{V.~I. Bogachev}.
\newblock \emph{Measure theory. {V}ol. {I}, {II}}.
\newblock Springer-Verlag, Berlin, 2007,  Vol. I: xviii+500 pp., Vol. II:
  xiv+575.
\newblock
  \myhref{doi:10.1007/978-3-540-34514-5}{https://dx.doi.org/10.1007/978-3-540-34514-5}.

\bibitem[Bou94]{Bourgain94}
\textsc{J.~Bourgain}.
\newblock Periodic nonlinear {S}chr\"{o}dinger equation and invariant measures.
\newblock \emph{Comm. Math. Phys.} \textbf{166}, no.~1, (1994), 1--26.
\newblock
  \myhref{doi:10.1007/BF02099299}{https://dx.doi.org/10.1007/BF02099299}.

\bibitem[CCHS22]{CCHSPrep}
\textsc{A.~{Chandra}}, \textsc{I.~{Chevyrev}}, \textsc{M.~{Hairer}}, and
  \textsc{H.~{Shen}}.
\newblock {Stochastic quantisation of Yang-Mills-Higgs in 3D}.
\newblock \emph{arXiv e-prints} (2022).
\newblock \myhref{arXiv:2201.03487}{https://arxiv.org/abs/2201.03487}.

\bibitem[CG13]{CG13}
\textsc{N.~Charalambous} and \textsc{L.~Gross}.
\newblock The {Y}ang-{M}ills heat semigroup on three-manifolds with boundary.
\newblock \emph{Comm. Math. Phys.} \textbf{317}, no.~3, (2013), 727--785.
\newblock
  \myhref{doi:10.1007/s00220-012-1558-0}{https://dx.doi.org/10.1007/s00220-012-1558-0}.

\bibitem[CH16]{CH16}
\textsc{A.~{Chandra}} and \textsc{M.~{Hairer}}.
\newblock {An analytic BPHZ theorem for regularity structures}.
\newblock \emph{ArXiv e-prints} (2016).
\newblock \myhref{arXiv:1612.08138}{https://arxiv.org/abs/1612.08138}.

\bibitem[Cha19]{Chatterjee18}
\textsc{S.~Chatterjee}.
\newblock Yang-{M}ills for probabilists.
\newblock In \emph{Probability and analysis in interacting physical systems},
  vol. 283 of \emph{Springer Proc. Math. Stat.},  1--16. Springer, Cham, 2019.
\newblock
  \myhref{doi:10.1007/978-3-030-15338-0_1}{https://dx.doi.org/10.1007/978-3-030-15338-0_1}.

\bibitem[Che19]{Chevyrev18YM}
\textsc{I.~Chevyrev}.
\newblock Yang-{M}ills measure on the two-dimensional torus as a random
  distribution.
\newblock \emph{Comm. Math. Phys.} \textbf{372}, no.~3, (2019), 1027--1058.
\newblock
  \myhref{doi:10.1007/s00220-019-03567-5}{https://dx.doi.org/10.1007/s00220-019-03567-5}.

\bibitem[CW16]{TomTrees}
\textsc{T.~{Cass}} and \textsc{M.~P. {Weidner}}.
\newblock {Tree algebras over topological vector spaces in rough path theory}.
\newblock \emph{ArXiv e-prints} (2016).
\newblock \myhref{arXiv:1604.07352}{https://arxiv.org/abs/1604.07352}.

\bibitem[DeT83]{deturck83}
\textsc{D.~M. DeTurck}.
\newblock Deforming metrics in the direction of their {R}icci tensors.
\newblock \emph{J. Differential Geom.} \textbf{18}, no.~1, (1983), 157--162.
\newblock
  \myhref{doi:10.4310/JDG/1214509286}{https://dx.doi.org/10.4310/JDG/1214509286}.

\bibitem[DH87]{DH87}
\textsc{P.~H. Damgaard} and \textsc{H.~H\"uffel}.
\newblock Stochastic quantization.
\newblock \emph{Phys. Rep.} \textbf{152}, no. 5-6, (1987), 227--398.
\newblock
  \myhref{doi:10.1016/0370-1573(87)90144-X}{https://dx.doi.org/10.1016/0370-1573(87)90144-X}.

\bibitem[DK90]{DK90}
\textsc{S.~K. Donaldson} and \textsc{P.~B. Kronheimer}.
\newblock \emph{The geometry of four-manifolds}.
\newblock Oxford Mathematical Monographs. The Clarendon Press, Oxford
  University Press, New York, 1990,  x+440.
\newblock Oxford Science Publications.

\bibitem[Don85]{Donaldson}
\textsc{S.~K. Donaldson}.
\newblock Anti self-dual {Y}ang-{M}ills connections over complex algebraic
  surfaces and stable vector bundles.
\newblock \emph{Proc. London Math. Soc. (3)} \textbf{50}, no.~1, (1985), 1--26.
\newblock
  \myhref{doi:10.1112/plms/s3-50.1.1}{https://dx.doi.org/10.1112/plms/s3-50.1.1}.

\bibitem[Dri89]{Driver89}
\textsc{B.~K. Driver}.
\newblock Y{M{${}_2$}}: continuum expectations, lattice convergence, and
  lassos.
\newblock \emph{Comm. Math. Phys.} \textbf{123}, no.~4, (1989), 575--616.
\newblock
  \myhref{doi:10.1007/BF01218586}{https://dx.doi.org/10.1007/BF01218586}.

\bibitem[FH20]{FrizHairer}
\textsc{P.~K. Friz} and \textsc{M.~Hairer}.
\newblock \emph{A course on rough paths}.
\newblock Universitext. Springer, Cham, [2020] \copyright 2020,  xvi+346.
\newblock With an introduction to regularity structures, Second edition.
\newblock
  \myhref{doi:10.1007/978-3-030-41556-3}{https://dx.doi.org/10.1007/978-3-030-41556-3}.

\bibitem[FV10]{FV10}
\textsc{P.~K. Friz} and \textsc{N.~B. Victoir}.
\newblock \emph{Multidimensional stochastic processes as rough paths}, vol. 120
  of \emph{Cambridge Studies in Advanced Mathematics}.
\newblock Cambridge University Press, Cambridge, 2010,  xiv+656.
\newblock Theory and applications.
\newblock
  \myhref{doi:10.1017/CBO9780511845079}{https://dx.doi.org/10.1017/CBO9780511845079}.

\bibitem[GH19a]{MateBoundary}
\textsc{M.~Gerencs\'{e}r} and \textsc{M.~Hairer}.
\newblock Singular {SPDE}s in domains with boundaries.
\newblock \emph{Probab. Theory Related Fields} \textbf{173}, no. 3-4, (2019),
  697--758.
\newblock
  \myhref{doi:10.1007/s00440-018-0841-1}{https://dx.doi.org/10.1007/s00440-018-0841-1}.

\bibitem[GH19b]{Mate2}
\textsc{M.~Gerencs\'{e}r} and \textsc{M.~Hairer}.
\newblock A solution theory for quasilinear singular {SPDE}s.
\newblock \emph{Comm. Pure Appl. Math.} \textbf{72}, no.~9, (2019), 1983--2005.
\newblock \myhref{doi:10.1002/cpa.21816}{https://dx.doi.org/10.1002/cpa.21816}.

\bibitem[GH21]{GH21}
\textsc{M.~Gubinelli} and \textsc{M.~Hofmanov\'{a}}.
\newblock A {PDE} construction of the {E}uclidean {$\phi_3^4$} quantum field
  theory.
\newblock \emph{Comm. Math. Phys.} \textbf{384}, no.~1, (2021), 1--75.
\newblock
  \myhref{doi:10.1007/s00220-021-04022-0}{https://dx.doi.org/10.1007/s00220-021-04022-0}.

\bibitem[GHM22]{AndrisKonstantin}
\textsc{A.~Gerasimovics}, \textsc{M.~Hairer}, and \textsc{K.~Matetski}.
\newblock Directed mean curvature flow in noisy environment.
\newblock \emph{arXiv e-prints} (2022).
\newblock \myhref{arXiv:2201.08807}{https://arxiv.org/abs/2201.08807}.

\bibitem[GIP15]{GIP15}
\textsc{M.~Gubinelli}, \textsc{P.~Imkeller}, and \textsc{N.~Perkowski}.
\newblock Paracontrolled distributions and singular {PDE}s.
\newblock \emph{Forum Math. Pi} \textbf{3}, (2015), e6, 75.
\newblock \myhref{arXiv:1210.2684v3}{https://arxiv.org/abs/1210.2684v3}.
\newblock
  \myhref{doi:10.1017/fmp.2015.2}{https://dx.doi.org/10.1017/fmp.2015.2}.

\bibitem[Gro85]{Gross85}
\textsc{L.~Gross}.
\newblock A {P}oincar{\'e} lemma for connection forms.
\newblock \emph{J. Funct. Anal.} \textbf{63}, no.~1, (1985), 1--46.
\newblock
  \myhref{doi:10.1016/0022-1236(85)90096-5}{https://dx.doi.org/10.1016/0022-1236(85)90096-5}.

\bibitem[{Gro}17]{GrossSingular}
\textsc{L.~{Gross}}.
\newblock {Stability of the Yang-Mills heat equation for finite action}.
\newblock \emph{ArXiv e-prints} (2017).
\newblock \myhref{arXiv:1711.00114}{https://arxiv.org/abs/1711.00114}.

\bibitem[Gro22]{Gross2}
\textsc{L.~Gross}.
\newblock The {Y}ang-{M}ills heat equation with finite action in three
  dimensions.
\newblock \emph{Mem. Amer. Math. Soc.} \textbf{275}, no. 1349, (2022), v+111.
\newblock \myhref{arXiv:math/0101239}{https://arxiv.org/abs/math/0101239}.
\newblock \myhref{doi:10.1090/memo/1349}{https://dx.doi.org/10.1090/memo/1349}.

\bibitem[Hai14]{Hairer14}
\textsc{M.~Hairer}.
\newblock A theory of regularity structures.
\newblock \emph{Invent. Math.} \textbf{198}, no.~2, (2014), 269--504.
\newblock \myhref{arXiv:1303.5113}{https://arxiv.org/abs/1303.5113}.
\newblock
  \myhref{doi:10.1007/s00222-014-0505-4}{https://dx.doi.org/10.1007/s00222-014-0505-4}.

\bibitem[HM18a]{HM18}
\textsc{M.~Hairer} and \textsc{K.~Matetski}.
\newblock Discretisations of rough stochastic {PDE}s.
\newblock \emph{Ann. Probab.} \textbf{46}, no.~3, (2018), 1651--1709.
\newblock
  \myhref{doi:10.1214/17-AOP1212}{https://dx.doi.org/10.1214/17-AOP1212}.

\bibitem[HM18b]{HM16}
\textsc{M.~Hairer} and \textsc{J.~Mattingly}.
\newblock The strong {F}eller property for singular stochastic {PDE}s.
\newblock \emph{Ann. Inst. Henri Poincar\'{e} Probab. Stat.} \textbf{54},
  no.~3, (2018), 1314--1340.
\newblock
  \myhref{doi:10.1214/17-AIHP840}{https://dx.doi.org/10.1214/17-AIHP840}.

\bibitem[HP21]{EtienneCLT}
\textsc{M.~Hairer} and \textsc{E.~Pardoux}.
\newblock Fluctuations around a homogenised semilinear random {PDE}.
\newblock \emph{Arch. Ration. Mech. Anal.} \textbf{239}, no.~1, (2021),
  151--217.
\newblock
  \myhref{doi:10.1007/s00205-020-01574-8}{https://dx.doi.org/10.1007/s00205-020-01574-8}.

\bibitem[HS22]{HS19}
\textsc{M.~Hairer} and \textsc{P.~Sch\"{o}nbauer}.
\newblock The support of singular stochastic partial differential equations.
\newblock \emph{Forum Math. Pi} \textbf{10}, (2022), Paper No. e1, 127.
\newblock \myhref{arXiv:1909.05526}{https://arxiv.org/abs/1909.05526}.
\newblock
  \myhref{doi:10.1017/fmp.2021.18}{https://dx.doi.org/10.1017/fmp.2021.18}.

\bibitem[JW06]{JW06}
\textsc{A.~Jaffe} and \textsc{E.~Witten}.
\newblock Quantum {Y}ang-{M}ills theory.
\newblock In \emph{The millennium prize problems},  129--152. Clay Math. Inst.,
  Cambridge, MA, 2006.

\bibitem[Kec95]{KechrisSet}
\textsc{A.~S. Kechris}.
\newblock \emph{Classical descriptive set theory}, vol. 156 of \emph{Graduate
  Texts in Mathematics}.
\newblock Springer-Verlag, New York, 1995,  xviii+402.
\newblock
  \myhref{doi:10.1007/978-1-4612-4190-4}{https://dx.doi.org/10.1007/978-1-4612-4190-4}.

\bibitem[Kna02]{Knapp02}
\textsc{A.~W. Knapp}.
\newblock \emph{Lie groups beyond an introduction}, vol. 140 of \emph{Progress
  in Mathematics}.
\newblock Birkh\"{a}user Boston, Inc., Boston, MA, second ed., 2002,
  xviii+812.
\newblock
  \myhref{doi:10.1007/978-1-4757-2453-0}{https://dx.doi.org/10.1007/978-1-4757-2453-0}.

\bibitem[L{\'e}v03]{Levy03}
\textsc{T.~L{\'e}vy}.
\newblock Yang-{M}ills measure on compact surfaces.
\newblock \emph{Mem. Amer. Math. Soc.} \textbf{166}, no. 790, (2003), xiv+122.
\newblock \myhref{arXiv:math/0101239}{https://arxiv.org/abs/math/0101239}.
\newblock \myhref{doi:10.1090/memo/0790}{https://dx.doi.org/10.1090/memo/0790}.

\bibitem[L{\'e}v06]{Levy06}
\textsc{T.~L{\'e}vy}.
\newblock Discrete and continuous {Y}ang-{M}ills measure for non-trivial
  bundles over compact surfaces.
\newblock \emph{Probab. Theory Related Fields} \textbf{136}, no.~2, (2006),
  171--202.
\newblock
  \myhref{doi:10.1007/s00440-005-0478-8}{https://dx.doi.org/10.1007/s00440-005-0478-8}.

\bibitem[LN06]{LevyNorris06}
\textsc{T.~L{\'e}vy} and \textsc{J.~R. Norris}.
\newblock Large deviations for the {Y}ang-{M}ills measure on a compact surface.
\newblock \emph{Comm. Math. Phys.} \textbf{261}, no.~2, (2006), 405--450.
\newblock
  \myhref{doi:10.1007/s00220-005-1450-2}{https://dx.doi.org/10.1007/s00220-005-1450-2}.

\bibitem[Lyo94]{Lyons94}
\textsc{T.~Lyons}.
\newblock Differential equations driven by rough signals. {I}. {A}n extension
  of an inequality of {L}. {C}. {Y}oung.
\newblock \emph{Math. Res. Lett.} \textbf{1}, no.~4, (1994), 451--464.
\newblock
  \myhref{doi:10.4310/MRL.1994.v1.n4.a5}{https://dx.doi.org/10.4310/MRL.1994.v1.n4.a5}.

\bibitem[Man89]{Mandelkern89}
\textsc{M.~Mandelkern}.
\newblock Metrization of the one-point compactification.
\newblock \emph{Proc. Amer. Math. Soc.} \textbf{107}, no.~4, (1989),
  1111--1115.
\newblock \myhref{doi:10.2307/2047675}{https://dx.doi.org/10.2307/2047675}.

\bibitem[Mei75]{Ears}
\textsc{G.~H. Meisters}.
\newblock Polygons have ears.
\newblock \emph{Amer. Math. Monthly} \textbf{82}, (1975), 648--651.
\newblock \myhref{doi:10.2307/2319703}{https://dx.doi.org/10.2307/2319703}.

\bibitem[MV81]{Mitter}
\textsc{P.~K. Mitter} and \textsc{C.-M. Viallet}.
\newblock On the bundle of connections and the gauge orbit manifold in
  {Y}ang-{M}ills theory.
\newblock \emph{Comm. Math. Phys.} \textbf{79}, no.~4, (1981), 457--472.
\newblock
  \myhref{doi:10.1007/BF01209307}{https://dx.doi.org/10.1007/BF01209307}.

\bibitem[MW17a]{MW17Phi43}
\textsc{J.-C. Mourrat} and \textsc{H.~Weber}.
\newblock The dynamic {$\Phi^4_3$} model comes down from infinity.
\newblock \emph{Comm. Math. Phys.} \textbf{356}, no.~3, (2017), 673--753.
\newblock \myhref{arXiv:1601.01234}{https://arxiv.org/abs/1601.01234}.
\newblock
  \myhref{doi:10.1007/s00220-017-2997-4}{https://dx.doi.org/10.1007/s00220-017-2997-4}.

\bibitem[MW17b]{MW17Phi42}
\textsc{J.-C. Mourrat} and \textsc{H.~Weber}.
\newblock Global well-posedness of the dynamic {$\Phi^4$} model in the plane.
\newblock \emph{Ann. Probab.} \textbf{45}, no.~4, (2017), 2398--2476.
\newblock
  \myhref{doi:10.1214/16-AOP1116}{https://dx.doi.org/10.1214/16-AOP1116}.

\bibitem[MW20]{MoinatWeber18}
\textsc{A.~Moinat} and \textsc{H.~Weber}.
\newblock Space-time localisation for the dynamic {$\Phi^4_3$} model.
\newblock \emph{Comm. Pure Appl. Math.} \textbf{73}, no.~12, (2020),
  2519--2555.
\newblock \myhref{doi:10.1002/cpa.21925}{https://dx.doi.org/10.1002/cpa.21925}.

\bibitem[PW81]{ParisiWu}
\textsc{G.~Parisi} and \textsc{Y.~S. Wu}.
\newblock Perturbation theory without gauge fixing.
\newblock \emph{Sci. Sinica} \textbf{24}, no.~4, (1981), 483--496.
\newblock
  \myhref{doi:10.1360/ya1981-24-4-483}{https://dx.doi.org/10.1360/ya1981-24-4-483}.

\bibitem[Rad92]{Rade92}
\textsc{J.~Rade}.
\newblock On the {Y}ang-{M}ills heat equation in two and three dimensions.
\newblock \emph{J. Reine Angew. Math.} \textbf{431}, (1992), 123--163.
\newblock
  \myhref{doi:10.1515/crll.1992.431.123}{https://dx.doi.org/10.1515/crll.1992.431.123}.

\bibitem[{Sch}18]{PhilippMall}
\textsc{P.~{Sch{\"o}nbauer}}.
\newblock {Malliavin calculus and density for singular stochastic partial
  differential equations}.
\newblock \emph{ArXiv e-prints} (2018).
\newblock \myhref{arXiv:1809.03570}{https://arxiv.org/abs/1809.03570}.

\bibitem[Sen92]{Sengupta92}
\textsc{A.~Sengupta}.
\newblock The {Y}ang-{M}ills measure for {$S^2$}.
\newblock \emph{J. Funct. Anal.} \textbf{108}, no.~2, (1992), 231--273.
\newblock
  \myhref{doi:10.1016/0022-1236(92)90025-E}{https://dx.doi.org/10.1016/0022-1236(92)90025-E}.

\bibitem[Sen97]{Sengupta97}
\textsc{A.~Sengupta}.
\newblock Gauge theory on compact surfaces.
\newblock \emph{Mem. Amer. Math. Soc.} \textbf{126}, no. 600, (1997), viii+85.
\newblock \myhref{doi:10.1090/memo/0600}{https://dx.doi.org/10.1090/memo/0600}.

\bibitem[She21]{Shen18}
\textsc{H.~Shen}.
\newblock Stochastic quantization of an {A}belian gauge theory.
\newblock \emph{Comm. Math. Phys.} \textbf{384}, no.~3, (2021), 1445--1512.
\newblock
  \myhref{doi:10.1007/s00220-021-04114-x}{https://dx.doi.org/10.1007/s00220-021-04114-x}.

\bibitem[ST12]{MonCat}
\textsc{S.~Stolz} and \textsc{P.~Teichner}.
\newblock Traces in monoidal categories.
\newblock \emph{Trans. Amer. Math. Soc.} \textbf{364}, no.~8, (2012),
  4425--4464.
\newblock
  \myhref{doi:10.1090/S0002-9947-2012-05615-7}{https://dx.doi.org/10.1090/S0002-9947-2012-05615-7}.

\bibitem[Whi34]{Whitney}
\textsc{H.~Whitney}.
\newblock Analytic extensions of differentiable functions defined in closed
  sets.
\newblock \emph{Trans. Amer. Math. Soc.} \textbf{36}, no.~1, (1934), 63--89.
\newblock \myhref{doi:10.2307/1989708}{https://dx.doi.org/10.2307/1989708}.

\bibitem[You36]{Young}
\textsc{L.~C. Young}.
\newblock An inequality of the {H}\"older type, connected with {S}tieltjes
  integration.
\newblock \emph{Acta Math.} \textbf{67}, no.~1, (1936), 251--282.
\newblock
  \myhref{doi:10.1007/BF02401743}{https://dx.doi.org/10.1007/BF02401743}.

\bibitem[Zwa81]{zwanziger81}
\textsc{D.~Zwanziger}.
\newblock Covariant quantization of gauge fields without {G}ribov ambiguity.
\newblock \emph{Nuclear Phys. B} \textbf{192}, no.~1, (1981), 259--269.
\newblock
  \myhref{doi:10.1016/0550-3213(81)90202-9}{https://dx.doi.org/10.1016/0550-3213(81)90202-9}.

\end{thebibliography}

\end{document}